\documentclass[reqno,english,11pt]{amsart}

\usepackage[margin=1in]{geometry}

\usepackage[latin9]{inputenc}
\usepackage{amsbsy}
\usepackage{amstext}
\usepackage{amsthm}
\usepackage{amssymb}
\usepackage{bm}
\usepackage{xcolor}
\usepackage{amsfonts,euscript,mathrsfs,color,amsmath,latexsym}
\numberwithin{equation}{section}
\numberwithin{figure}{section}
\theoremstyle{plain}
\newtheorem{thm}{\protect\theoremname}[section]
\theoremstyle{remark}
\newtheorem{rem}[thm]{\protect\remarkname}
\theoremstyle{definition}
\newtheorem{defn}{\protect\definitionname}[section]

\newtheorem{lemma}{Lemma}[section]
\newtheorem{proposition}{Proposition}[section]

\newtheorem{cor}{Corollary}[section]

\usepackage{hyperref}

\newcommand{\norm}[1]{\left\lVert#1\right\rVert}

\makeatother

\usepackage{babel}
\providecommand{\definitionname}{Definition}
\providecommand{\remarkname}{Remark}
\providecommand{\theoremname}{Theorem}

\begin{document}
	
	\title	[Modulational instability]{Nonlinear modulational instabililty of the Stokes waves \\in 2d full water waves }

	\author{Gong Chen}
	\author{Qingtang Su}
	\address[Chen]{Fields Institute for Research in Mathematical Sciences, 222 College Street Toronto, Ontario M5S 2E4, Canada }
\email{gc@math.toronto.edu, gc@math.uchicago.edu}
\address[Su]{Department of Mathematics, University of Southern California, Los Angeless, CA, 90089, USA}
\email{qingtang@usc.edu}
\date{\today}
	\begin{abstract}
The well-known Stokes waves refer to periodic traveling waves under the gravity at the free surface of a two dimensional full water wave system. In this paper,  we prove that small-amplitude Stokes waves with infinite depth
are nonlinearly unstable under long-wave perturbations. Our approach is based on the modulational approximation of the water wave system and the instability mechanism of the focusing cubic nonlinear Schr\"odinger equation. 
	\end{abstract}
	\maketitle
	
	\setcounter{tocdepth}{1}

\tableofcontents

	\section{Introduction}
In this paper, we establish the nonlinear modulational instability of the small-amplitude Stokes waves  under long-wave perturbations in the context of 2d  full water waves with infinite depth. The famous Stokes wave  refers to  a periodic steady wave traveling at a constant speed, which was first studied by Stokes in 1847 \cite{stokes1880theory}. The existence of Stokes waves was rigorously proved in the 1920s for the small-amplitude cases \cite{nekrasov1921steady,Levi-Civita,struik1926determination},  and in the early 1960s for the large-amplitude settings \cite{krasovskii1960theory,krasovskii1962theory}. These periodic traveling waves are of crucial importance in both theoretical and practical studies of water waves.

\subsection{Modulational instability}

 The modulational instability,  which  is also known as the Benjamin-Feir or the sideband instability, is a very important instability mechanism in a diverse range of dispersive and fluid models. Roughly speaking, this is a phenomenon whereby deviations from a periodic waveform are reinforced by the nonlinearity, leading to the generation of spectral-sidebands and the eventual breakup of the waveform into a train of pulses. This instability mechanism has been wildly observed in experiments and in nature, such as water waves and their asymptotic models.  In 1967, by Benjamin and Feir \cite{benjamin1967disintegration,benjamin1967instability}, this phenomenon was first discovered for periodic surface gravity waves, i.e. Stokes waves, on the deep water. This is the context of our main interest in this paper. 



\smallskip

The modulational instability also exists in various dispersive equations.
The literature on this topic is extensive and without trying to be exhaustive, we mention  the work by  Whitham \cite{whitham1967non}, Benny and Newell \cite{benney1967propagation}, Ostrovsky \cite{ostrovskii1967propagation}, Zakharov \cite{zakharov1968stability}, Lighthill \cite{lighthill1970contributions}.
We also refer interested reader to the excellent survey by Ostrovsky and Zakharov \cite{zakharov2009modulation} for more details on the history and
physical applications of modulational instability. Moving beyond the linear level, recently the nonlinear modulational instability for a class of dispersive models was proved by Jin, Liao, and Lin in \cite{jin2019nonlinear}.

\smallskip	
	
Returning to the water wave problem, nevertheless, the rigorous proof of the linear modulational instability (spectral instability) for the full water waves was quite recent. In the 1990s, Bridges and Mielke \cite{bridges1995proof} were able to prove the spectral modulational instability for the finite-depth water waves linearized near a small-amplitude Stokes wave. Under long-wave perturbations, i.e. frequencies near zero,  recently, Nguyen and Strauss in \cite{nguyen2020proof} proved the spectral modulational instability of the Stokes waves in infinite depth case. See also \cite{hur2020benjamin} for a simplified proof by Hur. The nonlinear instability of the full water waves remains open and is our main result in this paper. The quasilinear feature and the nonlocality of water wave systems make the nonlinear analysis here exceedingly difficult.

\subsection{Water wave system}
Now we introduce the full water wave system. Consider the motion of an inviscid and incompressible ideal fluid with a free surface in two space dimensions (that is, the interface separating the fluid and the vacuum is one dimensional). We refer such fluid as \emph{water waves}. For simplicity, we consider the infinite depth case, that is, without a finite bottom. Denote the fluid region by $\Omega(t)$ and the free interface by $\Sigma(t)$. The equations of motion are Euler's equations, coupled to the motion of the boundary, and with  the vanishing boundary condition for the pressure. It is assumed that the fluid region is below the air region. Assume that the density of the fluid is $1$, and the gravitational field is normalized as $-(0,1)$. In two dimensions, if the surface tension is zero, then the motion of the fluid is described by
	\begin{equation}\label{euler}
	\begin{cases}
	v_t+v\cdot\nabla v=-\nabla P-(0,1)\quad \quad & \text{on} ~\Omega(t),\quad t\geq 0  \\
	\text{div}~v=0,\quad \text{curl}~v=0,  & \text{on}~\Omega(t), \quad  t\geq 0\\
	P\Big|_{\Sigma(t)}\equiv 0,  & t\geq 0\\
	(1,v)~\text{ is tangent to the free surface } (t,\Sigma(t)).
	\end{cases}
	\end{equation}
	Here, $v$ is the fluid velocity, and $P$ is the pressure. We shall consider the water waves such that $|v(z,t)|\rightarrow 0$ as $\Im\{z\}\rightarrow -\infty$.

\smallskip

This system, along with many variants and generalizations, has been extensively studied in the literature. The so-called Taylor sign condition (also referred as Rayleigh-Taylor sign condition in many literature) $-\frac{\partial P}{\partial\Vec{n}}>0$ on the pressure is an important stability condition for the water waves problem. If the Taylor sign condition fails, the system is, in general, unstable, see, for example, \cite{beale1993growth, birkhoff1962helmholtz, taylor1950instability,ebin1987equations, su2020transition}. In the irrotational case without a bottom, the validity of the Taylor sign condition was shown by Wu \cite{Wu1997, Wu1999}, and was the key to obtain the first local-in-time existence results for large data in Sobolev spaces. In the case of non-trivial vorticity or with a bottom, the Taylor sign condition can fail and the sign condition has to be assumed  for the initial data.  In the irrotational case, Nalimov \cite{Nalimov}, Yosihara \cite{Yosihara} and Craig \cite{Craig} proved local well-posedness for 2d water waves equation for small initial data. In S. Wu's breakthrough works \cite{Wu1997, Wu1999} she proved the local-in-time well-posedness without smallness assumptions. Since then, a lot of interesting local well-posedness results were obtained, see for example \cite{alazard2014cauchy, ambrose2005zero, christodoulou2000motion, coutand2007well, iguchi2001well, lannes2005well, lindblad2005well, ogawa2002free, shatah2006geometry, zhang2008free, ai2019low, ai2020low, miao2020well, ai2019two}, and the references therein. See also \cite{Wu1, Wu2, Wu3, wu2019wellposedness} for water waves with non-smooth interfaces. For the formation of splash singularities, see for example \cite{castro2012finite, castro2013finite, coutand2014finite, coutand2016impossibility}.
Regarding the local-in-time wellposedness with regular vorticity, see \cite{iguchi1999free, ogawa2002free, ogawa2003incompressible, christodoulou2000motion, christodoulou2000motion, lindblad2005well, zhang2008free} and \cite{su2020long} for water waves with point vortices.  In the irrotational case, almost global and global well-posedness for water waves were proved in \cite{Wu2009,Wu2011, germain2012global, Ionescu2015, alazard2015global},  and see also \cite{HunterTataruIfrim1,HunterTataruIfrim2, wang2018global, zheng2019long, ai2020two}. In the rotational case, see \cite{ifrim2015two, bieri2017motion, ginsberg2018lifespan}, and \cite{su2020long}. In \cite{wu2020quartic}, Wu obtained long-time existence results without imposing size restrictions on the slope of the initial interface and the magnitude of the initial velocity. In particular these allow the interface to have arbitrarily large steepnesses and initial velocities to have arbitrarily large magnitudes.

\smallskip
\subsection{Asymptotic models}\label{subsec:AM}
To understand the behavior of the water waves, one can study the system in various asymptotic regimes. It is well-known that the 1d cubic nonlinear Schr\"odinger equation (NLS) 
\begin{equation}\label{NLS}
iu_t+u_{xx}+|u|^2u=0
\end{equation}
is related to the full water wave system in the sense that asymptotically it is the envelope equation for the free interface of the water waves. Formally speaking, consider
the modulational approximation to the solution of
2d water waves equations,  i.e., 
%
a solution $\zeta(\alpha,t)$ of the parametrized free interface whose leading order is a wave packet of the form
\begin{equation}\label{amplitudeevolve}
W(\alpha,t):=\alpha+\epsilon B(X, T)e^{i(k\alpha+\omega t)}  \quad  (\epsilon ~\text{small},\quad k, \omega: \text{constant}),
\end{equation}
then from the multi-scale analysis, we obtain that  $X=\epsilon(\alpha+\frac{1}{2\omega} t),\, T=\epsilon^2t$,  $\omega=\sqrt{k}$ and $B$ solves the 1d focusing cubic NLS.  One observes that the envelope $B$ is a profile that
it travels at the group velocity $\frac{1}{2\omega}=\frac{d\omega}{dk}$ determined  by the dispersion relation of the water wave equations and it  evolves according to the NLS on the time scale $O(\epsilon^{-2})$. 

\smallskip

This discovery was derived formally by Zakharov \cite{zakharov1968stability} for the infinite-depth case,  and by
Hasimoto and Ono \cite{hasimoto1972nonlinear} for the finite-depth case. In \cite{craig1992nonlinear}, Craig, Sulem and Sulem applied
the modulation analysis to the finite depth 2D water wave equation. They derived an approximate
solution in the form of a wave packet and showed that the modulation approximation
satisfies the 2D finite-depth water wave equation to the leading order. In \cite{schneider2011justification}, Schneider and Wayne justified the NLS as the modulation approximation for a quasilinear model that
captures some of the main features of the water wave equations.

\smallskip

The rigorous justification of the NLS for the full water waves was given by Totz and Wu \cite{Totz2012} in infinite-depth case with data in Sobolev spaces. The justification in a canal of finite depth was proved by D\"{u}ll, Schneider and Wayne \cite{dull2016justification}. See also \cite{ifrim2018nls}. In \cite{su2020partial}, the second author rigorously justified the NLS from the full 2d infinite-depth water waves with data of the form $H^s(\mathbb{R})+H^{s'}(\mathbb{T})$ and therefore justified the Peregrine soliton from the water waves. 

\smallskip
As mentioned before, to analyze the instability of Stokes waves using the water wave system directly could be complicated. In this paper, instead of working on the full system
directly, we further explore the modulational approximation to the water wave via the NLS with the appearance of Stokes waves and incorporate the instability of the NLS.

\smallskip

Our mativation is that numerical results, for example  \cite{Deconinck2011} by Deconinck and Oliveras, showed that the spectrum of the linearized operator given by the Stokes wave in the 2d full water waves was qualitatively resembled by the spectrum of the linearized operator given by the special solution $e^{it}$ in the 1d cubic NLS. Therefore, it is natural to conjecture that the mechanism of modulational instability in the full water waves is governed by the 1d NLS.  


	
	\subsection{Basic setting and main results}
In this subsection, we formulate the basic setting of the problem and state the main result.  More details and  estimates are presented in Section \ref{section:proof}.


\smallskip
First of all, we fix some constants. Throughout this paper, we assume that the pressure $P=0$ on the interface, the gravity is given by $(0,-1)$ and the density of the fluid is $1$. Denote $\mathbb{T}:=\mathbb{R}/2\pi$. We identify $(x,y)\in \mathbb{R}^2$ with $x+iy\in \mathbb{C}$.

\subsubsection{Wu's modified Lagrangian formulation}\label{subsubsection:lagrangian}
	It implies from $\text{div}~v=0$ and $\text{curl}~v=0$ that $\bar{v}$ is holomoprhic in $\Omega(t)$, so $v$ is completely determined by its boundary value on $\Sigma(t)$. Let  the interface $\Sigma(t)$ be parametrized by $z=z(\alpha,t)$, with $\alpha\in\mathbb R$ as the Lagrangian coordinate, i.e., $\alpha$ is chosen in such a way that $z_t(\alpha,t)=v(z(\alpha,t),t)$. So we have $v_t+v\cdot\nabla v\Big|_{\Sigma(t)}=z_{tt}$. Because $P(z(\alpha,t),t)\equiv 0$, we can write $\nabla P\Big|_{\Sigma(t)}=-iaz_{\alpha}$, where $a:=-\frac{\partial P}{\partial \boldmath{n}}\frac1{|z_\alpha|}$ is a real-valued function. Therefore the momentum equation $v_t+v\cdot \nabla v=-(0,1)-\nabla P$ along $\Sigma(t)$ can be written as
	\begin{equation}
	z_{tt}-iaz_{\alpha}=-i.
	\end{equation}
	Since $\bar{z}_t$ is the boundary value of $\bar{v}$, 
	the water wave equations (\ref{euler}) is equivalent to 
	\begin{equation}\label{system_boundaryold}
	\begin{cases}
	z_{tt}-iaz_{\alpha}=-i\\
	\bar{z}_t~\text{is}~\text{holomorphic},
	\end{cases}
	\end{equation}
	where by $\bar{z}_t$ holomorphic, we mean that there is a bounded holomorphic function $\mathcal{V}(\cdot,t)$ on $\Omega(t)$ with $\mathcal{V}(x+iy,t)\rightarrow 0$ as $y\rightarrow -\infty$ such that $\bar{z}_t(\alpha,t)=\mathcal{V}(z(\alpha,t),t)$.

	While the coordinate above is well-suited to quasilinearize the water wave system and prove the local wellposedness, see \cite{Wu1997}, it is not convenient to study the long-time behavior of the system due to some quadratic terms appearing in $a$ and the quasilinearzation, see \cite{Wu2009}.  To obtain the nonlinear instability of the Stokes wave, we have to solve the water wave system for sufficiently long time. Following \cite{Wu2009}, we introduce Wu's modified Lagrangian coordinate. \footnote{Throughout this paper, we will call it Wu's coordinate or modified Lagrangian coordinate interchangeably} 

		Let $\kappa(\cdot,t):\mathbb{R}\rightarrow \mathbb{R}$ be a diffeomorphism. Denote $\zeta:=z\circ \kappa^{-1}$. We pick the $\kappa$ such that $\bar{\zeta}(\alpha,t)-\alpha$ is holomorphic in the sense as above.
Then composing the equation \eqref{system_boundaryold} with the diffeomorphism $\kappa$, we know that 	$\zeta$ solves
%
	\begin{equation}\label{system_boundary}
	\begin{cases}
	(D_t^2-iA\partial_{\alpha})\zeta=-i\\
	D_t\bar{\zeta}, \quad \bar{\zeta}-\alpha~\text{are}~\text{holomorphic},
	\end{cases}
	\end{equation}
where  we used the notations \begin{equation}\label{eq:DAintro}
	D_{t}=\partial_t+b\partial_{\alpha},\,\,
	b=(\partial_{t}\kappa)\circ \kappa^{-1},\,\,
	A:=(a\kappa_{\alpha})\circ \kappa^{-1}.
	\end{equation}
	Such coordinates system was first used in \cite{Wu2009} to prove the almost global wellposedness of 2d water waves with small and localized data. Then it  has been used in \cite{Wu2011,Totz2012,Totz2015,su2020partial,su2020long} to study the water wave problems on the long-time scale. 
	Once knowing the existence of such coordinates, one can directly work on the water wave system in this coordinate without invoking the Lagrangian coordinates. In \S \ref{section:structure}, by deriving formulae for the quantities $b$ and $A$, we formulate the water wave system in the $\zeta$ variables directly and avoid  using the change of variables $\kappa$.

	\subsubsection{The Stokes waves}
	
	 A Stokes wave is a periodic steady wave traveling at a fixed speed to the system \eqref{system_boundary}. Under Wu's coordinate, we can use a triple $\big(\omega, \zeta_{ST},D_t^{ST}\zeta_{ST}\big)$ to represent a Stokes wave $\zeta_{ST}$ with the traveling velocity $\omega$
	 and velocity field given by $D^{ST}_t \zeta_{ST}$. In this paper, we shall consider small-amplitude Stokes waves. 	 Notice that \eqref{system_boundary} is time reversible, that is, if $(\zeta, D_t\zeta)$ is a solution to \eqref{system_boundary}, then $(\hat{\zeta}, \hat{D}_t\hat{\zeta})$ is also a solution to \eqref{system_boundary} where 
	\begin{equation}
	    \hat{\zeta}(\alpha,t):=\zeta(\alpha,-t), \quad \hat{D}_t\hat{\zeta}(\alpha,t):=D_t\zeta(\alpha,-t).
	\end{equation}
	Without loss of generality, we consider those Stokes waves traveling to the left with period $2\pi$.

	Regarding the existence of the Stokes waves, one has the following result in Wu's coordinate.

	\begin{proposition}\label{prop:Stokes} There exists a smooth curve $(\omega, Z_{ST},D^{ST}_t Z_{ST})$ of periodic traveling wave solutions to the water wave system \eqref{system_boundary} parametrized by the a parameter $0\leq|\epsilon|\ll1$ which we call the amplitude
	of $Z_{ST}$. For each solution $(\omega, Z_{ST}, D^{ST}_t Z_{ST})$ on this curve,  one can write it as$$Z_{ST}(\alpha,t)=\alpha+F(\alpha+\omega t),\quad D^{ST}_t Z_{ST}=G(\alpha+\omega t)$$
	where $F$ and $G$ satisfy the following properties:
	\begin{itemize}
	    \item $F(\alpha)$ and $G(\alpha)$  are $2\pi$ periodic
	    \item	 $\Re\{F\}$ and $\Re\{G\}$ are odd, $\Im\{F\}$ and $\Im\{G\}$ are even.
	\end{itemize}
For the element of the curve with amplitude $\epsilon$, we denote it as
\begin{equation}\label{eq:stokesnotation}
\Big(\omega(\epsilon),\,Z_{ST}(\epsilon),\,D_t^{ST}Z_{ST}(\epsilon)\Big)
\end{equation}
and it has the following asymptotic expansions
\begin{equation}
        Z_{ST}(\epsilon;\alpha,t)=\alpha+i\epsilon e^{i\alpha+i\omega t}+i\epsilon^2+\frac{i}{2}\epsilon^3 e^{-i\alpha-i\omega t}+O(\epsilon^4).
\end{equation}
and
\begin{equation}
        \omega(\epsilon)=1+\epsilon^2/2+\mathcal{O}(\epsilon^3).
\end{equation}
	\end{proposition}
In Subsection \S \ref{appendix:existencestokes} and Subsection \S \ref{subsec:stokesexpansion}, we provide for the proof of Proposition \ref{prop:Stokes}. 

\smallskip

Throughout of this paper, for the sake of simplicity,  we will only focus on the case that $0\leq\epsilon\ll1$ since the case that $\epsilon\leq0$ can be treated in the same manner.

\smallskip
By the translational symmetry of the system \eqref{system_boundary},  for any Stokes wave $\zeta_{ST}$ with amplitude $\epsilon$, one can find a unique $\phi\in\mathbb{T}$ and the solution $(\omega,Z_{ST},D^{ST}_tZ_{ST})$ from the curve in Proposition \ref{prop:Stokes} associated with the amplitude $\epsilon$, such that we can write\begin{equation}\label{eq:stokestrans}
    \zeta_{ST}(\alpha,t)=Z_{ST}(\epsilon;\alpha+\phi,t).
\end{equation}
From Proposition \ref{prop:Stokes} and \eqref{eq:stokestrans}, using the notation \eqref{eq:stokesnotation}, we define the family of Stokes waves of small amplitude,
\begin{align}\label{eq:family}
   \Large\{\zeta_{ST}^{\gamma,\phi}\,|\,\,\,\, \zeta_{ST}^{\gamma,\phi}(\alpha,t)=Z_{ST}(\gamma;\alpha+\phi,t),\,\phi\in\mathbb{T},\,0\leq\gamma\leq \epsilon_0\Large\}
\end{align}
where $\epsilon_0$ is a given small number.


	\subsubsection{The main result}
	With preparations above, we are ready to state the main result in this paper.
\begin{thm}\label{thm:mainintro2}
	 There exists a sufficiently small number $\epsilon_0\in (0, 1)$ such that for all $0<\epsilon\leq \epsilon_0$ {\color{red},}  a Stokes wave $\zeta_{ST}$ with amplitude $\epsilon$ is nonlinearly modulational unstable in the following sense: 
	 
	 Let $s\geq 4$ be a fixed positive number. For any $q\in\mathbb{Q}_{+}$ satisfying  $q\geq \frac{1}{\epsilon}$ and any $0\leq \delta\ll 1$,  there exists a solution $\zeta(\alpha,t)$ to \eqref{system_boundary} satisfying the following conditions:
	
	\begin{itemize}
	\item  Its initial data
 $\zeta(\alpha,0)=\zeta_0$ and $D_t \zeta(\alpha,0)=v_0$ satisfy 
\begin{equation}\label{eq:langrangianPerbintro}
    \norm{(\zeta_0, v_0)-(\zeta_{ST}(\cdot,0), \partial_t \zeta_{ST}(\cdot,0))}_{H^{s+1}(q\mathbb{T})\times H^{s+1/2}(q\mathbb{T})}\leq \epsilon^{1/2}\delta.
\end{equation}
	    \item The solution $\zeta(t)$ exists  on $[0, \epsilon^{-2}\log\frac{\mu}{\delta}]$ and it satisfies $$(\zeta_{\alpha}-1, D_t\zeta)\in C([0, \epsilon^{-2}\log\frac{\mu}{\delta}];H^s(q\mathbb{T})\times H^{s+1/2}(q\mathbb{T}))$$
where $\delta\ll\mu<1$ is a fixed number.
\item The solution $\zeta(t)$ satisfies
\begin{equation}\label{eq:mainthminsta}
        \sup_{t\in [0,\,\epsilon^{-2}\log\frac{\mu}{\delta}]}\inf_{\phi\in \mathbb{T}}\inf_{\gamma\in (0,\epsilon_0)}\norm{ \zeta(\cdot,t)-\zeta_{ST}^{\gamma,\phi}(\cdot,t)}_{L^2(q\mathbb{T})}\geq c\epsilon^{1/2}
        \end{equation} 
        for some constant $c>0$ which is uniform in $\delta$ and $\epsilon$.
	\end{itemize}

\end{thm}	

Before discussing the main ideas of this paper, let us give a few comments on this result.

\begin{rem}
Notice that \eqref{eq:langrangianPerbintro} is an open condition. Therefore, our instability construction is stable under perturbations of size $\delta\epsilon^{3/2}$ in $H^s(q\mathbb{T})$.
\end{rem}

\begin{rem}
Although we measure the instability using the $L^2$ norm in \eqref{eq:mainthminsta}, the solution we constructed satisfies $(\partial_\alpha \zeta-1, D_t \zeta)\in H^s(q\mathbb{T})\times H^{s+1/2}(q\mathbb{T})$. Moreover, clearly from \eqref{eq:mainthminsta}, we also obtain the instability in $H^s(q\mathbb{T})$.
\end{rem}
\begin{rem}
The quantitative estimate \eqref{eq:mainthminsta} 
precisely implies that the waveform of the Stokes wave
is broken under long-wave perturbations in $H^s(q\mathbb{T})$. Moreover, by our explicit construction and Sobolev embedding, this phenomenon is also captured pointwisely.
\end{rem}
\begin{rem}
By construction, the solution $\zeta$ we constructed here has the fundamental period $2q\pi$. See Section \ref{section:proof} for details.
\end{rem}
\begin{rem}
One can translate the instability to the Eulerian coordinate in term of the elevations. Again see Section \ref{section:proof} for details.
\end{rem}
\begin{rem}
The long-time existence of  size $\mathcal{O}(\epsilon^{-2}\log\frac{1}{\delta})$ holds for more general flows around the Stokes wave. See Theorem \ref{thm:main1zeta} for details.  
\end{rem}

\subsection{Essential ideas and outline of the proof}
In this subsection, we highlight the key features and present the outline of our approach.
\subsubsection{\textbf{\emph{Choice of coordinates}}}
The first key part of our work is to find good coordinates to perform expansions and the long-time estimates.

In Eulerian coordinates,  it is known that the elevation of the free interface of a given Stokes wave, $\eta$, the following expansion
	\begin{equation}\label{eq:etaexpintro1}
	    \eta(x,t)=\epsilon \cos (x+\omega t)+\frac{1}{2}\epsilon^2 \cos(2(x+\omega t))+\epsilon^3\Big\{\frac{1}{8}\cos (x+\omega t)+\frac{3}{8}\cos(3(x+\omega t))\Big\}+O(\epsilon^4)
	\end{equation}
holds. See \cite{stokes1880theory} and \cite{nguyen2020proof} for details. We should immediately notice that in this setting, up to order $\mathcal{O}(\epsilon^4)$,  there are three nontrivial frequencies. Putting this expansion into the nonlinear problem, one should expect that due to the interaction of frequencies, the nonlinear analysis would be very complicated. 

Moreover, although numerical simulations tell us that the NLS is a good asymptotic model to analyze the modulational instability problem of the deep water wave problem, in the Eulerian coordinates, it is highly nontrival to see the connection between the NLS and the water wave system.

In this paper, we utilize Wu's modified Lagrangian coordinates. In this setting, the Stokes wave has a remarkable asymptotic expansion:
\begin{equation}\label{intro:stokesexpan}
    \zeta_{ST}(\alpha,t)=\alpha+i\epsilon e^{i\alpha+i\omega t}+i\epsilon^2+\frac{i}{2}\epsilon^3 e^{-i\alpha-i\omega t}+O(\epsilon^4).
\end{equation}
Compared with the asymptotic expansion 	\eqref{eq:etaexpintro1}
 in Eulerian coordinates, up to an error of $O(\epsilon^4)$, (\ref{intro:stokesexpan}) has only one nonzero fundamental frequency. This  fact plays an essential role in our work and simplifies the analysis.

Furthermore, using Wu's modified Lagrangian coordinates, it is relatively clear how to derive the relation between the NLS and the water wave system. The NLS has been derived from water wave system using Wu's coordinates in other settings, see \cite{Totz2012} and \cite{su2020partial}. Due to the appearance of Stokes waves as will be explained in \S \ref{subsub:introdevNLS}, the current situation is quite different than those earlier works.
%
%

From the perspective of the long-time existence, under Wu's coordinates, one can derive structures without quadratic nonlinearities which we call cubic structures for both the Stokes wave and the perturbed flow which are fundamental for us to capture the instability. We will discuss this in details in Section \S \ref{section:structure}.


\subsubsection{\textbf{\emph{Derivation of the NLS and its instability }}}\label{subsub:introdevNLS}
The second main step is to derive the NLS from the water wave system via the modulational approximation.

From the expansion  (\ref{intro:stokesexpan}), the leading order term of the Stokes wave is given as $\zeta_{ST}=\alpha+i\epsilon e^{i\alpha+i\omega t}$ which is in the form of a plane wave. It is natural that when we perturb the Stokes wave, the perturbed flow should be written as a wave packet\begin{equation}
    \zeta(\alpha,t)=\alpha+\epsilon\zeta^{(1)}+\mathcal{O}(\epsilon^2)
\end{equation}
where from the view of the modulational approximation, see \S \ref{subsec:AM}, $\zeta^{(1)}=B(X,T)e^{i\alpha+i\omega t}$ with $X=\epsilon(\alpha+\frac{1}{2\omega}t),\,T=\epsilon^2t$ and $B$ will be chosen appropriately to solve a NLS.

To derive the NLS in our current setting, as mentioned above,  is different than earlier works.  
In the  current setting,  the parameter $\omega$ in the phase also depends on $\epsilon$ since the velocity of the Stokes wave depends on its amplitude. Whereas, in earlier works, when proceeding the modulational approximation analysis, it is always assumed that $\omega^2=k$, see \eqref{amplitudeevolve}, which is independent of $\epsilon$.  We need some extra care to handle this dependence when we perform the multi-scale expansion.  This extra dependence is also crucial for us to obtain the approximate solutions to Stokes wave and the perturbed flow.

In Section \S \ref{subsection:multiscale}, we obtain that the NLS for $B$ is 
\begin{equation}\label{eq:NLSintro}
    iB_T+\frac{1}{8}B_{XX}+\frac{1}{2}|B|^2 B-\frac{1}{2}B=0\quad\text{on }\,q_1\mathbb{T},\quad q_1:=\epsilon q.
\end{equation}
We also note that the coefficient $i$ in front of $\epsilon e^{i\alpha+i\omega t}$  in the expansion \eqref{intro:stokesexpan} is a special solution to the NLS above.

Letting $u(x,t)=e^{it}B(\frac{x}{2},2t)$, then by a direct computation, it solves
\begin{equation}\label{eq:cubicNLSintro}
    iu_t+{u}_{xx}+|u|^2 u=0,\quad x\in q_2\mathbb{T},\,\,q_2:=\epsilon q/2
\end{equation}
which is the standard cubic NLS. In this setting, the special solution corresponding to the Stokes wave is $u_0=ie^{it}$.

Therefore, in the scale of the NLS, the stability problem of the Stokes wave is reduced to the corresponding problem for the special solution
\[
u_0\left(x,t\right)=ie^{it},
\]
to \eqref{eq:cubicNLSintro}.

Consider the perturbation of the form
\[
u(x,t)=ie^{it}\left(1+w\left(x,t\right)\right).
\]
Plugging the ansatz above into \eqref{eq:cubicNLSintro}, we have
\[
i\partial_{t}w+\partial_{x}^{2}w+2\Re w=N\left(w\right)
\]
where $N$ is the nonlinear term.

To see the instability at the linear level, we take the real part of $w$, $\phi=\Re w$. Then $\phi$ satisfies
\[
\partial_{t}^{2}\phi+2\partial_{x}^{2}\phi+\partial_{x}^{4}\phi=0.
\]
Consider the plane wave, $\phi=e^{i\frac{k}{q_2}x-it\omega}$. Then direct computations give us
\[
-\omega^{2}-\frac{2}{q_2^{2}}k^{2}+\frac{k^{4}}{q_2^{4}}=0.
\]
Solving $\omega$ from the expression above,
\[
i\omega=\pm i\frac{|k|}{q_2}\sqrt{\frac{k^2}{q^2_2}-2}\]  
we can conclude that
\begin{itemize}
    \item When $\frac{k^{2}}{q_2^{2}}>2$, the linear solution is dispersive.
    \item When $\frac{k^{2}}{q_2^{2}}<2$ , the linear solution produces the exponential instability.

\end{itemize}

In Appendix \ref{sec:NLS}, we show  that the instability above is persisted in the nonlinear problem. The conclusion is that for any $0<\delta\ll1$, there exists solution $u$ to \eqref{eq:NLSintro} such that\begin{equation}
    \norm{u(\cdot,0)-i}_{H^s(q_2\mathbb{T})}=\delta
\end{equation} 
but at $t^*\sim\log\frac{1}{\delta}$ \begin{equation}\label{eq:nlsinstaintro}
        \norm{u(\cdot,t^*)-ie^{it^*}}_{H^s(q_2\mathbb{T})}=\mathcal{O}(1).
\end{equation}
We use this as the deriving force for the nonlinear modulational instability of water waves.
\subsubsection{\textbf{\emph{Construction of perturbations}}}

The third step of our analysis is to construct the unstable perturbation of the Stokes wave.

Performing the multi-scale analysis in Section \S \ref{section:multiscale}, we further expand 
the perturbed flow as \begin{equation}\label{eq:zetaexpandintro}
    \zeta(\alpha,t)=\alpha+\epsilon\zeta^{(1)}+\epsilon^2 \zeta^{(2)}+\epsilon^3 \zeta^{(3)}+\mathcal{O}(\epsilon^4)
\end{equation}
where from the discussion above, $\zeta^{(1)}=B(X,T)e^{i\alpha+i\omega t}$.

Under this setting, suppose that  we can control the flow for a sufficiently long time interval such that $$\zeta-\alpha-\epsilon\zeta^{(1)}=\mathcal{O}(\epsilon^2)$$ then the dominated behavior should be given by $B$ which turns out to solve the NLS \eqref{eq:NLSintro}.

From the discussion before, the NLS has the strong instability  around the special solution $i$  given by the Stokes wave.
We can always use the initial data $u(\cdot,0)$ of the NLS  to construct the initial data to the water wave system using the multi-scale analysis.  Taking the initial data constructed via $u(\cdot,0)$ which causes the instability of the NLS, from \eqref{eq:nlsinstaintro}, returning back to the water wave system, we obtain that at $t_*\sim\epsilon^{-2}\log\frac{1}{\delta}$,
\begin{equation}
    \norm{\zeta(\cdot, t_*)-\alpha-\epsilon B(\cdot,t_*)e^{i\alpha+i\omega t}}_{H^s(q\mathbb{T)}}=\mathcal{O}(\epsilon^{3/2})
\end{equation}
which implies the instability.  Therefore the problem now is reduced to control the solution for a sufficiently long time interval.

%
    %
 \subsubsection{\textbf{\emph{Long-time existence}}}
 From the construction above, we notice that in order to see the instability in the water wave system, we need to solve the system on a time interval of size $\mathcal{O}(\epsilon^{-2}\log\frac{1}{\delta})$. In the general setting, there is no global-in-time theory for periodic water waves. At this stage, the best lifespan for the general periodic water wave systems with initial data of size $\epsilon$ is $\mathcal{O}(\epsilon^{-3})$, see \cite{berti2018birkhoff} and \cite{wu2020quartic}.
 But the $\delta$ above could be arbitrarily small, say, $\delta\ll e^{-\epsilon^{-2}}$. So any lifespan independent of $\delta$ will not be sufficient for us to see the instability. In order to achieve our goal, we need to design appropriate perturbations such that the perturbed flow $\zeta$ exists at least $[0, \epsilon^{-2}\log\frac{1}{\delta}]$ and it needs to satisfy
 \begin{equation}
     \norm{\partial_{\alpha}(\zeta(\cdot,0)-\zeta_{ST}(\cdot,0))}_{H^{s}(q\mathbb{T})}=\mathcal{O}(\delta \epsilon^{1/2}) 
 \end{equation}
 but\begin{equation}
          \norm{\partial_\alpha\Big(\zeta(\cdot,t)-\zeta_{ST}(\cdot,t)\Big)}_{H^{s}(q\mathbb{T})}=\mathcal{O}(\epsilon^{1/2}) 
 \end{equation} 
 at some $t\in[0,\epsilon^{-2}\log\frac{1}{\delta}]$.
 This is another place where this work differs from
 \cite{Totz2012,su2020partial} where the existence were proved for $\mathcal{O}(\epsilon^{-2})$.


To achieve the long-time existence and obtain the corresponding the estimates, the fact that the Stokes wave is a global solution of size $\epsilon$ plays a pivotal role.

By the multi-scale analysis and expanding the solutions in terms of powers of $\epsilon$, we can obtain the asymptotic expansions for the perturbed flow $\zeta$ as \eqref{eq:zetaexpandintro} and for the Stokes wave \eqref{intro:stokesexpan}.  Then we define
\begin{equation}
	    \tilde{\zeta}:=\alpha+\epsilon\zeta^{(1)}+\epsilon^2\zeta^{(2)}+\epsilon^3\zeta^{(3)}
	\end{equation}
and
\begin{equation}
	    \tilde{\zeta}_{ST}:=\alpha+i\epsilon e^{i\alpha+i\omega t}+i\epsilon^2+\frac{i}{2}\epsilon^3 e^{-i\alpha-i\omega t}.
	\end{equation}
With these two notations, we define the approximate solution as
	\begin{equation}
	\zeta_{app}(\alpha,t)=\zeta_{ST}+(\tilde{\zeta}-\tilde{\zeta}_{ST}).
	\end{equation}
Notice that the approximate solution $\zeta_{app}$ defined above is different from those in \cite{Totz2012,su2020partial}.
This definition of approximate solution $\zeta_{app}$ together with the modified Lagrangian coordinate allow us to gain extra long-time existence of solutions.

Finally, we define the remainder term as
	\begin{equation}
	    r:=\zeta-\zeta_{app}.
\end{equation}
To control the remainder term $r$, one uses the following functional
\begin{equation}
	E_s(t)^{1/2}:=\norm{D_t r(\cdot,t)}_{H^{s+1/2}(q\mathbb{T})}+\norm{(r)_{\alpha}(\cdot,t)}_{H^{s}(q\mathbb{T})}+\norm{D_t^2r(\cdot,t)}_{H^s(q\mathbb{T})}.
	\end{equation}
In Sections \S\ref{section:error}, \S \ref{section:aprioribound}, \S \ref{section:energyestimates}, we establish that  for
$0\leq t\lesssim \frac{\log \delta^{-1}}{\epsilon^2}$, one has the following estimate:
    \begin{equation} 
        E_s(t)\leq E_s(0)+\int_0^t C\epsilon^{5} \delta^2 e^{2\epsilon^2 \tau}\, d\tau.
    \end{equation}
In particular, a direct computation of the time integral gives
    \begin{equation}
        E_s(t)\leq E_s(0)+C\epsilon^3 \delta^2 (e^{2\epsilon^2 t}-1)\leq C\epsilon^3 \delta^2 e^{2\epsilon^2 t}.
    \end{equation}
Then the bootstrap argument  and the local wellposedness theory for water wave systems will ensure the long-time existence and estimates. The factor $\delta$ on the right-hand side allows us to gain the extended lifespan $\mathcal{O}(\epsilon^{-2}\log\frac{1}{\delta})$. 

To obtain the conclusion above is far from being standard. Here we briefly illustrate the idea of computations.  Formally, the quantity $r$ satisfies$$(D^2_t-iA\partial_\alpha)r=N$$
with the initial data satisfying 
	\begin{equation}
	    \norm{\partial_\alpha r(\cdot, 0)}_{H^s}+\norm{D_t r(\cdot, 0)}_{H^{s+\frac{1}{2}}}\leq C\delta \epsilon^\frac{3}{2}.
	\end{equation}
Ignoring higher order terms,  the operator $D^2_t-iA\partial_\alpha$ is morally like $	    \partial_t^2+|\partial_{\alpha}|
$ when acting on anti-holomorphic functions. The nonlinear term $N$ on the right-hand side consists of many cubic and higher order terms, for example, some cubic structures in terms of $\zeta$, $\tilde{\zeta}$, $\zeta_{ST}$ and $\tilde{\zeta}_{ST}$. Although one can bound each of them separately in terms of powers of $\epsilon$, it is not sufficient for us. We need to explore the cancellations among cubic terms to recast $r=(\zeta-\tilde{\zeta})-(\zeta_{ST}-\tilde{\zeta}_{ST})$. 

After a careful analysis, superficially, we obtain an equation of the form
	\begin{equation}
	    (\partial_t^2+|\partial_{\alpha}|)r=h f r+g r^2+r^3.
	\end{equation}
	for some functions $f,h, g$ of size $\epsilon$. Then from the structure of the right-hand side of the equation, one can expect that the energy estimate above and a lifespan of size $\mathcal{O}(\epsilon^{-2}\log\frac{1}{\delta})$.
For full details, see Sections \S\ref{section:error}, \S \ref{section:aprioribound}, \S \ref{section:energyestimates}.

\subsubsection{\textbf{\emph{Development of instability}}}
With all the preparations above, the nonlinear instability is produced naturally.  Choosing the unstable solution to the NLS to construct the corresponding solution to the water wave system, then solution $\zeta$ satisfies
\begin{equation}
     \norm{\zeta(\cdot,0)-\zeta_{ST}(\cdot,0)}_{H^{s+1}(q\mathbb{T})}=\mathcal{O}(\delta \epsilon^{1/2}) 
 \end{equation}
 and\begin{equation}
          \norm{\partial_\alpha\Big(\zeta(\cdot,t)-\zeta_{ST}(\cdot,t)\Big)}_{H^{s}(q\mathbb{T})}=\mathcal{O}(\epsilon^{1/2}) 
 \end{equation} at time $t \sim \epsilon^{-2}\log\frac{1}{\delta}$.

By construction, the frequencies of the leading order term of the instability are in completely different scales from the Fourier modes of the family of Stokes waves.  Therefore, under the long-wave perturbation, the solution will deviate from the family of Stokes waves and completely change the structure of the wave trains.

The mechanism here  gives the dynamical and mathematical description of the modulational instability under long-wave perturbations: the instability of periodic wave trains due to \emph{self modulation} and the development of \emph{large scale} structures.

\subsubsection{\textbf{\emph{General remarks}}}









As we pointed out before, the quasilinear feature and the nonlocality of water wave systems make the nonlinear analysis exceedingly difficult.  One should expect that in the quasilinear setting, the interactions of long-waves and short-waves should be fairly complicated.  Consequently, to obtain the long-time existence and estimate could be hard. Our modulational approximation approach and the well-chosen coordinates could get rid of these difficulties. 

\smallskip

Our approach is quite general. To study the (in)stability problem directly in quaslinear problems could be very elaborate. On the other hand, since many quasilinear problems can be approximated by semilinear equations, we believe that our method has broader application to other problems. In particular, long-wave perturbations problems fit well into the general idea here.

\subsection{Outline of the paper}
	In \S\ref{Prelim} we will provide some analytical tools and the basic definitions that will be used in later sections. In \S \ref{section:structure}, we discuss the local wellposedness of the water waves in Wu's coordinates and derive the corresponding formulas for a few quantities. 
	In \S \ref{subsec:stokeseq}, the existence of Stokes waves and their expansions are present in Wu's coordinates. 
	In \S \ref{section:multiscale}, we use the multiscale analysis to derive the NLS from the full water waves with the desired properties. In \S \ref{section:error}, we derive the governing equations for the error term and define the energy functionals. In \S \ref{section:aprioribound}, we estimate the important quantities  used in the energy estimates. In \S \ref{section:energyestimates}, we obtain the a priori energy estimates. In \S \ref{section:proof}, we prove the modulational instability. In Appendix \ref{appendix:cauchy}, we study the Cauchy integral in the periodic setting bounded by a nonflat curve. In Appendix \S \ref{appendix:identities}, we provide the proof of some important identities that are used in this paper.
In the Appendix \S \ref{estimatesforsomequantities}, we provide the proof for some estimates in the previous sections for the sake of completeness. In the Appendix \S \ref{sec:NLS}, the instability of the NLS is analyzed in details. Finally we list the frequently used notations in the Appendix \ref{appendix:notation}. 
	\subsection{Notations}
	For positive quantities $a$ and $b$, we write $a\lesssim b$ for
$a\leq Cb$ where $C$ is some prescribed constant. Throughout, we use $u_{t}:=\frac{\partial}{\partial_{t}}u$, for the derivative in the time variable and 
$u_{x}:=\frac{\partial}{\partial x}u$ for the derivative in the space variable. These two way of denoting are used interchangebly. We use $\Re\{f\}$, $\Im\{f\}$ to represent the real and imaginary part of $f$, respectively.
\subsection*{Acknowledgement}
G.C. was supported by Fields Institute for Research in Mathematical Sciences via Thematic Program on Mathematical Hydrodynamics. 

	\section{Preliminaries}\label{Prelim}
In this section, we collect some basic definitions and tools which will be used in the later part of the this paper.
	\subsection{Functional spaces}
We start with the functional spaces we use in this paper.	
\begin{defn}
Let $q\in \mathbb{R}_+$. The Fourier transform or the Fourier coefficient of a function $f$ on $q\mathbb{T}$ is defined by	
\begin{equation}\label{eq:Fouriercoef}
		f_k:=\int_{-q\pi}^{q\pi}f(x)e^{-ik\frac{x}{q}}\,dx
		\end{equation}
and the Fourier inversion is given as
\begin{equation}\label{eq:FourierInv}
    f(x)=\frac{1}{2q\pi}\sum_{k\in \mathbb{Z}}f_k e^{i\frac{k}{q}x}.
\end{equation}

\end{defn}
	\begin{defn}
		Let $q\in \mathbb{R}_+$, $s\geq 0$, we define $H^s(q\mathbb{T})$ by
		\begin{equation}
		H^s(q\mathbb{T})=\{f\in L^2(q\mathbb{T}): \|f\|_{H^s(q\mathbb{T})}^2:=\sum_{n=-\infty}^{\infty}(1+|k/q|)^{2s}|f_k|^2<\infty\},
		\end{equation}
		where $f_k$ is given by \eqref{eq:Fouriercoef}.
	\end{defn}
	\begin{rem}
	Notice that by the definition above, for $f\sim 1,\,f\in H^s(q\mathbb{T})$, we have $$\|f\|_{H^s(q\mathbb{T})}\sim\sqrt{q}.$$
	\end{rem}
	\begin{rem}
	For simplicity, we take $q\in \mathbb{N}$.
	\end{rem}
	
	\begin{lemma}[Sobolev embedding]\label{sobolev}
	    Let $f\in H^s(q\mathbb{T})$, $s>1/2$. Then $f\in L^{\infty}(q\mathbb{T})$. Moreover, 
	    \begin{equation}
	        \|f\|_{L^{\infty}(q\mathbb{T})}\leq Cq^{-1/2}\|f\|_{H^s(q\mathbb{T})},
	    \end{equation}
	    where $C>0$ is an absolute constant.
	\end{lemma}
	\begin{proof}
By the Fourier inversion formula, we write
$f(x)=\frac{1}{2q\pi}\sum_{k\in \mathbb{Z}}f_k e^{i\frac{k}{q}x}$.
From $$\sum_{k\in \mathbb{Z}}(1+|k/q|)^{-2s}\leq \frac{2}{2s-1}q,$$we conclude
\begin{align*}
    |f(x)|\leq & \frac{1}{2q\pi}\sum_{k\in \mathbb{Z}}|f_k|\leq \frac{1}{2q\pi}\sum_{k\in \mathbb{Z}}|f_k (1+|k/q|)^s| (1+|k/q|)^{-s}\\
    \leq & \frac{1}{2q\pi} \|f\|_{H^s(q\mathbb{T})}\Big(\sum_{k\in \mathbb{Z}}(1+|k/q|)^{-2s}\Big)^{1/2}\\
    \leq & \frac{C}{\sqrt{2s-1}}q^{-1/2} \|f\|_{H^s(q\mathbb{T})}
\end{align*} as desired.
	
	\end{proof}
\begin{defn}
Given $s\geq0$, we define $W^{s,\infty}(q\mathbb{T})$ as
		\begin{equation}
		W^{s,\infty}(q\mathbb{T})=\left\{f:|		\|f\|_{W^{s,\infty}(q\mathbb{T})}:
=\left\Vert\frac{1}{2q\pi}\sum_{k=-\infty}^{\infty}\left(1+|k/q|\right)^s f_{k}e^{i\frac{k}{q}x}
\right\Vert_{L^{\infty}}<\infty\right\}
		\end{equation}where again $f_k$ is given by \eqref{eq:Fouriercoef}.
\end{defn}	
	
	\subsection{Hilbert transform and double layer potential}
Next, we define the Hilbert transform and the double layer potential used in the analysis of water wave systems.	
	\subsubsection{The Hilbert transform}
	\begin{defn}\label{hilbert}
		Given $q>0$. Assume that $\gamma(\alpha)-\alpha$ is $2q\pi$ periodic and satisfies
		\begin{equation}\label{chordchordarcarc}
		\beta_0|\alpha-\beta|\leq |\gamma(\alpha)-\gamma(\beta)|\leq \beta_1|\alpha-\beta|, \quad \quad \forall~~ \alpha,\beta\in \mathbb{R},
		\end{equation}
		where $0<\beta_0<\beta_1<\infty$ are constants.
		The Hilbert transform associates with the curve $\gamma(\alpha)$ is defined as 
		\begin{equation}
		\mathcal{H}_{\gamma}f(\alpha):=\frac{1}{2q\pi i}\text{p.v.}\int_{-q\pi}^{q\pi} \gamma_{\beta}(\beta)\cot(\frac{\gamma(\alpha)-\gamma(\beta)}{2q})f(\beta)\,d\beta.
		\end{equation}
		We define $H_0$ to be the Hilbert transform associated with $\gamma:=\alpha$, that is, 
		\begin{equation}
		H_0f(\alpha,t):=\frac{1}{2q\pi i}\text{p.v.}\int_{-q\pi}^{q\pi}\cot\Big(\frac{\alpha-\beta}{2q}\Big) f(\beta,t)\,d\beta.
		\end{equation}
	\end{defn}
	\begin{lemma}\label{lemma:holoboundary}
		Let $\gamma$ be a curve satisfying the condition \eqref{chordchordarcarc} in Definition \ref{hilbert}. Denote $\Omega$ as the region in $\mathbb{R}^2$ below the graph of $\gamma$. Let $f\in H^s(q\mathbb{T})$. 
		\begin{itemize}
			\item [1.] If $f(\alpha,t)=F(\zeta(\alpha,t),t)$ for some holomorphic function $F$ in $\Omega$ such that $F(x+iy,t)=F(x+2q\pi+iy,t)$ and $\lim_{y\rightarrow -\infty}F(x+iy,t)=0$. Then 
			\begin{equation}
			(I-\mathcal{H}_{\gamma})f=0.
			\end{equation}
			
			\item [2.] If $(I-\mathcal{H}_{\gamma})f=0$, then there is a holomorphic function $F$ in $\Omega$ such that $F(x+iy,t)=F(x+2q\pi+iy,t)$ and $\lim_{y\rightarrow -\infty}F(x+iy,t)=0$.
			
			\item [3.] In particular, $(I-H_0)(I-H_0)f=2(I-H_0)f+M(f)$, where $M(f):=\frac{1}{2q\pi}\int_{-q\pi}^{q\pi}f(\alpha)d\alpha$.
		\end{itemize}
	\end{lemma}
	\noindent For the proof of Lemma \ref{lemma:holoboundary}, see \S \ref{appendix:cauchy}, in particular, Corollary \ref{corollary:hilbertboundary}.
	
	It is well-known (see, for example, the celebrated paper by Guy David \cite[Theorem 6]{david1984operateurs}) that if $\gamma(\alpha)$ satisfies (\ref{chordchordarcarc}), then $\mathcal{H}_{\gamma}$ is bounded on $L^2(q\mathbb{T})$.
	\begin{lemma}\label{boundednesshilbert}We have the following bounds for Hilbert transforms.
		\begin{itemize}
			\item [(1)] Assume that $\gamma(\alpha)$ satisfies (\ref{chordchordarcarc}), then 
			\begin{equation}
			\|\mathcal{H}_{\gamma}f\|_{L^2(q\mathbb{T})}\leq C\|f\|_{L^2(q\mathbb{T})},
			\end{equation}
			for some constant $C$ depending on $\beta_0$ and $\beta_1$ only.
			
			\item [(2)] Assume in addition that $\gamma-\alpha\in H^s(q\mathbb{T})$, then 
			\begin{equation}
			\|\mathcal{H}_{\gamma}f\|_{H^s(q\mathbb{T})}\leq C\|f\|_{H^s(q\mathbb{T})},
			\end{equation}
			where $C$ depends on $\|\gamma_{\alpha}-1\|_{H^{s-1}(q\mathbb{T})}$.
			
			\item [(3)] Assume that $\gamma_1-\alpha, \gamma_2-\alpha\in H^s(q\mathbb{T})$ and there exists constants $\beta_{0,j}<\beta_{1,j}$($j=1,2$) such that 
					\begin{equation}
		\beta_{0,j}|\alpha-\beta|\leq |\gamma_j(\alpha)-\gamma_j(\beta)|\leq \beta_{1,j}|\alpha-\beta|, \quad \quad \forall~~ \alpha,\beta\in \mathbb{R}.
		\end{equation}
		Then we have 
		\begin{equation}
		\begin{split}
		    &\norm{(\mathcal{H}_{\gamma_1}-\mathcal{H}_{\gamma_2})f}_{H^s(q\mathbb{T})}\\
		    \leq & C\min\Big\{\|\partial_{\alpha}(\gamma_1-\gamma_2)\|_{W^{s-1,\infty}(q\mathbb{T})}\|f\|_{H^s(q\mathbb{T})}, \|\partial_{\alpha}(\gamma_1-\gamma_2)\|_{H^{s-1}(q\mathbb{T})}\|f\|_{W^{s-1,\infty}(q\mathbb{T})}\Big\}.
		    \end{split}
		\end{equation}

		\end{itemize}
	\end{lemma}

	\subsubsection{Double layer potentials and its adjoint}

	We define the double layer potential operator as follows.
	\begin{defn}[Double layer potential]\label{def:double}
		Suppose $\gamma(\alpha)$ satisfies (\ref{chordchordarcarc}) and $\gamma(\alpha)-\alpha$ is $2q\pi$ periodic. Let $\Sigma$ be the curve parametrized by $\gamma(\alpha)$, and $\Omega$ be the region in $\mathbb{R}^2$ bounded above by $\Sigma$. Let $\vec{n}$ be the outward normal of $\Omega$. The  double layer potential operator $\mathcal{K}_{\gamma}$ is defined by, for $f\in L^2(q\mathbb{T})$,
		\begin{equation}
		\mathcal{K}_{\gamma} f(\alpha):=\text{p.v.}\int_{-q\pi}^{q\pi} \Re\Big\{\frac{1}{2q\pi i}\gamma_{\beta}(\beta)\cot(\frac{\gamma(\alpha)-\gamma(\beta)}{2q})\Big\} f(\beta)\,d\beta.
		\end{equation}
		The adjoint $\mathcal{K}_{\gamma}^{\ast}$ of the double layer potential is defined as
		\begin{equation}
		\mathcal{K}_{\gamma}^{\ast} f(\alpha):=\text{p.v.}\int_{-q\pi}^{q\pi} \Re\Big\{-\frac{1}{2q\pi i} \frac{\gamma_{\alpha}}{|\gamma_{\alpha}|}|\gamma_{\beta}| \cot(\frac{\gamma(\alpha)-\gamma(\beta)}{2q})\Big\} f(\beta)\,d\beta.
		\end{equation}
	\end{defn}
	By Lemma \ref{boundednesshilbert}, $\|\mathcal{K}_{\gamma} f\|_{L^2(q\mathbb{T})}\leq C\|f\|_{L^2(q\mathbb{T})}$,  $\mathcal{K}_{\gamma} f$ is well-defined as an $L^2(q\mathbb{T})$ function. Similarly, for $f\in L^2(q\mathbb{T})$, $\|\mathcal{K}_{\gamma}^{\ast}f\|_{L^2(q\mathbb{T})}\leq C\|f\|_{L^2(q\mathbb{T})}$, so $\mathcal{K}_{\gamma}^{\ast}f$ is also well-defined as an $L^2(q\mathbb{T})$ function. Moreover, we have the following celebrated results due to Verchota \cite{verchota1984layer}. See also \cite{coifman1982integrale}, \cite{taylor2007tools}.
	\begin{lemma}\label{layer}
		Let $\Sigma$ be a Jordan curve parametrized by $\gamma(\alpha)$ such that $\gamma(\alpha)-\alpha$ is $2q\pi$ periodic and $\gamma(\alpha)$ satisfies  (\ref{chordchordarcarc}). Then $I\pm\mathcal{K}_{\gamma}: L^2(q\mathbb{T})\rightarrow L^2(q\mathbb{T})$ and their adjoints $I\pm \mathcal{K}_{\gamma}^{\ast}: L^2(q\mathbb{T})\rightarrow L^2(q\mathbb{T})$ are invertible, with
		\begin{equation}
		\|(I\pm \mathcal{K}_{\gamma})^{-1}f\|_{L^2(q\mathbb{T})}\leq C\|f\|_{L^2(q\mathbb{T})},
		\end{equation}
		and
		\begin{equation}\label{adjoint}
		\|(I\pm \mathcal{K}_{\gamma}^{\ast})^{-1}f\|_{L^2(q\mathbb{T})}\leq C\|f\|_{L^2(q\mathbb{T})},
		\end{equation}
		for some constant $C>0$ depending only on $\beta_0$ and $\beta_1$.
	\end{lemma}
	
	\subsubsection{Some relevant notations} Throughout this paper, we parametrize the interface $\Sigma(t)$ in modified Lagrangian coordinates by $\zeta(\alpha,t)$. So we will frequently use the notation $\mathcal{H}_{\zeta}$, $\mathcal{K}_{\zeta}$, $\mathcal{K}_{\zeta}^{\ast}$.

	\subsection{Identities}
	Here we collect some commutator identities which are frequently used later on.
	\begin{lemma}\label{lemmmaeight}
		Let $T_0>0$ be fixed. Assume that $f\in C^2_{t,x}([0,T_0]\times q\mathbb{T})$. We have 
		\begin{equation}\label{U1}
		[\partial_{\alpha}, \mathcal{H}_{\zeta}]f=[\zeta_{\alpha}, \mathcal{H}_{\zeta}]\frac{f_{\alpha}}{\zeta_{\alpha}}.
		\end{equation}
		\begin{equation}\label{U2}
		[g\partial_{\alpha}, \mathcal{H}_{\zeta}]f=[g\zeta_{\alpha}, \mathcal{H}_{\zeta}]\frac{f_{\alpha}}{\zeta_{\alpha}}, \quad \quad \forall~g\in L^{\infty}(q\mathbb{T}).
		\end{equation}
		\begin{equation}\label{U3}
		[\partial_t, \mathcal{H}_{\zeta}]f=[\zeta_t, \mathcal{H}_{\zeta}]\frac{f_{\alpha}}{\zeta_{\alpha}}.
		\end{equation}
		
		\begin{equation}\label{U4}
		[D_t, \mathcal{H}_{\zeta}]f=[D_t\zeta, \mathcal{H}_{\zeta}]\frac{f_{\alpha}}{\zeta_{\alpha}}.
		\end{equation}
		
		\begin{equation}\label{U5}
		\begin{split}
		[D_t^2, \mathcal{H}_{\zeta}]f=&[D_t^2\zeta, \mathcal{H}_{\zeta}]\frac{f_{\alpha}}{\zeta_{\alpha}}+2[D_t\zeta, \mathcal{H}_{\zeta}]\frac{\partial_{\alpha}D_t f}{\zeta_{\alpha}}\\
		&-\frac{1}{4\pi q^2 i}\int_{-q\pi}^{q\pi}\Big(\frac{D_t\zeta(\alpha)-D_t\zeta(\beta)}{\sin(\frac{1}{2q}(\zeta(\alpha)-\zeta(\beta)))}\Big)^2 f_{\beta}\,d\beta.
		\end{split}
		\end{equation}
		
		\begin{equation}\label{U6}
		\begin{split}
		[D_t^2-iA\partial_{\alpha}, \mathcal{H}_{\zeta}]f=&2[D_t\zeta, \mathcal{H}_{\zeta}]\frac{\partial_{\alpha}D_tf}{\zeta_{\alpha}}-\frac{1}{4\pi q^2 i}\int_{-q\pi}^{q\pi}\Big(\frac{D_t\zeta(\alpha)-D_t\zeta(\beta)}{\sin(\frac{1}{2q}(\zeta(\alpha)-\zeta(\beta)))}\Big)^2 f_{\beta}\,d\beta.
		\end{split}
		\end{equation}
	\end{lemma}
These identities on the full real line were studied and proved by Wu in \cite [Lemma 2.1]{Wu2009}.	For the proof of them in the current setting, see \S \ref{appendix:identities} in the Appendix.
	\subsection{Expansion of $\mathcal{H}_{\zeta}$}
In this subsection, we formally derive the expansion of the Hilbert transform associated with a curve of small amplitude. This expansion will one of the basic tools for us to derive the asymptotic expansion of solutions.

Consider a curve $\zeta$ of small amplitude. Formally, it has the expansion
\begin{equation}\label{eq:zetaexpansionpre}
    	\zeta(\alpha,t)=\alpha+\epsilon \zeta^{(1)}(\alpha,t)+\epsilon^2 \zeta^{(2)}(\alpha,t)+\epsilon^3\zeta^{(3)}(\alpha,t)+...
\end{equation}
	We rewrite $\mathcal{H}_{\zeta}f(\alpha,t)$ as 
	\begin{equation}
	\mathcal{H}_{\zeta}f(\alpha,t)=\frac{1}{\pi i}p.v. \int_{-q\pi}^{q\pi}\log(\sin(\frac{\zeta(\alpha,t)-\zeta(\beta,t))}{2q}) f_{\beta}(\beta,t)\,d\beta.
	\end{equation}
	Using the expansion \eqref{eq:zetaexpansionpre}, we have
	\begin{equation}
	\begin{split}
	\frac{\zeta(\alpha,t)-\zeta(\beta,t)}{2q}=\frac{\alpha-\beta}{2q}+\epsilon\frac{\zeta^{(1)}(\alpha,t)-\zeta^{(1)}(\beta,t)}{2q}+\epsilon^2 \frac{\zeta^{(2)}(\alpha,t)-\zeta^{(2)}(\beta,t)}{2q}+O(\frac{\epsilon^3}{2q}).
	\end{split}
	\end{equation}
	By a simple Taylor expansion,
	\begin{equation}
	\log\sin (x+a)=\log\sin a+x\cot a-\frac{1}{2\sin^2a}x^2+O(x^3),
	\end{equation}
	we obtain (with $a=\frac{\alpha-\beta}{2q}$ and $x=\epsilon\frac{\zeta^{(1)}(\alpha,t)-\zeta^{(1)}(\beta,t)}{2q}+\epsilon^2 \frac{\zeta^{(2)}(\alpha,t)-\zeta^{(2)}(\beta,t)}{2q}+O(\frac{\epsilon^3}{2q})$)
	\begin{align*}
	&\log\sin(\frac{\zeta(\alpha,t)-\zeta(\beta,t))}{2q})\\
	=&\log \sin\frac{\alpha-\beta}{2q}\\
	&+\Big[\epsilon\frac{\zeta^{(1)}(\alpha,t)-\zeta^{(1)}(\beta,t)}{2q}+\epsilon^2 \frac{\zeta^{(2)}(\alpha,t)-\zeta^{(2)}(\beta,t)}{2q}+O(\frac{\epsilon^3}{2q})\Big] \cot\frac{\alpha-\beta}{2q}\\
	&-\frac{1}{2(\sin\frac{\alpha-\beta}{2q})^2 }\Big[\epsilon\frac{\zeta^{(1)}(\alpha,t)-\zeta^{(1)}(\beta,t)}{2q}+\epsilon^2 \frac{\zeta^{(2)}(\alpha,t)-\zeta^{(2)}(\beta,t)}{2q}+O(\frac{\epsilon^3}{2q})\Big]^2+O(\frac{\epsilon^3}{2q})\\
	=&\log \sin \frac{\alpha-\beta}{2q}+\epsilon\frac{\zeta^{(1)}(\alpha,t)-\zeta^{(1)}(\beta,t)}{2q}\cot\frac{\alpha-\beta}{2q}\\
	&+\epsilon^2\Big[\frac{\zeta^{(2)}(\alpha,t)-\zeta^{(2)}(\beta,t)}{2q}\cot\frac{\alpha-\beta}{2q}-\frac{(\zeta^{(1)}(\alpha,t)-\zeta^{(1)}(\beta,t))^2}{2(2q)^2(\sin\frac{\alpha-\beta}{2q})^2 }\Big]+O(\frac{\epsilon^3}{2q}).
	\end{align*}
From the the explicit computations above, we regroup everything with respect to the power of $\epsilon$ and get 
	\begin{equation}\label{eq:expansionH}
	\mathcal{H}_{\zeta}f(\alpha,t)=H_0f(\alpha,t)+\epsilon H_1f(\alpha,t)+\epsilon^2 H_2 f(\alpha,t)+O(\frac{\epsilon^3}{2q}),
	\end{equation}
	where
	\begin{equation}\label{eq:H0}
	H_0f(\alpha,t)=\frac{1}{2q\pi i}\text{p.v.}\int_{-q\pi}^{q\pi}f(\beta,t)\cot\frac{\alpha-\beta}{2q}\, d\beta,
	\end{equation}
	\begin{equation}\label{eq:H1}
	H_1f(\alpha,t)=[\zeta^{(1)}, H_0]f_{\alpha}(\alpha,t),
	\end{equation}
	\begin{equation}\label{eq:H2}
	\begin{split}
	H_2f(\alpha,t)=&[\zeta^{(2)}, H_0]f_{\alpha}(\alpha,t)-\frac{1}{2\pi i}\text{p.v.}\int_{-q\pi}^{q\pi}\frac{((\zeta^{(1)}(\alpha,t)-\zeta^{(1)}(\beta,t)))^2}{(2q)^2(\sin\frac{\alpha-\beta}{2q})^2 } f_{\beta}(\beta,t)\,d\beta\\
=	&[\zeta^{(2)}, H_0]f_{\alpha}-[\zeta^{(1)}, H_0]\zeta_{\alpha}^{(1)}f_{\alpha}+\frac{1}{2}[\zeta^{(1)}, [\zeta^{(1)}, H_0]]\partial_{\alpha}^2f.
	\end{split}
	\end{equation}

	\subsection{Commutator estimates}
	For the following commutator estimates, we refer the interested reader to Propositions 3.2, 3.3 in \cite{Wu2009} and Theorem 2.1, Proposition 2.3 in \cite{Totz2012} for the versions on the whole real line.
	
	\noindent Given $A_j$ such that $A_j\in W^{s-1,\infty}(q\mathbb{T})$, $j=1,2$, we define
	\begin{equation}\label{eq:commutator1}
	S_{\zeta}(f,g):=[g,\mathcal{H}_{\zeta}]\frac{f_{\alpha}}{\zeta_{\alpha}}.
	\end{equation}
	\begin{equation}
	S_{2}(A,f)=\frac{1}{4q^2\pi i}\int_{-q\pi}^{q\pi} \prod_{j=1}^2 \frac{A_j(\alpha)-A_j(\beta)}{\sin(\frac{1}{2q}(\zeta_j(\alpha,t)-\zeta_j(\beta,t))} f_{\beta}(\beta)\,d\beta.
	\end{equation}
	For the commutators above, we have the following estimates. 
	\begin{proposition}\label{singularperiodic}
		Assume $\zeta$ and $\zeta_j$ ($j=1,2$) satisfy the conditions
		\begin{equation}\label{chordarc1}
		C_{0,j}|\alpha-\beta|\leq |\zeta_j(\alpha,t)-\zeta_j(\beta,t)|\leq C_{1,j}|\alpha-\beta|,
		\end{equation}
		\begin{equation}
		    C_{0}|\alpha-\beta|\leq |\zeta(\alpha,t)-\zeta(\beta,t)|\leq C_{1}|\alpha-\beta|,
		\end{equation}
		Then one has
		\begin{equation}
		\norm{S_{\zeta}(f,g)}_{H^s(q\mathbb{T})}\leq C\norm{f}_{Z}\norm{g}_{Y},
		\end{equation}
		\begin{equation}\label{singularversiontwo}
		\norm{S_{2}(A,f)}_{H^s}\leq C\prod_{j=1}^2\norm{A_j'}_{Y}\norm{f}_{Z},
		\end{equation}
		where the constant $C$ depends on $\norm{\zeta_{\alpha}-1}_{H^{s-1}(q\mathbb{T})}$, and  $Y=H^{s-1}(q\mathbb{T})$ or $W^{s-2,\infty}(q\mathbb{T})$, $Z=H^s(q\mathbb{T})$ or $W^{s-1,\infty}(q\mathbb{T})$. Moreover, only one of the $Y$ and $Z$ norms takes the $H^k(q\mathbb{T})$ norm  ($k=s-1$ if $Y$ takes the $H^k(q\mathbb{T})$ norm and $k=s$ if $Z$ takes the $H^k(q\mathbb{T})$ norm).
	\end{proposition}
	
Next, we estimate the differences of commutators produced by different curves.  These will be helpful when we analyze the differences of different solutions.
	\begin{proposition}\label{commutator:difference:quadratic}
	    Let $\zeta_j,\,j=1,2,3,4$ satisfy
	    \begin{equation}
	       C_{0,j}|\alpha-\beta|\leq |\zeta_j(\alpha,t)-\zeta_j(\beta,t)|\leq C_{1,j}|\alpha-\beta|,
	    \end{equation}
	    for some constants $C_{0,j}$, $C_{1,j}$, $j=1,2,3,4$. Then we have
	    \begin{equation}
	    \begin{split}
	        \norm{\Big(S_{\zeta_1}-S_{\zeta_2}\Big)(f,g)}_{H^{s}(q\mathbb{T})}\leq & C\|\partial_{\alpha}(\zeta_1-\zeta_2)\|_{Y}\|f\|_{Z}\|g\|_{Y},
	        \end{split}
	    \end{equation} 
	    \begin{equation}\label{eqn:quadraticsingular}
	        \begin{split}
	        \norm{S_{\zeta_1}(f_1,g_2)-S_{\zeta_2}(f_2, g_2)}_{H^s(q\mathbb{T})}
	        &\leq C\|\partial_{\alpha}(\zeta_1-\zeta_2)\|_{Y}\|f\|_{Z}\|g\|_{Y}+C\norm{f_1-f_2}_{Z}\norm{g_1}_{Y}\\
	        &+
	        C\norm{f_2}_{Y}\norm{g_1-g_2}_{Z}.
	        \end{split}
	    \end{equation}

	    Furthermore,  one has
	    \begin{equation}\label{eq:diffcommutator2}
	    \begin{split}
	        &\norm{ \Big(S_{\zeta_1}-S_{\zeta_2}-(S_{\zeta_3}-S_{\zeta_4})\Big)(f,g)}_{H^s(q\mathbb{T})}\\
	        \leq & C\|\partial_{\alpha}(\zeta_1-\zeta_2-(\zeta_3-\zeta_4))\|_{Y}\|f\|_{Z} \|g\|_{Y}\\
	        &+ C\norm{\partial_{\alpha}(\zeta_1-\zeta_2-(\zeta_3-\zeta_4))}_{Y}\Big(\norm{\partial_{\alpha}(\zeta_1-\zeta_2)}_{Y}+\norm{\partial_{\alpha}(\zeta_3-\zeta_4)}_{Y}\Big)^2 \|f\|_Z \|g\|_{Y}\\
	        &+C\|\partial_{\alpha}(\zeta_3-\zeta_4)\|_{Y}\Big( \norm{\partial_{\alpha}(\zeta_1-\zeta_3)}_{Y}+\norm{\partial_{\alpha}(\zeta_2-\zeta_4)}_{Y}\Big)\sum_{j=1}^4\Big(1+\norm{\partial_{\alpha}\zeta_j-1}_{Y}\Big)
	       \|f\|_Z \|g\|_{Y}.
	        \end{split}
	    \end{equation}
	    for some constant $C$ depends on $\|\partial_{\alpha}\zeta_j\|_{H^{s-1}(q\mathbb{T})}$, $C_{0,j}$, $C_{1,j}$, $j=1,2,3,4$. Moreover, only one of the $Y$ and $Z$ norms takes the $H^k(q\mathbb{T})$ norm ( $k=s-1$ if $Y$ takes the $H^k(q\mathbb{T})$ norm and $k=s$ if $Z$ takes the $H^k(q\mathbb{T})$ norm).
	    
	    \vspace*{1ex}

	\end{proposition}
Proposition \ref{commutator:difference:quadratic} is proved in Appendix \S \ref{proof:commutatordiffquadratic}.

Using the same idea, we can also prove the following estimate for the differences of Hilbert transforms assoiated to different curves.
	\begin{proposition}\label{proposition:Hilbertquartic}
	    Suppose we have four curves  $\zeta_j\,j=1,2,3,4$ satisfying
	    \begin{equation}
	       C_{0,j}|\alpha-\beta|\leq |\zeta_j(\alpha,t)-\zeta_j(\beta,t)|\leq C_{1,j}|\alpha-\beta|,
	    \end{equation}
	    for some constants $C_{0,j}$, $C_{1,j}$, $j=1,2,3,4$. Then 
	    \begin{equation}
	        \begin{split}
	           & \norm{\Big(\mathcal{H}_{\zeta_1}-\mathcal{H}_{\zeta_2}-(\mathcal{H}_{\zeta_3}-\mathcal{H}_{\zeta_4})\Big)f}_{H^s(q\mathbb{T})}\\
	           \leq & C\norm{\partial_{\alpha}\Big(\zeta_1-\zeta_2-(\zeta_3-\zeta_4)\Big)}_{Y}\norm{f}_{Y}+\norm{\partial_{\alpha}(\zeta_1-\zeta_3)}_{Y}\norm{\partial_{\alpha}(\zeta_2-\zeta_4)}_{Y}\norm{f}_{Z},
	        \end{split}
	    \end{equation}
	    where the $Y$ norm is either $H^{s}(q\mathbb{T})$ or $W^{s-1,\infty}(q\mathbb{T})$. Moreover, only one of these $Y$-norms takes the $H^s(q\mathbb{T})$ norm.
	\end{proposition}
By construction, the commutator $S_\zeta(f,g)$ defined by \eqref{eq:commutator1} can be regarded as a trilinear form in terms of  the triple $(\zeta,f,g)$.  The following estimate for the differences of commutators produced by different triples are useful in our analysis.
	\begin{proposition}\label{commutator:difference:cubic}
	   With notations above and the same assumption as Proposition \ref{commutator:difference:quadratic}, we have the following estimate
	    \begin{equation}
	    \begin{split}
	        &\norm{\Big(S_{\zeta_1}(g_1,f_1)-S_{\zeta_2}(g_2, f_2)\Big)-\Big(S_{\zeta_3}(g_3, f_3)-S_{\zeta_4}(g_4, f_4)\Big)}_{H^s(q\mathbb{T})}\\
	        \leq & C\norm{\Big(S_{\zeta_1}-S_{\zeta_2}-(S_{\zeta_3}-S_{\zeta_4})\Big)(g_1, f_2+f_3-f_4)}_{H^s(q\mathbb{T})}\\
	        &+C\norm{\Big(S_{\zeta_3}-S_{\zeta_4}\Big)(g_1-g_3, f_2+f_3-f_4)}_{H^s(q\mathbb{T})}\\
&+C\norm{\Big(S_{\zeta_3}-S_{\zeta_4}\Big)(g_3, f_2-f_4)}_{H^s(q\mathbb{T})}\\
	        &+C\norm{S_{\zeta_2}\Big(g_1-g_2-(g_3-g_4), f_2+f_3-f_4\Big)}_{H^s(q\mathbb{T})}+C\norm{S_{\zeta_2}\Big(g_3-g_4, f_2-f_4\Big)}_{H^s(q\mathbb{T})}\\
    &+C\norm{\Big(S_{\zeta_2}-S_{\zeta_4}\Big)(g_3-g_4, f_3)}_{H^s(q\mathbb{T})}\\
    &+C\norm{S_{\zeta_2}(g_2-g_4, f_3-f_4)}_{H^s(q\mathbb{T})}+C\norm{\Big(S_{\zeta_2}-S_{\zeta_4}\Big)(g_4, f_3-f_4)}_{H^s(q\mathbb{T})}.
	        \end{split}
	    \end{equation}
	\end{proposition}
For the proof of Proposition \ref{commutator:difference:cubic}, see Appendix \S \ref{proof:commutatordiffcubic}.

	\section{Local wellposedness, cubic structure }\label{section:structure}
In this section, we first write down the equations for the water waves system in Wu's coordinates directly and  derive explicit formulae for important quantities. Then we record the basic results on the local well-posedness of the water waves system. Finally, we derive a cubic structure for studying the long time existence of the water waves. The derivations of these formulae and the cubic structure were first used by Wu in \cite{Wu2009} in the Euclidean setting.

\subsection{Water waves in Wu's coordinates}
As discussed in \S\ref{subsubsection:lagrangian}, we formulate the water waves by the following
	\begin{equation}\label{system_newvariables}
	\begin{cases}
	(D_t^2-iA\partial_{\alpha})\zeta=-i\\
	D_t\bar{\zeta}, ~~\bar{\zeta}(\alpha,t)-\alpha\quad \text{holomorphic}
	\end{cases}
	\end{equation}
where
\begin{equation}\label{eq:Dt}
    D_{t}:=\partial_t+b\partial_{\alpha}
\end{equation}for some function $b$.

By Lemma \ref{lemma:holoboundary}, due to the holomorphic conditions in \eqref{system_newvariables}, we have
	\begin{equation}\label{holo:Dtzeta}
	    (I-\mathcal{H}_{\zeta})D_t\bar{\zeta}=0,
	\end{equation} and 
	\begin{equation}\label{holo:barzeta}
	    (I-\mathcal{H}_{\zeta})(\bar{\zeta}(\alpha,t)-\alpha)=0.
	\end{equation}
\subsection{Formulae for important quantities}
To obtain a closed system in \eqref{system_newvariables}, we need to derive formulae for $b$ and $A$ in terms of the unknown $\zeta$. 


	\subsubsection{Formula for $b$}
To derive a formula for $b$, by \eqref{eq:Dt}, we write
\begin{equation}
    D_t \bar{\zeta}= D_t (\bar{\zeta}-\alpha)+b.
\end{equation}
Applying $I-\mathcal{H}_\zeta$ on both sides of the equation above, by the holomorphic conditions \eqref{holo:Dtzeta}, \eqref{holo:barzeta} and the commutator identity \eqref{U4}, we have 
	\begin{equation}\label{formula:b}
(I-\mathcal{H}_{\zeta})b=-[D_t\zeta, \mathcal{H}_{\zeta}]\frac{\bar{\zeta}_{\alpha}-1}{\zeta_{\alpha}}.
	\end{equation}

	
		\subsubsection{Formula for $A$}
 From the momentum equation, the first equation in \eqref{system_newvariables}, we have 
	\begin{align}\label{eq:A1}
	    D_tD_t\bar{\zeta}+iA(\bar{\zeta}_{\alpha}-1)+iA=i.
	\end{align}
	Since $D_t\bar{\zeta}$ is holomorphic, it follows that
	\begin{equation}\label{eq:A11}
	    D_tD_t\bar{\zeta}=D_t\zeta \frac{\partial_{\alpha}D_t\bar{\zeta}}{\zeta_{\alpha}}+G(\zeta(\alpha,t),t),
	\end{equation}
	where $G$ is a holomorphic function in $\Omega(t)$ satisfies $G(x+iy, t)\rightarrow 0$ as $y\rightarrow-\infty$.
	
	Plugging \eqref{eq:A11} into \eqref{eq:A1} and then applying $I-\mathcal{H}_\zeta$ to the resulting equation, we obtain 
	\begin{equation}\label{convenient:expansion}
(I-\mathcal{H}_{\zeta})(A-1)=i[D_t\zeta, \mathcal{H}_{\zeta}]\frac{\partial_{\alpha}D_t\bar{\zeta}}{\zeta_{\alpha}}+i[D_t^2\zeta, \mathcal{H}_{\zeta}]\frac{\bar{\zeta}_{\alpha}-1}{\zeta_{\alpha}}.
	\end{equation}

	\subsubsection{Formula for quantities of the form $[D_t^2-iA\partial_{\alpha}, D_t]f$}\label{subsubsection:PDT}
	Next we derive a formula for quantities of the form $[D_t^2-iA\partial_{\alpha}, D_t]f$. In order to simplify the calculation of the commutators, we consider a change of variables $\kappa:\mathbb{R}\rightarrow \mathbb{R}$ defined by 
	\begin{equation}
	    \kappa_{t}\circ\kappa^{-1}:=b
	\end{equation}
where $b$ is given by \eqref{formula:b}. 
	Set $z(\alpha,t):=\zeta(\kappa(\alpha,t),t)$. We also need the quantity  $a$ which is defined by $(a\kappa_{\alpha})\circ\kappa^{-1}:=A$. Then $z$ solves
	\begin{equation}
	z_{tt}-iaz_{\alpha}=-i.
	\end{equation}
	Composing $[D_t^2-iA\partial_{\alpha}, D_t]f$ with $\kappa$ yields $[\partial_t^2-ia\partial_{\alpha}, \partial_t]f\circ\kappa$. So we obtain
	\begin{equation}\label{cal:lagrangian}
	    [\partial_t^2-ia\partial_{\alpha}, \partial_t]f\circ\kappa=ia_t (f\circ\kappa)_{\alpha}.
	\end{equation}
	Going back to the original coordinate by composing $\kappa^{-1}$ on both sides of (\ref{cal:lagrangian}), one has
	\begin{equation}\label{formula:PDT}
	    [D_t^2-iA\partial_{\alpha}, D_t]f=i\frac{a_t}{a}\circ\kappa^{-1}Af_{\alpha}.
	\end{equation}
	
\subsubsection{Formula for $\frac{a_t}{a}\circ\kappa^{-1}$}	
Using the same calculation as in \S \ref{subsubsection:PDT}, one has
	\begin{equation}\label{intermediate:ata}
	    (D_t^2+iA\partial_{\alpha})D_t\bar{\zeta}=-i\frac{a_t}{a}\circ\kappa^{-1}A\bar{\zeta}_{\alpha}.
	\end{equation}
	Applying $I-\mathcal{H}_{\zeta}$ on both sides of (\ref{intermediate:ata}), by (\ref{U2}), (\ref{U5}), the identity $iA\zeta_{\alpha}=D_t^2\zeta+i$, and $(I-\mathcal{H}_{\zeta})D_t\bar{\zeta}=0$, it follows
	\begin{align*}
	   & -i(I-\mathcal{H}_{\zeta})\frac{a_t}{a}\circ\kappa^{-1}A\bar{\zeta}_{\alpha}= (I-\mathcal{H}_{\zeta}) (D_t^2+iA\partial_{\alpha})D_t\bar{\zeta}
	    =[D_t^2+iA\partial_{\alpha}, \mathcal{H}_{\zeta}]D_t\bar{\zeta}\\
	    =& 2[D_t^2\zeta, \mathcal{H}_{\zeta}]\frac{\partial_{\alpha}D_t\bar{\zeta}}{\zeta_{\alpha}}+2[D_t\zeta, \mathcal{H}_{\zeta}]\frac{\partial_{\alpha}D_t^2\bar{\zeta}}{\zeta_{\alpha}}-\frac{1}{4\pi q^2 i}\int_{-q\pi}^{q\pi}\Big(\frac{D_t(\alpha,t)-D_t\zeta(\beta,t)}{\sin(\frac{1}{2q}(\zeta(\alpha,t)-\zeta(\beta,t)))}\Big)^2 \partial_{\beta}D_t\bar{\zeta}(\beta,t)\,d\beta
\end{align*}

%
	\subsection{Local wellposedness}
By formulae \eqref{convenient:expansion} and \eqref{formula:b},  \eqref{system_newvariables} is a closed fully nonlinear system. One way to achieve the well-posedness is to quasilinearize the system by differentiating it with respect to $D_t$. We only state the local well-posedness result here and  refer \cite{Wu1997} for the proof.


\begin{thm}[Local well-posedness]\label{localperiodic}
Let $s\geq 4$. 
Given $(\zeta_0, v_0)$ with $(\partial_{\alpha}\zeta_0-1, v_0)\in H^{s}(q\mathbb{T})\times H^{s+1/2}(q\mathbb{T})$, there is $T_0>0$ depending on $\norm{(\partial_{\alpha}\zeta_0-1, v_0)}_{H^{s}(q\mathbb{T})\times H^{s+1/2}(q\mathbb{T})}$ such that  the water waves system (\ref{system_boundary}) with initial data $(\zeta(\alpha,0), D_t(\alpha,0)):=(\zeta_0, v_0)$ has a unique solution $(\zeta(\cdot,t), D_t\zeta(\cdot,t))$ for $t\in [0,T_0]$, satisfying 
\begin{equation}
(\zeta_{\alpha}-1, D_t\zeta(\cdot,t), D_t^2\zeta(\cdot,t))\in C([0,T_0]; H^{s}(q\mathbb{T})\times H^{s+1/2}(q\mathbb{T})\times H^s(q\mathbb{T})).
\end{equation}
Moreover, if $T_{max}$ is the supremum over all such times $T_0$, then either $T_{max}=\infty$, or $T_{max}<\infty$, but
\begin{equation}\label{blowup:one}
\begin{split}
\lim_{t\uparrow T_{max}}  &\norm{(D_t\zeta, D_t^2\zeta)}_{H^{s}(q\mathbb{T})\times H^s(q\mathbb{T})}=\infty,
\end{split}
\end{equation}
or
\begin{equation}\label{blowup:two}
    \sup_{\alpha\neq \beta}\Big|\frac{\zeta(\alpha,t)-\zeta(\beta,t)}{\alpha-\beta}\Big|+\sup_{\alpha\neq \beta}\Big|\frac{\alpha-\beta}{\zeta(\alpha,t)-\zeta(\beta,t)}\Big|=\infty.
\end{equation}

\end{thm}
	
\begin{rem}\label{rem:s}
Throughout this paper, we fix a $s$ satisfying the condition in the theorem above.
\end{rem}

	\subsection{Cubic structure}
The key advantage to use Wu's coordinate is that it is well-suited to study the long-time existence of the water wave system. 
To obtain the long-time existence of $\zeta$, we need to derive a cubic structure for the system \eqref{system_newvariables}.

Setting $\theta:=(I-\mathcal{H}_\zeta)(\zeta-\alpha)$. Since $(I-\mathcal{H}_{\zeta})(\bar{\zeta}-\alpha)=0$, then one has
\begin{equation}
    \theta=(I-\mathcal{H}_{\zeta})(\zeta-\bar{\zeta}).
\end{equation}
Then we have 
	\begin{equation}\label{cubic}
	\begin{split}
	(D_t^2-iA\partial_{\alpha})\theta=&-2[D_t\zeta, \mathcal{H}_{\zeta}\frac{1}{\zeta_{\alpha}}+\bar{\mathcal{H}}_{\zeta}\frac{1}{\bar{\zeta}_{\alpha}}]\partial_{\alpha}D_t\zeta+\frac{1}{ 4\pi q^2 i}\int_{-q\pi}^{q\pi} \Big(\frac{D_t\zeta(\alpha)-D_t\zeta(\beta)}{\sin(\frac{1}{2q}(\zeta(\alpha)-\zeta(\beta)))}\Big)^2\partial_{\beta}(\zeta-\bar{\zeta})\,d\beta\\
	:=& G_1+G_2
	\end{split}
	\end{equation}
where $G_1$ and $G_2$ are cubic or higher-power nonlinearities.

Indeed, by (\ref{U6}), we have 
	\begin{equation}\label{eq:cubicstructureStokes}
	\begin{split}
	    &(D_t^2-iA\partial_{\alpha})\theta\\
	    =& (I-\mathcal{H}_{\zeta})(D_t^2-iA\partial_{\alpha})(\zeta-\bar{\zeta})-[(D_t^2-iA\partial_{\alpha}), \mathcal{H}_{\zeta}](\zeta-\bar{\zeta})\\
	    =& -2(I-\mathcal{H}_{\zeta})(D_t)^2\bar{\zeta}-2[D_t\zeta, \mathcal{H}_{\zeta}]\frac{\partial_{\alpha}D_t(\zeta-\bar{\zeta})}{\partial_{\beta}\zeta}\\
	    &+\frac{1}{ 4\pi q^2 i}\int_{-q\pi}^{q\pi} \Big(\frac{D_t\zeta(\alpha)-D_t\zeta(\beta)}{\sin(\frac{1}{2q}(\zeta(\alpha)-\zeta(\beta)))}\Big)^2\partial_{\beta}(\zeta-\bar{\zeta})\,d\beta.
	   \end{split}
 	\end{equation}
	Using $(I-\mathcal{H}_{\zeta})D_t\bar{\zeta}=0$, one has 
	\begin{equation}\label{eq:cancellationfromholomorphic}
	    -2(I-\mathcal{H}_{\zeta})D_t^2\bar{\zeta}=-2[D_t, \mathcal{H}_{\zeta}]D_t\bar{\zeta}=2[D_t\zeta, \mathcal{H}_{\zeta}]\frac{\partial_{\alpha}D_t\bar{\zeta}}{\partial_{\alpha}\zeta}.
	\end{equation}
Finally,	(\ref{cubic}) follows by combining (\ref{eq:cubicstructureStokes}) and (\ref{eq:cancellationfromholomorphic}). 
	
	\section{Stokes waves in Wu's coordinates}\label{subsec:stokeseq}
In this section, we study Stokes waves in Wu's coordinates. We will first write the equations for the Stokes waves and formulae for some important quantities as last section. The goal here is to introduce some notations specific to Stokes waves.  Then we will show the existence of small-amplitude Stokes waves in this coordinate and give the asymptotic expansions of them.

\subsection{Equations for Stokes waves}

	
		We denote a given Stokes wave as $\zeta_{ST}$. It is a special solution to \eqref{system_newvariables}.  As in Section \ref{section:structure}, we denote $$D_t^{ST}:=\partial_t+b_{ST}\partial_{\alpha},$$ where $b_{ST}$ is given by
			\begin{equation}\label{formula:bST}
		    (I-\mathcal{H}_{\zeta_{ST}})b_{ST}=-[(D_t^{ST})\zeta_{ST}, \mathcal{H}_{\zeta_{ST}}]\frac{\partial_{\alpha}\bar{\zeta}_{ST}-1}{\partial_{\alpha}\zeta_{ST}}.
		\end{equation}
Let $A_{ST}$ be a real valued function given by
\begin{equation}\label{formula:Astokes}
(I-\mathcal{H}_{\zeta_{ST}})(A_{ST}-1)=i[D_t^{ST}\zeta_{ST}, \mathcal{H}_{\zeta_{ST}}]\frac{\partial_{\alpha}D_t^{ST}\bar{\zeta}_{ST}}{\partial_{\alpha}\zeta_{ST}}+i[(D_t^{ST})^2\zeta_{ST}, \mathcal{H}_{\zeta_{ST}}]\frac{\partial_{\alpha}\bar{\zeta}_{ST}-1}{\partial_{\alpha}\zeta_{ST}}
\end{equation}
Then by \eqref{system_newvariables}, we have
	\begin{equation}\label{eq:stokeseqnew}
	    \begin{cases}
	    \Big((D_t^{ST})^2-iA_{ST}\partial_{\alpha}\Big)\zeta_{ST}=-i\\
	    (I-\mathcal{H}_{\zeta_{ST}})D_t^{ST}\bar{\zeta}_{ST}=0, \\
	     (I-\mathcal{H}_{\zeta_{ST}})(\bar{\zeta}_{ST}-\alpha)=0.
	    \end{cases}
	\end{equation}
The same as the general setting in Section \ref{section:structure} we have the cubic structure for the Stokes wave $\zeta_{ST}$.
Following \eqref{cubic}, the quantity
	\begin{equation}
	    \theta_{ST}(\alpha,t):=(I-\mathcal{H}_{\zeta_{ST}})(\zeta_{ST}-\alpha).
	\end{equation}
satisfies the cubic equation:
		\begin{equation}\label{cubicStokes}
	\begin{split}
	((D^{ST}_t)^2-iA_{ST}\partial_{\alpha})\theta_{ST}=&-2[D^{ST}_t\zeta_{ST}, \mathcal{H}_{\zeta_{ST}}\frac{1}{\partial_{\alpha}\zeta_{ST}}+\bar{\mathcal{H}}_{\zeta_{ST}}\frac{1}{\partial_{\alpha}\bar{\zeta}_{ST}}]\partial_{\alpha}D^{ST}_t{\zeta}_{ST}\\
	&+\frac{1}{ 4\pi q^2 i}\int_{-q\pi}^{q\pi} \Big(\frac{D^{ST}_t\zeta_{ST}(\alpha)-D^{ST}_t\zeta_{ST}(\beta)}{\sin(\frac{1}{2q}(\zeta_{ST}(\alpha)-\zeta_{ST}(\beta)))}\Big)^2\partial_{\beta}(\zeta_{ST}-\bar{\zeta}_{ST})\,d\beta\\
	:=& G_{ST,1}+G_{ST,2}.
	\end{split}
	\end{equation}
		\subsection{Existence of Stokes waves of small amplitude}\label{appendix:existencestokes}
In this subsection we prove Proposition \ref{prop:Stokes}. Our proof is based on the existence and uniqueness of Stokes waves in Eulerian coordinates. To begin with, consider the fluid domain
\begin{equation}
    \Omega(t)=\{(x,y): x\in\mathbb{R}, y<\eta(x,t)\},
\end{equation}
and $\Sigma(t)=\{(x, \eta(x,t): x\in\mathbb{R}\}$. Let $v$ be the velocity field as in \eqref{euler}. Denote
\begin{equation}
    u(x,t):=v(x, \eta(x,t),t).
\end{equation}
One has the following result on the existence of Stokes waves.
\begin{proposition}\label{prop:existenceInEuler}
    There exists a curve of smooth solutions $(\omega(\epsilon), \eta(\epsilon), u(\epsilon))$ to (\ref{euler}) parametrized by a small parameter $|\epsilon|\ll 1$ which we call the amplitude of $\eta$. For each solution $(\omega, \eta, u)$ on this curve, one can write it as
    \begin{equation*}
        \eta(x,t)=\eta_0(x+\omega t), \quad \quad   u(x,t)=u_0(x+\omega t),
    \end{equation*}
    where $\eta_0$ and $u_0$ satisfy the following properties:
\begin{itemize}
    \item [(i)] $\eta_0$ and $u_0$ are $2\pi$ periodic smooth functions.
    
    \item [(ii)] $\eta_0$ is even, $\Re\{u_0\}$ is odd and $\Im\{u_0\}$ is even.
\end{itemize}
Other than the trivial solutions (with $\eta=0$), the curve is unique. Moreover, for any given $k\in \mathbb{N}$, the following estimates hold
\begin{equation}
    \norm{\eta_0}_{H^{k}(\mathbb{T})}+\norm{u_0}_{H^{k}(q\mathbb{T})}\leq C_k\epsilon,
\end{equation}
for some constant $C_k$ depending on $k$ only.
\end{proposition}
Proposition \ref{prop:existenceInEuler} has been known for a century, see \cite{nekrasov1921steady} and \cite{Levi-Civita}. Also see Theorem 2.1 in  \cite{nguyen2020proof} for the version in the Zakharov-Craig-Sulem formulation.

\begin{rem}
We shall also use $(\omega(\epsilon), \eta_0(\epsilon), u_0(\epsilon))$ to represent the curve of solutions constructed in Proposition \ref{prop:existenceInEuler}.
\end{rem}

Using Proposition \ref{prop:existenceInEuler}, we obtain the existence of Stokes waves in the Lagrangian formulation.

\begin{proposition}\label{prop:eulertolagrangian}
    Let $(\omega(\epsilon), \eta_0(\epsilon), u_0(\epsilon))$ be a given Stokes wave of amplitude $\epsilon$ as given in Proposition \ref{prop:existenceInEuler}. There exists a unique  odd smooth function $g:\mathbb{T}\rightarrow \mathbb{R}$, such that if we denote
    \begin{equation}
        x(\alpha,t):=\alpha+g(\alpha+\omega t),
    \end{equation}
 \begin{equation}
     z(\alpha,t):=x(\alpha,t)+i\eta_0\big(x(\alpha,t)+\omega t\big), 
 \end{equation}
 then 
    \begin{equation}
        z_t(\alpha,t)=v(z(\alpha,t),t)=u_0(x+\omega t).
    \end{equation}
\end{proposition}
\begin{proof}
We prove the existence. The uniqueness follows easily. Suppose we have constructed $g$, since $g$ is odd, we have $g(\alpha)=\int_0^{\alpha} g'(\beta)\,d\beta$. Hence $g$ is determined by $g'$. 

Differentiating the expression of $z$ with respect to $t$, one has
\begin{equation}
    z_t(\alpha,t)=\omega g'(\alpha+\omega t)+i(\omega+\omega g'(\alpha+\omega t))\eta_0'(\alpha+\omega t+g(\alpha+\omega t)).
\end{equation}
We want to find $g$ to satisfy
\begin{equation}\label{eq: gderivative}
    \omega g'(\alpha+\omega t)+i(\omega +\omega g'(\alpha+\omega t))\eta_0'(\alpha+\omega t+g(\alpha+\omega t))=u_0(x(\alpha,t)+\omega t).
\end{equation}
Let $\Gamma:=\alpha+\omega t$. Then (\ref{eq: gderivative}) can be written as
\begin{equation}\label{eq:gderivativeequivalent}
    g'(\Gamma)=-i(1+g'(\Gamma))\eta_0'(\Gamma+g(\Gamma))+\frac{1}{\omega }u_0(\Gamma+g(\Gamma)).
\end{equation}
Integrating both sides of (\ref{eq:gderivativeequivalent}) yields
\begin{equation}
    g(\Gamma)=\int_0^{\Gamma} \Big(-i(1+g'(\beta))\eta_0'(\beta+g(\beta))+\frac{1}{\omega }u_0(\beta+g(\beta))\Big) \,d\beta.
\end{equation}
From the expression above, to find $g$ is equivalent to obtain a fixed point.  We use the standard iteration to find the fixed point.

Set
\begin{equation}
    g_0(\Gamma)=0, \quad  z_0(\alpha,t):=\alpha+g_0(\Gamma)+i\eta_0(\Gamma+g_0(\Gamma)).
\end{equation}
Define
\begin{equation}
    g_1(\Gamma):=\int_0^{\Gamma} \Big\{-i(1+g_0'(\beta))\eta_0'(\beta+g_0(\beta))+\frac{1}{\omega }u_0(\beta+g_0(\beta))\Big\}\,d\beta.
\end{equation}
Assume $g_n$ has been defined, then we can define
\begin{equation}
  z_n(\alpha,t):=\alpha+g_n(\Gamma)+i\eta_0(\Gamma+g_n(\Gamma)).
\end{equation}
Given $g_n$, we then define
\begin{equation}
    g_{n+1}(\Gamma):=\int_0^{\Gamma} \Big\{-i(1+g_n'(\beta))\eta_0'(\beta+g_n(\beta))+\frac{1}{\omega }u_0(\beta+g_n(\beta))\Big\}\,d\beta.
\end{equation}
It is straightforward to check that by construction, the smoothness of $u_0$ and $\eta_0$  implies that $g_n$  is smooth. Moreover, for each $k\in \mathbb{N}$,
\begin{equation}\label{estimates:banach}
    \norm{g_{n+1}-g_n}_{H^{k}(\mathbb{T})}\leq C_k\epsilon \norm{g_{n}-g_{n-1}}_{H^{k}(\mathbb{T})}
\end{equation}
where $C_k$ is a constant depending on $k$ only.
For  sufficiently small $\epsilon$ such that $C_0\epsilon<1$, the standard Banach fixed point theorem gives $g_n\rightarrow g\in L^2(\mathbb{T})$. Using the smoothness of $g_n$ and the estimates (\ref{estimates:banach}), one obtains that $g$ is also smooth. Finally define
\begin{equation}
    x(\alpha,t):=\alpha+g(\alpha+\omega t), \quad z(\alpha,t):=x(\alpha,t)+i\eta_0(x(\alpha,t)+\omega t).
\end{equation}
By construction, we have $z_t=v(z(\alpha,t),t)$ as desired.
\end{proof}
\begin{cor}
Let $(\omega , \eta_0, u_0)$ be a given Stokes wave of amplitude $\epsilon$ given in Proposition \ref{prop:existenceInEuler}, and $z(\alpha,t)$ the (Lagrangian) parametrization of of the interface as given in Proposition \ref{prop:eulertolagrangian}. There  is a real-valued function $a$ such that $z(\alpha,t)$ satisfies
\begin{equation}\label{system:lagrangiancoordinates}
\begin{cases}
    z_{tt}-iaz_{\alpha}=-i\\
    \bar{z}_t\quad \text{holomorphic}.
\end{cases}
\end{equation}
\end{cor}
\begin{proof}
The restriction of $v_t+v\cdot \nabla v$ on $\Sigma(t)$ can be written as $z_{tt}$. Since $P(z(\alpha,t),t)\equiv 0$, we have $\nabla P(z(\alpha,t),t)$ is in the direction normal to $\Sigma(t)$. So there is a real-valued function $a$ such that $\nabla P\Big|_{\Sigma(t)}=-iaz_{\alpha}$. We are done.
\end{proof}
With preparations above, we obtain the existence of small-amplitude Stokes waves in Wu's coordinates.

	\begin{thm}\label{thm:travelinginWu}
	There exists $\epsilon_0>0$ such that for all $\epsilon\in (0,\epsilon_0]$, there is a unique solution to the system (\ref{system_boundary}) such that 
	\begin{itemize}
	    \item [(A)] $\zeta(\alpha,t)=\alpha+F(\alpha+\omega t)$ and $D_t\zeta(\alpha,t)=G(\alpha+\omega t)$ for some $\omega\in\mathbb{R}$ and some $2\pi$ periodic smooth functions $F$ and $G$.
	    \item [(B)] $\Re\{F\}$ and $\Re\{G\}$ are odd, $\Im\{F\}$ and $\Im\{G\}$ are even.
	    
	    \item [(C)] $(I-\mathcal{H}_{\zeta})(\bar{\zeta}-\alpha)=0$, and $(I-\mathcal{H}_{\zeta})D_t\bar{\zeta}=0$.

	    \item [(D)] For each $k\in \mathbb{N}$, one has
	    \begin{equation}
	        \norm{F}_{H^k(\mathbb{T})}+\norm{G}_{H^k(\mathbb{T})}\leq C_k\epsilon,
	    \end{equation}
	    for some constant $C_k$ depending only on $k$.
	    
	\end{itemize}
	\end{thm}
	
	\begin{proof}
Given a  diffeomorphism $\kappa:\mathbb{R}\rightarrow \mathbb{R}$, we denote $\zeta(\alpha,t):=z\circ\kappa^{-1}(\alpha,t)$. Define $\kappa$ by
\begin{equation}
\kappa_t\circ\kappa^{-1}=b,
\end{equation}
where $b$ is given as in \eqref{formula:b}, and define $A$ by $A:=(a\kappa_{\alpha})\circ\kappa^{-1}$. Then any solution to \eqref{system:lagrangiancoordinates} gives to a solution to 
\begin{equation}\label{system:goodsystem}
    \begin{cases}
    (D_t^2-iA\partial_{\alpha})\zeta=-i\\
    (I-\mathcal{H}_{\zeta})D_t\bar{\zeta}=0, \quad (I-\mathcal{H}_{\zeta})(\bar{\zeta}-\alpha)=0.
    \end{cases}
\end{equation}
In particular, the  given Stokes wave $(\omega, \eta_0, u_0)$ from Proposition \ref{prop:eulertolagrangian} gives a solution $(\zeta, D_t\zeta)$ to (\ref{system:goodsystem}). It is straightforward to check (A)-(B)-(C)-(D) of the proposition. We are done.
	\end{proof}


 As a direct consequence of Theorem \ref{thm:travelinginWu}, we obtain the following.
 \begin{cor}\label{cor:travelingAB}
 Let $(\omega, \zeta_{ST}, D_t^{ST}\zeta_{ST})$ be a solution to (\ref{eq:stokeseqnew}) as given in Theorem \ref{thm:travelinginWu}. Then there exist $\phi_1, \phi_2\in C^{\infty}(\mathbb{T})$ such that 
 \begin{equation}
     b_{ST}(\alpha,t)=\phi_1(\alpha+\omega t), \quad \quad A_{ST}=\phi_2(\alpha+\omega t).
 \end{equation}
 \end{cor}

\subsection{Asymptotic expansion of the Stokes waves in Wu's coordinates}\label{subsec:stokesexpansion}
Finally, we present the asymptotic expansion of the Stokes wave. It is worthwhile to point out that in up to $\epsilon^4$, the expansion only has one non-trivial frequency. 
\begin{proposition}\label{prop:stokesexpansion}
    Let $\zeta_{ST}$ be a Stokes wave of amplitude $\epsilon$ and period $2\pi$. Then we have 
    \begin{equation}
        \zeta_{ST}(\alpha,t)=\alpha+i\epsilon e^{i\alpha+i\omega t}+\epsilon^2 i+\frac{i}{2}\epsilon^3 e^{-i\alpha-i\omega t}+\mathcal{O}(\epsilon^4),
    \end{equation}
    and 
    \begin{equation}
        \omega=1+\epsilon^2/2+\mathcal{O}(\epsilon^3).
    \end{equation}
\end{proposition}
	
In the remaining part of this subsection, we give detailed computations to show the proposition above.

	Given a Stokes wave $\zeta_{ST}$ of amplitude $\epsilon$ with velocity $\omega$, 
	we assume that the Stokes wave has an expansion of the form
	\begin{equation}\label{eq:ZetaST}
	    \zeta_{ST}=\alpha+\epsilon\zeta_{ST}^{(1)}+\epsilon^2\zeta_{ST}^{(2)}+\epsilon^3 \zeta_{ST}^{(3)}+\mathcal{O}(\epsilon^4),
	\end{equation}
	where we can further write
	\begin{equation}\label{eq:zetaSTtraveling}
	    \zeta_{ST}^{(j)}(\alpha,t)=Z_{ST}^{(j)}(\alpha+\omega t),\quad j=1,2,3
	\end{equation} 	 with some  
	$2\pi$-periodic  function $Z_{ST}^{(j)}$.

	For the velocity $\omega$, we assume that it has the expansion
	\begin{equation}
	    \omega=\omega^{(0)}+\epsilon \omega^{(1)} +\epsilon^2 \omega^{(2)} +\mathcal{O}(\epsilon^3).
	\end{equation}
	Our goal now is to  compute $\zeta_{ST}^{(j)},\,j=1,2,3$ and $\omega^{(j)},\,j=0,1,2,3$.
	
	\subsubsection{Computations for $\zeta^{(1)}_{ST}$}
	First of all, expanding the momentum equation, the first equation of \eqref{eq:stokeseqnew}, in terms of powers of $\epsilon$, at the level of $\epsilon$, we obtain
	\begin{equation}\label{eq:zetast1expand}
	    (\partial^2_t-i\partial_\alpha)\zeta_{ST}^{(1)}=0.
	\end{equation}
Then in terms of $Z_{ST}^{(1)}$, in the leading order, one has
\begin{equation}\label{eq:Zst1}
    ((\omega^{(0)})^2\partial^2_\alpha-i\partial_\alpha)Z^{(1)}_{ST} (\alpha)=0.
\end{equation}
Expanding $Z^{(1)}_{ST}$ using its Fourier series on $\mathbb{T}$, we get\begin{equation}
    Z^{(1)}_{ST}(\alpha)= \sum_{k=-\infty}^{\infty} h_k e^{ik\alpha}
\end{equation}where $h_k$ is the $k$th Fourier coefficient of $Z_{ST}^{(1)}$.

Writing \eqref{eq:Zst1} in terms of the Fourier series above, we  notice that it only allows one non-trivial Fourier mode. Since we consider the Stokes wave with fundamental period $2\pi$, we take $k=1$. Then \eqref{eq:Zst1} implies
\begin{equation}\label{eq:omega0}
    \omega^{(0)}=1
\end{equation}
To determine $h_1$, we recall that from Theorem \ref{thm:travelinginWu}, $\Re\{Z^{(1)}_{ST}\}$ is odd and $\Im\{Z^{(1)}_{ST}\}$ is even,  so $h_1$ has to be purely imaginary. By a Stokes wave of amplitude $\epsilon$, we really mean $\norm{\epsilon\zeta_{ST}^{(1)}}_{L^{\infty}}=\epsilon$.
So $Z_{ST}^{(1)}=\pm ie^{i\alpha}$. Without loss of generality, let's take 
	\begin{equation}\label{eq:zetaST1}
	    \zeta_{ST}^{(1)}=ie^{i(\alpha+\omega t)}.
	\end{equation}
	Next, we expand $b_{ST}$ and $A_{ST}$ as
	\begin{equation}
	    b_{ST}=\epsilon^2 b_{ST}^{(2)}+\epsilon^3 b_{ST}^{(3)}+O(\epsilon^4),
	\end{equation}
and
	\begin{equation}
	    A_{ST}=1+\epsilon^2 A_{ST}^{(2)}+\epsilon^3 A_{ST}^{(3)}+O(\epsilon^3)
	\end{equation}
where we used fact that $b_{ST}$ and $A_{ST}-1$ are at least quadratic in $\epsilon$ from formulae \eqref{formula:Astokes} and \eqref{formula:bST}.

We need to compute $b_{ST}^{(2)}$ and $A_{ST}^{(2)}$ in order to find $\zeta^{(2)}_{ST}$.
	\begin{proposition}\label{prop:bstAst2}Given notations above,  we have 
	    \begin{equation}
	        b_{ST}^{(2)}=-1,\,	        A_{ST}^{(2)}=0.
	    \end{equation}
	\end{proposition}
	\begin{proof}
	Expanding (\ref{formula:bST}) in terms of powers of $\epsilon$, at the $\epsilon^2$ level, we compute
	\begin{align}
	    (I-H_0)b_{ST}^{(2)}=&-[\partial_t \zeta_{ST}^{(1)}, H_0]\partial_{\alpha}\bar{\zeta}_{ST}^{(1)}\\
	    =&-\omega_0[\zeta_{ST}^{(1)}, H_0]\bar{\zeta}_{ST}^{(1)}\\
	    =& \omega_0[\zeta_{St}^{(1)}, I-H_0]\bar{\zeta}_{ST}^{(1)}\\
	    =& -\omega_0(I-H_0)|\zeta_{ST}^{(1)}|^2\\
	    =& -\omega_0 =-1.
	\end{align}
 So we get $b_{ST}^{(2)}=-1.$
	Using the same calculations, we obtain $A_{ST}^{(2)}=0$.
	\end{proof}
	
	\subsubsection{Computations for $\zeta_{ST}^{(2)}$} 
	
	At the $\epsilon^2$ level, the holomorphic condition  $(I-\mathcal{H}_{\zeta_{ST}})(\bar{\zeta}_{ST}-\alpha)=0$ from the equation \eqref{eq:stokeseqnew} implies
	\begin{align}
	    (I-H_0)\bar{\zeta}_{ST}^{(2)}=& H_1\bar{\zeta}_{ST}^{(1)}  = [\zeta_{ST}^{(1)}, H_0]\partial_{\alpha}\bar{\zeta}_{ST}^{(1)}\\
	    =& -i[\zeta_{ST}^{(1)}, H_0] \bar{\zeta}_{ST}^{(1)}= -i(I-H_0)|\zeta_{ST}^{(1)}|^2\\
	    =&-i.
	\end{align}
	Therefore, we can conclude that
	\begin{equation}\label{eq:zetaST2}
	    \zeta_{ST}^{(2)}=i.
	\end{equation}
	Note that such choice guarantees the $O(\epsilon^2)$ terms of $(I-\mathcal{H}_{\zeta_{ST}})D_t^{ST}\bar{\zeta}_{ST}=0$ which is also part of \eqref{eq:stokeseqnew}.
	
From the equation \eqref{eq:stokeseqnew}, and Proposition \ref{prop:bstAst2}, we know $$\Big((D_t^{ST})^2-iA_{ST}\partial_{\alpha}\Big)\zeta_{ST}=-i,\,b_{ST}^{(2)}=-1,\, A_{ST}^{(2)}=0.$$
Expanding the equations above in terms of powers of $\epsilon$ and using \eqref{eq:ZetaST}, we have 
	\begin{equation}\label{eq:zetaSTepsilon2}
	    \partial_{t}^2\zeta_{ST}^{(2)}+2b_{ST}^{(2)}\partial_{\alpha}\partial_t \alpha-i\partial_{\alpha}\zeta_{ST}^{(2)}=2\omega^{(1)}\zeta_{ST}^{(1)}.
	\end{equation}
Plugging \eqref{eq:zetaST2} back into the equation \eqref{eq:zetaSTepsilon2}, we obtain that
\begin{equation}\label{eq:omega1}
    \omega^{(1)}=0.
\end{equation}

	To find $\zeta_{ST}^{(3)}$, we will analyze the expansion of the equation \eqref{eq:stokeseqnew}  in the $\epsilon^3$ level. After choosing $\zeta_{ST}^{(1)}$ and $\zeta_{ST}^{(2)}$, in order to find $\zeta_{ST}^{(3)}$,  we need to compute $A^{(3)}_{ST}$, $G^{(3)}_{ST,1}$ and $G^{(3)}_{ST,2}$ which are the $\epsilon^3$ levels of $1-A_{ST}$, $G_{ST,1}$ and $G_{ST,2}$. We have the following conclusion.
	\begin{proposition}Using notations above, we have
	    $$A_{ST}^{(3)}=0,\, G_{ST,1}^{(3)}=0,\,G_{ST,2}^{(3)}=2|\zeta_{ST}^{(1)}|^2\zeta_{ST}^{(1)}.$$
	\end{proposition}
	\begin{proof}
	The results for $G_{ST,1}^{(3)}$ and $G_{ST,2}^{(3)}$ follow from direct inspection.

	To compute $A_{ST}^{(3)}$, using $\zeta_{ST}^{(2)}=i$, from \eqref{formula:Astokes}, we have 
	\begin{align}
	    (I-H_0)A_{ST}^{(3)}=& i[\partial_{t}^2\zeta_{ST}^{(1)}, H_1]\partial_{\alpha}\bar{\zeta}_{ST}^{(1)}+i[\partial_t \zeta_{ST}^{(1)}, H_1]\partial_{\alpha}\partial_t\bar{\zeta}_{ST}^{(1)}\\
	    =& i\partial_t^2\zeta_{ST}^{(1)}H_1\partial_{\alpha}\bar{\zeta}_{ST}^{(1)}-i H_1 \partial_t^2\zeta_{ST}^{(1)}\partial_{\alpha}\bar{\zeta}_{ST}^{(1)}+i\partial_t\zeta_{ST}^{(1)}H_1\partial_{\alpha}\partial_t\bar{\zeta}_{ST}^{(1)}-iH_1\partial_t\zeta_{ST}^{(1)}\partial_{\alpha}\partial_t\bar{\zeta}_{ST}^{(1)}\\
	    =& i(i\omega)^2(-i)\zeta_{ST}^{(1)}[\zeta_{ST}^{(1)},H_0]\partial_{\alpha}\bar{\zeta}_{ST}^{(1)}-i(i\omega)^2(-i)[\zeta_{ST}^{(1)}, H_0]\partial_{\alpha}(\zeta_{ST}^{(1)}\bar{\zeta}_{ST}^{(1)})\\
	    & +i(i\omega)(-i\omega)(-i)\zeta_{ST}^{(1)}[\zeta_{ST}^{(1)}, H_0]\partial_{\alpha}\bar{\zeta}_{ST}^{(1)}-i(i\omega)(-i\omega)(-i)[\zeta_{ST}^{(1)}, H_0]\partial_{\alpha}(\zeta_{ST}^{(1)}\bar{\zeta}_{ST}^{(1)})\\
	    =&0.
	\end{align}
	So $A_{ST}^{(3)}=0$  as desired.
	\end{proof}
	\subsubsection{Computations for $\zeta_{ST}^{(3)}$}
	The constraint $(I-\mathcal{H}_{\zeta_{ST}})(\bar{\zeta}_{ST}-\alpha)=0$ implies
\begin{align}
    (I-H_0)\bar{\zeta}_{ST}^{(3)}=& H_2\bar{\zeta}_{ST}^{(1)}+H_1\bar{\zeta}_{ST}^{(2)}\\
    =& [\zeta_{ST}^{(2)}, H_0]\partial_{\alpha}\bar{\zeta}_{ST}^{(1)}-[\zeta_{ST}^{(1)}, H_0]\partial_{\alpha}\zeta_{ST}^{(1)}\partial_{\alpha}\bar{\zeta}_{ST}^{(1)}+\frac{1}{2}[\zeta_{ST}^{(1)},[\zeta_{ST}^{(1)},H_0]]\partial_{\alpha}^2\bar{\zeta}_{ST}^{(1)}.
\end{align}
Then we compute each term  on the right-hand side of the equation above.

First of all,  note that $\zeta_{ST}^{(2)}=i$  and the choice of $\zeta_{ST}^{(1)}$ implies
$$[\zeta_{ST}^{(2)},H_0]\partial_{\alpha}\bar{\zeta}_{ST}^{(1)}=0.$$
Next, by the explicit formula of $\zeta^{(1)}_{ST}$  \eqref{eq:zetaST1},  we have 
	\begin{align}
	   -[\zeta_{ST}^{(1)}, H_0]\partial_{\alpha}\zeta_{ST}^{(1)}\partial_{\alpha}\bar{\zeta}_{ST}^{(1)}=&-(i)(-i)[\zeta_{ST}^{(1)},H_0]|\zeta_{ST}^{(1)}|^2=-\zeta_{ST}^{(1)}.
	\end{align}
Finally, invoking the explicit formula for $\zeta^{(1)}_{ST}$ again, one has
\begin{align}
    \frac{1}{2}[\zeta_{ST}^{(1)},[\zeta_{ST}^{(1)},H_0]]\partial_{\alpha}^2\bar{\zeta}_{ST}^{(1)}=& \frac{1}{2}\zeta_{ST}^{(1)}[\zeta_{ST}^{(1)}, H_0]\partial_{\alpha}^2\bar{\zeta}_{ST}^{(1)}-\frac{1}{2}[\zeta_{ST}^{(1)}, H_0]\zeta_{ST}^{(1)}\partial_{\alpha}^2\bar{\zeta}_{ST}^{(1)}\\
    =& (-i)^2\frac{1}{2}\zeta_{ST}^{(1)}\zeta_{ST}^{(1)}H_0\bar{\zeta}_{ST}^{(1)}-(-i)^2\frac{1}{2}H_0|\zeta_{ST}^{(1)}|^2\\
    &-\frac{1}{2}(-i)^2\zeta_{ST}^{(1)}H_0|\zeta_{ST}^{(1)}|^2+\frac{1}{2}(-i)^2H_0 \zeta_{ST}^{(1)}|\zeta_{ST}^{(1)}|^2\\
    =& 0.
\end{align}	
Putting things together,  we obtain
	\begin{equation}
	    (I-H_0)\bar{\zeta}_{ST}^{(3)}=-\zeta_{ST}^{(1)}.
	\end{equation}
Therefore we can take 
	\begin{equation}\label{eq:zetaST3}
	    \zeta_{ST}^{(3)}=  \frac{i}{2}e^{-i\alpha-i\omega t}
	\end{equation}
To find $\omega^{(2)}$, we notice that from \eqref{eq:zetaST1}, it follows
	\begin{equation}
	\begin{split}
	    (\partial_t^2-i\partial_{\alpha})(I-H_0)\zeta_{ST}^{(1)}=&2((i\omega)^2+1)\zeta_{ST}^{(1)}\\
	    =&2(-\omega^2+1)\zeta_{ST}^{(1)}\\
	    =&-4\omega^{(2)}\epsilon^2 \zeta_{ST}^{(1)}+O(\epsilon^3).
	    \end{split}
	\end{equation}
Using \eqref{cubicStokes}, $b_{ST}^{(2)}=-1$ and $A_{ST}^{(2)}=0$, at the leading level, we have
	\begin{align}
	    (\partial_t^2-i\partial_{\alpha})(I-H_0)\zeta_{ST}^{(3)}=&-2b_{ST}^{(2)}\partial_{\alpha}\partial_t(I-H_0)\zeta_{ST}^{(1)}+G_{ST,2}^{(3)}-(\partial_t^2-i\partial_{\alpha})(I-H_0)\zeta_{ST}^{(1)}\\
	    =& -4(-1)(i)(i\omega)\zeta_{ST}^{(1)}+2|\zeta_{ST}^{(1)}|^2\zeta_{ST}^{(1)}+4\omega^{(2)}\zeta_{ST}^{(1)}\\
	   =&(4\omega^{(2)}-2)\zeta^{(1)}_{ST}.
\end{align}
Therefore, one has
\begin{equation}\label{eq:zetaSTepsion3}
    (\partial_t^2-i\partial_{\alpha})(I-H_0)\zeta_{ST}^{(3)}=	(4\omega^{(2)}-2)\zeta^{(1)}_{ST}.
\end{equation}
Plugging \eqref{eq:zetaST3} into the equation \eqref{eq:zetaSTepsion3}, duo the holomorphicity, we obtain
\begin{equation}\label{eq:omega2}
    \omega^{(2)}=\frac{1}{2}.
\end{equation}
From all computations above, we conclude
\begin{equation}
    	    \omega=1 + \frac{1}{2}\epsilon^2+\mathcal{O}(\epsilon^3).
\end{equation}
and
\begin{equation}
    \zeta^{(1)}_{ST}=ie^{i\alpha+\omega t},\quad\zeta_{ST}^{(2)}=i,\quad\zeta_{ST}^{(3)}=\frac{i}{2}e^{-i\alpha-\omega t}.
    \end{equation}
We finish the proof of Proposition \ref{prop:stokesexpansion}.

\subsubsection{Approximation of the Stokes wave}
Using the $\zeta_{ST}^{(j)}$ obtained above, we are able to defined an approximate form of the Stokes wave.

Define the approximation of the Stokes wave $\tilde{\zeta}_{ST}$ as
\begin{equation}
    \tilde{\zeta}_{ST}= \alpha+\epsilon i e^{i(\alpha+\omega t)}+ \epsilon^2 i+ \epsilon^3 \frac{i}{2}e^{-i\alpha-i\omega t}.
\end{equation}

Then by construction, we know that
\begin{align*}
    |\zeta_{ST}-\tilde{\zeta}_{ST}|=O(\epsilon^4).
\end{align*}
Here we note that in  $\tilde{\zeta}_{ST}$ only one non-trivial fundamental frequency is involved. This fact will be crucial in the analysis.

	\section{Multiscale analysis and derivation of the NLS from the full water waves}\label{section:multiscale}
	In this section, we shall use the water waves system (\ref{system_newvariables}) to perform the multiscale analysis and derive the NLS.
	
	\subsection{Basic setting}
Given a Stoke wave of amplitude $\epsilon$ from Proposition \ref{prop:Stokes} and Section \S \ref{subsec:stokeseq}, by Proposition \ref{prop:stokesexpansion}, it has the following expansion in terms of $\epsilon$:
	\begin{align}	\label{eq:Stokes1}
    {\zeta}_{ST}(\alpha,t)=&\alpha+\epsilon \zeta^{(1)}_{ST}+ \epsilon^2 \zeta^{(2)}_{ST}+\epsilon^3 \zeta^{(3)}_{ST}+\dots\\
    =&\alpha+i\epsilon e^{i(\alpha+\omega t)}+\epsilon^2 i +\frac{1}{2}\epsilon^3 e^{-i(\alpha+\omega t)}+\dots.\nonumber
	\end{align}
Consider the solution to the system \eqref{system_newvariables} as a perturbation of the Stokes wave above in $H^s({q\mathbb{T}})$. Expanding the solution  in terms of the power of $\epsilon$, we seek for $\zeta(\alpha,t)$ of the form
	\begin{equation}\label{eq:zetaexpansion1}
	\zeta(\alpha,t)=\alpha+\epsilon \zeta^{(1)}(\alpha,t)+\epsilon^2 \zeta^{(2)}(\alpha,t)+\epsilon^3\zeta^{(3)}(\alpha,t)+...
	\end{equation}
where $\zeta^{(1)}$ satisfies the ansatz
	\begin{equation}\label{ansatz:zeta1}
	\zeta^{(1)}=B(\alpha_1,t_1,t_2)e^{i(\alpha+\omega t)},\quad \quad \alpha_1:=\epsilon\alpha, t_1:=\epsilon t, t_2:=\epsilon^2 t.
	\end{equation}
for some periodic function $B$ on $q_1\mathbb{T}$ wtih $q_1:=\epsilon q$ from the view of the modulational approximation and the long-wave perturbation. 


Since $\zeta$ is the perturbation of $\zeta_{ST}$, from the coefficient of $\epsilon$ in \eqref{eq:Stokes1}, we will choose
$$\quad \quad B=i(1+\psi).$$
Eventually, we will show that in order make the expansion \eqref{eq:zetaexpansion1} valid,  $B$ will admit the form $B(X,T)$ with $X=\epsilon(\alpha+\frac{1}{2\omega}t)$, $T=\epsilon^2 t$ and solve
	$$iB_T+\frac{1}{8} B_{XX}+\frac{1}{2} |B|^2B-\frac{1}{2} B=0.$$
In the remaining part of this section, we first give some estimates for almost holmorphicity of wavepackets which will be useful to handle functions with slow variables.  With these preparations, we then analyze the expansion for $A-1$, $b$ from \eqref{system_newvariables} and find $\zeta^{(1)}$, $\zeta^{(2)}$, $\zeta^{(3)}$ in \eqref{eq:zetaexpansion1}.

	\subsection{Almost holomorphicity of wavepackets}
	Let $\lambda$ be an integer, we know that $e^{i\lambda\alpha}$ is the boundary value of the holomorphic function $e^{i\lambda(\alpha+i y)}$ in the lower half plane. If $\lambda<0$, then $e^{i\lambda(\alpha+i y)}\rightarrow 0$ as $y\rightarrow -\infty$ which implies that $(I-H_0)e^{i\lambda\alpha}=0$. In general, if $f\in H^s(q\mathbb{T})$, $\|(I-H_0)fe^{i\lambda\alpha}\|_{H^s(q\mathbb{T})}$ can be comparable with $\|f\|_{H^s(q\mathbb{T})}$. However, given $m\in \mathbb{N}$, if $f\in H^{s+m}(q_1\mathbb{T})$, then $\|(I-H_0)f(\epsilon\alpha)e^{i\lambda\alpha}\|_{H^s(q\mathbb{T})}$ is as small as $\epsilon^{m-1/2}$. \footnote{Recall that $q_1=\epsilon q$.}
	
	\begin{lemma}\label{lemma: slowvaryingalmostholomorphic}
Let $m\in \mathbb{N}$ be given. Let $f\in H^{s}(q_1\mathbb{T})$ and $\lambda\in \mathbb{Z}$. Then
			\begin{equation}
			\| (I+\text{sgn}(\lambda)H_0)f(\epsilon \alpha)e^{i\lambda\alpha}\|_{H^s(q\mathbb{T})}\leq C\epsilon^{m-1/2}\|f\|_{H^{s+m}(q_1\mathbb{T})},
			\end{equation}
			where $C$ depends on $s$ only.
	\end{lemma}
	\begin{proof}
		We consider the case that $\lambda<0$. By the Fourier series on $q_1\mathbb{T}$, we write 
		\begin{equation}
		f(\alpha)=\frac{1}{2\pi q_1}\sum_{k\in \mathbb{Z}}a_k e^{i k\frac{\alpha}{q_1}},
		\end{equation}
		where
		\begin{equation}
		a_k=\int_{-q_1\pi}^{q_1\pi}f(\alpha)e^{-ik\frac{\alpha}{q_1}}\,d\alpha.
		\end{equation}
		Therefore, one has
		\begin{equation}
		f(\epsilon \alpha)e^{i\lambda\alpha}=\frac{1}{2q_1\pi}\sum_{k\in \mathbb{Z}}a_k e^{i(2k/q-|\lambda|)\alpha}.
		\end{equation}
		By Parseval's identity, it follows
		\begin{equation}
		\|f\|_{H^\ell(q_1\mathbb{T})}^2=\frac{1}{2q_1\pi}\sum_{n=0}^\ell\sum_{k\in \mathbb{Z}}\frac{|k|^{2n}}{q_1^{2n}}a_k^2.
		\end{equation}
	 Notice that for $2k<|\lambda|q$, we have $(I-H_0)a_ke^{i(2k/q-|\lambda|)\alpha}=0$.
		Therefore, we get
		\begin{align*}
		\norm{(I-H_0)fe^{-i |\lambda|\alpha}}_{H^{s}(q\mathbb{T})}^2=&\frac{1}{(2q_1\pi)^2}\norm{(I-H_0)\sum_{2k=q|\lambda|}^{\infty}a_k e^{i(-|\lambda|+2k/q)\alpha}}_{H^{s}(q\mathbb{T})}^2\\
		\leq & \frac{q}{q_1^2\pi^2}\sum_{n=0}^{s}\sum_{2k=q|\lambda|}^{\infty}(-|\lambda|+2k/q)^{2n}a_k^2\\
		\leq &\frac{q}{q_1^2\pi^2}\sum_{n=0}^{s}\sum_{2k=q|\lambda|}^{\infty}(2k/q)^{2n}(k/q_1)^{-2(m+s-n)}(k/q_1)^{2(m+s-n)}a_k^2\\
		= &\frac{q}{q_1^2\pi^2}\sum_{n=0}^{s}\sum_{2k=q|\lambda|}^{\infty}\Big((k/q_1)^{2n}\epsilon^{-2n}\Big)(k/q_1)^{-2(m+s-n)}(k/q_1)^{2(m+s-n)}a_k^2\\
		=&\frac{q}{q_1^2\pi^2}\sum_{n=0}^{s}\epsilon^{-2n}\sum_{2k=q|\lambda|}^{\infty}(k/q_1)^{-2(m+s-n)}(k/q_1)^{2(m+s)}a_k^2\\
		\leq & C\frac{q}{q_1^2\pi^2}\epsilon^{
		-2m}\sum_{n=0}^{\infty} (k/q_1)^{2(m+s)}a_k^2\\
		 \leq & C\epsilon^{2m-1}\norm{f}_{H^{s+m}(q_1\mathbb{T})}^2.
		\end{align*}
		as desired. Here in the last step, we used the estimate
		\begin{equation}
		    \epsilon^{-2n}(k/q_1)^{-2(m+s-n)}\leq \epsilon^{2m+(2s-2n)}\leq \epsilon^{2m}
		\end{equation}
		for $k\geq |\lambda |q/2$. We are done.
	\end{proof}

	\subsection{Multi-scale expansion}\label{subsection:multiscale}
	In this subsection, we compute the multi-scale expansion. The general strategy here is  similar to \cite{Totz2012, su2020partial} (\S 3 of \cite{Totz2012}, \S 4 of \cite{su2020partial}). 
	We note that the nonlinear Schr\"odinger equation has not been derived in the setting with Stokes wave.  Also notice that under the current setting, in the leading order $Be^{i\alpha+i\omega t}$, $\omega$ can depend on $\epsilon$, which is different from previous works, for example, \cite{Totz2012,su2020partial}. 
	
	Our goal is  to choose $\zeta(\alpha,t)$ such that 
	$$\zeta(\alpha,t)=\alpha+\epsilon \zeta^{(1)}+\epsilon^2\zeta^{(2)}+\epsilon^3\zeta^{(3)}+O(\epsilon^4)$$
	and \footnote{See Proposition \ref{prop:stokesexpansion}.}
\begin{align}\label{eq:tildezetaST}
        \tilde{\zeta}_{ST}(\alpha,t)=&\alpha+i\epsilon e^{i(\alpha+\omega t)}+\epsilon^2 i +\frac{i}{2}\epsilon^3 e^{i(-\alpha-\omega t)}\\
    =&\alpha+\epsilon \zeta^{(1)}_{ST}+ \epsilon^2 \zeta^{(2)}_{ST}+\epsilon^3 \zeta^{(3)}_{ST}\nonumber
\end{align}
	and the followings hold
	\begin{itemize}
	
	\item [(1)] $\zeta$ solves (\ref{cubic}).
	    \item [(2)] $(I-\mathcal{H}_{\zeta})(\bar{\zeta}-\alpha)=0$.
	    
	    \item [(3)] $(I-\mathcal{H}_{\zeta})D_t\bar{\zeta}=0$.


\item [(4)]  $\|\alpha+\epsilon\zeta^{(1)}+\epsilon^2\zeta^{(2)}+\epsilon^3\zeta^{(3)}-\tilde{\zeta}_{ST}\|_{H^{s'}(q\mathbb{T})}\leq C\epsilon^{1/2}\delta$.

	\end{itemize}
By Corollary \ref{corollary:hilbertboundary},  $(I-\mathcal{H}_{\zeta})(\bar{\zeta}-\alpha)=0$ and $(I-\mathcal{H}_{\zeta})D_t\bar{\zeta}=0$ imply
	\begin{equation}\label{constraints}
	    \int_{-q\pi}^{q\pi}\zeta_{\beta}(\bar{\zeta}(\beta,t)-\alpha)\,d\beta=0,\quad \quad \int_{-q\pi}^{q\pi}\zeta_{\beta}D_t\bar{\zeta}(\beta,t)\,d\beta=0.
	\end{equation}
%
In this section,  we will use the notation $\phi:=\alpha+\omega t$.

	\subsubsection{$O(\epsilon)$ Level}
	Plugging \eqref{eq:zetaexpansion1} into the the momentum equation, the first equation of \eqref{system_newvariables}, we know that $\zeta^{(1)}$ satisfies $(\partial_t^2-i\partial_{\alpha})\zeta^{(1)}=O(\epsilon^3)$ because 
	\begin{equation}
	    \omega=1+O(\epsilon^2)
	\end{equation}
	by construction.
	\subsubsection{$O(\epsilon^2)$ Level}
	Using the system \eqref{system_newvariables} again, we need
	\begin{align}\label{eq:zeta2equation}
(\partial_{t_0}^2-i\partial_{\alpha_0})(I-H_0)\zeta^{(2)}=-(2\partial_{t_0}\partial_{t_1}-i\partial_{\alpha_1})(I-H_0)\zeta^{(1)}+(\partial_{t_0}^2-i\partial_{\alpha_0})H_1 \zeta^{(1)}.
	\end{align}
	Note that by \eqref{eq:H1} and the explicit choice of \eqref{ansatz:zeta1}, one has	\begin{align*}
	(\partial_{t_0}^2-i\partial_{\alpha_0})H_1 \zeta^{(1)}=&  (\partial_{t_0}^2-i\partial_{\alpha_0})[\zeta^{(1)},H_0]\partial_{\alpha_0}Be^{i\phi}\\
	=&  i(\partial_{t_0}^2-i\partial_{\alpha_0})[\zeta^{(1)},I+H_0]Be^{i\phi}\\
	=& O(\epsilon^4)
	\end{align*}
where in the last step we applied Lemma \ref{lemma: slowvaryingalmostholomorphic}.  To avoid secular terms, we choose $\zeta^{(1)}$ such that 
	\begin{equation}\label{eqn:epsilon2}
	    -(2\partial_{t_0}\partial_{t_1}-i\partial_{\alpha_1})(I-H_0)\zeta^{(1)}=0.
	\end{equation}
Using Lemma \ref{lemma: slowvaryingalmostholomorphic} again, we have $(I-H_0)\zeta^{(1)}=2\zeta^{(1)}+O(\epsilon^4)$. So plugging  the ansatz, \eqref{ansatz:zeta1},  into  (\ref{eqn:epsilon2}), it follows 
	$$ B_{t_1}-\frac{1}{2\omega}B_{\alpha_1}=O(\epsilon^4).$$
	So we choose $B=B(X,T)$, with
\begin{equation}\label{eq:XT}
	X= \alpha_1+\frac{1}{2\omega}t_1=\epsilon(\alpha+\frac{1}{2\omega}t),\quad	T=t_2=\epsilon^2 t.
\end{equation}
To choose $\zeta^{(2)}$, we use $(I-\mathcal{H}_\zeta)(\bar{\zeta}-\alpha)=0$. The $O(\epsilon^2)$ terms give
\begin{align*}
(I-H_0)\bar{\zeta}^{(2)}=&H_1 \bar{\zeta}^{(1)}=[\zeta^{(1)}, H_0]\partial_{\alpha_0} \bar{\zeta}^{(1)}\\
=&-i[\zeta^{(1)},H_0]\bar{B}e^{-i\phi}\\
=& i[\zeta^{(1)},I-H_0]\bar{B}e^{-i\phi}\\
=&-i(I-H_0)|B|^2+O(\epsilon).
\end{align*}
We pick
\begin{equation}\label{eqn:eqnzeta2}
    \zeta^{(2)}=\frac{i}{2}(I+H_0)|B|^2+\frac{i}{2}M(|B|^2),
\end{equation}
where $M(|B|^2)=\frac{1}{2q\pi}\int_{-q\pi}^{q\pi}|B(X,T)|^2\,d\alpha$.

	\subsubsection{Expansion of $b$}
	We expand $b$ as
	\begin{equation}
	b=b^{(0)}+\epsilon b^{(1)}+ \epsilon^2b^{(2)}+O(\epsilon^3).
	\end{equation}	
Since	$(I-\mathcal{H}_\zeta)b=-[D_t\zeta,\mathcal{H}]\frac{\bar{\zeta}_{\alpha}-1}{\zeta_{\alpha}}$ is quadratic, we have  $b^{(0)}=b^{(1)}=0$. 
For $b_2$, one has
\begin{equation}
\begin{split}
(I-H_0)b^{(2)}=&-[\partial_{t_0}\zeta^{(1)},H_0]\partial_{\alpha_0}\bar{\zeta}^{(1)}\\
=&-\omega [\zeta^{(1)},H_0]\bar{\zeta}^{(1)}=\omega [\zeta^{(1)}, I-H_0]\bar{\zeta}^{(1)}=-\omega(I-H_0)|B|^2.
\end{split}
\end{equation}

\noindent Since $b^{(2)}$ is real, we get
\begin{equation}\label{eq:b2}
b^{(2)}=-\omega|B|^2.
\end{equation}

\subsubsection{Expansion of $A$}
We need also to expand $A=\sum_{n\geq 0}\epsilon^n A^{(n)}$. 

Since $A-1$ is quadratic by the formula \eqref{convenient:expansion}, clearly, $A^{(0)}=1$, and $A^{(1)}=0$.

To find $A^{(2)}$, again from \eqref{convenient:expansion}, at the $\epsilon^2$ level, we have 
\begin{align*}
(I-H_0)A^{(2)}=i[\partial_{t_0}^2\zeta^{(1)},H_0]\partial_{\alpha_0}\bar{\zeta}^{(1)}+i[\partial_{t_0}\zeta^{(1)},H_0]\partial_{\alpha_0}\partial_{t_0}\bar{\zeta}^{(1)}=0.
\end{align*}
Since $A^{(2)}$ is real, we get $A^{(2)}=0$.  


	\subsubsection{Expansions of $G_1$ and $G_2$}
	From formula \eqref{cubic}, by  direct inspection, we obtain
	\begin{equation}
G_1=O(\epsilon^4).
\end{equation}
For $G_2$, integration by parts,  
\begin{align*}
    G_2=& \frac{\epsilon^3}{ 4\pi q^2 i}\int_{-q\pi}^{q\pi} \Big(\frac{\partial_{t_0}\zeta^{(1)}(\alpha,t)-\partial_{t_0}\zeta^{(1)}(\beta,t)}{\sin(\frac{\alpha-\beta}{2q})}\Big)^2\partial_{\beta_0}(\zeta^{(1)}-\bar{\zeta}^{(1)})\,d\beta+O(\epsilon^4)\\
    =& 2\epsilon^3[\partial_{t_0}\zeta^{(1)}, H_0](\partial_{t_0}\partial_{\alpha_0}\zeta^{(1)}\partial_{\alpha_0}(\zeta^{(1)}-\bar{\zeta}^{(1)})
    -[\partial_{t_0}\zeta^{(1)}, [\partial_{t_0}\zeta^{(1)}, H_0]]\partial_{\alpha_0}^2 (\zeta^{(1)}-\bar{\zeta}^{(1)})+O(\epsilon^4)\\
    =& 2\epsilon^3 |B|^2Be^{i\phi}+O(\epsilon^4).
\end{align*}

	\subsubsection{$O(\epsilon^3)$ Level}
	We first note that
	\begin{equation}
	    (\partial_{t_0}^2-i\partial_{\alpha_0})(I-H_0)e^{i\phi}=2(1-\omega^2)e^{i\phi}=2\left(1-\Big(1+\frac{\epsilon^2}{2}+O(\epsilon)^3\Big)^2\right)e^{i\phi}=-2\epsilon^2 e^{i\phi}.
	\end{equation}
	At the  $O(\epsilon^3)$ level of $(\partial_{t_0}^2-i\partial_{\alpha_0})Be^{i\phi}$, we have    $-2Be^{i\phi}$.
Now we expand the momentum equation, the first equation of \eqref{system_newvariables}, in term of powers of $\epsilon$ as before. At the level of $O(\epsilon^3)$, we have 
	\begin{equation}\label{epsiloncubicterms}
	\begin{split}
	&(\partial_{t_0}^2-i\partial_{\alpha_0})(I-H_0)\zeta^{(3)}\\
	=&-(\partial_{t_0}^2-i\partial_{\alpha_0})(-H^{(1)})\zeta^{(2)}-(\partial_{t_0}^2-i\partial_{\alpha_0})(-H^{(2)})\zeta^{(1)}\\
	&-(2\partial_{t_0}\partial_{t_1}-i\partial_{\alpha_1})(I-H_0)\zeta^{(2)}-(2\partial_{t_0\partial_{t_1}}-i\partial_{\alpha_1})(-H^{(1)})\zeta^{(1)}\\
	&-(2\partial_{t_0t_2}+\partial_{t_1}^2+2b^{(2)}\partial_{t_0}\partial_{\alpha_0})(I-H_0)\zeta^{(1)}+G_2+2Be^{i\phi}.\\
	\end{split}
	\end{equation}
$\bullet$ Noticing that $\zeta^{(2)}$ is slowly varying, one has
	\begin{equation}
	    -(\partial_{t_0}^2-i\partial_{\alpha_0})(-H^{(1)})\zeta^{(2)}=O(\epsilon).
	\end{equation}
 
 \noindent $\bullet$ By our choice of $\zeta^{(2)}$, \eqref{eqn:eqnzeta2},  and applying (3) of Lemma \ref{lemma:holoboundary}, we have,
	\begin{equation}
	\zeta^{(2)}=\frac{i}{2}(I+H_0)|B|^2+\frac{i}{2}M(|B|^2)=\frac{i}{2}(I+H_0)(|B|^2-M(|B|^2))+iM(|B|^2).
	\end{equation}
Applying (3) of Lemma \ref{lemma:holoboundary} again, we obtain
	\begin{align}
	    (I-H_0)\zeta^{(2)}=& (I-H_0)\Big\{\frac{i}{2}(I+H_0)\Big(|B|^2-M(|B|^2)\Big)+iM(|B|^2)\Big\}=iM(|B|^2).
	\end{align}
Since $M(|B|^2)$ is slowly varying in $t$, we obtain
	\begin{equation}
	    -(2\partial_{t_0}\partial_{t_1}-i\partial_{\alpha_1})(I-H_0)\zeta^{(2)}=O(\epsilon).
	\end{equation}
Using Lemma \ref{lemma: slowvaryingalmostholomorphic}, we have
	\begin{align*}
	    (-H^{(1)})\zeta^{(1)}=&-[\zeta^{(1)}, H_0]\partial_{\alpha_0}\zeta^{(1)}=O(\epsilon^4).
	\end{align*}
So we obtain
\begin{align*}
    &-(2\partial_{t_0}\partial_{t_1}-i\partial_{\alpha_1})(I-H_0)(\zeta^{(2)}-\bar{\zeta}^{(2)})-(2\partial_{t_0\partial_{t_1}}-i\partial_{\alpha_1})(-H^{(1)})(\zeta^{(1)}-\bar{\zeta}^{(1)})
    =O(\epsilon).
\end{align*}

\vspace*{1ex}

 \noindent $\bullet$ For $H^{(2)}\zeta^{(1)}$, by \eqref{eq:H2}, Lemma \ref{lemma: slowvaryingalmostholomorphic} and the fact that $\zeta^{(2)}$ is slowly varying, we obtain
	\begin{align*}
	    H^{(2)}\zeta^{(1)}=&[\zeta^{(2)}, H_0]\partial_{\alpha_0}\zeta^{(1)}-[\zeta^{(1)}, H_0]\zeta_{\alpha_0}^{(1)}\partial_{\alpha_0}\zeta^{(1)}+\frac{1}{2}[\zeta^{(1)}, [\zeta^{(1)}, H_0]]\partial_{\alpha_0}^2\zeta^{(1)}=O(\epsilon^4).
	\end{align*}

\vspace*{1ex}
\noindent $\bullet$ Also we have
\begin{align}
   & -(2\partial_{t_0t_2}+\partial_{t_1}^2+2b^{(2)}\partial_{t_0}\partial_{\alpha_0})(I-H_0)\zeta^{(1)}\\
   =&-2 \Big\{ 2(i\omega) \partial_T+(\frac{1}{2\omega})^2\partial_{XX}+2(-\omega|B|^2)(i\omega)(i)\Big\}Be^{i\phi}\\
   =& -2\Big\{ 2i\omega B_T+(\frac{1}{2\omega})^2B_{XX}+2\omega^2|B|^2B\Big\}e^{i\phi}
\end{align}	
Overall, from all computations above, we obtain 
	\begin{align}
	    &	(\partial_{t_0}^2-i\partial_{\alpha_0})(I-H_0)\zeta^{(3)}
	= -2\Big\{2i\omega B_T+(\frac{1}{2\omega})^2 B_{XX}+(2\omega^2-1)|B|^2B-B\Big\}e^{i\phi}.
	\end{align}
Since $\omega=1+O(\epsilon^2)$, we have
\begin{equation}
    -2\Big\{2i\omega B_T+(\frac{1}{2\omega})^2 B_{XX}+(2\omega^2-1)|B|^2B-B\Big\}e^{i\phi}=-2\Big\{2i B_T+\frac{1}{4}B_{XX}+|B|^2B-B\Big\}+O(\epsilon^2).
\end{equation}
	To avoid the secular growth, we choose $B$ such that 
	\begin{equation}
	    2iB_T+\frac{1}{4} B_{XX}+|B|^2B-B=0.
	\end{equation}
	or equivalently,
		\begin{equation}\label{rescaledNLS}
	    iB_T+\frac{1}{8}B_{XX}+\frac{1}{2}|B|^2B-\frac{1}{2}B=0.
	\end{equation}
\begin{rem}
	Note that $B\equiv i$ is an exact solution to (\ref{rescaledNLS}), which justifies our assumption that 
	\begin{equation}
	    B=i+\text{perturbation}
	\end{equation} in the long-wave perturbation setting.
\end{rem}
With the choice of $B$ above,  (\ref{epsiloncubicterms}) becomes 
	$$(\partial_{t_0}^2-i\partial_{\alpha_0})(I-H_0)\zeta^{(3)}=O(\epsilon).$$
From $(I-\mathcal{H}_{\zeta})(\bar{\zeta}-\alpha)=0$, we have 
\begin{align*}
(I-H_0)\bar{\zeta}^{(3)}=& H^{(1)}\bar{\zeta}^{(2)}+H^{(2)}\bar{\zeta}^{(1)}\\
=&[\zeta^{(1)},H_0]\partial_{\alpha_0}\bar{\zeta}^{(2)}+[\zeta^{(2)},H_0]\partial_{\alpha_0}\bar{\zeta}^{(1)}+[\zeta^{(1)},H_0]\partial_{\alpha_1}\zeta^{(1)}\\
&-[\zeta^{(1)},H_0]\overline{\partial_{\alpha_0}\zeta^{(1)}}\partial_{\alpha_0}\zeta^{(1)}+\frac{1}{2}[\zeta^{(1)}, [\zeta^{(1)}, H_0]]\partial_{\alpha_0}^2\bar{\zeta}^{(1)}.
\end{align*}
\noindent $\bullet$ Since $\bar{\zeta}^{(2)}$ is slowly varying, we have 
	\begin{equation}
	    [\zeta^{(1)},H_0]\partial_{\alpha_0}\bar{\zeta}^{(2)}=O(\epsilon).
	\end{equation}
	
\noindent $\bullet$ For $[\zeta^{(2)},H_0]\partial_{\alpha_0}\bar{\zeta}^{(1)}$, since $\zeta^{(2)}$ slowly varying, using Lemma \ref{almostNLSpacket}, 
\begin{align}
    [\zeta^{(2)},H_0]\partial_{\alpha_0}\bar{\zeta}^{(1)}=[\zeta^{(2)}, I+H_0] \partial_{\alpha_0}\bar{\zeta}^{(1)}=  &O(\epsilon^4).
\end{align}

\noindent $\bullet$ For $[\zeta^{(1)},H_0]\partial_{\alpha_1}\bar{\zeta}^{(1)}$, one has 
\begin{align}
    [\zeta^{(1)},H_0]\partial_{\alpha_1}\bar{\zeta}^{(1)}=&[\zeta^{(1)},H_0]\partial_{\alpha_1}\bar{\zeta}^{(1)}=(I-H_0)BB_X.
\end{align}

\noindent $\bullet$ For $-[\zeta^{(1)},H_0]\overline{\partial_{\alpha_0}\zeta^{(1)}}\partial_{\alpha_0}\bar{\zeta}^{(1)}$, it is easy to obtain
\begin{align}
    -[\zeta^{(1)},H_0]\overline{\partial_{\alpha_0}\zeta^{(1)}}\partial_{\alpha_0}\bar{\zeta}^{(1)}=O(\epsilon^4).
\end{align}

\noindent $\bullet$ For $\frac{1}{2}[\zeta^{(1)}, [\zeta^{(1)}, H_0]]\partial_{\alpha_0}^2\bar{\zeta}^{(1)}$, we have 
\begin{align}
    &\frac{1}{2}[\zeta^{(1)}, [\zeta^{(1)}, H_0]]\partial_{\alpha_0}^2\bar{\zeta}^{(1)}=-\frac{1}{2}[\zeta^{(1)}, [\zeta^{(1)}, H_0]]\bar{B}e^{-i\phi}\\
    =&-\frac{1}{2}\zeta^{(1)}[\zeta^{(1)}, H_0]\bar{B}e^{-i\phi}+\frac{1}{2}[\zeta^{(1)}, H_0]|B|^2\\
    =& -\frac{1}{2}\zeta^{(1)}[\zeta^{(1)}, I-H_0]\bar{B}e^{-i\phi}+\frac{1}{2}[\zeta^{(1)}, I+H_0]|B|^2\\
    =& \frac{1}{2}\zeta^{(1)}(I-H_0)|B|^2+\frac{1}{2}\zeta^{(1)}(I+H_0)|B|^2\\
    =& |B|^2\zeta^{(1)}.
\end{align}
Therefore we conclude that
\begin{equation}
    (I-H_0)\bar{\zeta}^{(3)}=-B|B|^2e^{i\phi}+(I-H_0)B\bar{B}_X
\end{equation}
Now we can choose 
	\begin{equation}\label{eq:zeta_3}
	\begin{split}
	    \zeta^{(3)}=&-\frac{1}{2}\bar{B}|B|^2e^{-i\phi}+\frac{1}{2}(I+H_0)(\bar{B}B_X).
	    \end{split}
	\end{equation}

	

	\subsection{The approximate solution}
After choosing $\zeta^{(j)},\,j=1,2,3$ above, with \eqref{eq:tildezetaST}, \eqref{ansatz:zeta1}, \eqref{eqn:eqnzeta2} and \eqref{eq:zeta_3}, we can define the approximate solution as
	\begin{equation}\label{eq:zetaapp}
	\zeta_{app}(\alpha,t)=\zeta_{ST}+(\alpha+\epsilon\zeta^{(1)}+\epsilon^2\zeta^{(2)}+\epsilon^3\zeta^{(3)}-\tilde{\zeta}_{ST}).
	\end{equation}
Explicitly plugging in the choices of $\zeta^{(j)}$, one has
\begin{equation}
    \begin{split}
        \alpha+\epsilon\zeta^{(1)}+\epsilon^2\zeta^{(2)}+\epsilon^3\zeta^{(3)}-\tilde{\zeta}_{ST}
        =&\epsilon (B-i)e^{i\phi}+\epsilon^2\Big(\frac{ik}{2} (I+H_0)|B|^2+\frac{1}{2}M(|B|^2)-i\Big)\\
        &+\epsilon^3 \Big(-\frac{1}{2}\bar{B}|B|^2e^{-i\phi}+\frac{1}{2}(I+H_0)(\bar{B}B_X)+\frac{i}{2}e^{-i\phi}\Big).
    \end{split}
\end{equation}
Therefore with this choice of $\zeta_{app}$, we have
\begin{equation}
    \norm{\zeta_{app}-\zeta_{ST}}_{H^{s+1}(q\mathbb{T})}\leq C\epsilon^{1/2}\norm{B-i}_{H^{s+7}(q_1\mathbb{T})},
\end{equation}
where $C>0$ is a constant depending on $s$ only.

\subsection{NLS estimates}
In the final part of this section, we discuss the behavior of the function $B$ coming from the expansion \eqref{eq:zetaexpansion1} and \eqref{ansatz:zeta1}. From our multi-scale analysis, $B$ solves
\begin{equation}\label{eq:BNLS}
iB_T+\frac{1}{8} B_{XX}+\frac{1}{2} |B|^2B-\frac{1}{2} B=0
\end{equation}
and the Stokes wave gives a special solution $B_0=i$. We perturb the special solution by considering solution of the form $B=i(1+\psi)$. We have the following result on its instability.

\begin{proposition}\label{prop:nlsestimates}
For any given $0<\delta\ll1$, there exist $\mu$ satisfying $\left|\delta\right|\ll\mu<1$
and $T_{0}=\log\left(\frac{\mu}{\delta}\right)$ such that for any solution  $B$ to \eqref{eq:BNLS} with initial data
\begin{equation}\label{eq:Binitial}
    \norm{B(\cdot,0)-i}_{H^{s'}(q_1\mathbb{T})}= \delta
\end{equation}
it satisfies the following growth estimate:
\begin{equation}
\left\Vert B(\cdot,t)-i\right\Vert _{H^{s'}(q_1\mathbb{T})}\leq 2 \delta e^{t},\,\forall t\in\left[0,T_{0}\right]\label{eq:H1boundB}.
\end{equation}
Moreover, there exists solution $B$ satisfying the initial condition \eqref{eq:Binitial}, growth estimate and the the following unstable condition
\begin{equation}
\left\Vert B(\cdot, T_0)-i\right\Vert _{H^{s'}(q_1\mathbb{T})}\geq\frac{1}{4}\mu\gg\delta.\label{eq:unstableboundB}
\end{equation}

\end{proposition}
\begin{proof}
Making the change of variable $u(y,s)=e^{is}B(y/2,2s)$, then it solves
\begin{equation}\label{eq:uNLS}
i\partial_s u+\partial^2_y u+|u|^2u=0
\end{equation} and the special solution given by the Stokes wave is $u_0(y,s)=ie^{is}$.  Then applying Theorem \ref{thm:NLSinsta} from Appendix \ref{sec:NLS}, the desired estimates follow.

\end{proof}
	
	\section{The error equation}\label{section:error}
	In this section, we derive governing equations for the remainder term. Let's denote
	\begin{equation}\label{eq:tildezeta}
	    \tilde{\zeta}:=\alpha+\epsilon\zeta^{(1)}+\epsilon^2\zeta^{(2)}+\epsilon^3\zeta^{(3)}.
	\end{equation}
	Then 
	\begin{equation}\label{eq:tildeapproximate}
	    \zeta_{app}=\zeta_{ST}+(\tilde{\zeta}-\tilde{\zeta}_{ST}).
	\end{equation}
	Define the error term as 
	\begin{equation}\label{eq:r}
	    r:=\zeta-\zeta_{app}.
\end{equation}
	
	\subsection{Notations}
We first introduce some notations here.

	Denote
	\begin{equation}\label{eq:notationb}
	    \tilde{b}=\epsilon^2b^{(2)}, \quad \tilde{b}_{ST}:=\epsilon^2 b_{ST}^{(2)},
	\end{equation}
	
	\begin{equation}
	    \tilde{A}:=1,\quad \quad \tilde{A}_{ST}:=1,
	\end{equation}
	
	\begin{equation}
	    \tilde{D}_t:=\partial_t+\tilde{b}\partial_{\alpha},\quad \tilde{D}_t^{ST}:=\partial_t+\tilde{b}_{ST}\partial_{\alpha},
	\end{equation}

	\begin{equation}
	    \tilde{\theta}:=(I-\mathcal{H}_{\tilde{\zeta}})(\tilde{\zeta}-\alpha),\quad	    	    \tilde{\theta}_{ST}:=(I-\mathcal{H}_{\tilde{\zeta}_{ST}})(\tilde{\zeta}_{ST}-\alpha),
	\end{equation}
	
	\begin{equation}
	    \theta:=(I-\mathcal{H}_{\zeta})(\zeta-\alpha),\quad
	    \theta_{ST}:=(I-\mathcal{H}_{\zeta_{ST}})(\zeta_{ST}-\alpha),
	\end{equation}
	
	\begin{equation}
	    \tilde{\mathcal{P}}=\tilde{D}_t^2-i\tilde{A}\partial_{\alpha},\quad
	    \tilde{\mathcal{P}}_{ST}=(\tilde{D}_t^{ST})^2-i\tilde{A}_{ST}\partial_{\alpha},
	\end{equation}	
	
	\begin{equation}
	    \begin{cases}
	        \mathcal{Q}:=\mathcal{P}(I-\mathcal{H}_{\zeta}),\\
	         \mathcal{Q}_{ST}:=\mathcal{P}_{ST}(I-\mathcal{H}_{\zeta_{ST}})\\
	         \tilde{\mathcal{Q}}:=\tilde{\mathcal{P}}(I-\mathcal{H}_{\tilde{\zeta}})\\
	         \tilde{\mathcal{Q}}_{ST}:=\tilde{\mathcal{P}}_{ST}(I-\mathcal{H}_{\tilde{\zeta}_{ST}}).
	    \end{cases}
	\end{equation}
	
With notations above, we immediately conclude the following:
	\begin{lemma}
	    We have 
	    \begin{equation}
	        \norm{\tilde{\mathcal{P}}\tilde{\theta}-\tilde{\mathcal{P}}_{ST}\tilde{\theta}_{ST}}_{H^{s+1}(q\mathbb{T})}\leq C\epsilon^{7/2}\|B-i\|_{H^{s+7}(q_1\mathbb{T})}.
	    \end{equation}
	\end{lemma}
	
	\subsection{Governing equation for $r$} 
From the cubic structure for $\zeta$, \eqref{cubic}, we have
\begin{equation}
    \mathcal{P}\theta=G.
\end{equation}
And similarly, one has
	\begin{equation}
	    \mathcal{P}_{ST}\theta_{ST}=G_{ST}.
	\end{equation}
Consider the quantity $\rho$ defined by
	\begin{equation}
	\rho:=(I-\mathcal{H}_{\zeta})\Big[\theta-\theta_{ST}-(\tilde{\theta}-\tilde{\theta}_{ST})\Big].
	\end{equation}
	Then $\rho$ is holomorphic in $\Omega(t)^c$. In Lemma \ref{equivalencequantities}, we will show that $\rho$ is equivalent to $r$.
	
	To control $\rho$ and therefore control $r$, we need to derive a nice structure for $\rho$. To achieve the nonlinear instability, it is necessary to obtain the control of $\rho$ for $t\in [0, \epsilon^{-2}\log\frac{\mu}{\delta}]$, where $\delta:=\|B(\cdot,0)-i\|_{H^{s+7}(q_1\mathbb{T})}$. So we need to derive an equation for $\rho$ of the form
	\begin{equation}
	    \mathcal{P}\rho=O(\epsilon^2)O(\rho)
	   	\end{equation}
	   such that
	   \begin{equation}
	       	    \|D_t\rho(\cdot,  t)\|_{H^{s+1/2}(q\mathbb{T})}+\norm{\partial_{\alpha}\rho(\cdot,t)}_{H^s(q\mathbb{T})}\sim \delta\epsilon^{3/2}.
	   \end{equation}
%
%
By the definition of $\rho$, using the notation above, we have 
	\begin{align*}
	    \mathcal{P}\rho=&\mathcal{P}(I-\mathcal{H}_{\zeta})\theta-\mathcal{P}(I-\mathcal{H}_{\zeta})\theta_{ST}-\mathcal{P}(I-\mathcal{H}_{\zeta})(\tilde{\theta}-\tilde{\theta}_{ST})\\
	    =&\mathcal{Q}\theta-\mathcal{Q}\theta_{ST}-\mathcal{Q}(\tilde{\theta}-\tilde{\theta}_{ST})\\
	    =& \mathcal{Q}\theta-\mathcal{Q}_{ST}\theta_{ST}-(\mathcal{Q}-\mathcal{Q}_{ST})\theta_{ST}-\tilde{\mathcal{Q}}\tilde{\theta}-(\mathcal{Q}-\tilde{\mathcal{Q}})\tilde{\theta}+\tilde{\mathcal{Q}}_{ST}\tilde{\theta}_{ST}+(\mathcal{Q}-\tilde{\mathcal{Q}}_{ST})\tilde{\theta}_{St}\\
	    =& G-G_{ST}-(\tilde{G}-\tilde{G}_{ST})+\Big\{-(\mathcal{Q}-\mathcal{Q}_{ST})\theta_{ST}-(\mathcal{Q}-\tilde{\mathcal{Q}})\tilde{\theta}+(\mathcal{Q}-\tilde{\mathcal{Q}}_{ST})\tilde{\theta}_{St}\Big\}.
	\end{align*}
	We regroup $\Big\{-(\mathcal{Q}-\mathcal{Q}_{ST})\theta_{ST}-(\mathcal{Q}-\tilde{\mathcal{Q}})\tilde{\theta}+(\mathcal{Q}-\tilde{\mathcal{Q}}_{ST})\tilde{\theta}_{ST}\Big\}$ as
	\begin{align*}
	    &(\mathcal{Q}-\tilde{\mathcal{Q}}_{ST})\tilde{\theta}_{ST}-(\mathcal{Q}-\mathcal{Q}_{ST})\theta_{ST}-(\mathcal{Q}-\tilde{\mathcal{Q}})\tilde{\theta}\\
	    =&\Big\{(\mathcal{Q}-\tilde{\mathcal{Q}}_{ST})\tilde{\theta}_{ST} -(\mathcal{Q}-\mathcal{Q}_{ST})\tilde{\theta}_{ST}\Big\}+\Big\{(\mathcal{Q}-\mathcal{Q}_{ST})(\tilde{\theta}_{ST}-\theta_{ST})\Big\}-(\mathcal{Q}-\tilde{\mathcal{Q}})\tilde{\theta}\\
	    =&(\mathcal{Q}_{ST}-\tilde{\mathcal{Q}}_{ST})\tilde{\theta}_{ST}  -(\mathcal{Q}-\tilde{\mathcal{Q}})\tilde{\theta} +\Big\{(\mathcal{Q}-\mathcal{Q}_{ST})(\tilde{\theta}_{ST}-\theta_{ST})\Big\}\\
	    =&\Big(\mathcal{Q}_{ST}-\tilde{\mathcal{Q}}_{ST}-(\mathcal{Q}-\tilde{\mathcal{Q}})\Big)\tilde{\theta}_{ST}+(\mathcal{Q}-\tilde{\mathcal{Q}})(\tilde{\theta}_{ST}-\tilde{\theta})+\Big\{(\mathcal{Q}-\mathcal{Q}_{ST})(\tilde{\theta}_{ST}-\theta_{ST})\Big\}\\
	\end{align*}
To sum up,  we obtain
	\begin{equation}
	    \mathcal{P}\rho=N_1+N_2+N_3+N_4,
	\end{equation}
	where
	\begin{equation}
	    N_1=G-G_{ST}-(\tilde{G}-\tilde{G}_{ST}),
	\end{equation}
	\begin{equation}
	    N_2=\Big(\mathcal{Q}_{ST}-\tilde{\mathcal{Q}}_{ST}-(\mathcal{Q}-\tilde{\mathcal{Q}})\Big)\tilde{\theta}_{ST},
	\end{equation}
\begin{equation}
    N_3=(\mathcal{Q}-\tilde{\mathcal{Q}})(\tilde{\theta}_{ST}-\tilde{\theta}),
\end{equation}
\begin{equation}
    N_4=(\mathcal{Q}-\mathcal{Q}_{ST})(\tilde{\theta}_{ST}-\theta_{ST}).
\end{equation}
	
	\subsection{Equation governing $D_t\rho$} We next derive an equation for $D_t\rho$. Define
	\begin{equation}
	   \sigma:=(I-\mathcal{H}_{\zeta})D_t\rho.
	\end{equation}
Applying $\mathcal{P}$ and using the equation for $\rho$ obtained above, we have
\begin{align*}
    \mathcal{P}\sigma=& \mathcal{P}D_t (I-\mathcal{H}_{\zeta})\rho+\mathcal{P}[D_t, \mathcal{H}_{\zeta}]\rho\\
    =& D_t\mathcal{P}(I-\mathcal{H}_{\zeta})\rho+[\mathcal{P}, D_t]\rho+ \mathcal{P}[D_t, \mathcal{H}_{\zeta}]\rho\\
    =& D_tN_1+D_tN_2+D_tN_3+D_tN_4+[\mathcal{P}, D_t]\rho+ \mathcal{P}[D_t, \mathcal{H}_{\zeta}]\rho\\
    :=& M_1+M_2+M_3+M_4+M_5+M_6.
\end{align*}
	
	\subsection{Energy functional}\label{subsection:energyfunctional}
	
	\begin{lemma}[Basic lemma, Lemma 4.1 in \cite{Wu2009}]\label{basic}
		Let $\Theta$ satisfy the equation
		$$(D_t^2-iA\partial_{\alpha})\Theta=G.$$ Suppose that  $\Theta, D_t\Theta, G\in H^{s}(q\mathbb{T})$ for some $s\geq 4$. Define
		\begin{equation}
		E_0(t):=\int_{-q\pi}^{q\pi}\frac{1}{A}|D_t\Theta(\alpha,t)|^2+i\Theta(\alpha,t)\partial_{\alpha}\bar{\Theta}(\alpha,t)\,d\alpha.
		\end{equation}
		Then 
		\begin{equation}
		\frac{dE_0}{dt}=\int_{-q\pi}^{q\pi} \frac{2}{A}\Re(D_t\Theta \bar{G})-\frac{a_t}{a}\circ\kappa^{-1}\frac{1}{A}|D_t\Theta|^2\, d\alpha.
		\end{equation}
		Moreover, if $\Theta$ is the boundary value of a holomorphic function in $\Omega(t)^c$, then 
		\begin{equation}
		\int_{-q\pi}^{q\pi} i\Theta\partial_{\alpha}\bar{\Theta}d\alpha=-\int_{-q\pi}^{q\pi} i\bar{\Theta}\partial_{\alpha}\Theta\, d\alpha \geq 0.
		\end{equation}
	\end{lemma}
	
	\vspace*{2ex}
	
	\noindent \textbf{Notations:} Denote
	\begin{equation}
	\rho^{(n)}:=\partial_{\alpha}^n\rho,\quad \quad \quad \sigma^{(n)}:=\partial_{\alpha}^n \sigma.
	\end{equation}
	Because $\rho^{(n)}$ and $\sigma^{(n)}$ are not necessarily holomorphic in $\Omega(t)^c$, we decompose them as
	\begin{equation}\label{decompose}
	\begin{split}
	\rho^n=&\frac{1}{2}(I-\mathcal{H}_{\zeta})\rho^n+\frac{1}{2}(I+\mathcal{H}_{\zeta})\rho^{(n)}:=\phi^{(n)}+\mathcal{R}^{(n)}\\
	\sigma^{(n)}=&\frac{1}{2}(I-\mathcal{H}_{\zeta})\sigma^{(n)}+\frac{1}{2}(I+\mathcal{H}_{\zeta})\sigma^n:=\Psi^{(n)}+\mathcal{S}^{(n)}.
	\end{split}
	\end{equation}
	and define
	\begin{equation}
	\mathcal{E}_{n}(t):=\int_{-q\pi}^{q\pi} \frac{1}{A}|D_t\rho^{(n)}|^2+i\phi^{(n)}\partial_{\alpha}\bar{\phi}^{(n)}\,d\alpha.
	\end{equation}
	
	\begin{equation}
	\mathcal{F}_{n}(t):=\int_{-q\pi}^{q\pi} \frac{1}{A}|D_t\sigma^{(n)}|^2+i\sigma^{(n)}\partial_{\alpha}\bar{\sigma}^{(n)}\,d\alpha.
	\end{equation}
	Define the energy as
	\begin{equation}
	\mathcal{E}(t):=\sum_{n=0}^s ( \mathcal{E}_{n}(t)+ \mathcal{F}_{n}(t)).
	\end{equation}
	By lemma \ref{basic},  each $\mathcal{E}_n$ is positive.

\subsection{Evolution of $\mathcal{E}_n$ and $\mathcal{F}_n$}  To show that $r$ remains small (in the sense of some appropriate norm), we need to show that the energy $\mathcal{E}$ remains small for a long time. To achieve this goal, we  analyze the evolution of $\mathcal{E}_n$ and $\mathcal{F}_n$.  Note that
\begin{equation}\label{hhaa}
\begin{split}
(D_t^2-iA\partial_{\alpha})\rho^{(n)}=& \partial_{\alpha}^n (D_t^2-iA\partial_{\alpha})\rho+[D_t^2-iA\partial_{\alpha}, \partial_{\alpha}^n]\rho\\
=& \sum_{m=1}^4 \partial_{\alpha}^n N_m +[D_t^2-iA\partial_{\alpha}, \partial_{\alpha}^n]\rho\\
:=& \mathcal{C}_{1,n}.
\end{split}
\end{equation}
Similarly, we derive the governing equation for $\sigma^{(n)}=\partial_{\alpha}^n \sigma$. By direct computations, one has 
\begin{equation}\label{hhaa1}
\begin{split}
(D_t^2-iA\partial_{\alpha})\sigma^{(n)}=& \partial_{\alpha}^n (D_t^2-iA\partial_{\alpha})\sigma+[D_t^2-iA\partial_{\alpha}, \partial_{\alpha}^n]\sigma\\
:=& C_{2,n}.
\end{split}
\end{equation}

\noindent By the basic lemma \ref{basic}, equations (\ref{hhaa}) and (\ref{hhaa1}), we have 
\begin{equation}
\begin{split}
\frac{d}{dt}\mathcal{E}_n(t)=& \int_{-q\pi}^{q\pi} \frac{2}{A}\Re(D_t\rho^{(n)} \bar{\mathcal{C}}_{1,n}) -\frac{a_t}{a}\circ\kappa^{-1}\frac{1}{A}|D_t\rho^{(n)}|^2 \,d\alpha\\
&+2\Im \int_{-q\pi}^{q\pi} \partial_t\mathcal{R}^{(n)}\partial_{\alpha}\bar{\phi}^{(n)}+\partial_t\mathcal{\phi}^{(n)}\partial_{\alpha}\bar{\mathcal{R}}_{}^{(n)}
+\partial_t\mathcal{R}^{(n)}\partial_{\alpha}\bar{\mathcal{R}}^{(n)}\,d\alpha
\end{split}
\end{equation}
and
\begin{equation}
\begin{split}
\frac{d}{dt}\mathcal{F}_n(t)=& \int_{-q\pi}^{q\pi} \frac{2}{A}\Re(D_t\sigma^{(n)} \bar{\mathcal{C}}_{2,n}) -\frac{a_t}{a}\circ\kappa^{-1}\frac{1}{A}|D_t\sigma^{(n)}|^2 \,d\alpha.
\end{split}
\end{equation}

	\section{Preparations for energy estimates}\label{section:aprioribound}

	In this section, we estimate the quantities which will be used in the energy estimates in the next section. We bound these quantities in terms of an auxiliary quantity $E_s$, which is essentially equivalent to the energy $\mathcal{E}$. 
	
	\subsection{An auxiliary quantity for the energy functional and a priori assumptions}
	The energy functional $\mathcal{E}$ is not very convenient in the energy estimates, so we introduce the quantity
	\begin{equation}
	E_s(t)^{1/2}:=\norm{D_t r(\cdot,t)}_{H^{s+1/2}(q\mathbb{T})}+\norm{r_{\alpha}(\cdot,t)}_{H^{s}(q\mathbb{T})}+\norm{D_t^2r(\cdot,t)}_{H^s(q\mathbb{T})}.
	\end{equation}
	Let $T_0>0$. We make the following a priori assumptions.
	
	\begin{itemize}
	    \item [1.] (Bootstrap assumption 1)
	    \begin{equation}\label{boot}
	\sup_{t\in [0,T_0]}E_s(t)^{1/2}\leq C\epsilon^{3/2}\delta e^{\epsilon^2 t}, \quad \quad \sup_{t\in [0,T_0]}\norm{D_t\zeta(\cdot,t)}_{H^s(q\mathbb{T})}\leq C\epsilon q^{1/2}.
	\end{equation}
	\item [2.] (Assumption 2) We assume $B$ satisfies
	\begin{equation}\label{aprioritwo}
	    \sup_{t\in [0,T_0]} \norm{B(\epsilon(\alpha+\frac{1}{2\omega}t),\epsilon^2 t)-i}_{H^{s'}(q\mathbb{T})}\leq C\epsilon^{-1/2}\delta e^{\epsilon^2 t}.
	\end{equation}
	\end{itemize}
	Here, $C>0$ is a constant depending on $s$ only.

	\begin{rem}
	The bootstrap assumption \eqref{boot} will be justified easily by the energy estimates in \S \ref{section:energyestimates}. The assumption \eqref{aprioritwo} is satisfied for $B(\alpha,t)$ given in Proposition \ref{prop:nlsestimates}.
	\end{rem}
	
	\vspace*{1ex}

	We will control $\frac{d\mathcal{E}}{dt}$ in terms of $E_s$ and $\epsilon$. Then we can obtain energy estimates on a lifespan of length $O(\epsilon^{-2}\log\frac{\mu}{\delta})$. For this purpose, we control the quantities appear in the energy estimates in terms of $E_s$ and $\epsilon$. 
	
	\noindent \textbf{Convention.} In this and the next sections, if not specified,  we assume $$0\leq t\leq \min\{T_0, \epsilon^{-2}\log\frac{\mu}{\delta}\}$$
	and the bootstrap assumption (\ref{boot}) holds. Here $0<\delta\ll 1$ is an arbitrary given number and $\delta\ll\mu< 1 $ is fixed, independent of $\delta$ and $\epsilon$. These $\delta$ and $\mu$ are given as in Proposition \ref{prop:nlsestimates}. If not specified, $C>0$ is a constant depending on $s$ only.
	\subsubsection{Consequences of the a priori assumptions}
	
	\begin{lemma}\label{lemma:tildezetaminustildezetaST}
	    Assume (\ref{aprioritwo}), then we have 
	    \begin{equation}
	        \sup_{t\in [0,T_0]}\norm{\tilde{\zeta}-\tilde{\zeta}_{ST}}_{H^{s'}(q\mathbb{T})}\leq C\epsilon^{1/2}\delta e^{\epsilon^2 t}.
	    \end{equation}
	\end{lemma}
	\begin{proof}
	This is a direct consequence of the definitions of $\tilde{\zeta}$ and $\tilde{\zeta}_{ST}$ and the assumption (\ref{aprioritwo}).
	\end{proof}

	\begin{lemma}\label{lemmazeta}
		Assuming the bootstrap assumption (\ref{boot}), we have 
		\begin{itemize}
	\item [1.]	\begin{equation}
		\sup_{t\in [0,T_0]}\|\zeta_{\alpha}(\cdot,t)-1\|_{W^{s-1,\infty}}\leq C\epsilon.
		\end{equation}
\item [2.]		\begin{equation}\label{estimate:zetaMinuszetaST}
		    \sup_{t\in [0,T_0]}\norm{\partial_{\alpha}(\zeta(\cdot,t)-\zeta_{ST}(\cdot,t))}_{H^s(q\mathbb{T})}\leq C\epsilon^{1/2}\delta e^{\epsilon^2 t}.
		\end{equation}

\item [3.]			\begin{equation}\label{estimate:DtzetaminusDtST}
    \sup_{t\in [0,T_0]}\norm{D_t\zeta-D_t^{ST}\zeta_{ST}}_{H^s(q\mathbb{T})}\leq C\norm{b-b_{ST}}_{H^s(q\mathbb{T})}+C\epsilon^{1/2}\delta e^{\epsilon^2 t}.
	\end{equation}  
\item [4.]	\begin{equation}
	 \sup_{t\in [0,T_0]}\norm{D_t\zeta-D_t^{ST}\zeta_{ST}-(\tilde{D}_t\tilde{\zeta}-\tilde{D}_t^{ST}\tilde{\zeta}_{ST})}_{H^s(q\mathbb{T})}\leq C\epsilon\norm{b-b_{ST}}_{H^s(q\mathbb{T})}+C\epsilon^{3/2}\delta e^{\epsilon^2 t}.
	\end{equation}
	\end{itemize}
	\end{lemma}
	\begin{proof}
		We decompose $\zeta_\alpha-1$ as
		\begin{align*}
		\zeta_{\alpha}-1=r_{\alpha}+\partial_{\alpha}(\tilde{\zeta}-\tilde{\zeta}_{ST})+\partial_{\alpha}\zeta_{ST}-1.
		\end{align*}
	Taking the norm, we get
		\begin{equation}\label{zetaalpha}
		\norm{\zeta_{\alpha}-1}_{W^{s-1,\infty}}\leq \norm{r_{\alpha}}_{W^{s-1,\infty}}+\norm{\partial_{\alpha}(\tilde{\zeta}-\tilde{\zeta}_{ST})+\partial_{\alpha}\zeta_{ST}-1}_{W^{s-1,\infty}}\leq C\epsilon.
		\end{equation}
where we applied the bootstrap assumption for the first term and used the definitions for $\tilde{\zeta}$ and $\tilde{\zeta}_{ST}$, see \eqref{eq:tildezeta} and \eqref{eq:tildezetaST} for the second term.

		To estimate (\ref{estimate:zetaMinuszetaST}), we notice that 
		\begin{equation}
		    \zeta-\zeta_{ST}=r+(\tilde{\zeta}-\tilde{\zeta}_{ST}).
		\end{equation}by the definition of $r$, \eqref{eq:r}.
By Lemma \ref{lemma:tildezetaminustildezetaST}, we have 
	\begin{align*}
	    \norm{\partial_{\alpha}(\zeta-\zeta_{ST})}_{H^s(q\mathbb{T})}\leq & \norm{\partial_{\alpha}(\tilde{\zeta}-\tilde{\zeta}_{ST})}_{H^s(q\mathbb{T})}+\norm{r_{\alpha}}_{H^s(q\mathbb{T})}\leq C\epsilon^{1/2} \delta e^{\epsilon^2 t}.
	\end{align*}
	For \eqref{estimate:DtzetaminusDtST}, using
	\begin{equation}\label{eqn:DtzetaDtstzest}
	\begin{split}
	    D_t\zeta-D_t^{ST}\zeta_{ST}=& D_t(\zeta-\zeta_{ST}-(\tilde{\zeta}-\tilde{\zeta}_{ST}))+D_t(\tilde{\zeta}-\tilde{\zeta}_{ST})+(D_t-D_t^{ST})\zeta_{ST}\\
	    =&D_tr+(b-b_{ST})\partial_{\alpha}(\tilde{\zeta}-\tilde{\zeta}_{ST})+D_t^{ST}(\tilde{\zeta}-\tilde{\zeta}_{ST})+(b-b_{ST})\partial_{\alpha}\zeta_{ST}
	    \end{split}
	\end{equation} 
	one has
	\begin{align*}
	    \norm{D_t\zeta-D_t^{ST}\zeta_{ST}}_{H^s(q\mathbb{T})}\leq &\norm{D_tr}_{H^s}+C\epsilon^{1/2}\delta e^{\epsilon^2 t}+C \norm{b-b_{ST}}_{H^s(q\mathbb{T})}\leq C\epsilon^{1/2}\delta e^{\epsilon^2 t}+C \norm{b-b_{ST}}_{H^s(q\mathbb{T})}.
	\end{align*}
	Combing \eqref{eqn:DtzetaDtstzest} together with
	\begin{equation}
	    \tilde{D}_t\tilde{\zeta}-\tilde{D}_t^{ST}\tilde{\zeta}_{ST}=\tilde{D}_t(\tilde{\zeta}-\tilde{\zeta}_{ST})+(\tilde{b}-\tilde{b}_{ST})\tilde{\zeta}_{ST},
	\end{equation}
	we obtain
	\begin{equation}
	    \norm{D_t\zeta-D_t^{ST}\zeta_{ST}-(\tilde{D}_t\tilde{\zeta}-\tilde{D}_t^{ST}\tilde{\zeta}_{ST})}_{H^s(q\mathbb{T})}\leq C\epsilon\norm{b-b_{ST}}_{H^s(q\mathbb{T})}+C\epsilon^{3/2}\delta e^{\epsilon^2 t}.
	\end{equation}
We are done.
	\end{proof}
	
	\begin{rem}
	Using  (\ref{estimate:DtbminusbST}), one has
	\begin{align*}
	    \norm{D_t\zeta}_{H^s(q\mathbb{T})}\leq & \norm{D_t^{ST}\zeta_{ST}}_{H^s(q\mathbb{T})}+\norm{D_t\zeta-D_t^{ST}\zeta_{ST}}_{H^s(q\mathbb{T})}.
	\end{align*}
	So the bootstrap assumption $\norm{D_t\zeta}_{H^s(q\mathbb{T})}\leq C\epsilon q^{1/2}$ can be easily justified once we establish the estimate $\norm{b-b_{ST}}_{H^s(q\mathbb{T})}\leq C\epsilon^{3/2}\delta e^{\epsilon^2 t}$.
	\end{rem}
	
	\begin{lemma}\label{lemma:apriori:layer}
		Assuming the bootstrap assumption (\ref{boot}), we have \begin{equation}
		    \sup_{t\in [0,T_0]}\norm{(I\pm \mathcal{K}_{\zeta}(\cdot,t))^{-1}f}_{H^s(q\mathbb{T})}\leq 2\norm{f}_{H^s(q\mathbb{T})}, \quad \quad \sup_{t\in [0,T_0]}\norm{(I\pm \mathcal{K}_{\zeta_{ST}}(\cdot,t))^{-1}f}_{H^s(q\mathbb{T})}\leq 2\norm{f}_{H^s(q\mathbb{T})},
		\end{equation}
		and 
		\begin{equation}
		    \sup_{t\in [0,T_0]}\norm{(I\pm \mathcal{K}_{\zeta}^{\ast}(\cdot,t))^{-1}f}_{H^s(q\mathbb{T})}\leq 2\norm{f}_{H^s(q\mathbb{T})}, \quad \quad \sup_{t\in [0,T_0]}\norm{(I\pm \mathcal{K}_{\zeta_{ST}}^{\ast}(\cdot,t))^{-1}f}_{H^s(q\mathbb{T})}\leq 2\norm{f}_{H^s(q\mathbb{T})},
		\end{equation}
where $\mathcal{K}_\gamma$ and $\mathcal{K}^*_\gamma$ are the double layer potential operator and its adjoint associated with $\gamma$, see Definition \ref{def:double}.
	\end{lemma}
	Lemma \ref{lemma:apriori:layer} is the direct consequence of Lemma \ref{layer} and the bootstrap assumption (\ref{boot}). Alternatively, we can also prove this lemma as follows.
	\begin{proof}
	 Let $f$ be a real-valued function, then
	\begin{equation}
	    \mathcal{K}_{\zeta}f(\alpha,t)=\Re\Big\{\mathcal{H}_{\zeta}f(\alpha,t)-H_0f(\alpha,t)\Big\}.
	\end{equation}
The bootstrap assumption \eqref{boot} and the Sobolev embedding Lemma \ref{sobolev} give us
	\begin{align*}
	    \norm{\mathcal{K}_{\zeta}f}_{H^s(q\mathbb{T})}\leq & C\norm{\zeta_{\alpha}-1}_{W^{s-1,\infty}(q\mathbb{T})}\norm{f}_{H^s(q\mathbb{T})}+ C\norm{\zeta_{\alpha}-1}_{H^{s}(q\mathbb{T})}\norm{f}_{W^{s-1,\infty}(q\mathbb{T})}\\
	    \leq & C\epsilon \norm{f}_{H^s(q\mathbb{T})}.
	\end{align*}
	So we have 
	\begin{equation}
	    \norm{\mathcal{K}_{\zeta}}_{\mathcal{L}_{H^s\rightarrow H^s}}\leq C\epsilon,
	\end{equation}
	which implies 
	\begin{equation}
	    \norm{(I\pm \mathcal{K}_{\zeta})^{-1}f}_{H^s(q\mathbb{T})}\leq (1+C\epsilon)\norm{f}_{H^s(q\mathbb{T})}\leq 2\norm{f}_{H^s(q\mathbb{T})},
	\end{equation}
	provided that $C\epsilon<1$ and $t\leq \epsilon^{-2}\log\frac{\mu}{\delta}$. Other inequalities can be proved similarly.
	\end{proof}
	
	\begin{lemma}\label{realinverse1}
		Assume the bootstrap assumption (\ref{boot}). Let $g, h$ be real functions. Suppose
		$$(I-\mathcal{H}_{\zeta})h\bar{\zeta}_{\alpha}=g\quad \quad \text{or}\quad \quad (I-\mathcal{H}_{\zeta})h=g.$$
		Then we have for any $t\in [0,T_0]$,
		\begin{equation}
		\|h\|_{H^s(q\mathbb{T})}\leq 2\|g\|_{H^s(q\mathbb{T})}.
		\end{equation}
	\end{lemma}
	\begin{proof}
	We consider the case $(I-\mathcal{H}_{\zeta})h=g$ only, the case for $(I-\mathcal{H}_{\zeta})h\bar{\zeta}_{\alpha}$ follows in a similar manner. Since $h$ is real, taking the real parts on both sides of $(I-\mathcal{H}_{\zeta})h=g$, we obtain
	\begin{equation}
	    (I-\mathcal{K}_{\zeta})h=\Re\{g\}.
	\end{equation}
	By Lemma \ref{lemma:apriori:layer}, we get
	\begin{equation}
	    \norm{h}_{H^s(q\mathbb{T})}\leq 2\norm{g}_{H^s(q\mathbb{T})}
	\end{equation}as desired.
	\end{proof}
	
	\subsubsection{The equivalence of $\rho$ and $r$}
	\begin{lemma}\label{equivalencequantities}
		Assume the a priori assumption (\ref{boot}).  We have 
		\begin{equation}
		\norm{\partial_{\alpha}(\rho-2r)}_{H^s(q\mathbb{T})}\leq C\delta e^{\epsilon^2t}\epsilon^{5/2},
		\end{equation}
		\begin{equation} \norm{D_t\rho-2D_tr}_{H^{s+1/2}(q\mathbb{T})}\leq C(\epsilon E_s^{1/2}+\delta e^{\epsilon^2t}\epsilon^{5/2})\leq C\epsilon^{5/2}\delta e^{\epsilon^2 t}.
		\end{equation}
				\begin{equation} \norm{D_t(D_t\rho-2D_tr)}_{H^{s}(q\mathbb{T})}\leq  C\epsilon^{5/2}\delta e^{\epsilon^2 t}.
		\end{equation}
	\end{lemma}
	The proof is similar to that of Lemma 8.3 in \cite{su2020partial}. However, we need an additional gain of the factor $\delta e^{\epsilon^2 t}$. For the sake of completeness, we provide the proof in Appendix \ref{estimatesforsomequantities}.
	
	\begin{cor}\label{equivalencerho1}
		Assuming the a priori assumption (\ref{boot}), we have 
		\begin{equation}
		\norm{\partial_{\alpha}\rho}_{H^s(q\mathbb{T})}\leq C\epsilon^{3/2}\delta e^{\epsilon^2 t},\quad \quad \norm{D_t\rho}_{H^{s+1/2}(\mathbb{T})}\leq C\epsilon^{3/2}\delta e^{\epsilon^2 t}.
		\end{equation}
	\end{cor}
	
	\subsection{Bound $\tilde{b}$, $b$, $b_{ST}$, $\tilde{b}_{ST}$, $b-\tilde{b}$, $b_{ST}-\tilde{b}_{ST}$ and $b-\tilde{b}-(b_{ST}-\tilde{b}_{ST})$} 
	
	\subsubsection{Estimate $\tilde{b}_{ST}$}\label{subsection:estimatesforb}
	Recalling  the notation \eqref{eq:notationb}, from Proposition \ref{prop:bstAst2}, we know that $$\tilde{b}_{ST}=-\epsilon^2\omega.$$
	
	\subsubsection{Estimate $b_{ST}$} Recall that $b_{ST}$ is given by
	\begin{equation}
	    (I-\mathcal{H}_{\zeta_{ST}})b_{ST}=-[(D_t^{ST})\zeta_{ST}, \mathcal{H}_{\zeta_{ST}}]\frac{\partial_{\alpha}\bar{\zeta}_{ST}-1}{\partial_{\alpha}\zeta_{ST}}.
	\end{equation}
	Also recall that $b_{ST}(\alpha,t)=b_{ST}(\alpha+\omega t,0)$ since the Stokes wave is a traveling wave. Also see Corollary  \ref{cor:travelingAB}. Since $b_{ST}$ is real, by Lemma \ref{realinverse1} and Proposition \ref{singularperiodic}, we have 
	\begin{equation}\label{estimate: bST}
	\begin{split}
	    \sup_{t\in \mathbb{R}}\norm{b_{ST}}_{H^{s'}(q\mathbb{T})}\leq & C\norm{[(D_t^{ST})\zeta_{ST}, \mathcal{H}_{\zeta_{ST}}]\frac{\partial_{\alpha}\bar{\zeta}_{ST}-1}{\partial_{\alpha}\zeta_{ST}}}_{H^{s'}(q\mathbb{T})}\\
	    \leq & C\norm{(D_t^{ST}\zeta_{ST}(\cdot,0)}_{W^{s'-1,\infty}}\norm{\partial_{\alpha}\zeta_{ST}(\cdot,0)-1}_{H^{s'}(q\mathbb{T})}\\
	    &+C\norm{(D_t^{ST}\zeta_{ST}(\cdot,0)}_{W^{s'-1,\infty}}\norm{\partial_{\alpha}\zeta_{ST}(\cdot,0)-1}_{H^{s'}(q\mathbb{T})}\\
	    \leq & C\epsilon^2 q^{1/2}.
	    \end{split}
	\end{equation}
	By the Sobolev embedding, Lemma \ref{sobolev}, we have 
	\begin{equation}
	    \norm{b_{ST}}_{W^{s'-1}}\leq C\epsilon^2.
	\end{equation}

		\subsubsection{Estimate $b_{ST}-\tilde{b}_{ST}$}
By direct computations, one has
\begin{equation}
    \norm{b_{ST}-\tilde{b}_{ST}}_{H^{s'}(q\mathbb{T})}\leq C\epsilon^3 q^{1/2}.
\end{equation}
Applying Lemma \ref{sobolev}, we obtain
	\begin{equation}
	    \norm{b_{ST}-\tilde{b}_{ST}}_{W^{s'-1}}\leq C\epsilon^3.
	\end{equation}

	\subsubsection{Estimate $\tilde{b}$} Recall that $\tilde{b}=-\epsilon^2\omega |B|^2$, see \eqref{eq:b2}. Here, $B=B(X,T)$. By the assumption (\ref{aprioritwo}), one has
	\begin{equation}
	    \norm{\tilde{b}}_{H^{s'}(q\mathbb{T})}\leq C\epsilon^2 \norm{|B|^2}_{H^{s'}(q\mathbb{T})}\leq C\epsilon^2q^{1/2}.
	\end{equation}
Using the Sobolev embedding again, Lemma \ref{sobolev}, it follows
	\begin{equation}
	    \norm{\tilde{b}}_{W^{s'-1,\infty}(q\mathbb{T})}\leq C\epsilon^2.
	\end{equation}
Moreover, we also obtain
	\begin{equation}\label{estimates:tildeb:quadratic}
	    \sup_{t\in [0,T_0]}\norm{\tilde{b}-\tilde{b}_{ST}}_{H^s(q\mathbb{T})}\leq C\epsilon^{3/2}\delta e^{\epsilon^2 t}.
	\end{equation}
\subsubsection{Estimate $b$ and $b-b_{ST}$}\label{estimates:bandbminusbst}
Recall that 
\begin{equation*}
	(I-\mathcal{H}_{\zeta})b=-[D_t\zeta, \mathcal{H}_{\zeta}]\frac{\bar{\zeta}_{\alpha}-1}{\zeta_{\alpha}}.
	\end{equation*}
	Also, 
	\begin{equation}
	    (I-\mathcal{H}_{\zeta_{ST}})b_{ST}=-[D_t^{ST}\zeta_{ST}, \mathcal{H}_{\zeta_{ST}}]\frac{\partial_{\alpha}\bar{\zeta}_{ST}-1}{\partial_{\alpha}\zeta_{ST}}.
	\end{equation}
Taking the difference of two expressions above, after some simple manipulations, we get 
\begin{equation}\label{formula:bminusbST}
    (I-\mathcal{H}_{\zeta_{ST}})(b-b_{ST})=(\mathcal{H}_{\zeta}-\mathcal{H}_{\zeta_{ST}})b_{ST}-\Big( [D_t\zeta, \mathcal{H}_{\zeta}]\frac{\bar{\zeta}_{\alpha}-1}{\zeta_{\alpha}}-[D_t^{ST}\zeta_{ST}, \mathcal{H}_{\zeta_{ST}}]\frac{\partial_{\alpha}\bar{\zeta}_{ST}-1}{\partial_{\alpha}\zeta_{ST}}\Big).
\end{equation}
	By (3) of Lemma \ref{boundednesshilbert}, (\ref{estimate:zetaMinuszetaST}), and (\ref{estimate: bST}), we have 
	\begin{equation}
	    \norm{(\mathcal{H}_{\zeta}-\mathcal{H}_{\zeta_{ST}})b_{ST}}_{H^s(q\mathbb{T})}\leq C\norm{\partial_{\alpha}(\zeta-\zeta_{ST})}_{H^s(q\mathbb{T})}\norm{b_{ST}}_{W^{s,\infty}}\leq C\epsilon^{5/2}\delta e^{\epsilon^2 t}.
	\end{equation}
	Applying \eqref{eqn:quadraticsingular} of Proposition \ref{commutator:difference:quadratic} with $\zeta_1:=\zeta$, $\zeta_2:=\zeta_{ST}$, $f_1:=D_t\zeta$, $f_2:=D_t^{ST}\zeta_{ST}$, $g_1:=\bar{\zeta}_{\alpha}$, and $g_2:=\bar{\zeta}_{ST}$, together with \eqref{boot} and \eqref{aprioritwo}, one has
	\begin{align*}
	    &\norm{[D_t\zeta, \mathcal{H}_{\zeta}]\frac{\bar{\zeta}_{\alpha}-1}{\zeta_{\alpha}}-[D_t^{ST}\zeta_{ST}, \mathcal{H}_{\zeta_{ST}}]\frac{\partial_{\alpha}\bar{\zeta}_{ST}-1}{\partial_{\alpha}\zeta_{ST}}}_{H^s(q\mathbb{T})}\\
	    \leq & C\norm{\partial_{\alpha}(\zeta-\zeta_{ST})}_{H^s(q\mathbb{T})}\norm{D_t\zeta}_{W^{s-1,\infty}}\norm{\bar{\zeta}_{\alpha}}_{W^{s-1,\infty}}+C\norm{D_t\zeta-D_t^{ST}\zeta_{ST}}_{H^s(q\mathbb{T})}\norm{\bar{\zeta}_{\alpha}-1}_{W^{s-1,\infty}}\\
	    &+\norm{D_t^{ST}\zeta_{ST}}_{W^{s-1,\infty}}\norm{\partial_{\alpha}(\zeta-\zeta_{ST})}_{H^s(q\mathbb{T})}\\
	    \leq & C\epsilon^{3/2} \delta e^{\epsilon^2 t}+C\epsilon\norm{b-b_{ST}}_{H^s(q\mathbb{T})}.
	\end{align*}
%
From computations above, we conclude that
	\begin{align*}
	    \norm{b-b_{ST}}_{H^s(q\mathbb{T})}\leq C\epsilon^{3/2}\delta e^{\epsilon^2 t}+C\epsilon \norm{b-b_{ST}}_{H^s(q\mathbb{T})},
	\end{align*}
	which implies 
	\begin{equation}\label{estimates:bminusbST}
	    \norm{b-b_{ST}}_{H^s(q\mathbb{T})}\leq C\epsilon^{3/2}\delta e^{\epsilon^2 t}.
	\end{equation}
	As a consequence, we obtain
	\begin{equation}
	    \sup_{t\in [0,T_0]}\norm{b}_{H^s(q\mathbb{T})}\leq C\epsilon^2 q^{1/2}+C\epsilon^{3/2}\delta e^{\epsilon^2 t}\leq C\epsilon^2 q^{1/2},
	\end{equation}
	\begin{equation}
	    \sup_{t\in [0,T_0]}\norm{b}_{W^{s-1,\infty}}\leq C\epsilon^2.
	\end{equation}
	\begin{equation}
	    \sup_{t\in [0,T_0]}\norm{b-b_{ST}}_{W^{s-1,\infty}}\leq C\epsilon^{3/2}\delta e^{\epsilon^2 t}q^{-1/2}.
	\end{equation}

	\subsubsection{Estimate $b-b_{ST}-(\tilde{b}-\tilde{b}_{ST})$}\label{estimates:difference:b} Note that by the explicit formula of $\tilde{b}$ and $\tilde{b}_{ST}$,
	\begin{equation}
	    \tilde{b}-\tilde{b}_{ST}=-[\tilde{D}_t\tilde{\zeta}, \mathcal{H}_{\tilde{\zeta}}]\frac{\partial_{\alpha}\tilde{\zeta}-1}{\tilde{\zeta}_{\alpha}}+[\tilde{D}_t^{ST}\tilde{\zeta}_{ST}, \mathcal{H}_{\tilde{\zeta}_{ST}}]\frac{\partial_{\alpha}\tilde{\zeta}_{ST}-1}{\partial_{\alpha}\tilde{\zeta}_{ST}}+e,
	\end{equation}
	where
	\begin{equation}
	    \norm{e}_{H^s(q\mathbb{T})}\leq C\epsilon^{5/2}\delta e^{\epsilon^2 t}.
	\end{equation}
	Applying Proposition \ref{commutator:difference:cubic} with $\zeta_1:=\zeta$, $\zeta_2:=\zeta_{ST}$, $\zeta_3:=\tilde{\zeta}$, $\zeta_4:=\tilde{\zeta}_{ST}$, $g_1:=D_t\zeta$, $g_2:=D_t^{ST}\zeta_{ST}$, $g_3:=\tilde{D}_t\tilde{\zeta}$, $g_4:=\tilde{D}_t^{ST}\tilde{\zeta}_{ST}$, $f_1:=\partial_{\alpha}\bar{\zeta}-1$, $f_2:=\partial_{\alpha}\bar{\zeta}_{ST}-1$, $f_3:=\partial_{\alpha}\tilde{\zeta}-1$, $f_4:=\partial_{\alpha}\tilde{\zeta}_{ST}-1$, from Proposition \ref{commutator:difference:quadratic} and Proposition \ref{commutator:difference:cubic}
, Proposition \ref{proposition:Hilbertquartic}, Lemma \ref{lemmazeta}, it follows
\begin{align*}
    &\norm{\Big([D_t\zeta, \mathcal{H}_{\zeta}]\frac{\bar{\zeta}_{\alpha}-1}{\zeta_{\alpha}}-[D_t^{ST}\zeta_{ST}, \mathcal{H}_{\zeta_{ST}}]\frac{\partial_{\alpha}\bar{\zeta}_{ST}-1}{\partial_{\alpha}\zeta_{ST}}\Big)-\Big([\tilde{D}_t\tilde{\zeta}, \mathcal{H}_{\tilde{\zeta}}]\frac{\partial_{\alpha}\tilde{\zeta}-1}{\tilde{\zeta}_{\alpha}}-[\tilde{D}_t^{ST}\tilde{\zeta}_{ST}, \mathcal{H}_{\tilde{\zeta}_{ST}}]\frac{\partial_{\alpha}\tilde{\zeta}_{ST}-1}{\partial_{\alpha}\tilde{\zeta}_{ST}}\Big)}_{H^s(q\mathbb{T})}\\
    \leq & C\epsilon\norm{D_t\zeta-D_t^{ST}\zeta_{ST}-(\tilde{D}_t\tilde{\zeta}-\tilde{D}_t\zeta_{ST})}_{H^s(q\mathbb{T})}+C\epsilon^{5/2}\delta e^{\epsilon^2t}\\
    \leq & C\epsilon\norm{D_t\zeta-D_t^{ST}\zeta_{ST}}_{H^s(q\mathbb{T})}+C\norm{\tilde{D}_t\tilde{\zeta}-\tilde{D}_t\zeta_{ST}}_{H^s(q\mathbb{T})}+C\epsilon^{5/2}\delta e^{\epsilon^2t}\\
    \leq &C\epsilon^{5/2}\delta e^{\epsilon^2 t}.
\end{align*}

Therefore, we obtain the following results:
	\begin{equation}\label{estimats:b:quintic}
	      \sup_{t\in [0,T_0]}\norm{b-b_{ST}-(\tilde{b}-\tilde{b}_{ST})}_{H^s(q\mathbb{T})}\leq C\epsilon^{5/2}\delta e^{\epsilon^2 t}.
	\end{equation}
	\begin{equation}
	      \sup_{t\in [0,T_0]}\norm{b-b_{ST}-(\tilde{b}-\tilde{b}_{ST})}_{W^{s-1,\infty}(q\mathbb{T})}\leq C\epsilon^{5/2}\delta e^{\epsilon^2 t}q^{-1/2}.
	\end{equation}

	\subsection{Bound $\tilde{D}_t\tilde{b}$, $D_tb$, $D_t^{ST}b_{ST}$, $D_t^{ST}b_{ST}-\tilde{D}_t^{ST}\tilde{b}_{ST}$, $D_tb-\tilde{D}_t\tilde{b}$, $D_tb-D_t^{ST}b_{ST}-(\tilde{D}_t\tilde{b}-\tilde{D}_t^{ST}\tilde{b}_{ST})$} 
	
	\subsubsection{Estimate $\tilde{D}_t\tilde{b}$} Since $\tilde{b}=-\epsilon^2\omega |B|^2$, $\tilde{D}_t=\partial_t +\tilde{b}\partial_{\alpha}$, and $$\partial_t B=\frac{\epsilon}{2\omega} B_X+\epsilon^2 B_T, \partial_{\alpha}B=\epsilon B_X,$$we directly obtain
	\begin{equation}
	    \norm{\tilde{D}_t\tilde{b}}_{H^s(q\mathbb{T})}\leq  C\delta e^{\epsilon^2 t}\epsilon^{5/2},
	\end{equation}
	and
	\begin{equation}
	    \norm{\tilde{D}_t\tilde{b}}_{W^{s-1,\infty}(q\mathbb{T})}\leq C\delta e^{\epsilon^2 t}\epsilon^{5/2}q^{-1/2}.
	\end{equation}
	
	\subsubsection{Estimate $\tilde{D}_t^{ST}\tilde{b}_{ST}$ , $D_t^{ST}b_{ST}$ and $D_t^{ST}b_{ST}-\tilde{D}_t^{ST}\tilde{b}_{ST}$} We simply have 
	\begin{equation} 
	\tilde{D}_t^{ST}\tilde{b}_{ST}=0
	\end{equation}
	due to $\tilde{b}_{ST}=-\omega$. Since 
	\begin{equation}
	    D_t^{ST}b_{ST}=\partial_t b_{ST}+b_{ST}\partial_{\alpha}b_{ST} 
	\end{equation}
	and $b_{ST}(\alpha,t)=b_{ST}(\alpha+\omega t)$, we have (noting that $b_{ST}=-\omega \epsilon^2+O(\epsilon^3)$)
	\begin{equation}
	    \norm{D_t^{ST}b_{ST}}_{H^{s}(q\mathbb{T})}\leq C\epsilon^3 q^{1/2}.
	\end{equation}
	Therefore,
	\begin{equation}
	    \norm{D_t^{ST}b_{ST}-\tilde{D}_t^{ST}\tilde{b}_{ST}}_{H^s(q\mathbb{T})}\leq C\epsilon^3 q^{1/2}.
	\end{equation}
	
	\begin{cor}\label{cor:Dtsquare}
	Assuming the bootstrap assumption (\ref{boot}), we have 
	\begin{equation}
	    \norm{D_t^2\zeta-(D_t^{ST})^2\zeta_{ST}-\Big(\tilde{D}_t^2\tilde{\zeta}-(\tilde{D}_t^{ST})^2\tilde{\zeta}_{ST}\Big)}_{H^s(q\mathbb{T})}\leq C\epsilon^{3/2}\delta e^{\epsilon^2t}+\norm{D_tb-D_t^{ST}b_{ST}}_{H^s(q\mathbb{T})}.
	\end{equation}
	\end{cor}
	\begin{proof}
	Writing $\zeta=r+\zeta_{ST}+\tilde{\zeta}-\tilde{\zeta}_{ST}$, one has 
	\begin{align*}
	    & D_t^2\zeta -(D_t^{ST})^2\zeta_{ST}-\Big(\tilde{D}_t^2\tilde{\zeta}-(\tilde{D}_t^{ST})^2\Big)\tilde{\zeta}_{ST}\\
	    =& D_t^2r+\Big(D_t^2-(D_t^{ST})^2\Big)\zeta_{ST}+\Big(D_t^2-\tilde{D}_t^2\Big)\tilde{\zeta}-\Big(D_t^2-(\tilde{D}_t^{ST})^2\Big)\tilde{\zeta}_{ST}\\
	    =& D_t^2 r+\Big(D_t^2-(D_t^{ST})^2\Big)\zeta_{ST}+\Big(D_t^2-\tilde{D}_t^2-\Big(D_t^2-(\tilde{D}_t^{ST})^2\Big)\Big)\tilde{\zeta}-\Big(D_t^2-(\tilde{D}_t^{ST})^2\Big)(\tilde{\zeta}_{ST}-\tilde{\zeta})\\
	    =& D_t^2 r+\Big(D_t^2-(D_t^{ST})^2\Big)\zeta_{ST}-\Big(\tilde{D}_t^2-(\tilde{D}_t^{ST})^2\Big)\tilde{\zeta}-\Big(D_t^2-(\tilde{D}_t^{ST})^2\Big)(\tilde{\zeta}_{ST}-\tilde{\zeta})
	\end{align*}
	Using
	\begin{equation}
	\begin{split}
	  D_t^2-(D_t^{ST})^2= & D_t(b-b_{ST})\partial_{\alpha}+(b-b_{ST})\partial_{\alpha}\partial_t+(b-b_{ST})b\partial_{\alpha}^2+(b-b_{ST})\partial_{\alpha}D_t^{ST},
	  \end{split}
	\end{equation}
	one obtains
	\begin{equation}
	    \norm{(D_t^2-(D_t^{ST})^2)\zeta_{ST}}_{H^s(q\mathbb{T})}\leq C\norm{D_t\zeta-D_t^{ST}\zeta_{ST}}_{H^s(q\mathbb{T})}+C\epsilon^{3/2}\delta e^{\epsilon^2 t}.
	\end{equation}
By the similar algebraic manipulations, one has
\begin{equation}
    \norm{\Big(\tilde{D}_t^2-(\tilde{D}_t^{ST})^2\Big)\tilde{\zeta}+\Big(D_t^2-(\tilde{D}_t^{ST})^2\Big)(\tilde{\zeta}_{ST}-\tilde{\zeta})}_{H^s(q\mathbb{T})}\leq C\norm{D_t\zeta-D_t^{ST}\zeta_{ST}}_{H^s(q\mathbb{T})}+C\epsilon^{3/2}\delta e^{\epsilon^2 t}.
\end{equation}
So we conclude the proof of the corollary.
	\end{proof}
	
\subsubsection{Estimate $D_tb-D_t^{ST}b_{ST}-(\tilde{D}_t\tilde{b}-\tilde{D}_t^{ST}\tilde{b}_{ST})$}
Applying $D_t$ on both sides of $
    (I-\mathcal{H}_{\zeta})b=-[D_t\zeta, \mathcal{H}_{\zeta}]\frac{
    \bar{\zeta}_{\alpha}-1)}{\zeta_{\alpha}},
$
we obtain (for the derivation in the Euclidean setting, see Proposition 2.7 of \cite{Wu2009}.)
\begin{equation}\label{Dtb}
\begin{split}
(I-\mathcal{H}_{\zeta})D_tb=& [D_t\zeta,\mathcal{H}_{\zeta}]\frac{\partial_{\alpha}(2b-D_t\bar{\zeta})}{\zeta_{\alpha}}-[D_t^2\zeta,\mathcal{H}_{\zeta}]\frac{\bar{\zeta}_{\alpha}-1}{\zeta_{\alpha}}\\
&+\frac{1}{4q^2\pi i}\int_{-q\pi}^{q\pi} \Big(\frac{D_t\zeta(\alpha)-D_t\zeta(\beta)}{\sin(\frac{\zeta(\alpha)-\zeta(\beta)}{2q})}\Big)^2 (\bar{\zeta}_{\beta}(\beta)-1)\,d\beta.
\end{split}
\end{equation}
Similarly, we have 
\begin{equation}
\begin{split}
    (I-\mathcal{H}_{\zeta_{ST}})D_t^{ST}b_{ST}=& [D_t^{ST}\zeta^{ST},\mathcal{H}_{\zeta_{ST}}]\frac{\partial_{\alpha}(2b_{ST}-D_t^{ST}\bar{\zeta}_{ST})}{\partial_{\alpha}\zeta_{ST}}-[(D_t^{ST})^2\zeta_{ST},\mathcal{H}_{\zeta_{ST}}]\frac{\partial_{\alpha}\bar{\zeta}_{ST}-1}{\partial_{\alpha}\zeta_{ST}}\\
&+\frac{1}{4q^2\pi i}\int_{-q\pi}^{q\pi} \Big(\frac{D_t^{ST}\zeta_{ST}(\alpha)-D_t^{ST}\zeta_{ST}(\beta)}{\sin(\frac{\zeta_{ST}(\alpha)-\zeta_{ST}(\beta)}{2q})}\Big)^2 (\partial_{\beta}\bar{\zeta}_{ST}-1)\,d\beta.
    \end{split}
\end{equation}
Taking the difference of two expressions above, we obtain 
\begin{align*}
   & (I-\mathcal{H}_{\zeta})(D_tb-D_t^{ST}b_{ST})\\
   =& (\mathcal{H}_{\zeta}-\mathcal{H}_{\zeta_{ST}})D_t^{ST} b_{ST}+\Big\{[D_t\zeta,\mathcal{H}_{\zeta}]\frac{\partial_{\alpha}(2b-D_t\bar{\zeta})}{\zeta_{\alpha}}-[D_t^{ST}\zeta^{ST},\mathcal{H}_{\zeta_{ST}}]\frac{\partial_{\alpha}(2b_{ST}-D_t^{ST}\bar{\zeta}_{ST})}{\partial_{\alpha}\zeta_{ST}}\Big\}\\
   &-\Big\{ [D_t^2\zeta,\mathcal{H}_{\zeta}]\frac{\bar{\zeta}_{\alpha}-1}{\zeta_{\alpha}}-[(D_t^{ST})^2\zeta_{ST},\mathcal{H}_{\zeta_{ST}}]\frac{\partial_{\alpha}\bar{\zeta}_{ST}-1}{\partial_{\alpha}\zeta_{ST}}\Big\}\\
    &+\Big\{\frac{1}{4q^2\pi i}\int_{-q\pi}^{q\pi} \Big(\frac{D_t\zeta(\alpha)-D_t\zeta(\beta)}{\sin(\frac{\zeta(\alpha)-\zeta(\beta)}{2q})}\Big)^2 (\bar{\zeta}_{\beta}(\beta)-1)\,d\beta\\
    &-\frac{1}{4q^2\pi i}\int_{-q\pi}^{q\pi} \Big(\frac{D_t^{ST}\zeta_{ST}(\alpha)-D_t^{ST}\zeta_{ST}(\beta)}{\sin(\frac{\zeta_{ST}(\alpha)-\zeta_{ST}(\beta)}{2q})}\Big)^2 (\partial_{\beta}\bar{\zeta}_{ST}-1)\,d\beta\Big\}.
\end{align*}
Similarly,
\begin{align*}
   & (I-\mathcal{H}_{\tilde{\zeta}})(\tilde{D}_t\tilde{b}-\tilde{D}_t^{ST}\tilde{b}_{ST})\\
   =& (\mathcal{H}_{\tilde{\zeta}}-\mathcal{H}_{\tilde{\zeta}_{ST}})\tilde{D}_t^{ST} \tilde{b}_{ST}+\Big\{[\tilde{D}_t\tilde{\zeta},\mathcal{H}_{\tilde{\zeta}}]\frac{\partial_{\alpha}(2\tilde{b}-\tilde{D}_t\bar{\tilde{\zeta}})}{\tilde{\zeta}_{\alpha}}-[\tilde{D}_t^{ST}\tilde{\zeta}^{ST},\mathcal{H}_{\tilde{\zeta}_{ST}}]\frac{\partial_{\alpha}(2\tilde{b}_{ST}-\tilde{D}_t^{ST}\bar{\tilde{\zeta}}_{ST})}{\partial_{\alpha}\tilde{\zeta}_{ST}}\Big\}\\
   &-\Big\{ [\tilde{D}_t^2\tilde{\zeta},\mathcal{H}_{\tilde{\zeta}}]\frac{\bar{\tilde{\zeta}}_{\alpha}-1}{\tilde{\zeta}_{\alpha}}-[(\tilde{D}_t^{ST})^2\tilde{\zeta}_{ST},\mathcal{H}_{\tilde{\zeta}_{ST}}]\frac{\partial_{\alpha}\bar{\tilde{\zeta}}_{ST}-1}{\partial_{\alpha}\tilde{\zeta}_{ST}}\Big\}\\
    &+\Big\{\frac{1}{4q^2\pi i}\int_{-q\pi}^{q\pi} \Big(\frac{\tilde{D}_t\tilde{\zeta}(\alpha)-\tilde{D}_t\zeta(\beta)}{\sin(\frac{\tilde{\zeta}(\alpha)-\tilde{\zeta}(\beta)}{2q})}\Big)^2 (\bar{\tilde{\zeta}}_{\beta}(\beta)-1)\,d\beta\\
    &-\frac{1}{4q^2\pi i}\int_{-q\pi}^{q\pi} \Big(\frac{\tilde{D}_t^{ST}\zeta_{ST}(\alpha)-\tilde{D}_t^{ST}\zeta_{ST}(\beta)}{\sin(\frac{\tilde{\zeta}_{ST}(\alpha)-\tilde{\zeta}_{ST}(\beta)}{2q})}\Big)^2 (\partial_{\beta}\bar{\tilde{\zeta}}_{ST}-1)\,d\beta\Big\}+e,
\end{align*}
where
\begin{equation}
    \norm{e}_{H^s(q\mathbb{T})}\leq C\epsilon^{5/2}\delta e^{\epsilon^2t}.
\end{equation}
Applying (3) of Lemma \ref{boundednesshilbert}, one has
\begin{equation}
    \norm{(\mathcal{H}_{\zeta}-\mathcal{H}_{\zeta_{ST}})D_t^{ST} b_{ST}}_{H^s(q\mathbb{T})}\leq C\norm{\partial_{\alpha}(\zeta-\zeta_{ST})}_{H^s(q\mathbb{T})}\norm{D_t^{ST}b_{ST}}_{W^{s,\infty}}\leq C\epsilon^{5/2}\delta e^{\epsilon^2 t}.
\end{equation}
Similarly,
\begin{equation}
    \norm{(\mathcal{H}_{\tilde{\zeta}}-\mathcal{H}_{\tilde{\zeta}_{ST}})\tilde{D}_t^{ST} \tilde{b}_{ST}}_{H^s(q\mathbb{T})}\leq C\epsilon^{5/2}\delta e^{\epsilon^2 t}.
\end{equation}
Using the same argument as for estimating $\norm{b-b_{ST}}_{H^s(q\mathbb{T})}$, together with the estimates \eqref{estimate:DtzetaminusDtST}, \eqref{estimates:bminusbST}, one has 
\begin{equation}
    \begin{split}
        \norm{[D_t\zeta,\mathcal{H}_{\zeta}]\frac{\partial_{\alpha}(2b)}{\zeta_{\alpha}}-[D_t^{ST}\zeta^{ST},\mathcal{H}_{\zeta_{ST}}]\frac{\partial_{\alpha}(2b_{ST})}{\partial_{\alpha}\zeta_{ST}}}_{H^s(q\mathbb{T})}\leq C\epsilon^{5/2}\delta e^{\epsilon^2 t}.
    \end{split}
\end{equation}
Similarly,
\begin{equation}
    \norm{[\tilde{D}_t\tilde{\zeta},\mathcal{H}_{\tilde{\zeta}}]\frac{\partial_{\alpha}(2\tilde{b})}{\tilde{\zeta}_{\alpha}}-[\tilde{D}_t^{ST}\tilde{\zeta}^{ST},\mathcal{H}_{\tilde{\zeta}_{ST}}]\frac{\partial_{\alpha}(2\tilde{b}_{ST})}{\partial_{\alpha}\tilde{\zeta}_{ST}}}_{H^s(q\mathbb{T})}\leq C\epsilon^{5/2}\delta e^{\epsilon^2 t}.
\end{equation}
Taking the differences and using Proposition \ref{singularperiodic}, we obtain
\begin{equation}
    \begin{split}
        &\Big|\Big|\frac{1}{4q^2\pi i}\int_{-q\pi}^{q\pi} \Big(\frac{D_t\zeta(\alpha)-D_t\zeta(\beta)}{\zeta(\alpha)-\zeta(\beta)}\Big)^2 (\bar{\zeta}_{\beta}(\beta)-1)\,d\beta\\
        &-\frac{1}{4q^2\pi i}\int_{-q\pi}^{q\pi} \Big(\frac{D_t^{ST}\zeta_{ST}(\alpha)-D_t^{ST}\zeta_{ST}(\beta)}{\zeta_{ST}(\alpha)-\zeta_{ST}(\beta)}\Big)^2 (\partial_{\beta}\bar{\zeta}_{ST}-1)\,d\beta\Big|\Big|_{H^s(q\mathbb{T})}\\
        \leq & C\norm{D_t\zeta-D_t^{ST}\zeta_{ST}}_{H^s(q\mathbb{T})}\Big(\norm{D_t\zeta}_{W^{s-1,\infty}}+\norm{D_t^{ST}}_{W^{s-1,\infty}}\Big)\Big(\norm{\zeta_{\alpha}-1}_{W^{s-1,\infty}}+\norm{\partial_{\alpha}\zeta_{ST}-1}_{W^{s-1,\infty}}\Big)\\
        & +\Big(\norm{D_t\zeta}_{W^{s-1,\infty}}^2+\norm{D_t^{ST}\zeta_{ST}}_{W^{s-1,\infty}}\Big)^2\norm{\partial_{\alpha}(\zeta-\zeta_{ST})}_{H^s(q\mathbb{T})}\\
        \leq & C\epsilon^{5/2}\delta e^{\epsilon^2 t}.
    \end{split}
\end{equation}
Similarly,
\begin{equation}
\begin{split}
    &\Big|\Big|\frac{1}{4q^2\pi i}\int_{-q\pi}^{q\pi} \Big(\frac{\tilde{D}_t\tilde{\zeta}(\alpha)-\tilde{D}_t\tilde{\zeta}(\beta)}{\tilde{\zeta}(\alpha)-\tilde{\zeta}(\beta)}\Big)^2 (\bar{\tilde{\zeta}}_{\beta}(\beta)-1)\,d\beta\\
        &-\frac{1}{4q^2\pi i}\int_{-q\pi}^{q\pi} \Big(\frac{\tilde{D}_t^{ST}\tilde{\zeta}_{ST}(\alpha)-\tilde{D}_t^{ST}\tilde{\zeta}_{ST}(\beta)}{\tilde{\zeta}_{ST}(\alpha)-\tilde{\zeta}_{ST}(\beta)}\Big)^2 (\partial_{\beta}\bar{\tilde{\zeta}}_{ST}-1)\,d\beta\Big|\Big|_{H^s(q\mathbb{T})}\leq  C\epsilon^{5/2}\delta e^{\epsilon^2 t}.
\end{split}
\end{equation}
Applying Proposition \ref{commutator:difference:cubic}, Lemma \ref{lemmazeta}, \eqref{estimates:bminusbST}, we have 
\begin{equation}
    \begin{split}
       & \Big|\Big|[D_t\zeta,\mathcal{H}_{\zeta}]\frac{\partial_{\alpha}D_t\bar{\zeta}}{\zeta_{\alpha}}-[D_t^{ST}\zeta^{ST},\mathcal{H}_{\zeta_{ST}}]\frac{\partial_{\alpha}D_t^{ST}\bar{\zeta}_{ST}}{\partial_{\alpha}\zeta_{ST}}\\
        &-\Big([\tilde{D}_t\tilde{\zeta},\mathcal{H}_{\tilde{\zeta}}]\frac{\partial_{\alpha}\tilde{D}_t\bar{\tilde{\zeta}}}{\tilde{\zeta}_{\alpha}}-[\tilde{D}_t^{ST}\tilde{\zeta}^{ST},\mathcal{H}_{\tilde{\zeta}_{ST}}]\frac{\partial_{\alpha}\tilde{D}_t^{ST}\bar{\tilde{\zeta}}_{ST}}{\partial_{\alpha}\tilde{\zeta}_{ST}}\Big)\Big|\Big|_{H^s(q\mathbb{T})}\\
        \leq & C\epsilon\norm{D_t\zeta-D_t^{ST}\zeta_{ST}-(\tilde{D}_t\tilde{\zeta}-\tilde{D}_t^{ST}\tilde{\zeta}_{ST})}_{H^S(q\mathbb{T})}+C\epsilon^2\norm{\partial_{\alpha}r}_{H^s(q\mathbb{T})}\\
        \leq & C\epsilon^{5/2}\delta e^{\epsilon^2 t}.
    \end{split}
\end{equation}
By Proposition \ref{commutator:difference:cubic} and using Corollary \ref{cor:Dtsquare}, we obtain
\begin{equation}
\begin{split}
   & \Big|\Big|[D_t^2\zeta,\mathcal{H}_{\zeta}]\frac{\bar{\zeta}_{\alpha}-1}{\zeta_{\alpha}}-[(D_t^{ST})^2\zeta_{ST},\mathcal{H}_{\zeta_{ST}}]\frac{\partial_{\alpha}\bar{\zeta}_{ST}-1}{\partial_{\alpha}\zeta_{ST}}\\
   &-\Big\{ [\tilde{D}_t^2\tilde{\zeta},\mathcal{H}_{\tilde{\zeta}}]\frac{\bar{\tilde{\zeta}}_{\alpha}-1}{\tilde{\zeta}_{\alpha}}-[(\tilde{D}_t^{ST})^2\tilde{\zeta}_{ST},\mathcal{H}_{\tilde{\zeta}_{ST}}]\frac{\partial_{\alpha}\bar{\tilde{\zeta}}_{ST}-1}{\partial_{\alpha}\tilde{\zeta}_{ST}} \Big\} \Big|\Big|\\
   \leq & C\epsilon^{5/2}\delta e^{\epsilon^2 t}+C\epsilon \norm{D_tb-D_t^{ST}b_{ST}-(\tilde{D}_t\tilde{b}-\tilde{D}_t^{ST}\tilde{b}_{ST})}_{H^s(q\mathbb{T})}.
\end{split}
\end{equation}
So we obtain
\begin{equation*}
\begin{split}
    &\norm{(I-\mathcal{H}_{\zeta})(D_tb-D_t^{ST}b_{ST})-\Big\{(I-\mathcal{H}_{\tilde{\zeta}})(\tilde{D}_t\tilde{b}-\tilde{D}_t^{ST}\tilde{b}_{ST})\Big\}}_{H^s(q\mathbb{T})}\\
    \leq & C\epsilon^{5/2}\delta e^{\epsilon^2t}+C\epsilon \norm{D_tb-D_t^{ST}b_{ST}-(\tilde{D}_t\tilde{b}-\tilde{D}_t^{ST}\tilde{b}_{ST})}_{H^s(q\mathbb{T})},
    \end{split}
\end{equation*}
which implies 
\begin{align*}
    \norm{(I-\mathcal{H}_{\zeta})(D_tb-D_t^{ST}b_{ST}-(\tilde{D}_t\tilde{b}-\tilde{D}_t^{ST}\tilde{b}_{ST})}_{H^s(q\mathbb{T})}  \leq & C\epsilon \norm{D_tb-D_t^{ST}b_{ST}-(\tilde{D}_t\tilde{b}-\tilde{D}_t^{ST}\tilde{b}_{ST})}_{H^s(q\mathbb{T})}\\
    &+C\epsilon^{5/2}\delta e^{\epsilon^2t}.
\end{align*}
Here, we used the fact 
\begin{align*}
    \norm{(\mathcal{H}_{\zeta}-\mathcal{H}_{\tilde{\zeta}})(\tilde{D}_t\tilde{b}-\tilde{D}_t^{ST}\tilde{b}_{ST})}_{H^s(q\mathbb{T})}  \leq & C\epsilon^{5/2}\delta e^{\epsilon^2t}.
\end{align*}
Therefore, we can use Lemma \ref{realinverse1} to conclude
\begin{equation}\label{estimates:difference:Dtbquintic1}
    \norm{D_tb-D_t^{ST}b_{ST}-(\tilde{D}_t\tilde{b}-\tilde{D}_t^{ST}\tilde{b}_{ST})}_{H^s(q\mathbb{T})}  \leq  C\epsilon^{5/2}\delta e^{\epsilon^2t}.
\end{equation}
Since $\tilde{D}_t^{ST}\tilde{b}_{ST}=0$, and \begin{equation}
    \tilde{D}_t \tilde{b}=-\epsilon^2\omega k(\partial_t+b^{(2)}\partial_{\alpha})|B(X,T)|^2,
\end{equation}
(\ref{estimates:difference:Dtbquintic1}) implies 
\begin{equation}
    \norm{D_tb-D_t^{ST}b_{ST}}_{H^s(q\mathbb{T})}\leq C\epsilon^{5/2}\delta e^{\epsilon^2 t}.
\end{equation}
Since $D_t(b-b_{ST})=D_tb-D_t^{ST}b_{ST}+(b_{ST}-b)\partial_{\alpha}b_{ST}$, we also conclude
\begin{equation}\label{estimate:DtbminusbST}
    \norm{D_t(b-b_{ST})}_{H^s(q\mathbb{T})}\leq C\epsilon^{5/2}\delta e^{\epsilon^2 t}.
\end{equation}
Using the same argument, we can also conclude
\begin{equation}\label{estimates:difference:Dtbquintic2}
   \norm{ (D_t(b-b_{ST}) -\tilde{D}_t(\tilde{b}-\tilde{b}_{ST})}_{H^s(q\mathbb{T})}\leq C\epsilon^{5/2}\delta e^{\epsilon^2 t}.
\end{equation}

	\subsection{Bounds for $A,\,A_{ST},\,\tilde{A},\,\tilde{A}_{ST}$}
First of all, by construction, one has $\tilde{A}=\tilde{A}_{ST}=1$.
Using the same arguments as for $b$, $b_{ST}$, respectively, we obtain
\begin{equation}
\left\Vert A-1\right\Vert _{H^{s}\left(q\mathbb{T}\right)}\leq C\epsilon^{3}q^{\frac{1}{2}}\label{eq:Aesti}
\end{equation}
which also implies 
\begin{equation}
\left\Vert A-1\right\Vert _{W^{s-1,\infty}\left(q\mathbb{T}\right)}\leq C\epsilon^{3}\label{eq:AestiW}
\end{equation}
by the Sobolev embedding. As a consequence, for $t\in [0,T_0]$,
\begin{equation}\label{estimate:Alower}
    \inf_{\alpha\in q\mathbb{T})}A\geq \frac{1}{2}.
\end{equation}

\begin{equation}
\left\Vert A_{ST}-1\right\Vert _{H^{s}\left(q\mathbb{T}\right)}\leq C\epsilon^{3}q^{\frac{1}{2}}\label{eq:ASTesti}
\end{equation}
and
\begin{equation}
\left\Vert A_{ST}-1\right\Vert _{W^{s-1,\infty}\left(q\mathbb{T}\right)}\leq C\epsilon^{3}.\label{eq:ASTestiW}
\end{equation}
Using the same argument as for 
$D_tb-D_t^{ST}b_{ST}-(\tilde{D}_t \tilde{b}-\tilde{D}_t^{ST}\tilde{b}_{ST})$, we can conclude
that
\begin{equation}
\left\Vert A-A_{ST}\right\Vert _{H^{s}\left(q\mathbb{T}\right)}\leq C\epsilon^{\frac{5}{2}}\delta e^{\epsilon^{2}t}\label{eq:DiffASob}
\end{equation}
which also implies 
\begin{equation}
\left\Vert A-A_{ST}\right\Vert _{W^{s-1,\infty}\left(q\mathbb{T}\right)}\leq C\epsilon^{\frac{5}{2}}\delta e^{\epsilon^{2}t}q^{-\frac{1}{2}}\label{eq:DiffAW}
\end{equation}
by the Sobolev embedding.

\subsection{Estimate sums of the form $\sum_{j=1}^4 (-1)^{j-1}G_{\zeta_j}(f_j, g_j, h_j)$} Define
\begin{equation}\label{def:singularcubic}
    G_{\zeta}(g, h, f):=\frac{1}{4\pi q^2 i}\int_{-q\pi}^{q\pi}\frac{(g(\alpha)-g(\beta))(h(\alpha)-h(\beta))}{\Big(\sin(\frac{1}{2q}(\zeta(\alpha)-\zeta(\beta)))\Big)^2} f_{\beta}\,d\beta
\end{equation}
The expression defined above can be regarded as an alternating sums of quadrilinear form in terms of $(\zeta,g,h,f)$.  In the remaining part, we estimates the differences of $G_\zeta(g,h,f)$ associated with different quadruples.

In the view of our bootstrap assumptions, we assume that 

\vspace*{1ex}

\begin{itemize}
\item [(H1)] For all $t\in [0,T_0]$, $j\in \{1,2,3,4\}$,
    \begin{equation}
        \norm{f_j}_{W^{s-1,\infty}(q\mathbb{T})}+\norm{g_j}_{W^{s-1,\infty}(q\mathbb{T})}+\norm{h_j}_{W^{s-1,\infty}(q\mathbb{T})}+\norm{\partial_{\alpha}\zeta_j-1}_{W^{s-1,\infty}(q\mathbb{T})}\leq C\epsilon.
    \end{equation}

\vspace*{1ex}

\item [(H2)] For $j=1, 2$,
\begin{equation}
\begin{split}
    \norm{f_j-f_{j+2}}_{W^{s-1,\infty}}&+\norm{g_j-g_{j+2}}_{W^{s-1,\infty}}+\norm{h_j-h_{j+2}}_{W^{s-1,\infty}}\\
    &+\norm{\partial_{\alpha}(\zeta_j-\zeta_{j+2})}_{W^{s-1,\infty}(q\mathbb{T})}\leq C\epsilon^2.
    \end{split}
\end{equation}

   \item [(H3)] 
   \begin{equation}
       \begin{split}
          \norm{\sum_{j=1}^4 (-1)^{j-1}f_j}_{H^s(q\mathbb{T})}& +\norm{\sum_{j=1}^4 (-1)^{j-1}g_j}_{H^s(q\mathbb{T})}+\norm{\sum_{j=1}^4 (-1)^{j-1}h_j}_{H^s(q\mathbb{T})}\\
          &+\norm{\sum_{j=1}^4 (-1)^{j-1}\partial_{\alpha}\zeta_j}_{H^s(q\mathbb{T})}\leq C\epsilon^{3/2}\delta e^{\epsilon^2 t}.
       \end{split}
   \end{equation}
\end{itemize}
Our goal is to estimate $\norm{\sum_{j=1}^4 (-1)^{j-1}G_{\zeta_j}(g_j, h_j, f_j)}_{H^s(q\mathbb{T})}$.
\begin{lemma}\label{lemma:quarticdifference}
Under the assumptions (H1)-(H2)-(H3),  we have
	\begin{equation}
	    \norm{\sum_{j=1}^4 (-1)^{j-1}G_{\zeta_j}(g_j, h_j, f_j)}_{H^s(q\mathbb{T})}\leq C\epsilon^{7/2}\delta e^{\epsilon^{2}t}.
	\end{equation}
\end{lemma} 
The proof is provided in Appendix \S \ref{proof:lemmadifferencequartic}. We will apply Lemma \ref{lemma:quarticdifference} with $\zeta_1:=\zeta$, $\zeta_2:=\zeta_{ST}$, $\zeta_3:=\tilde{\zeta}$, $\zeta_4:=\tilde{\zeta}_{ST}$. $f_j$, $g_j$ and $h_j$ are quantities associated with $\zeta_j$. For example, we take $f_1=D_t\zeta$, $f_2=D_t^{ST}\zeta_{ST}$, $f_3=\tilde{D}_t\tilde{\zeta}$, $f_4=\tilde{D}_t^{ST}\tilde{\zeta}_{ST}$. According to the bootstrap assumption \eqref{boot} and the assumption \eqref{aprioritwo}, $\{\zeta_j\}, \{f_j\}, \{g_j\}, \{h_j\}$ satisfy (H1)-(H2)-(H3).

	\subsection{Estimate $N_1:=G-G_{ST}-(\tilde{G}-\tilde{G}_{ST})$}\label{estimates:N1:G1} We can write $G=G_1+G_2$, where
	\begin{equation}
	    G_1=-2[D_t\zeta, \mathcal{H}_{\zeta}\frac{1}{\zeta_{\alpha}}+\bar{\mathcal{H}}_{\zeta}\frac{1}{\bar{\zeta}_{\alpha}}]\partial_{\alpha}D_t\zeta,
	\end{equation}
	and
	\begin{equation}
	    G_2=\frac{1}{ 4\pi q^2 i}\int_{-q\pi}^{q\pi} \Big(\frac{D_t\zeta(\alpha)-D_t\zeta(\beta)}{\sin(\frac{1}{2q}(\zeta(\alpha)-\zeta(\beta)))}\Big)^2\partial_{\beta}(\zeta-\bar{\zeta})\,d\beta.
	\end{equation}
	The terms $G_{j, ST}$, $\tilde{G}_j$, and $\tilde{G}_{j,ST}$ are defined similarly.
	Applying Lemma \ref{lemma:quarticdifference} with 
	\begin{equation}
	    \zeta_1:=\zeta\quad \quad \zeta_2:=\zeta_{ST}, \quad \quad \zeta_3:=\tilde{\zeta}, \quad \quad \zeta_4:=\tilde{\zeta}_{ST}
	    \end{equation}
	    \begin{equation}
	        g_1=h_1:=D_t\zeta, \quad g_2=h_2:=D_t^{ST}\zeta_{ST}, \quad g_3=h_3:=\tilde{D}_t\tilde{\zeta}, \quad g_4=h_4:=\tilde{D}_t^{ST}\tilde{\zeta}_{ST},
	    \end{equation}
	    \begin{equation}
	        f_1:=\zeta-\bar{\zeta}, \quad f_2:=\zeta_{ST}-\bar{\zeta}_{ST}, \quad f_3:=\tilde{\zeta}-\bar{\tilde{\zeta}}, \quad f_4:=\tilde{\zeta}_{ST}-\bar{\tilde{\zeta}}_{ST},
	    \end{equation}
	   \begin{equation}
	       G_{\zeta_1}:=G_2, \quad\quad G_{\zeta_2}:=G_{2,ST}, \quad \quad  G_3:=\tilde{G}_2, \quad \quad  G_4:=\tilde{G}_{2,ST}
	   \end{equation} 
one has
	\begin{align*}
	   & \norm{G_2-G_{2,ST}-(\tilde{G}_2-\tilde{G}_{2,ST})}_{H^s(q\mathbb{T})}\\
\leq & C\epsilon^2 E_s^{1/2}+C\epsilon^{7/2}\delta e^{\epsilon^2 t}.
	\end{align*}
We rewrite $G_1$ as 
\begin{equation}\label{formula:G1}
\begin{split}
    G_1=& -\int_{-q\pi}^{q\pi} \Big(\frac{1}{q\pi i}\frac{\cos(\frac{\zeta(\alpha,t)-\zeta(\beta,t)}{2q})}{\sin(\frac{\zeta(\alpha,t)-\zeta(\beta,t)}{2q})}+\overline{\frac{1}{q\pi i}\frac{\cos(\frac{\zeta(\alpha,t)-\zeta(\beta,t)}{2q})}{\sin(\frac{\zeta(\alpha,t)-\zeta(\beta,t)}{2q})}}\Big)(D_t\zeta(\alpha,t)-D_t\zeta(\beta,t))\partial_{\beta}D_t\zeta(\beta,t)\,d\beta\\
=&
-\frac{2}{q\pi}\int_{-q\pi}^{q\pi}\Im \Big\{\frac{\cos(\frac{\zeta(\alpha,t)-\zeta(\beta,t)}{2q})}{\sin(\frac{\zeta(\alpha,t)-\zeta(\beta,t)}{2q})}\Big\}(D_t\zeta(\alpha,t)-D_t\zeta(\beta,t)) \partial_{\beta}D_t\zeta(\beta,t)\,d\beta\\
=& -\frac{2}{q\pi}\int_{-q\pi}^{q\pi}\frac{\Im \Big\{\cos(\frac{\zeta(\alpha,t)-\zeta(\beta,t)}{2q})\overline{\sin(\frac{\zeta(\alpha,t)-\zeta(\beta,t)}{2q})}\Big\}}{\Big|\sin(\frac{\zeta(\alpha,t)-\zeta(\beta,t)}{2q})\Big|^2}(D_t\zeta(\alpha,t)-D_t\zeta(\beta,t))\partial_{\beta}D_t\zeta(\beta,t)\,d\beta
\end{split}
\end{equation}
Note that
\begin{equation}
    \overline{\sin(\frac{\zeta(\alpha,t)-\zeta(\beta,t)}{2q})}=\sin(\frac{\bar{\zeta}(\alpha,t)-\bar{\zeta}(\beta,t)}{2q}).
\end{equation}
Denoting $L_{\zeta}(\alpha,\beta):=\frac{\zeta(\alpha,t)-\zeta(\beta,t)}{2q}$, we have 
\begin{equation}
\begin{split}
   & \cos(\frac{\zeta(\alpha,t)-\zeta(\beta,t)}{2q})\sin(\frac{\bar{\zeta}(\alpha,t)-\bar{\zeta}(\beta,t)}{2q})\\
   =& \frac{1}{2}\Big( \sin (L_{\zeta}(\alpha,\beta)+\bar{L}_{\zeta}(\alpha,\beta))-\sin (L_{\zeta}(\alpha,\beta)-\bar{L}_{\zeta}(\alpha,\beta))  \Big)
    \end{split}
\end{equation}
With the computations above, $G_1$ can be written as 
\begin{equation}
\begin{split}
G_1=& -\frac{1}{q\pi}\int_{-q\pi}^{q\pi}\frac{\Im \Big\{\sin (L_{\zeta}(\alpha,\beta)+\bar{L}_{\zeta}(\alpha,\beta))-\sin (L_{\zeta}(\alpha,\beta)-\bar{L}_{\zeta}(\alpha,\beta))\Big\}}{\Big|\sin(\frac{\zeta(\alpha,t)-\zeta(\beta,t)}{2q})\Big|^2}\\
&\times (D_t\zeta(\alpha,t)-D_t\zeta(\beta,t))\partial_{\beta}D_t\zeta(\beta,t)\,d\beta
\end{split}
\end{equation}
Noticing that 
\begin{equation}
    \Im \Big\{\sin (L_{\zeta}(\alpha,\beta)+\bar{L}_{\zeta}(\alpha,\beta))-\sin (L_{\zeta}(\alpha,\beta)-\bar{L}_{\zeta}(\alpha,\beta))\Big\}\sim \Im\{\zeta(\alpha,t)-\zeta(\beta,t)\},
\end{equation}
from which we see that $G_1$ is a cubic term. Moreover, using the same argument as for $G_2$, we obtain
\begin{align*}
	   & \norm{G_1-G_{1,ST}-(\tilde{G}_1-\tilde{G}_{1,ST})}_{H^s(q\mathbb{T})}\\
\leq & C\epsilon^2 E_s^{1/2}+C\epsilon^{7/2}\delta e^{\epsilon^2 t}.
	\end{align*}
So we conclude that
\begin{equation}
    \norm{G-G_{ST}-(\tilde{G}-\tilde{G}_{ST})}_{H^s(q\mathbb{T})} \leq C\epsilon^2 E_s^{1/2}+C\epsilon^{7/2}\delta e^{\epsilon^2 t}.
\end{equation}

\subsection{Estimate $N_4:=(\mathcal{Q}-\mathcal{Q}_{ST})(\tilde{\theta}_{ST}-\theta_{ST})$} Recall that $\mathcal{Q}=\mathcal{P}(I-\mathcal{H}_{\zeta})=(D_t^2-iA\partial_{\alpha})(I-\mathcal{H}_{\zeta})$, and $\mathcal{Q}_{ST}=\mathcal{P}_{ST}(I-\mathcal{H}_{\zeta_{ST}})=((D_t^{ST})^2-iA_{ST}\partial_{\alpha})(I-\mathcal{H}_{\zeta_{ST}})$. Taking the difference,  we have 
\begin{align*}
    \mathcal{Q}-\mathcal{Q}_{ST}= & (\mathcal{P}-\mathcal{P}_{ST})(I-\mathcal{H}_{\zeta})-\mathcal{P}_{ST}(\mathcal{H}_{\zeta}-\mathcal{H}_{\zeta_{ST}}).
\end{align*}
By a direct computation we can write
\begin{align*}
    \mathcal{P}-\mathcal{P}_{ST}=& ((\partial_t+b\partial_{\alpha})^2-iA\partial_{\alpha})-((\partial_t+b_{ST}\partial_{\alpha})^2-iA_{ST}\partial_{\alpha})\\
    =& D_t (D_t-D_t^{ST})+(D_t-D_t^{ST})D_t-i(A-A_{ST})\partial_{\alpha}\\
    =& D_t (b-b_{ST})\partial_{\alpha}+(b-b_{ST})\partial_{\alpha}D_t^{ST}-i(A-A_{ST})\partial_{\alpha}.
\end{align*}
From  computations from previous sections, we have the following:
\begin{equation}
    \norm{D_t(b-b_{ST})}_{H^{s}(q\mathbb{T})}\leq C\epsilon^{5/2}\delta e^{\epsilon^2 t},
\end{equation}
\begin{equation}
    \norm{b-b_{ST}}_{H^{s}(q\mathbb{T})}\leq C\epsilon^{3/2}\delta e^{\epsilon^2 t},
\end{equation}
\begin{equation}
    \norm{D_t\partial_{\alpha}(\tilde{\theta}_{ST}-\theta_{ST})}_{W^{s-1,\infty}}\leq C\epsilon^4,
\end{equation}
\begin{equation}
    \norm{D_t\partial_{\alpha}(\tilde{\theta}_{ST}-\theta_{ST})}_{H^{s}(q\mathbb{T})}\leq C\epsilon^4q^{1/2},
\end{equation}
\begin{equation}
    \norm{\partial_{\alpha}D_t^{ST}(\tilde{\theta}_{ST}-\theta_{ST})}_{W^{s,\infty}}\leq C\epsilon^4.
\end{equation}
Therefore, one has
\begin{equation}
    \norm{-i(A-A_{ST})\partial_{\alpha}(\tilde{\theta}_{ST}-\theta_{ST})}_{H^s(q\mathbb{T})}\leq C\epsilon^{7/2}\delta e^{\epsilon^2 t}.
\end{equation}
\begin{equation}
    \norm{(\mathcal{P}-\mathcal{P}_{ST})(I-\mathcal{H}_{\zeta})(\tilde{\theta}_{ST}-\theta_{ST})}_{H^s(q\mathbb{T})}\leq C\epsilon^{7/2}\delta e^{\epsilon^2 t},
\end{equation}
and
\begin{equation}
    \norm{\mathcal{P}_{ST}(\mathcal{H}_{\zeta}-\mathcal{H}_{\zeta_{ST}})(\tilde{\theta}_{ST}-\theta_{ST})}_{H^s(q\mathbb{T})}\leq C\epsilon^{7/2}\delta e^{\epsilon^2 t}.
\end{equation}
Hence, we conclude
\begin{equation}
    \norm{N_4}_{H^s(q\mathbb{T})}\leq C\epsilon^{7/2}\delta e^{\epsilon^2 t}.
\end{equation}

\subsection{Estimate $N_3:=(\mathcal{Q}-\tilde{\mathcal{Q}})(\tilde{\theta}_{ST}-\tilde{\theta})$} We can write
\begin{align*}
    \mathcal{Q}-\tilde{\mathcal{Q}}= & (\mathcal{P}-\tilde{\mathcal{P}})(I-\mathcal{H}_{\zeta})-\tilde{\mathcal{P}}(\mathcal{H}_{\zeta}-\mathcal{H}_{\tilde{\zeta}}).
\end{align*}
Taking the difference, we have
\begin{align*}
    \mathcal{P}-\tilde{\mathcal{P}}
    =& \Big(D_t(b-\tilde{b})\Big) \partial_{\alpha}+(b-\tilde{b})D_t\partial_{\alpha}+(b-\tilde{b})\partial_{\alpha}\tilde{D}_t-i(A-\tilde{A})\partial_{\alpha}.
\end{align*}
Using the similar argument to that for $N_4$, we obtain
\begin{equation}
    \norm{N_3}_{H^s(q\mathbb{T})}\leq C\epsilon^{7/2}\delta e^{\epsilon^2 t}.
\end{equation}

\subsection{Estimate $N_2:=\Big(\mathcal{Q}_{ST}-\tilde{\mathcal{Q}}_{ST}-(\mathcal{Q}-\tilde{\mathcal{Q}})\Big)\tilde{\theta}_{ST}$}
Taking the differences, one has\begin{align*}
   & \mathcal{Q}_{ST}-\tilde{\mathcal{Q}}_{ST}-(\mathcal{Q}-\tilde{\mathcal{Q}})\\
   =& \Big\{ (\mathcal{P}_{ST}-\tilde{\mathcal{P}}_{ST})(I-\mathcal{H}_{\zeta_{ST}})-\tilde{\mathcal{P}}_{ST}(\mathcal{H}_{\zeta_{ST}}-\mathcal{H}_{\tilde{\zeta}_{ST}}) \Big\}-  \Big\{(\mathcal{P}-\tilde{\mathcal{P}})(I-\mathcal{H}_{\zeta})-\tilde{\mathcal{P}}(\mathcal{H}_{\zeta}-\mathcal{H}_{\tilde{\zeta}})  \Big\}\\
   =& \Big\{ (\mathcal{P}_{ST}-\tilde{\mathcal{P}}_{ST})(I-\mathcal{H}_{\zeta_{ST}})-(\mathcal{P}-\tilde{\mathcal{P}})(I-\mathcal{H}_{\zeta}) \Big\}+\Big\{\tilde{\mathcal{P}}(\mathcal{H}_{\zeta}-\mathcal{H}_{\tilde{\zeta}})-\tilde{\mathcal{P}}_{ST}(\mathcal{H}_{\zeta_{ST}}-\mathcal{H}_{\tilde{\zeta}_{ST}})  \Big\}\\
   =& -\Big\{ (\mathcal{P}-\mathcal{P}_{ST})-(\tilde{\mathcal{P}}-\tilde{\mathcal{P}}_{ST})\Big\}(I-\mathcal{H_{\zeta}})+\Big\{(\mathcal{P}_{ST}-\tilde{\mathcal{P}}_{ST})(\mathcal{H}_{\zeta}-\mathcal{H}_{\zeta_{ST}})\Big\}\\
   &+(\tilde{\mathcal{P}}-\tilde{\mathcal{P}}_{ST})(\mathcal{H}_{\zeta}-\mathcal{H}_{\tilde{\zeta}})+\tilde{\mathcal{P}}_{ST}\Big(\mathcal{H}_{\zeta}-\mathcal{H}_{\tilde{\zeta}}-(\mathcal{H}_{\zeta_{ST}}-\mathcal{H}_{\mathcal{\tilde{\zeta}_{ST}}})\Big)
\end{align*}
and
\begin{align*}
    &\Big((\mathcal{P}-\mathcal{P}_{ST})-(\tilde{\mathcal{P}}-\tilde{\mathcal{P}}_{ST})\Big)(I-\mathcal{H}_{\zeta})\tilde{\theta}_{ST}\\
    =& \Big\{ \Big(D_t(b-b_{ST})\Big) \partial_{\alpha}+(b-b_{ST})D_t\partial_{\alpha}+(b-b_{ST})\partial_{\alpha}D_t^{ST}-i(A-A_{ST})\partial_{\alpha}\Big\}(I-\mathcal{H}_{\zeta})\tilde{\theta}_{ST}\\
    & -\Big\{ \Big(\tilde{D}_t(\tilde{b}-\tilde{b}_{ST})\Big) \partial_{\alpha}+(\tilde{b}-\tilde{b}_{ST})\tilde{D}_t\partial_{\alpha}+(\tilde{b}-\tilde{b}_{ST})\partial_{\alpha}\tilde{D}_t^{ST}-i(\tilde{A}-\tilde{A}_{ST})\partial_{\alpha}\Big\}(I-\mathcal{H}_{\zeta})\tilde{\theta}_{ST}\\
    =& \Big\{\Big(D_t(b-b_{ST})\Big) \partial_{\alpha}-\Big(\tilde{D}_t(\tilde{b}-\tilde{b}_{ST})\Big) \partial_{\alpha}\Big\}(I-\mathcal{H}_{\zeta})\tilde{\theta}_{ST}\\
    & +\Big\{(b-b_{ST})D_t\partial_{\alpha}-(\tilde{b}-\tilde{b}_{ST})\tilde{D}_t\partial_{\alpha}\Big\}(I-\mathcal{H}_{\zeta})\tilde{\theta}_{ST}\\
    &+\Big\{(b-b_{ST})\partial_{\alpha}D_t^{ST}-(\tilde{b}-\tilde{b}_{ST})\partial_{\alpha}\tilde{D}_t^{ST}\Big\}(I-\mathcal{H}_{\zeta})\tilde{\theta}_{ST}\\
    &-i\Big(A-A_{ST}-(\tilde{A}-\tilde{A}_{ST})\Big)\partial_{\alpha}(I-\mathcal{H}_{\zeta})\tilde{\theta}_{ST}\\
    :=& I_1+I_2+I_3+I_4.
\end{align*}
We estimate each term on the right-hand side  of the expression above.

By (\ref{estimates:difference:Dtbquintic2}), we have 
\begin{align*}
    \norm{I_1}_{H^s(q\mathbb{T})}\leq \norm{(D_t(b-b_{ST}) -\tilde{D}_t(\tilde{b}-\tilde{b}_{ST})}_{H^s(q\mathbb{T})}\norm{\partial_{\alpha}\tilde{\theta}_{ST}}_{W^{s,\infty}}\leq C\epsilon^{7/2}\delta e^{\epsilon^2 t}.
\end{align*}
Note that $I_2$ can be written as
\begin{align*}
    I_2=& (I-\mathcal{H}_{\zeta})\Big\{(b-b_{ST})D_t\partial_{\alpha}-(\tilde{b}-\tilde{b}_{ST})\tilde{D}_t\partial_{\alpha}\Big\}\tilde{\theta}_{ST}-\Big[(b-b_{ST})D_t\partial_{\alpha}-(\tilde{b}-\tilde{b}_{ST})\tilde{D}_t\partial_{\alpha}, \mathcal{H}_{\zeta}\Big]\tilde{\zeta}_{ST}.
\end{align*}
Using (\ref{estimats:b:quintic}) and (\ref{estimates:tildeb:quadratic}), one has
\begin{align*}
    \norm{I_2}_{H^s(q\mathbb{T})}\leq & \norm{\Big((b-b_{ST})-(\tilde{b}-\tilde{b}_{ST})\Big)D_t\partial_{\alpha}\tilde{\theta}_{ST}}_{H^s(q\mathbb{T})}+\norm{(\tilde{b}-\tilde{b}_{ST})(b-\tilde{b})\partial_{\alpha}^2\tilde{\theta}_{ST}}_{H^s(q\mathbb{T})}\\
    \leq & C\epsilon^{7/2}\delta e^{\epsilon^2 t}.
\end{align*}
We estimate for $I_3$ by the same way as for $I_2$, and obtain
\begin{equation}
    \norm{I_3}_{H^s(q\mathbb{T})}\leq C\epsilon^{7/2}\delta e^{\epsilon^2 t}.
\end{equation}
Using (\ref{eq:DiffASob}) and $\tilde{A}=\tilde{A}_{ST}=1$, we have 
\begin{equation}
    \norm{I_4}_{H^s(q\mathbb{T})}\leq C\epsilon^{7/2}\delta e^{\epsilon^2 t}.
\end{equation}
Therefore, we can conclude
\begin{equation}
    \norm{\Big((\mathcal{P}-\mathcal{P}_{ST})-(\tilde{\mathcal{P}}-\tilde{\mathcal{P}}_{ST})\Big)(I-\mathcal{H}_{\zeta})\tilde{\theta}_{ST}}_{H^s(q\mathbb{T})}\leq C\epsilon^{7/2}\delta e^{\epsilon^2 t}.
\end{equation}
Similarly, we also have
\begin{equation}
    \norm{(\mathcal{P}_{ST}-\tilde{\mathcal{P}}_{ST})(\mathcal{H}_{\zeta}-\mathcal{H}_{\zeta_{ST}})\tilde{\theta}_{ST}}_{H^s(q\mathbb{T})} \leq C\epsilon^{7/2}\delta e^{\epsilon^2 t},
\end{equation}
\begin{equation}
    \norm{(\tilde{\mathcal{P}}-\tilde{\mathcal{P}}_{ST})(\mathcal{H}_{\zeta}-\mathcal{H}_{\tilde{\zeta}})\tilde{\theta}_{ST}}_{H^s(q\mathbb{T})}\leq C\epsilon^{7/2}\delta e^{\epsilon^2 t}.
\end{equation}
The quantity $\tilde{\mathcal{P}}_{ST}\Big(\mathcal{H}_{\zeta}-\mathcal{H}_{\tilde{\zeta}}-(\mathcal{H}_{\zeta_{ST}}-\mathcal{H}_{\mathcal{\tilde{\zeta}_{ST}}})\Big)\tilde{\theta}_{ST}$ can be written as 
\begin{align*}
    &\tilde{\mathcal{P}}_{ST}\Big(\mathcal{H}_{\zeta}-\mathcal{H}_{\tilde{\zeta}}-(\mathcal{H}_{\zeta_{ST}}-\mathcal{H}_{\mathcal{\tilde{\zeta}_{ST}}})\Big)\tilde{\theta}_{ST}\\
    =& \Big(\mathcal{H}_{\zeta}-\mathcal{H}_{\tilde{\zeta}}-(\mathcal{H}_{\zeta_{ST}}-\mathcal{H}_{\mathcal{\tilde{\zeta}_{ST}}})\Big)\tilde{\mathcal{P}}_{ST}\tilde{\theta}_{ST}+[\tilde{\mathcal{P}}_{ST},\mathcal{H}_{\zeta}-\mathcal{H}_{\tilde{\zeta}}-(\mathcal{H}_{\zeta_{ST}}-\mathcal{H}_{\mathcal{\tilde{\zeta}_{ST}}}) ]\tilde{\theta}_{ST}\\
\end{align*}

Using the above identity and Proposition \ref{proposition:Hilbertquartic}, we obtain
\begin{equation}
\begin{split}
    &\norm{\tilde{\mathcal{P}}_{ST}\Big(\mathcal{H}_{\zeta}-\mathcal{H}_{\tilde{\zeta}}-(\mathcal{H}_{\zeta_{ST}}-\mathcal{H}_{\mathcal{\tilde{\zeta}_{ST}}})\Big)\tilde{\theta}_{ST}}_{H^s(q\mathbb{T})}\\
    \leq & C\norm{\partial_{\alpha}(\theta-\tilde{\theta}-(\theta_{ST}-\tilde{\theta}_{ST}))}_{H^s(q\mathbb{T})}\norm{\partial_{\alpha}\tilde{\theta}_{ST}}_{W^{s,\infty}}+C\epsilon^{7/2}\delta e^{\epsilon^2 t}\\
    \leq & C\epsilon^{7/2}\delta e^{\epsilon^2 t}.
    \end{split}
\end{equation}
Hence, finally, we achieve
\begin{equation}
    \norm{N_2}_{H^s(q\mathbb{T})}\leq C\epsilon^{7/2}\delta e^{\epsilon^2 t}.
\end{equation}

\subsection{Estimate $M_1:=D_t N_1$}\label{estimates:M1=DtN1}
Recall that 
\begin{equation}
    M_1=D_t (G-G_{ST}-(\tilde{G}-\tilde{G}_{ST})).
\end{equation}
Recall that $G=G_1+G_2$. We first estimate $\norm{D_t(G_1-G_{1,ST}-(\tilde{G}_1-\tilde{G}_{1,ST}))}_{H^s(q\mathbb{T})}$. The estimate for $\norm{D_t(G_2-G_{2,ST}-(\tilde{G}_2-\tilde{G}_{2,ST}))}_{H^s(q\mathbb{T})}$ follows from the similar yet even easier calculations. 

Using (\ref{formula:G1}), we have 
\begin{align*}
    D_t G_1=&-\frac{2}{q\pi}D_t\int_{-q\pi}^{q\pi}\frac{\Im \Big\{\cos(\frac{\zeta(\alpha,t)-\zeta(\beta,t)}{2q})\overline{\sin(\frac{\zeta(\alpha,t)-\zeta(\beta,t)}{2q})}\Big\}}{\Big|\sin(\frac{\zeta(\alpha,t)-\zeta(\beta,t)}{2q})\Big|^2}(D_t\zeta(\alpha,t)-D_t\zeta(\beta,t))\partial_{\beta}D_t\zeta(\beta,t)\,d\beta.
\end{align*}

Let $\kappa:\mathbb{R}\rightarrow \mathbb{R}$ be the diffeomorphism be defined by $\kappa_t\circ\kappa^{-1}=b$ and let $z(\alpha,t):=\zeta(\kappa(\alpha,t),t)$. Composing\footnote{The advantage of the original coordinate is that $\partial_t$ commutes with $\partial_\beta$ in this coordinate.}
 with the diffeomorphism $\kappa$, one has
\begin{align*}
&\partial_t G_1\circ\kappa\\
=& -\frac{2}{q\pi}\partial_t\int_{-q\pi}^{q\pi}\frac{\Im \Big\{\cos(\frac{z(\alpha,t)-z(\beta,t)}{2q})\overline{\sin(\frac{z(\alpha,t)-z(\beta,t)}{2q})}\Big\}}{\Big|\sin(\frac{z(\alpha,t)-z(\beta,t)}{2q})\Big|^2}(z_t(\alpha,t)-z_t (\beta,t))\partial_{\beta}z_t(\beta,t)\,d\beta\\
=& -\frac{2}{q\pi}\int_{-q\pi}^{q\pi}\frac{\Im \Big\{\cos(\frac{z(\alpha,t)-z(\beta,t)}{2q})\overline{\sin(\frac{z(\alpha,t)-z(\beta,t)}{2q})}\Big\}}{\Big|\sin(\frac{z(\alpha,t)-z(\beta,t)}{2q})\Big|^2}(z_{tt}(\alpha,t)-z_{tt} (\beta,t))\partial_{\beta}z_t(\beta,t)\,d\beta\\
&-\frac{2}{q\pi}\int_{-q\pi}^{q\pi}\frac{\Im \Big\{\cos(\frac{z(\alpha,t)-z(\beta,t)}{2q})\overline{\sin(\frac{z(\alpha,t)-z(\beta,t)}{2q})}\Big\}}{\Big|\sin(\frac{z(\alpha,t)-z(\beta,t)}{2q})\Big|^2}(z_t(\alpha,t)-z_t (\beta,t))\partial_{\beta}z_{tt}(\beta,t)\,d\beta\\
& -\frac{2}{q\pi}\int_{-q\pi}^{q\pi}\partial_t\Big\{\frac{\Im \Big\{\cos(\frac{z(\alpha,t)-z(\beta,t)}{2q})\overline{\sin(\frac{z(\alpha,t)-z(\beta,t)}{2q})}\Big\}}{\Big|\sin(\frac{z(\alpha,t)-z(\beta,t)}{2q})\Big|^2}\Big\}(z_t(\alpha,t)-z_t (\beta,t))\partial_{\beta}z_t(\beta,t)\,d\beta.
\end{align*}
Using
\begin{align*}
   & \frac{\Im \Big\{\cos(\frac{z(\alpha,t)-z(\beta,t)}{2q})\overline{\sin(\frac{z(\alpha,t)-z(\beta,t)}{2q})}\Big\}}{\Big|\sin(\frac{z(\alpha,t)-z(\beta,t)}{2q})\Big|^2}=\Im\{\cot(\frac{z(\alpha,t)-z(\beta,t)}{2q})\Big\},
\end{align*}
we obtain
\begin{equation}
    \partial_t \Big(\frac{\Im \Big\{\cos(\frac{z(\alpha,t)-z(\beta,t)}{2q})\overline{\sin(\frac{z(\alpha,t)-z(\beta,t)}{2q})}\Big\}}{\Big|\sin(\frac{z(\alpha,t)-z(\beta,t)}{2q})\Big|^2}\Big)=-\frac{1}{2q}\Im\Big\{\frac{z_t(\alpha,t)-z_t(\beta,t)}{\sin^2(\frac{z(\alpha,t)-z(\beta,t)}{2q})}\Big\}
\end{equation}
Changing of variables, we get
\begin{align*}
    &D_t G_1(\alpha,t)\\
    =& -\frac{2}{q\pi}\int_{-q\pi}^{q\pi}\frac{\Im \Big\{\cos(\frac{\zeta(\alpha,t)-\zeta(\beta,t)}{2q})\overline{\sin(\frac{\zeta(\alpha,t)-\zeta(\beta,t)}{2q})}\Big\}}{\Big|\sin(\frac{\zeta(\alpha,t)-\zeta(\beta,t)}{2q})\Big|^2}(D_t^2\zeta(\alpha,t)-D_t^2\zeta (\beta,t))\partial_{\beta}D_t\zeta(\beta,t)\,d\beta\\
&-\frac{2}{q\pi}\int_{-q\pi}^{q\pi}\frac{\Im \Big\{\cos(\frac{\zeta(\alpha,t)-\zeta(\beta,t)}{2q})\overline{\sin(\frac{\zeta(\alpha,t)-\zeta(\beta,t)}{2q})}\Big\}}{\Big|\sin(\frac{\zeta(\alpha,t)-\zeta(\beta,t)}{2q})\Big|^2}(D_t\zeta(\alpha,t)-D_t\zeta (\beta,t))\partial_{\beta}D_t^2\zeta(\beta,t)\,d\beta\\
&+\frac{1}{q^2\pi}\int_{-q\pi}^{q\pi}\Im\Big\{\frac{D_t\zeta(\alpha,t)-D_t\zeta(\beta,t)}{\sin^2(\frac{\zeta(\alpha,t)-\zeta(\beta,t)}{2q})}\Big\}(D_t\zeta(\alpha,t)-D_t\zeta (\beta,t))\partial_{\beta}D_t\zeta(\beta,t)\,d\beta\\
:=& H_{1,\zeta}+H_{2,\zeta}+H_{3,\zeta}.
\end{align*}
Since 
\begin{equation}
\begin{split}
    &\Im\Big(\frac{D_t\zeta(\alpha,t)-D_t\zeta(\beta,t)}{\sin^2(\frac{\zeta(\alpha,t)-\zeta(\beta,t)}{2q})}\Big)\\
    =&-\frac{i}{2}\Big\{ \frac{D_t\zeta(\alpha,t)-D_t\zeta(\beta,t)}{\sin^2(\frac{\zeta(\alpha,t)-\zeta(\beta,t)}{2q})}-\overline{\frac{D_t\zeta(\alpha,t)-D_t\zeta(\beta,t)}{\sin^2(\frac{\zeta(\alpha,t)-\zeta(\beta,t)}{2q})}}\Big\},
    \end{split}
\end{equation}
so $H_{3,\zeta}$ is essentially of the same type of $G_{\zeta}$ in (\ref{def:singularcubic}). Using the same argument  for $N_1$, we obtain
\begin{equation}
    \norm{D_tG-D_t^{ST} G_{ST}-(\tilde{D}_t\tilde{G}-\tilde{D}_t^{ST}\tilde{G}_{ST})}_{H^s(q\mathbb{T})}\leq C\epsilon^{7/2}\delta e^{\epsilon^2 t}.
\end{equation}
We can write\begin{equation*}
    M_1=\Big(D_tG-D_t^{ST} G_{ST}-(\tilde{D}_t\tilde{G}-\tilde{D}_t^{ST}\tilde{G}_{ST})\Big)+(D_t-D_t^{ST})G_{ST}+(D_t-\tilde{D}_t)\tilde{G}-(D_t-\tilde{D}_{ST})\tilde{G}_{ST},
\end{equation*}
and
\begin{align*}
    & (D_t-D_t^{ST})G_{ST}+(D_t-\tilde{D}_t)\tilde{G}-(D_t-\tilde{D}_{ST})\tilde{G}_{ST}\\
    =& (b-b_{ST})\partial_{\alpha}G_{ST}+\Big( (D_t-\tilde{D}_t)-(D_t-\tilde{D}_{ST})\Big)\tilde{G}+(D_t-\tilde{D}_{ST})(\tilde{G}-\tilde{G}_{ST})\\
    =& (b-b_{ST})\partial_{\alpha}G_{ST}+ (\tilde{b}_{ST}-\tilde{b})\partial_{\alpha}\tilde{G}+(b-\tilde{b}_{ST})\partial_{\alpha}(\tilde{G}-\tilde{G}_{ST}).
\end{align*}
Then we can bound
\begin{align*}
    \norm{M_1}_{H^s(q\mathbb{T})}\leq & \norm{\Big(D_tG-D_t^{ST} G_{ST}-(\tilde{D}_t\tilde{G}-\tilde{D}_t^{ST}\tilde{G}_{ST})\Big)}_{H^s(q\mathbb{T})}+\norm{(b-b_{ST})\partial_{\alpha}G_{ST}}_{H^s(q\mathbb{T})}\\
    &+\norm{(\tilde{b}_{ST}-\tilde{b})\partial_{\alpha}\tilde{G}}_{H^s(q\mathbb{T})}+\norm{(b-\tilde{b}_{ST})\partial_{\alpha}(\tilde{G}-\tilde{G}_{ST})}_{H^s(q\mathbb{T})}.
\end{align*}
Using
\begin{equation}
    \norm{b-b_{ST}}_{H^s(q\mathbb{T})}+\norm{\tilde{b}-\tilde{b}_{ST}}_{H^s(q\mathbb{T})}\leq C\epsilon^{3/2}\delta e^{\epsilon^2 t},
\end{equation}
\begin{equation}
   \norm{\partial_{\alpha}G_{ST}}_{W^{s,\infty}}+\norm{\partial_{\alpha}\tilde{G}}_{W^{s,\infty}}\leq C\epsilon^3,
\end{equation}
and 
\begin{equation}
    \norm{(b-\tilde{b}_{ST})}_{H^s(q\mathbb{T})}\leq C\epsilon^2 q^{1/2}, \quad \quad \norm{\partial_{\alpha}(\tilde{G}-\tilde{G}_{ST})}_{H^s(q\mathbb{T})}\leq C\epsilon^{3}\delta e^{\epsilon^2 t}q^{-1/2},
\end{equation}
we conclude that
\begin{equation}
    \norm{M_1}_{H^s(q\mathbb{T})}\leq C\epsilon^{7/2}\delta e^{\epsilon^2 t}.
\end{equation}

\subsection{Estimate $M_5$} Recall that $M_5=[\mathcal{P}, D_t](I-\mathcal{H}_{\zeta})\Big[\theta-\theta_{ST}-(\tilde{\theta}-\tilde{\theta}_{ST})\Big]$. Using \eqref{formula:PDT}, one has
\begin{equation}
    M_5=[\mathcal{P}, D_t](I-\mathcal{H}_{\zeta})\Big[\theta-\theta_{ST}-(\tilde{\theta}-\tilde{\theta}_{ST})\Big]=i\frac{a_t}{a}\circ\kappa^{-1}A\rho_{\alpha}.
\end{equation}
We can bound
\begin{align*}
    \norm{M_5}_{H^s(q\mathbb{T})}\leq & \norm{\frac{a_t}{a}\circ\kappa^{-1}}_{H^s}\norm{A\zeta_{\alpha}}_{W^{s-1,\infty}}\norm{\rho_{\alpha}}_{W^{s-1,\infty}(q\mathbb{T})}\\
    &+\norm{\frac{a_t}{a}\circ\kappa^{-1}}_{W^{s-1,\infty}}\norm{A-1}_{H^{s}(q\mathbb{T})}\norm{\rho_{\alpha}}_{H^s(q\mathbb{T})}
\end{align*}
Using Proposition \ref{singularperiodic},  and Lemma \ref{realinverse1}, one has
\begin{equation}\label{estimate:atainverse}
    \norm{\frac{a_t}{a}\circ\kappa^{-1}}_{H^s(q\mathbb{T})}\leq  C\epsilon^{2}q^{1/2}.
\end{equation}
By the Sobolev embedding, we have 
\begin{equation}\label{estimate:atainverinfty}
    \norm{\frac{a_t}{a}\circ \kappa^{-1}}_{W^{s-1,\infty}}\leq C\epsilon^{2}.
\end{equation}
Hence, we can conclude
\begin{equation}
    \norm{M_5}_{H^s(q\mathbb{T})}\leq C\epsilon^{7/2}\delta e^{\epsilon^2 t}+C\epsilon^{2}\norm{\partial_{\alpha}\rho}_{H^s(q\mathbb{T})}
\end{equation}

\subsection{Estimate $M_4$, $M_3$, and $M_2$} Recall that $M_4=D_tN_4$ and $N_4:=(\mathcal{Q}-\mathcal{Q}_{ST})(\tilde{\theta}_{ST}-\theta_{ST})$.
Taking $D_t$, we get
\begin{equation}
\begin{split}
    D_t N_4
    =& (\mathcal{Q}-\mathcal{Q}_{ST})D_t(\tilde{\theta}_{ST}-\theta_{ST})+[D_t, \mathcal{Q}-\mathcal{Q}_{ST}](\tilde{\theta}_{ST}-\theta_{ST}).
\end{split}
\end{equation}
We rewrite
\begin{align*}
   \mathcal{Q}-\mathcal{Q}_{ST}=& \Big\{ (D_t(b-b_{ST})) \partial_{\alpha}+(b-b_{ST})D_t\partial_{\alpha}+(b-b_{ST})\partial_{\alpha}D_t^{ST}-i(A-A_{ST})\partial_{\alpha}\Big\}(I-\mathcal{H}_{\zeta})\\
    &-\mathcal{P}_{ST}(\mathcal{H}_{\zeta}-\mathcal{H}_{\zeta_{ST}}).
\end{align*}
Using (\ref{estimate:DtbminusbST}) to estimate $D_t(b-b_{ST})$, (\ref{estimates:bminusbST}) to estimate $b-b_{ST}$, and (\ref{eq:DiffASob}) to estimate $A-A_{ST}$, we obtain
\begin{equation}
    \norm{(\mathcal{Q}-\mathcal{Q}_{ST})D_t(\tilde{\theta}_{ST}-\theta_{ST})}_{H^s(q\mathbb{T})}\leq C\epsilon^{7/2}\delta e^{\epsilon^2 t}.
\end{equation}
We write $[D_t, \mathcal{Q}-\mathcal{Q}_{ST}]$ as 
\begin{equation*}
    \begin{split}
        [D_t, \mathcal{Q}-\mathcal{Q}_{ST}]=& [D_t, \mathcal{P}(I-\mathcal{H}_{\zeta}]-[D_t, \mathcal{P}_{ST}(I-\mathcal{H}_{\zeta_{ST}})]\\
    =& [D_t, \mathcal{P}](I-\mathcal{H}_{\zeta})+\mathcal{P}[D_t, I-\mathcal{H}_{\zeta}]-\Big([D_t, \mathcal{P}_{ST}](I-\mathcal{H}_{\zeta_{ST}})+\mathcal{P}_{ST}[D_t, I-\mathcal{H}_{\zeta_{ST}}]\Big)
    \end{split}
\end{equation*}
Repeating the estimates for $M_4$ and $M_5$ if necessary, we obtain
\begin{equation}
    \norm{ [D_t, \mathcal{Q}-\mathcal{Q}_{ST}](\tilde{\theta}_{ST}-\theta_{ST})}_{H^s(q\mathbb{T})}\leq C\epsilon^{7/2}\delta e^{\epsilon^2 t}.
\end{equation}
So one has
\begin{equation}
    \norm{M_4}_{H^s(q\mathbb{T})}\leq C\epsilon^{7/2}\delta e^{\epsilon^2 t}.
\end{equation}
Similarly, we conclude
\begin{equation}
    \norm{M_2}_{H^s(q\mathbb{T})}+\norm{M_3}_{H^s(q\mathbb{T})}\leq C\epsilon^{7/2}\delta e^{\epsilon^2 t}.
\end{equation}

\subsection{Estimate $M_6$} Recall that $M_6= \mathcal{P}[D_t, \mathcal{H}_{\zeta}]\rho$. We decompose
\begin{align*}
    \mathcal{P}[D_t, \mathcal{H}_{\zeta}]\rho=& \mathcal{P}D_t\mathcal{H}_{\zeta}\rho-\mathcal{P}\mathcal{H}_{\zeta}D_t\rho\\
    =& D_t\mathcal{P}\mathcal{H}_{\zeta}\rho+[\mathcal{P}, D_t]\mathcal{H}_{\zeta}\rho-\mathcal{H}_{\zeta}\mathcal{P}D_t\rho-[\mathcal{P}, \mathcal{H}_{\zeta}]D_t\rho\\
    =& D_t\mathcal{H}_{\zeta}\mathcal{P}\rho+D_t[\mathcal{P},\mathcal{H}_{\zeta}]\rho+[\mathcal{P}, D_t]\mathcal{H}_{\zeta}\rho-\mathcal{H}_{\zeta}D_t\mathcal{P}\rho-\mathcal{H}_{\zeta}[\mathcal{P}, D_t]\rho-[\mathcal{P}, \mathcal{H}_{\zeta}]D_t\rho\\
    =& \mathcal{H}_{\zeta}D_t\mathcal{P}\rho+[D_t, \mathcal{H}_{\zeta}]\mathcal{P}\rho +D_t[\mathcal{P},\mathcal{H}_{\zeta}]\rho+[\mathcal{P}, D_t]\mathcal{H}_{\zeta}\rho-\mathcal{H}_{\zeta}D_t\mathcal{P}\rho\\
    &-\mathcal{H}_{\zeta}[\mathcal{P}, D_t]\rho-[\mathcal{P}, \mathcal{H}_{\zeta}]D_t\rho\\
    =& [D_t\zeta, \mathcal{H}_{\zeta}]\frac{\partial_{\alpha}\mathcal{P}\rho}{\zeta_{\alpha}}+D_t[\mathcal{P},\mathcal{H}_{\zeta}]\rho+i\frac{a_t}{a}\circ\kappa^{-1}A\zeta_{\alpha}\partial_{\alpha}(\mathcal{H}_{\zeta}\rho)\\
    &-i\mathcal{H}_{\zeta}\frac{a_t}{a}\circ\kappa^{-1}A\rho_{\alpha}-[\mathcal{P}, \mathcal{H}_{\zeta}]D_t\rho\\
    :=& \it{II}_1+\it{II}_2+\it{II}_3+\it{II}_4+\it{II}_5.
\end{align*}
By Proposition \ref{singularperiodic} and the estimates for $\mathcal{P}\rho=N_1+N_2+N_3+N_4$:
\begin{equation}
    \norm{\mathcal{P}\rho}_{H^s(q\mathbb{T})}\leq C\epsilon^{7/2}\delta e^{\epsilon^2 t}+C\epsilon^2 E_s(t),
\end{equation}
we obtain
\begin{equation}
    \norm{\it{II}_1}_{H^s}\leq C\epsilon^{7/2}\delta e^{\epsilon^2 t}+C\epsilon^2 E_s(t).
\end{equation}
For $\it{II}_3$ and $\it{II}_4$, we have 
\begin{align*}\small
    &\norm{\it{II}_3+\it{II}_4}_{H^s(q\mathbb{T})}\\
    \leq & \norm{\frac{a_t}{a}\circ\kappa^{-1}}_{H^s(q\mathbb{T})}\norm{A\zeta_{\alpha}}_{W^{s-1,\infty}(q\mathbb{T})}\norm{\rho_{\alpha}}_{W^{s-1,\infty}}+\norm{\frac{a_t}{a}\circ\kappa^{-1}}_{W^{s-1,\infty}(q\mathbb{T})}\norm{A}_{W^{s-1,\infty}(q\mathbb{T})}\norm{\rho_{\alpha}}_{H^{s}(q\mathbb{T})}\\
    &+\norm{\frac{a_t}{a}\circ\kappa^{-1}}_{H^s(q\mathbb{T})}\norm{A-1}_{H^s(q\mathbb{T})}\norm{\rho_{\alpha}}_{W^{s-1,\infty}}\\
    \leq & C\epsilon^{7/2}\delta e^{\epsilon^2 t}+C\epsilon^2 E_s(t).
\end{align*}
For $\it{II}_5$, by (\ref{U6}), we obtain
\begin{align*}
    [\mathcal{P}, \mathcal{H}_{\zeta}]D_t\rho=& [D_t^2-iA\partial_{\alpha}, \mathcal{H}_{\zeta}]D_t\rho\\
    =& 2[D_t\zeta, \mathcal{H}_{\zeta}]\frac{\partial_{\alpha}D_t\rho}{\zeta_{\alpha}}-\frac{1}{4\pi q^2 i}\int_{-q\pi}^{q\pi}\Big(\frac{D_t\zeta(\alpha,t)-D_t\zeta(\beta,t)}{\sin(\frac{\zeta(\alpha,t)-\zeta(\beta,t)}{2q})}\Big)^2 \partial_{\beta}D_t\rho(\beta,t)\,d\beta\\
    :=& \it{II}_{51}+\it{II}_{52}.
\end{align*}
For $\it{II}_{52}$, using Proposition \ref{singularperiodic}, one has
\begin{equation}
    \norm{\it{II}_{52}}_{H^s(q\mathbb{T})}\leq C\epsilon^2 \norm{D_t\rho}_{H^s(q\mathbb{T})}.
\end{equation}
For $\it{II}_{51}$, we decompose $ D_t\rho$ as
\begin{equation}
    D_t\rho=\frac{1}{2}(I+\overline{\mathcal{H}_{\zeta}})D_t\rho+\frac{1}{2}(I-\overline{\mathcal{H}_{\zeta}})D_t\rho:=\mathbb{P}_{A}D_t\rho+\mathbb{P}_H D_t\rho.
\end{equation}
Then we can write
\begin{equation}
    \it{II}_{51}=2[D_t\zeta, \mathcal{H}_{\zeta}]\frac{\partial_{\alpha}\mathbb{P}_A D_t\rho}{\zeta_{\alpha}}+2[D_t\zeta, \mathcal{H}_{\zeta}]\frac{\partial_{\alpha}\mathbb{P}_H D_t\rho}{\zeta_{\alpha}}:= \it{II}_{511}+\it{II}_{512}.
\end{equation}
For $\it{II}_{511}$, using\footnote{Note that $\mathbb{P}_A D_t\rho=\frac{1}{2}(I+\mathcal{H}_{\zeta})D_t\bar{\rho}$ is the boundary value of a bounded holomorphic function in $\Omega(t)$, so $\frac{\partial_{\alpha}\overline{\mathbb{P}_A D_t\rho}}{\zeta_{\alpha}}$ is also the boundary value of a bounded holomorphic function in $\Omega(t)$. So $D_t\bar{\zeta}\frac{\partial_{\alpha}\overline{\mathbb{P}_A D_t\rho}}{\zeta_{\alpha}}$ is the boundary value of a bounded holomorphic function in $\Omega(t)$ which approaches zero as $y\rightarrow-\infty$. }
\begin{equation}
    (I-\mathcal{H}_{\zeta})D_t\bar{\zeta}=0,\,\,    (I-\mathcal{H}_{\zeta})D_t\bar{\zeta}\frac{\partial_{\alpha}\overline{\mathbb{P}_A D_t\rho}}{\zeta_{\alpha}}=0.
\end{equation}
we have 
\begin{align*}
    \it{II}_{511}=& 2[D_t\zeta, \mathcal{H}_{\zeta}]\frac{\partial_{\alpha}\mathbb{P}_A D_t\rho}{\zeta_{\alpha}}\\
    =&  2[D_t\zeta, \mathcal{H}_{\zeta}\frac{1}{\zeta_{\alpha}}+\overline{\mathcal{H}_{\zeta}\frac{1}{\zeta_{\alpha}}}]\partial_{\alpha}\mathbb{P}_A D_t\rho
\end{align*}
Then we can use arguments as for $G_1$ in \S \ref{estimates:N1:G1} to conclude that 
\begin{equation}
    \norm{\it{II}_{511}}_{H^s(q\mathbb{T})}\leq C\epsilon^{7/2}\delta e^{\epsilon^2 t}+C\epsilon^2 \norm{D_t\rho}_{H^s(q\mathbb{T})}.
\end{equation}
For $\it{II}_{512}$, we decompose it as
\begin{align*}
   & \mathbb{P}_H D_t\rho= D_t \mathbb{P}_H \rho+[\mathbb{P}_H, D_t]\rho\\
    =& \frac{1}{2}D_t (I-\overline{\mathcal{H}_{\zeta}})(I-\mathcal{H}_{\zeta})(\theta-\theta_{ST}-(\tilde{\theta}-\tilde{\theta}_{ST}))-\frac{1}{2}[D_t\bar{\zeta}, \overline{\mathcal{H}_{\zeta}}]\frac{\partial_{\alpha}\rho}{\bar{\zeta}_{\alpha}}\\
    =&\frac{1}{2}D_t(I+\mathcal{H}_{\zeta})(I-\mathcal{H}_{\zeta})(\theta-\theta_{ST}-(\tilde{\theta}-\tilde{\theta}_{ST}))-\frac{1}{2}D_t(\mathcal{H}_{\zeta}+\overline{\mathcal{H}_{\zeta}})(I-\mathcal{H}_{\zeta})\rho-\frac{1}{2}[D_t\bar{\zeta}, \overline{\mathcal{H}_{\zeta}}]\frac{\partial_{\alpha}\rho}{\bar{\zeta}_{\alpha}}\\
    :=& F_1+F_2+F_3.
\end{align*}
For the first term, note that by Corollary \ref{corollary:hilbertboundary},
\begin{align*}
    &\frac{1}{2}(I+\mathcal{H}_{\zeta})(I-\mathcal{H}_{\zeta})(\theta-\theta_{ST}-(\tilde{\theta}-\tilde{\theta}_{ST}))\\
    =& \frac{1}{2}(I-\mathcal{H}_{\zeta})(I+\mathcal{H}_{\zeta})(\theta-\theta_{ST}-(\tilde{\theta}-\tilde{\theta}_{ST}))\\
    =& c(t),
\end{align*}
where
\begin{equation}
    c(t)=\frac{1}{q\pi}\int_{-q\pi}^{q\pi}\partial_{\alpha}\zeta(\beta,t)(I+\mathcal{H}_{\zeta})(\theta-\theta_{ST}-(\tilde{\theta}-\tilde{\theta}_{ST}))\,d\alpha.
\end{equation}
For $F_2$ and $F_3$, it's easy to obtain
\begin{equation}
    \norm{F_2}_{H^s(q\mathbb{T})}+\norm{F_3}_{H^s(q\mathbb{T})}\leq C\epsilon \norm{D_t\rho}_{H^s(q\mathbb{T})}+C\epsilon\norm{\partial_{\alpha}\rho}_{H^s(q\mathbb{T})}.
\end{equation}
From the computations above, it follows 
\begin{equation}
\begin{split}
    \norm{\it{II}_{512}}_{H^s(q\mathbb{T})}=& \norm{2[D_t\zeta, \mathcal{H}_{\zeta}]\frac{\partial_{\alpha}(c'(t)+F_2+F_3)}{\zeta_{\alpha}}}_{H^s(q\mathbb{T})}\\
    \leq & C\epsilon^2 \norm{D_t\rho}_{H^s(q\mathbb{T})}+C\epsilon^2\norm{\partial_{\alpha}\rho}_{H^s(q\mathbb{T})}.
    \end{split}
\end{equation}
So we obtain
\begin{equation}
    \norm{\it{II}_5}_{H^s(q\mathbb{T})}\leq C\epsilon^2\norm{\partial_{\alpha}\rho}_{H^s(q\mathbb{T})}+C\epsilon^2\norm{D_t\rho}_{H^s(q\mathbb{T})}.
\end{equation}

\vspace*{1ex}
To estimate $\it{II}_2$,  we use (\ref{U6}) to rewrite $D_t[\mathcal{P}, \mathcal{H}_{\zeta}]\rho$ as 
\begin{align*}
    D_t[\mathcal{P}, \mathcal{H}_{\zeta}]\rho=& D_t \Big\{ 2[D_t\zeta, \mathcal{H}_{\zeta}]\frac{\partial_{\alpha}D_t\rho}{\zeta_{\alpha}}-\frac{1}{4\pi q^2 i}\int_{-q\pi}^{q\pi}\Big(\frac{D_t\zeta(\alpha,t)-D_t\zeta(\beta,t)}{\sin(\frac{\zeta(\alpha,t)-\zeta(\beta,t)}{2q})}\Big)^2 \partial_{\beta}D_t\rho(\beta,t)\,d\beta\Big\}.
\end{align*}
Then $D_t [\mathcal{P}, \mathcal{H}_{\zeta}]\rho$ can be written as 
\begin{equation}
    D_t [\mathcal{P}, \mathcal{H}_{\zeta}]\rho=2[D_t\zeta, \mathcal{H}_{\zeta}]\frac{\partial_{\alpha}D_t^2\rho}{\zeta_{\alpha}}+2[D_t^2\zeta, \mathcal{H}_{\zeta}]\frac{\partial_{\alpha}D_t\rho}{\zeta_{\alpha}}+g,
\end{equation}
where 
\begin{align*}
    g=& D_t \Big\{ -\frac{1}{4\pi q^2 i}\int_{-q\pi}^{q\pi}\Big(\frac{D_t\zeta(\alpha,t)-D_t\zeta(\beta,t)}{\sin(\frac{\zeta(\alpha,t)-\zeta(\beta,t)}{2q})}\Big)^2 \partial_{\beta}D_t\rho(\beta,t)\,d\beta\Big\}\\
    &-\frac{1}{4q^2\pi i}\int_{-q\pi}^{q\pi}\Big(\frac{D_t\zeta(\alpha,t)-D_t\zeta(\beta,t)}{\sin(\frac{\zeta(\alpha,t)-\zeta(\beta,t)}{2q})}\Big)^2 \partial_{\beta} D_t\rho(\beta,t)\,d\beta\\
    :=& g_1+g_2.
\end{align*}
The calculation and the estimates for $g_1$ is similar to the term $D_t G_1$, see \S \ref{estimates:M1=DtN1}. We obtain
\begin{equation}
    \norm{g}_{H^s(q\mathbb{T})}\leq C\epsilon^2 \norm{D_t\rho}_{H^s(q\mathbb{T})}+C\epsilon^2 \norm{\partial_{\alpha}\rho}_{H^s(q\mathbb{T})}.
\end{equation}
For $2[D_t\zeta, \mathcal{H}_{\zeta}]\frac{\partial_{\alpha}D_t^2\rho}{\zeta_{\alpha}}+2[D_t^2\zeta, \mathcal{H}_{\zeta}]\frac{\partial_{\alpha}D_t\rho}{\zeta_{\alpha}}$, decomposing 
\begin{equation}
    \rho=\mathbb{P}_{H}\rho+\mathbb{P}_A\rho,
\end{equation}
and using the same arguments as for $\it{II}_{51}$, we obtain
\begin{equation}
    \norm{\it{II}_2}_{H^s(q\mathbb{T})} \leq C\epsilon^{7/2}\delta e^{\epsilon^2 t}+C\epsilon^2\norm{\partial_{\alpha}\rho}_{H^s(q\mathbb{T})}+C\epsilon^2\norm{D_t\rho}_{H^s(q\mathbb{T})}.
\end{equation}
Therefore, we finally conclude that
\begin{equation}
    \norm{M_6}_{H^s(q\mathbb{T})}\leq  C\epsilon^{7/2}\delta e^{\epsilon^2 t}+C\epsilon^2\norm{\partial_{\alpha}\rho}_{H^s(q\mathbb{T})}+C\epsilon^2\norm{D_t\rho}_{H^s(q\mathbb{T})}.
\end{equation}

\subsection{Summary of the estimates}
To conclude, by assuming the a priori assumptions (\ref{boot}) and  (\ref{aprioritwo}), together with Lemma \ref{equivalencequantities}, we obtain
\begin{equation}\label{sumarize}
\begin{split}
   & \sum_{j=1}^4 \norm{N_j(t)}_{H^s(q\mathbb{T})}+\sum_{j=1}^6 \norm{M_j(t)}_{H^s(q\mathbb{T})}
   \leq  C\epsilon^{7/2}\delta e^{\epsilon^2 t}+C\epsilon^2E_s^{1/2}(t)\leq C\epsilon^{7/2}\delta e^{\epsilon^2 t}.
    \end{split}
\end{equation}
		\begin{equation}
    \norm{b-b_{ST}}_{H^{s}(q\mathbb{T})}\leq C\epsilon^{3/2}\delta e^{\epsilon^2 t},
\end{equation}
and
\begin{equation}
    \norm{b-b_{ST}}_{W^{s-1,\infty}(q\mathbb{T})}\leq C\epsilon^{3/2}q^{-1/2}\delta e^{\epsilon^2 t},
\end{equation}

	\section{Energy estimates}\label{section:energyestimates}

	The goal of this section is to obtain the following energy estimates.
	\begin{proposition}\label{enegyproposition}
	Assuming the a priori assumptions (\ref{boot}) and (\ref{aprioritwo}), we have for $t\in [0,\epsilon^{-2} \log\frac{\mu}{\delta}]$,
		\begin{equation}
		\frac{d}{dt}\mathcal{E}(t)\leq C\epsilon^{5}\delta^2 e^{2\epsilon^2 t},
		\end{equation}
where $C>0$ is some constant depending on $s$ only.
	\end{proposition}
Recall that
\begin{equation}
\begin{split}
\frac{d}{dt}\mathcal{E}_n(t)=& \int_{-q\pi}^{q\pi} \frac{2}{A}\Re(D_t\rho^{(n)} \bar{\mathcal{C}}_{1,n}) -\frac{1}{A}\frac{a_t}{a}\circ \kappa^{-1} |D_t\rho^{(n)}|^2 \,d\alpha\\
&+2\Im \int_{-q\pi}^{q\pi} \partial_t\mathcal{R}^{(n)}\partial_{\alpha}\bar{\phi}^{(n)}+\partial_t\mathcal{\phi}^{(n)}\partial_{\alpha}\bar{\mathcal{R}}_{}^{(n)}
+\partial_t\mathcal{R}^{(n)}\partial_{\alpha}\bar{\mathcal{R}}^{(n)}\,d\alpha\\
:=&\mathfrak{E}_{1,n}+\mathfrak{E}_{2,n}+\mathfrak{E}_{3,n}+\mathfrak{E}_{4,n}+\mathfrak{E}_{5,n},
\end{split}
\end{equation}
where
\begin{equation}
    \mathcal{C}_{1,n}=\sum_{m=1}^4 \partial_{\alpha}^n N_m +[D_t^2-iA\partial_{\alpha}, \partial_{\alpha}^n]\rho,
\end{equation}
and
\begin{equation}
    \rho^{(n)}=\partial_{\alpha}^n\rho, \quad \phi^{(n)}=\frac{1}{2}(I-\mathcal{H}_{\zeta})\rho^n, \quad \quad \mathcal{R}^{(n)}=\frac{1}{2}(I+\mathcal{H}_{\zeta})\rho^{(n)}.
\end{equation}
From the computations in previous section, we have obtained the estimates for $N_m$. To close the energy estimates for $\mathcal{E}_n(t)$, we still need to obtain the estimates for $\int_{-q\pi}^{q\pi} |D_t\rho^{(n)}[D_t^2-iA\partial_{\alpha}, \partial_{\alpha}^n]\rho|$, and $\mathfrak{E}_j, j=2, 3, 4, 5$.

Recall also that
\begin{equation}
\begin{split}
\frac{d}{dt}\mathcal{F}_n(t)=& \int_{-q\pi}^{q\pi} \frac{2}{A}\Re(D_t\sigma^{(n)} \bar{\mathcal{C}}_{2,n}) -\frac{1}{A}\frac{a_t}{a}\circ \kappa^{-1} |D_t\sigma^{(n)}|^2\, d\alpha\\
:=& \mathfrak{F}_1+\mathfrak{F}_2.
\end{split}
\end{equation}
where
\begin{equation}
    \mathcal{C}_{2,n}=\partial_{\alpha}^n (D_t^2-iA\partial_{\alpha})\sigma+[D_t^2-iA\partial_{\alpha}, \partial_{\alpha}^n]\sigma
\end{equation}

\subsection{Estimate $\int_{-q\pi}^{q\pi} |D_t\rho^{(n)}[D_t^2-iA\partial_{\alpha}, \partial_{\alpha}^n]\rho|\,d\alpha$ and $\mathcal{C}_{1,n}$}\label{estimates:commutator:rho}
By direct commutator computations, we have 
\begin{equation}
\begin{split}
&[D_t^2-iA\partial_{\alpha}, \partial_{\alpha}^n]\rho\\
=&\sum_{m=1}^n \partial_{\alpha}^{n-m}[D_t^2-iA\partial_{\alpha},\partial_{\alpha}]\partial_{\alpha}^{m-1}\rho\\
=&\sum_{m=1}^n \partial_{\alpha}^{n-m}(D_t[D_t,\partial_{\alpha}]+[D_t,\partial_{\alpha}]D_t+iA_{\alpha}\partial_{\alpha})\partial_{\alpha}^{m-1}\rho\\
=&-\sum_{m=1}^n \partial_{\alpha}^{n-m}D_t(b_{\alpha}\partial_{\alpha})\partial_{\alpha}^{m-1}\rho-\sum_{m=1}^n \partial_{\alpha}^{n-m}b_{\alpha}D_t\partial_{\alpha}^{m-1}\rho+i\sum_{m=1}^n\partial_{\alpha}^{n-m}A_{\alpha}\partial_{\alpha}^{m}\rho.
\end{split}
\end{equation}
By writing $D_t(b_{\alpha}\partial_{\alpha})$ and $D_t\partial_{\alpha}^{m-1}\rho$ as
\begin{equation}
    D_t(b_{\alpha}\partial_{\alpha})=(D_t b_{\alpha})\partial_{\alpha}+b_{\alpha}\partial_{\alpha}D_t+b_{\alpha}[D_t,\partial_{\alpha}],
\end{equation}
\begin{equation}
    D_t\partial_{\alpha}^{m-1}\rho=\partial_{\alpha}^{m-1}D_t\rho+[D_t,\partial_{\alpha}^{m-1}]\rho,
\end{equation}
and using the estimates 
\begin{equation}
    \norm{b}_{H^s(q\mathbb{T})}\leq C\epsilon^{2}q^{1/2}, \quad \quad \norm{b}_{W^{s-1,\infty}(q\mathbb{T})}\leq C\epsilon^{2}, \quad \norm{D_tb}_{W^{s-1,\infty}}\leq C\epsilon^2,
\end{equation}
and
\begin{equation}
    \norm{A-1}_{H^s(q\mathbb{T})}\leq C\epsilon^3 q^{1/2}, \quad \quad  \norm{A-1}_{W^{s-1,\infty}(q\mathbb{T})}\leq C\epsilon^3,
\end{equation}
we obtain
\begin{equation}
    \norm{[D_t^2-iA\partial_{\alpha}, \partial_{\alpha}^n]\rho}_{H^s(q\mathbb{T})}\leq C\epsilon^2 \norm{D_t\rho}_{H^s(q\mathbb{T})}+C\epsilon^2 \norm{\partial_{\alpha}\rho}_{H^s(q\mathbb{T})}.
\end{equation}
Therefore, we conclude that
\begin{equation}\label{ineq:[D_t^2-]}
    \int_{-q\pi}^{q\pi} |D_t\rho^{(n)}[D_t^2-iA\partial_{\alpha}, \partial_{\alpha}^n]\rho\,d\alpha\leq C\epsilon^2 \norm{D_t\rho}_{H^s(q\mathbb{T})}^2+C\epsilon^2 \norm{\partial_{\alpha}\rho}_{H^s(q\mathbb{T})}^2.
\end{equation}
By \eqref{ineq:[D_t^2-]}, \eqref{sumarize} and the Cauchy Schwarz inequality, one has
\begin{equation}
    |\mathcal{C}_{1,n}|\leq C\epsilon^5\delta^2e^{2\epsilon^2 t}+C\epsilon^2 (\norm{D_t\rho}_{H^s(q\mathbb{T})}^2+ \norm{\partial_{\alpha}\rho}_{H^s(q\mathbb{T})}^2).
\end{equation}

\subsection{Estimate $\mathfrak{E}_2$} For $\mathfrak{E}_2$, recalling that by (\ref{estimate:atainverse}) and \eqref{estimate:atainverinfty},
\begin{equation}
    \norm{\frac{a_t}{a}\circ \kappa^{-1} }_{H^s(q\mathbb{T})}\leq C\epsilon^2 q^{1/2}, \quad \quad  \norm{\frac{a_t}{a}\circ \kappa^{-1}}_{W^{s-1,\infty}}\leq C\epsilon^{2},
\end{equation}
together with \eqref{estimate:Alower}, one has
\begin{equation}
\begin{split}
    |\mathfrak{E}_2|\leq  & C\norm{\frac{1}{A}}_{L^{\infty}} \Big\{\norm{\frac{a_t}{a}\circ \kappa^{-1} }_{H^s(q\mathbb{T})}\norm{D_t\rho}_{W^{s-1,\infty}}^2+ \norm{\frac{a_t}{a}\circ \kappa^{-1} }_{W^{s-1,\infty}(q\mathbb{T})}\norm{D_t\rho}_{H^s(q\mathbb{T})}^2\Big\}\\
    \leq & C\epsilon^2 \norm{D_t\rho}_{H^s(q\mathbb{T})}^2+C\epsilon^2 \norm{\partial_{\alpha}\rho}_{H^s(q\mathbb{T})}^2.
    \end{split}
\end{equation}
\subsection{Estimates $\mathfrak{E}_3$, $\mathfrak{E}_4$, and $\mathfrak{E}_5$} 
\subsubsection{Estimate $\phi^{(n)}$ and $\mathcal{R}^{(n)}$} Recall that $\mathfrak{E}_3=2\Im \int \partial_t\mathcal{R}^{(n)}\partial_{\alpha}\bar{\phi}^{(n)}\,d\alpha$. 
We have for $n\geq 1$,
\begin{align*}
\mathcal{R}^{(n)}=& \frac{1}{2}\partial_{\alpha}^n (I+\mathcal{H}_{\zeta})\rho-\frac{1}{2}[\partial_{\alpha}^n, \mathcal{H}_{\zeta}]\rho\\
=&\frac{1}{2}\partial_{\alpha}^n(I+\mathcal{H}_{\zeta})(I-\mathcal{H}_{\zeta})(\theta-\theta_{ST}-(\tilde{\theta}-\tilde{\theta}_{ST}))-\frac{1}{2}\sum_{m=1}^n \partial_{\alpha}^{n-m}[\zeta_{\alpha}-1, \mathcal{H}_{\zeta}]\frac{\partial_{\alpha}^{m}\rho}{\zeta_{\alpha}}\\
=&\partial_{\alpha}^n c(t)-\frac{1}{2}\sum_{m=1}^n \partial_{\alpha}^{n-m}[\zeta_{\alpha}-1, \mathcal{H}_{\zeta}]\frac{\partial_{\alpha}^{m}\rho}{\zeta_{\alpha}}\\
=& -\frac{1}{2}\sum_{m=1}^n \partial_{\alpha}^{n-m}[\zeta_{\alpha}-1, \mathcal{H}_{\zeta}]\frac{\partial_{\alpha}^{m}\rho}{\zeta_{\alpha}}.
\end{align*}
Here, we used the fact that $\frac{1}{2}(I+\mathcal{H}_{\zeta})(I-\mathcal{H}_{\zeta})(\theta-\theta_{ST}-(\tilde{\theta}-\tilde{\theta}_{ST}))=c(t)$ for some constant $c(t)$, and therefore\footnote{Here and after, we use $c(t)$ to denote constants depending on $t$ only. }
\begin{equation}
    \partial_{\alpha}^n(I+\mathcal{H}_{\zeta})(I-\mathcal{H}_{\zeta})(\theta-\theta_{ST}-(\tilde{\theta}-\tilde{\theta}_{ST}))=0.
\end{equation}
So we obtain for $1\leq n\leq s+1$,
\begin{equation}\label{estimates:mathcalRnpartialalpha}
    \norm{\partial_{\alpha}\mathcal{R}^{(n)}}_{L^2(q\mathbb{T})}\leq C\epsilon \norm{\partial_{\alpha}\rho}_{H^s(q\mathbb{T})},
\end{equation}
and for $\mathcal{R}_t^{(n)}$, one has
\begin{align*}
    \mathcal{R}_t^{(n)}=-\frac{1}{2}\partial_t \sum_{m=1}^n \partial_{\alpha}^{n-m}[\zeta_{\alpha}-1, \mathcal{H}_{\zeta}]\frac{\partial_{\alpha}^{m}\rho}{\zeta_{\alpha}}
\end{align*}
Then one obtain
\begin{equation}\label{estimates:mathcalRn}
    \norm{\partial_{t}\mathcal{R}^{(n)}}_{L^2(q\mathbb{T})}\leq C\epsilon \norm{D_t\rho}_{H^s(q\mathbb{T})}+C\epsilon\norm{\partial_{\alpha}\rho}_{H^s(q\mathbb{T})}.
\end{equation}
For $n=0$, we have 
\begin{equation}
    \mathcal{R}^{(0)}=\frac{1}{2}(I+\mathcal{H}_{\zeta})\rho=c(t), \quad \quad \phi^{(0)}=\frac{1}{2}(I-\mathcal{H}_{\zeta})\rho.
\end{equation}
\subsubsection{Estimate $\mathfrak{E}_3$}
Recll that  $\mathfrak{E}_3=2\Im \int_{-q\pi}^{q\pi} \partial_t\mathcal{R}^{(n)}\partial_{\alpha}\bar{\phi}^{(n)}\,d\alpha$. For $n=0$, we have 
\begin{align*}
    \mathfrak{E}_3=2\Im \int_{-q\pi}^{q\pi} \partial_t\mathcal{R}^{(0)}\partial_{\alpha}\bar{\phi}^{(0)}\,d\alpha=2\Im c(t)\int_{-q\pi}^{q\pi} \partial_{\alpha}\bar{\phi}^{(0)}\,d\alpha=0.
\end{align*}
For $1\leq n\leq s+1$, using $\mathcal{R}^{(n)}=\frac{1}{2}(I+\mathcal{H}_{\zeta})\mathcal{R}^{(n)}+c(t)$ and the fact $\int_{-q\pi}^{q\pi} \overline{\partial_{\alpha}\phi^{(n)}} d\alpha=0$, one has
\begin{equation}
    \int_{-q\pi}^{q\pi} \partial_t \mathcal{R}^{(n)}\overline{\partial_{\alpha}\phi^{(n)}}\,d\alpha=\frac{1}{2}\int_{-q\pi}^{q\pi} \partial_t (I+\mathcal{H}_{\zeta})\mathcal{R}^{(n)}\overline{\partial_{\alpha}\phi^{(n)}}\,d\alpha.
\end{equation}
We rewrite $\partial_{t}(I+\mathcal{H}_{\zeta})\mathcal{R}^{(n)}$ and $\partial_{\alpha}\phi^{(n)}$ as 
\begin{equation}
    \partial_t (I+\mathcal{H}_{\zeta})\mathcal{R}^{(n)}=(I+\mathcal{H}_{\zeta})\partial_t\mathcal{R}^{(n)}+[\zeta_t, \mathcal{H}_{\zeta}]\frac{\partial_{\alpha}\mathcal{R}^{(n)}}{\zeta_{\alpha}},
\end{equation}
and
\begin{equation}
    \partial_{\alpha}\phi^{(n)}=(I-\mathcal{H}_{\zeta})\partial_{\alpha}^{n+1}\rho-[\zeta_{\alpha}, \mathcal{H}_{\zeta}]\frac{\partial_{\alpha}^{n+1}\rho}{\zeta_{\alpha}}.
\end{equation}
Therefore, we rewrite
\begin{align*}
    \mathfrak{E}_3=& \Im\int_{-q\pi}^{q\pi}\partial_t (I+\mathcal{H}_{\zeta}) \mathcal{R}^{(n)}\overline{(I-\mathcal{H}_{\zeta})\partial_{\alpha}^{n+1}\rho}\,d\alpha-\Im \int_{-q\pi}^{q\pi} \partial_{t}(I+\mathcal{H}_{\zeta})\mathcal{R}^{(n)}\overline{[\zeta_{\alpha}, \mathcal{H}_{\zeta}]\frac{\partial_{\alpha}^{n+1}\rho}{\zeta_{\alpha}}}\,d\alpha\\
    :=& \mathfrak{E}_{31}+\mathfrak{E}_{32}.
\end{align*}
For $\mathfrak{E}_{32}$, using (\ref{estimates:mathcalRn}) and Proposition \ref{singularperiodic}, we have 
\begin{equation}
    |\mathfrak{E}_{32}|\leq C\epsilon^2 \norm{D_t\rho}_{H^s(q\mathbb{T})}^2+C\epsilon^2\norm{\partial_{\alpha}\rho}_{H^s(q\mathbb{T})}^2.
\end{equation}
For $\mathfrak{E}_{31}$, one has
\begin{align*}
    \mathfrak{E}_{31}=& \Im \int_{-q\pi}^{q\pi} (I+\mathcal{H}_{\zeta})\partial_t\mathcal{R}^{(n)} \overline{(I-\mathcal{H}_{\zeta})\partial_{\alpha}^{n+1}\rho}\,d\alpha-\Im\int_{-q\pi}^{q\pi} [\partial_t\zeta ,\mathcal{H}_{\zeta}] \frac{\partial_{\alpha}\mathcal{R}^{(n)}}{\zeta_{\alpha}} \overline{(I-\mathcal{H}_{\zeta})\partial_{\alpha}^{n+1}\rho}\,d\alpha\\
    :=& \mathfrak{E}_{311}+\mathfrak{E}_{312}.
\end{align*}
Using the same argument as for $\mathfrak{E}_{32}$, we have 
\begin{equation}
     |\mathfrak{E}_{312}|\leq C\epsilon^2 \norm{D_t\rho}_{H^s(q\mathbb{T})}^2+C\epsilon^2\norm{\partial_{\alpha}\rho}_{H^s(q\mathbb{T})}^2.
\end{equation}
For $\mathfrak{E}_{311}$, we split it into two pieces
\begin{align*}
    \mathfrak{E}_{311}=& \Im \int_{-q\pi}^{q\pi} \Big((I+\mathcal{H}_{\zeta})\partial_t\mathcal{R}^{(n)}\Big) (I-\overline{\mathcal{H}_{\zeta}})\partial_{\alpha}^{n+1}\bar{\rho}\,d\alpha\\
    =&  \Im \int_{-q\pi}^{q\pi} \Big((I+\mathcal{H}_{\zeta})\partial_t\mathcal{R}^{(n)}\Big) (I+H_0)\partial_{\alpha}^{n+1}\bar{\rho}\,d\alpha\\
    &-\Im \int_{-q\pi}^{q\pi} \Big((I+\mathcal{H}_{\zeta})\partial_t\mathcal{R}^{(n)}\Big) (H_0+\overline{\mathcal{H}_{\zeta}})\partial_{\alpha}^{n+1}\bar{\rho}\,d\alpha\\
    :=& \mathcal{C}_{311}^{(1)}+\mathcal{C}_{311}^{(2)}.
\end{align*}
Note that since 
\begin{equation}
    \norm{(H_0+\overline{\mathcal{H}_{\zeta}})\partial_{\alpha}^{n+1}\bar{\rho}}_{H^s(q\mathbb{T})}\leq C\epsilon \norm{\partial_{\alpha}\rho}_{H^s(q\mathbb{T})},
\end{equation}
one has
\begin{equation}
    |\mathfrak{E}_{311}^{(2)}|\leq C\epsilon^2 \norm{D_t\rho}_{H^s(q\mathbb{T})}^2+C\epsilon^2\norm{\partial_{\alpha}\rho}_{H^s(q\mathbb{T})}^2.
\end{equation}
For $\mathfrak{E}_{311}^{(1)}$, we further decompose it as
\begin{align*}
    \mathfrak{E}_{311}^{(1)}=& \Im \int \Big((I+H_0)\partial_t\mathcal{R}^{(n)}\Big) \Big((I+H_0)\partial_{\alpha}^{n+1}\bar{\rho}\Big)\,d\alpha\\
   & +\Im \int \Big((\mathcal{H}_{\zeta}-H_0)\partial_t\mathcal{R}^{(n)}\Big) \Big((I+H_0)\partial_{\alpha}^{n+1}\bar{\rho}\Big)\,d\alpha\\
    :=& L_1+L_2.
\end{align*}
Clearly, using (\ref{estimates:mathcalRn}), we have $$\norm{(\mathcal{H}_{\zeta}-H_0)\partial_t\mathcal{R}^{(n)}}_{L^2}\leq C\epsilon^2 \norm{D_t\rho}_{H^s(q\mathbb{T})}+C\epsilon^4\norm{\partial_{\alpha}\rho}_{H^s(q\mathbb{T})},$$
from which it follows 
\begin{equation}
    |L_2|\leq C\epsilon^2 \norm{D_t\rho}_{H^s(q\mathbb{T})}^2+C\epsilon^2\norm{\partial_{\alpha}\rho}_{H^s(q\mathbb{T})}^2.
\end{equation}
Finally, we analyze $L_1$. Note that $(I+H_0)\partial_{\alpha}^{n+1}\bar{\rho}=\partial_{\alpha}(I+H_0)\partial_{\alpha}^n\bar{\rho}$, and $(I+H_0)\partial_{\alpha}^{n}\bar{\rho}$ is the boundary value of a holomorphic function $\Phi_2(x+iy,t)$ in $\mathbb{P}_-$, so $(I+H_0)\partial_{\alpha}^{n+1}\bar{\rho}$ is the boundary value of $\partial_z \Phi_2$. Notice that $\partial_z \Phi_2(x+iy,t)\rightarrow 0$ as $y\rightarrow -\infty$. Therefore, 
$$\Big((I+H_0)\partial_t\mathcal{R}^{(n)}\Big) \Big((I+H_0)\partial_{\alpha}^{n+1}\bar{\rho}\Big)$$
is the boundary value of a holomorphic function $\Phi_3(x+iy,t)$, with 
\begin{equation}
\Phi_3(x+iy,t)\rightarrow 0, \quad \text{as}~y\rightarrow -\infty.
\end{equation}
Applying Cauchy's theorem, one has
\begin{equation}
    L_1=0.
\end{equation}
Therefore we conclude that 
\begin{equation}
    |\mathfrak{E}_3|\leq C\epsilon^2 \norm{D_t\rho}_{H^s(q\mathbb{T})}^2+C\epsilon^2 \norm{\partial_{\alpha}\rho}_{H^s(q\mathbb{T})}^2.
\end{equation}

\subsubsection{Estimate $\mathfrak{E}_4$} Integration by parts, the estimate for $\mathfrak{E}_4$ is similar to that for $\mathfrak{E}_3$. Hence we obtain
\begin{equation}
    |\mathfrak{E}_4|\leq C\epsilon^2 \norm{D_t\rho}_{H^s(q\mathbb{T})}^2+C\epsilon^2 \norm{\partial_{\alpha}\rho}_{H^s(q\mathbb{T})}^2.
\end{equation}

\subsubsection{Estimate $\mathfrak{E}_5$} Using (\ref{estimates:mathcalRnpartialalpha}) and (\ref{estimates:mathcalRn}), we obtain
\begin{equation}
    |\mathfrak{E}_5|\leq C\epsilon^2 \norm{D_t\rho}_{H^s(q\mathbb{T})}^2+C\epsilon^2 \norm{\partial_{\alpha}\rho}_{H^s(q\mathbb{T})}^2.
\end{equation}

\subsection{Estimate for $\sum_{n=0}^s\frac{d\mathcal{E}_n}{dt}$}
Putting all computations from previous subsections together, we conclude that
\begin{equation}
   \sum_{n=0}^s \frac{d\mathcal{E}_n}{dt}\leq C\epsilon^{5}\delta^2 e^{2\epsilon^2 t}+C\epsilon^2 \norm{D_t\rho}_{H^s(q\mathbb{T})}^2+C\epsilon^2 \norm{\partial_{\alpha}\rho}_{H^s(q\mathbb{T})}^2.
\end{equation}

\subsection{Estimate $ \mathcal{C}_{2,n}$}
Recall that $\mathcal{C}_{2,n}=\partial_{\alpha}^n (D_t^2-iA\partial_{\alpha})\sigma+[D_t^2-iA\partial_{\alpha}, \partial_{\alpha}^n]\sigma$. First we estimate $\partial_{\alpha}^n (D_t^2-iA\partial_{\alpha})\sigma$. Recall that $(D_t^2-iA\partial_{\alpha})\sigma=\sum_{j=1}^n M_j$. By (\ref{sumarize}), we have 
\begin{equation}
    \sum_{n=0}^s \norm{\partial_{\alpha}^n (D_t^2-iA\partial_{\alpha})\sigma}_{L^2}\leq  C\epsilon^{7/2}\delta e^{\epsilon^2 t}+C\epsilon^2E_s^{1/2}(t).
\end{equation}
We estimate $[D_t^2-iA\partial_{\alpha}, \partial_{\alpha}^n]\sigma$ by a way similar to $[D_t^2-iA\partial_{\alpha}, \partial_{\alpha}^n]\rho$ in \S \ref{estimates:commutator:rho}. We obtain
\begin{equation}
    \norm{[D_t^2-iA\partial_{\alpha}, \partial_{\alpha}^n]\sigma}_{H^s(q\mathbb{T})}\leq C\epsilon^2 \norm{D_t\rho}_{H^s(q\mathbb{T})}^2+C\epsilon^2 \norm{D_t^2\rho}_{H^s(q\mathbb{T})}^2+C\epsilon^2 \norm{\partial_{\alpha}\rho}_{H^s(q\mathbb{T})}^2.
\end{equation}
Therefore, we conclude that
\begin{equation}
\begin{split}
 |\mathfrak{F}_1|=&\Big|\int_{-q\pi}^{q\pi} \frac{2}{A}\Re(D_t\sigma^{(n)} \bar{\mathcal{C}}_{2,n})\Big|\\
 \leq & C\epsilon^{5}\delta^2 e^{2\epsilon^2 t}+C\epsilon^2 \norm{D_t\rho}_{H^s(q\mathbb{T})}^2+C\epsilon^2 \norm{D_t^2\rho}_{H^s(q\mathbb{T})}^2+C\epsilon^2 \norm{\partial_{\alpha}\rho}_{H^s(q\mathbb{T})}^2.
 \end{split}
\end{equation}

\subsection{Estimate $\mathfrak{F}_2$}
By direct computations, we have 
\begin{equation}
\begin{split}
    |\mathfrak{F}_2|=&\Big| \int_{-q\pi}^{q\pi} \frac{1}{A}\frac{a_t}{a}\circ \kappa^{-1} |D_t\sigma^{(n)}|^2\, d\alpha\Big|\\
    \leq & C\epsilon^2 \norm{D_t\rho}_{H^s(q\mathbb{T})}^2+C\epsilon^2 \norm{D_t^2\rho}_{H^s(q\mathbb{T})}^2+C\epsilon^2 \norm{\partial_{\alpha}\rho}_{H^s(q\mathbb{T})}^2.
    \end{split}
\end{equation}
\subsection{Conclude the proof of Proposition \ref{enegyproposition} }
From  computations in previous subsections, we have that for $t\in [0,T_0]$,
\begin{equation}
\begin{split}
   &\sum_{n=0}^s \frac{d}{dt}(\mathcal{E}_n(t)+\mathcal{F}_n(t))\\
   \leq & C\epsilon^{5}\delta^2 e^{2\epsilon^2 t}+C\epsilon^2 \norm{D_t\rho}_{H^s(q\mathbb{T})}^2+C\epsilon^2 \norm{D_t^2\rho}_{H^s(q\mathbb{T})}^2+C\epsilon^2 \norm{\partial_{\alpha}\rho}_{H^s(q\mathbb{T})}^2+C\epsilon^2 E_s(t).
   \end{split}
\end{equation}
By Lemma \ref{equivalencequantities} and the a priori assumption (\ref{boot}) and (\ref{aprioritwo}), one has
\begin{equation}
    \norm{D_t\rho-2D_tr}_{H^s(q\mathbb{T})}\leq C\epsilon E(t)^{1/2}+C\epsilon^{5/2}\delta e^{\epsilon^2 t}\leq C\epsilon^{5/2}\delta e^{\epsilon^2 t},
\end{equation}
\begin{equation}
\norm{\partial_{\alpha}\rho}_{H^s(q\mathbb{T})}\leq 2\norm{\partial_{\alpha}r}_{H^s(q\mathbb{T})}+C(\epsilon E_s^{1/2}+\delta e^{\epsilon^2t}\epsilon^{5/2})\leq C\epsilon^{3/2}\delta e^{\epsilon^2 t}.
\end{equation}
By the definition of $\rho$,  one has 
\begin{equation}
\begin{split}
    \norm{D_t^2\rho}_{H^s(q\mathbb{T})}\leq & C\epsilon \norm{D_tr}_{H^s(q\mathbb{T})}+C\epsilon\norm{r_{\alpha}}_{H^s(q\mathbb{T})}+2\norm{D_t^2r}_{H^s(q\mathbb{T})}\\
\leq & C\epsilon^{3/2}\delta e^{\epsilon^2 t}.
    \end{split}
\end{equation}
Finally, we can conclude that
\begin{equation}
\frac{d}{dt}\mathcal{E}(t)\leq C\epsilon^{5}\delta^2 e^{2\epsilon^2 t},
\end{equation}
which proves Proposition \ref{enegyproposition}.

\subsection{Equivalence of $E_s(t)$ and $\mathcal{E}(t)$} We can show that under the bootstrap assumption (\ref{boot}) and the assumption \eqref{aprioritwo}, for $0\leq t\leq T_0$,
\begin{equation}\label{control:ESTbymathcalEst}
    E_s^{1/2}(t)\leq C\mathcal{E}(t)^{1/2}+C\epsilon^{5/2}\delta e^{\epsilon^2 t}.
\end{equation}
The proof  here is similar to \S 9.16 in \cite{su2020partial}, so we omit the details.

\subsection{Control the growth of the error term}
By Proposition \ref{enegyproposition} and (\ref{control:ESTbymathcalEst}), we obtain the following energy estimate.
\begin{proposition}\label{prop:energycontrol} With notations above, under the bootstrap assumption (\ref{boot}) and the assumption \eqref{aprioritwo}, for $0\leq t\leq T_0$, we have 
    \begin{equation} 
        E_s(t)\leq E_s(0)+\int_0^t C\epsilon^{5} \delta^2 e^{2\epsilon^2 \tau}\, d\tau.
    \end{equation}
In particular, a direct computation of the time integral gives
    \begin{equation}\label{eq:bootlast} 
        E_s(t)\leq E_s(0)+C\epsilon^3 \delta^2 (e^{\epsilon^2 t}-1).
    \end{equation}for $0\leq t\leq T_0$.
\end{proposition}

\section{Modulational instability of the Stokes waves}\label{section:proof}
In this  final section, we prove the nonlinear modulational instability of the Stokes wave. To achieve this goal, the essential step is to use the energy estimates obtained from the previous section to establish the long-time existence of the reminder term. 



	\subsection{Existence of initial data with desired properties}
Notice that the approximate solution $\zeta_{app}$, \eqref{eq:tildeapproximate},  we obtained in Section \ref{section:multiscale} could not be taken as the initial data for the system \eqref{system_newvariables} since it does not satisfy the holomorphic conditions in the equation.

We should first construct initial data to the water wave system \eqref{system_newvariables} for long-wave perturbations of the Stokes wave $\zeta_{ST}$.

	\begin{proposition}\label{proposition:constructinitialdata}
		Suppose that $B_0\in H^{s'}(q_1\mathbb{T})$, and $\|B_0-i\|_{H^{s'}(q_1\mathbb{T})}\leq C_0$ for some absolute constant $C_0$. In addition, we assume that
\begin{equation}\label{eq:orthoB}
    \int_{-q\pi}^{q\pi} B_0(\epsilon \alpha) e^{i\alpha}\,d\alpha=0
\end{equation}
Then 
there exist $\zeta_0$ and $v_0$ such that $\zeta_0-\alpha\in H^{s+1}(q\mathbb{T})$ and $v_0\in H^{s+1}(q\mathbb{T})$ satisfying
		\begin{equation}
		(I-\mathcal{H}_{\zeta_0})(\bar{\zeta}_0-\alpha)=0,
		\end{equation}
		\begin{equation}
		    (I-\mathcal{H}_{\zeta_0})\bar{v}_0=0, 
		\end{equation}
		and				\begin{equation}\label{almostNLSpacket}
		\norm{\zeta_0-\zeta_{ST}(\alpha,0)-\epsilon (B(\epsilon\alpha)-i)e^{i\alpha}}_{H^{s+1}(q\mathbb{T})}\leq C\epsilon^{3/2}\|B_0(\alpha)-i\|_{H^{s'}(q_1\mathbb{T})}.
		\end{equation}
		
		\begin{equation}
		    \|v_0-D_t^{ST}\zeta_{ST}(\cdot,0)-i\omega\epsilon (B(\epsilon\alpha)-i)e^{i\alpha}\|_{H^{s+1}(q\mathbb{T})}\leq C\epsilon^{3/2}\|B_0(\alpha)-i\|_{H^{s'}(q_1\mathbb{T})},
		\end{equation}
		for some constant $C$ depending on $C_0$ and $s$.
		
	\end{proposition}
	\begin{proof}
The initial data here is obtained by iteration.	We start by defining $$\zeta_1:=\zeta_{ST}(\alpha,0)+\epsilon (B_0(\epsilon\alpha)-i)e^{i\alpha}.$$ Let $\mathcal{H}_{\zeta_1}$ be the Hilbert transform associated with $\zeta_1$. Define $\zeta_2$ by 
		\begin{equation}
		\bar{\zeta}_2(\alpha)-\alpha:=\frac{1}{2}(I+\mathcal{H}_{\zeta_1})\Big(\bar{\zeta}_{ST}(\cdot,0)-\alpha+\epsilon(\bar{B}_0(\epsilon\alpha)+i)e^{-i\alpha}\Big).
		\end{equation}
		Assuming that $\zeta_n$ has been constructed, we define
		\begin{equation}
	\bar{\zeta}_{n+1}(\alpha)-\alpha:=\frac{1}{2}(I+\mathcal{H}_{\zeta_n})\Big(\bar{\zeta}_{ST}(\cdot,0)-\alpha+\epsilon(\bar{B}_0(\epsilon\alpha)+i)e^{-i\alpha}\Big).
		\end{equation}
	Since 
	\begin{equation}
	    (I-\mathcal{H}_{\zeta_{ST}(\cdot,0)})(\bar{\zeta}_{ST}(\cdot,0)-\alpha)=0,
	\end{equation}
	and by Lemma \ref{lemma: slowvaryingalmostholomorphic},
	\begin{equation}
	    \norm{(I-H_0)(\bar{B}_0(\epsilon\alpha)+i)e^{-i\alpha}}_{H^{s+1}(q\mathbb{T})}\leq C\epsilon^{7/2}\norm{B_0(\alpha)-i}_{H^{s'}(q_1\mathbb{T})},
	\end{equation}
	one has
	\begin{equation}\small
	\begin{split}
	   & \norm{\bar{\zeta}_2-\bar{\zeta}_1}_{H^{s+1}(q\mathbb{T})}\\
	   \leq & \frac{1}{2}\norm{(\mathcal{H}_{\zeta_1}-\mathcal{H}_{\zeta_{ST}(\cdot,0)}(\bar{\zeta}_{ST}(\cdot,0)-\alpha)}_{H^{s+1}(q\mathbb{T})}+\frac{\epsilon}{2}\norm{(\mathcal{H}_{\zeta_1}-H_0)(\bar{B}_0(\epsilon\alpha)+i)e^{-i\alpha}}_{H^{s+1}(q\mathbb{T})}\\
	   &+\frac{\epsilon}{2}\norm{(I-H_0)(\bar{B}_0(\epsilon\alpha)+i)e^{-i\alpha}}_{H^{s+1}(q\mathbb{T})}\\
	    \leq & C\norm{\partial_{\alpha}(\zeta_1-\zeta_{ST}(\cdot,0))}_{H^{s+1}(q\mathbb{T})}\norm{\zeta_{ST}(\cdot,0)-\alpha}_{W^{s,\infty}(q\mathbb{T})}+C(\epsilon^{3/2}+ \epsilon^{7/2})\norm{B_0(\alpha)-i}_{H^{s'}(q_1\mathbb{T})}\\
	    \leq & C\epsilon^{3/2}\norm{B_0(\alpha)-i}_{H^{s'}(q_1\mathbb{T})}.
	    \end{split}
	\end{equation}
				It is straightforward to obtain
		\begin{equation}
		\|\zeta_{n+1}-\zeta_n\|_{H^{s+1}(q\mathbb{T})}\leq C\epsilon \|\zeta_n-\zeta_{n-1}\|_{H^{s+1}(q\mathbb{T})}.
		\end{equation}
		Therefore, after applying the fixed point theorem, $\zeta_{n+1}-\alpha$ converges in $H^{s+1}(q\mathbb{T})$ to a function $\xi\in H^{s+1}(q\mathbb{T})$. Define 
		$\gamma(\alpha)$ as
		\begin{equation}
		    \gamma:=\alpha+\xi(\alpha,0).
		\end{equation}
		By the definition of $\gamma$ and the iteration procedure above, there hold
			\begin{equation}\label{ineq:gamma0NoB}
		    \norm{\gamma-\zeta_{ST}(\cdot,0)}_{H^{s+1}(q\mathbb{T})}\leq C\epsilon^{1/2}\norm{B_0(\alpha)-i}_{H^{s'}(q_1\mathbb{T})}.
		\end{equation}
		
		\begin{equation}\label{ineq:gamma0}
		    \norm{\gamma-\Big(\zeta_{ST}(\cdot,0)+\epsilon(B_0(\epsilon\alpha)-i)e^{i\alpha}\Big)}_{H^{s+1}(q\mathbb{T})}\leq C\epsilon^{3/2}\norm{B_0(\alpha)-i}_{H^{s'}(q_1\mathbb{T})}.
		\end{equation}
	One has
		\begin{equation}
		    \bar{\gamma}-\alpha=\frac{1}{2}(I+\mathcal{H}_{\gamma})\Big(\bar{\zeta}_{ST}(\cdot,0)-\alpha+\epsilon(\bar{B}_0(\epsilon\alpha)+i)e^{-i\alpha}\Big).
		\end{equation}
		By Lemma \ref{boundaryvalueofcauchyintegral}, we know that
	\begin{equation}
	    (I-\mathcal{H}_{\gamma})\Big(\bar{\gamma}(\alpha)-\alpha\Big)=c_{\gamma},
	\end{equation}
	where 
	$$c_{\gamma}:=\frac{1}{2q\pi}\int_{-q\pi}^{q\pi}\partial_{\alpha}\gamma(\alpha)\Big(\bar{\zeta}_{ST}(\alpha,0)-\alpha+\epsilon(\bar{B}_0(\epsilon\alpha)+i)e^{-i\alpha}\Big)\,d\alpha.$$
	Note that $\int_{-q\pi}^{q\pi}\partial_{\alpha}\zeta_{ST}(\alpha,0) (\bar{\zeta}_{ST}(\alpha,0)-\alpha)\,d\alpha=0$ because $(I-\mathcal{H}_{\zeta_{ST}(\cdot,0)})(\bar{\zeta}_{ST}(\alpha,0)-\alpha)=0$. So one has 
	\begin{align*}
	    |c_{\gamma}|\leq \Big| \frac{1}{2q\pi}\int_{-q\pi}^{q\pi}\partial_{\alpha}(\gamma-\zeta_{ST})(\bar{\zeta}_{ST}(\alpha,0)-\alpha)d\alpha\Big|+\Big|\frac{1}{2q\pi}\int_{-q\pi}^{q\pi}\epsilon(\bar{B}_0(\epsilon\alpha)+i)e^{-i\alpha}d\alpha\Big|:= c_{\gamma,1}+c_{\gamma,2}.
	\end{align*}
	By \eqref{ineq:gamma0}, we have 
	\begin{equation}
	    c_{\gamma,1}\leq Cq^{-1}\norm{\partial_{\alpha}(\gamma-\zeta_{ST}(\alpha,0)}_{L^2(q\mathbb{T})}\norm{\partial_{\alpha}(\bar{\zeta}_{ST}(\alpha,0)-\alpha)}_{L^2(q\mathbb{T})}\leq C\epsilon^2\norm{B_0(\alpha)-i}_{H^{s'}(q_1\mathbb{T})}.
	\end{equation}
By the orthogonal condition \eqref{eq:orthoB}, one has
	\begin{equation}
	    c_{\gamma 2}=0.
	\end{equation}
	Define $\zeta_0(\alpha):=\gamma(\alpha)-c_{\gamma}$. Noting that $\mathcal{H}_{\zeta_0}=\mathcal{H}_{\gamma}$,  we have 
	\begin{align*}
	    &(I-\mathcal{H}_{\zeta_0})(\bar{\zeta}_0-\alpha)=0.
	\end{align*}
Moreover,
\begin{equation}\label{inequality:zeta_0}
    \norm{\zeta_0-\zeta_{ST}(\alpha,0)}_{H^{s+1}(q\mathbb{T})}\leq C\epsilon^{1/2}\norm{B_0(\alpha)-i}_{H^{s'}(q_1\mathbb{T})},
\end{equation}
and
\begin{equation}
		\norm{\zeta_0-\zeta_{ST}(\alpha,0)-\epsilon (B(\epsilon\alpha)-i)e^{i\alpha}}_{H^{s'}(q\mathbb{T})}\leq C\epsilon^{3/2}\|B_0(\alpha)-i\|_{H^{s'}(q_1\mathbb{T})}.
		\end{equation}
Define $v_0$ as
\begin{equation}
    \bar{v}_0(\alpha)=\frac{1}{2}(I+\mathcal{H}_{\zeta_0})\Big(D_t^{ST}\bar{\zeta}_{ST}(\alpha,0)-i\omega\epsilon (\bar{B}_0+i)e^{-i\alpha}\Big)+d_{v_0},
\end{equation}
where
\begin{equation}
    d_{v_0}:=-\frac{1}{4q\pi}\int_{-q\pi}^{q\pi}\partial_{\beta}\zeta_{0}\Big(D_t^{ST}\bar{\zeta}_{ST}(\beta,0)-i\omega\epsilon (\bar{B}_0+i)e^{-i\beta}\Big)\,d\beta
\end{equation}
By Corollary \ref{corollary:hilbertboundary}, we have 
\begin{equation}
    (I-\mathcal{H}_{\zeta_0})\bar{v}_0=0.
\end{equation}
Note that we can rewrite
\begin{align*}
   & \frac{1}{2}(I+\mathcal{H}_{\zeta_0})D_t^{ST}\bar{\zeta}_{ST}(\alpha,0)-\frac{1}{4q\pi}\int_{-q\pi}^{q\pi}\zeta_{\beta}D_t\bar{\zeta}_{ST}(\beta,0)\,d\beta\\
   =& \frac{1}{2}(I+\mathcal{H}_{\zeta_{ST}(\cdot,0)})D_t^{ST}\bar{\zeta}_{ST}(\alpha,0)-\frac{1}{4q\pi}\int_{-q\pi}^{q\pi}\partial_{\beta}\zeta_{ST}(\beta,0) D_t^{ST}\bar{\zeta}_{ST}(\beta,0)\,d\beta\\
   &+ \frac{1}{2}(\mathcal{H}_{\zeta_0}-\mathcal{H}_{\zeta_{ST}(\cdot,0)})D_t^{ST}\bar{\zeta}_{ST}(\cdot,0)-\frac{1}{4q\pi}\int_{-q\pi}^{q\pi}\partial_{\beta}(\zeta_0(\beta)-\zeta_{ST}(\beta,0))D_t^{ST}\bar{\zeta}_{ST}(\beta,0)\,d\beta
\end{align*}
Since $D_t^{ST}\bar{\zeta}_{ST}$ is holomorphic and vanishes as $y\rightarrow-\infty$, we have 
\begin{equation}
    \frac{1}{4q\pi}\int_{-q\pi}^{q\pi}\partial_{\beta}\zeta_{ST} D_t^{ST}\bar{\zeta}_{ST}(\beta,0)\,d\beta=0,
\end{equation}
and
\begin{equation}
    D_t^{ST}\bar{\zeta}_{ST}(\cdot,0)=\frac{1}{2}(I+\mathcal{H}_{\zeta_{ST}(\cdot,0)})D_t^{ST}\bar{\zeta}_{ST}.
\end{equation}
Therefore, by \eqref{inequality:zeta_0},
\begin{align*}
    &\norm{\frac{1}{2}(I+\mathcal{H}_{\zeta_0})D_t^{ST}\bar{\zeta}_{ST}(\alpha,0)-\frac{1}{4q\pi}\int_{-q\pi}^{q\pi}\zeta_{\beta}D_t\bar{\zeta}_{ST}(\beta,0)\,d\beta-D_t^{ST}\bar{\zeta}_{ST}(\alpha,0)}_{H^{s+1}(q\mathbb{T})}\\
    \leq & \norm{\frac{1}{2}(\mathcal{H}_{\zeta_0}-\mathcal{H}_{\zeta_{ST}(\cdot,0)})D_t^{ST}\bar{\zeta}_{ST}(\cdot,0)}_{H^{s'}(q\mathbb{T})}+\norm{\frac{1}{4q\pi}\int_{-q\pi}^{q\pi}\partial_{\beta}(\zeta_0-\zeta_{ST}(\beta,0))D_t^{ST}\bar{\zeta}_{ST}(\beta,0)\,d\beta}_{H^{s'}(q\mathbb{T})}\\
    \leq & C\epsilon^{3/2}\norm{B_0-i}_{H^{s'}(q_1\mathbb{T})}.
\end{align*}
By Lemma \ref{lemma: slowvaryingalmostholomorphic}, we can bound
\begin{align*}
    &\norm{\Big(-\frac{1}{2}i\omega \epsilon(I+H_0)(\bar{B}_0+i)e^{-i\alpha}\Big)-\Big(-i\omega \epsilon(\bar{B}_0+i)e^{-i\alpha}\Big)}_{H^{s+1}(q\mathbb{T})}\\
    \leq & C\epsilon^{7/2}\norm{B_0-i}_{H^{s'}(q_1\mathbb{T})}.
\end{align*}
Therefore, it follows that
\begin{align*}
        &\norm{\Big(-\frac{1}{2}i\omega \epsilon(I+\mathcal{H}_{\zeta_0})(\bar{B}_0+i)e^{-i\alpha}\Big)-\Big(-i\omega \epsilon(\bar{B}_0+i)e^{-i\alpha}\Big)}_{H^{s+1}(q'\mathbb{T})}\\
    \leq & C\epsilon^{7/2}\norm{B_0(\alpha)-i}_{H^{s'}(q_1\mathbb{T})}+\frac{1}{2}\omega \epsilon \norm{(\mathcal{H}_{\zeta_0}-H_0)(B_0(\epsilon\alpha)-i)}_{H^{s+1}(q\mathbb{T})}\\
    \leq & C\epsilon^{3/2}\norm{B_0(\alpha)-i}_{H^{s'}(q_1\mathbb{T})}.
\end{align*}
Also, using $q=\epsilon^{-1}q_1$, we split
\begin{align*}
   & \frac{i\omega\epsilon}{4q\pi}\int_{-q\pi}^{q\pi}\partial_{\alpha}\zeta_{0} (\bar{B}_0+i)e^{-i\alpha}\,d\alpha\\
   \leq &\frac{\omega\epsilon}{4q\pi}\int_{-q\pi}^{q\pi} \Big|(\partial_{\alpha}\zeta_0-1)(B_0(\epsilon\alpha)-i)\Big| \,d\alpha+\Big|\frac{\omega\epsilon}{4q\pi}\int_{-q\pi}^{q\pi}(\bar{B}_0(\epsilon\alpha)+i)e^{-i\alpha} \,d\alpha\Big|\\
    :=& I+II.
\end{align*}
For $I$, we have 
\begin{equation}
    I\leq C\epsilon^{3/2}\norm{B_0(\alpha)-i}_{H^{s'}(q_1\mathbb{T})}.
\end{equation}
For $II$, by the orthogonality condition \eqref{eq:orthoB},
one has
\begin{equation}
    II=0.
\end{equation}
By (\ref{almostNLSpacket}), we finally conclude that 
\begin{align*}
    \|\bar{v}_0-D_t^{ST}\bar{\zeta}_{ST}(\cdot,0)+i\omega\epsilon (\bar{B}_0(\epsilon\alpha)+i)e^{-i\alpha}\|_{H^{s+1}(q\mathbb{T})}
\leq & C\epsilon^{3/2}\|B_0(\alpha)-i\|_{H^{s'}(q_1\mathbb{T})}.
\end{align*}
We are done.
	\end{proof}

\begin{rem}
Indeed, by the asymptotic analysis in \S \ref{section:multiscale}, the initial data $(\zeta_0, v_0)$ constructed in Proposition \ref{proposition:constructinitialdata} satisfies 
\begin{equation}
    \norm{(\partial_{\alpha}\zeta_0-1, v_0)}_{H^s(q\mathbb{T})\times H^{s+1/2}(q\mathbb{T})}\leq C\epsilon^{7/2}\norm{B_0(\alpha)-i}_{H^{s'}(q_1\mathbb{T})}.
\end{equation}
In particular, for $\epsilon$ sufficiently small, one has 
\begin{equation}
    \norm{(\partial_{\alpha}\zeta_0-1, v_0)}_{H^s(q\mathbb{T})\times H^{s+1/2}(q\mathbb{T})}\leq \epsilon^{3/2}\norm{B_0(\alpha)-i}_{H^{s'}(q_1\mathbb{T})}.
\end{equation}
\end{rem}

\subsection{Extended lifespan}

Now we are ready to conclude the  result on the long-time existence.

Given a solution $B(X,T)\in H^{s'}(q_1\mathbb{T})$ to the NLS
\begin{equation}\label{eq:longtimenls}
iB_T+\frac{1}{8}B_{XX}=\frac{1}{2} |B|^2B-\frac{1}{2} B,
\end{equation}
we define\begin{equation}\label{eq:longtimezeta1}
    \zeta^{(1)}(\alpha,t)= B(\epsilon(\alpha+\frac{1}{2\omega}t), \epsilon^2 t) e^{i(\alpha+\omega t)}.
\end{equation}

\subsubsection{Long-time estimates}
By Theorem \ref{localperiodic},  Proposition \ref{prop:energycontrol}, Proposition \ref{proposition:constructinitialdata}, and the standard bootstrap argument, we obtain the following. 

\begin{thm}\label{thm:main1zeta}
Let $s\geq 4$ be given and $s'=s+7$. Let $\zeta_{ST}$ be a Stokes wave of period $2\pi$ and amplitude $\epsilon$. Let $0<\delta\ll1$ be an arbitrarily small but fixed number. For any given $q\in\mathbb{Q}_+$ with $q\geq\frac{1}{ \epsilon}$, and any solution $B$ to the NLS \eqref{eq:longtimenls} satisfying
\begin{equation}\label{eq:Bcondilife}
        \norm{B(\alpha,0)-i}_{H^{s'}(q_{1}\mathbb{T})}\leq \delta,\quad \quad\,q_{1}=\epsilon q
\end{equation}
and the orthogonality condition \eqref{eq:orthoB},
there exist $\zeta_0$ and $v_0$ such that $(\zeta_0-\alpha, v_0)\in H^{s+1}(q\mathbb{T})\times H^{s+1/2}(q\mathbb{T})$ and they satisfy the estimate
\begin{equation}\label{eq:estimateinitial}
    \norm{(\zeta_0, v_0)-\epsilon(\zeta^{(1)}(\cdot,0), \partial_t \zeta^{(1)}(\cdot,0))}_{H^{s+1}(q\mathbb{T})\times H^{s+1/2}(q\mathbb{T})}\leq \epsilon^{3/2}\delta
\end{equation}
where $\zeta^{(1)}$ is constructed as \eqref{eq:longtimezeta1}.
For all such data $(\zeta_0,v_0)$, the water wave system (\ref{system_boundary}) admits a unique solution $\zeta(\alpha,t)$ on $[0, \epsilon^{-2}\log\frac{\mu}{\delta}]$ with $$(\zeta_{\alpha}-1, D_t\zeta)\in C([0, \epsilon^{-2}\log\frac{\mu}{\delta}];H^s(q\mathbb{T})\times H^{s+1/2}(q\mathbb{T}))$$satisfying the following estimate
\begin{equation}\label{eq:zetaperb}
    \norm{\Big(\partial_{\alpha}\zeta(\alpha,t)-1, D_t\zeta\Big)-\epsilon\Big(\partial_{\alpha}\zeta^{(1)}, \partial_t\zeta^{(1)}\Big)}_{H^s(q\mathbb{T})\times H^{s+1/2}(q\mathbb{T})}\leq C\epsilon^{3/2}\delta e^{\epsilon^2 t},
\end{equation}
for all $t\in [0,\, \epsilon^{-2}\log\frac{\mu}{\delta}]$ where $\delta\ll\mu<1$ is a fixed number.
\end{thm}

\begin{proof}
For given $\zeta_{ST}$, $\delta$ and $B(\alpha,0)$ satisfying \eqref{eq:Bcondilife},
the existence of $(\zeta_0, v_0)$ satisfying \eqref{eq:estimateinitial} is guaranteed by Proposition \ref{proposition:constructinitialdata}.  By Proposition \ref{prop:nlsestimates}, there exists a fixed number $\mu\in (0,1)$ uniform in $\delta$ and $\epsilon$ such that for all solution $B$ to the NLS \eqref{eq:longtimenls} satisfying \eqref{eq:Bcondilife}, one has 
\begin{equation}
    \norm{B(\alpha,t)-i}_{H^{s'}(q_1\mathbb{T})}\leq 2\delta e^{t}, \quad t\in [0,\log\frac{\mu}{\delta}].
\end{equation}
So the existence of $B$ satisfying \eqref{aprioritwo} is guaranteed.

Assume the bootstrap assumption \eqref{boot} with constant $C>2$. Clearly, the bootstrap assumption (\ref{boot}) is satisfied at $t=0$. Using the a priori energy estimates provided in Proposition \ref{prop:energycontrol}, the estimate \eqref{eq:bootlast}, the constant appearing the bootstrap assumption is improved since $E_s(0)+C\epsilon^3 \delta^2 (e^{\epsilon^2 t}-1)<C\epsilon^3 \delta^2 e^{\epsilon^2 t}$. Therefore together with the blowup criterion \eqref{blowup:one} and \eqref{blowup:two}, we are able to use the bootstrap argument to prove that the solution $(\zeta_{\alpha}-1, D_t\zeta)\in C([0,\epsilon^{-2}\log\frac{\mu}{\delta}]; H^{s}(q\mathbb{T})\times H^{s+1/2}(q\mathbb{T}))$, and \eqref{eq:zetaperb} holds.
\end{proof}
\begin{rem}
Note that by construction, the initial data $(\zeta_0, v_0)$ is a long-wave perturbation (with fundamental period $2q\pi$) of the Stokes wave $\zeta_{ST}$.
\end{rem}

\begin{rem}
This theorem also shows the validity of the modulational approximation via NLS.  It might  be interesting to point out that due to the fact that the Stokes wave is a global solution to the water wave system, the valid time scale $\mathcal{O}(\epsilon^{-2}\log\frac{1}{\delta})$ for the modulational approximation of the perturbed flow is longer than other settings.  See for example  \cite{Totz2012,su2020partial} where the valid time scale is  of order $\mathcal{O}(\epsilon^{-2})$.
\end{rem}





\subsection{Nonlinear instability}
With the long-time existence and estimates, we now analyze the instability of Stokes waves under long-wave perturbations. 









\subsubsection{Growth of large scales}
By Theorem \ref{thm:main1zeta}, with estimates between the difference of $\zeta-\alpha$ and $\epsilon \zeta^{(1)}$, it suffices to analyze the growth of $B$.  In this setting, from the proof of the NLS instability, Appendix \S \ref{sec:NLS},  for $B$, we can take the initial data
\begin{equation}
    B(\epsilon \alpha,0)=i+\frac{i}{\sqrt{q_1}}\left(\delta_{1}e^{i\frac{k_0}{q_1}\epsilon\alpha}+\delta_{2}e^{-i\frac{k_0}{q_1}\epsilon\alpha}+\eta_{1}e^{i\frac{\epsilon\alpha}{q_1}}+\eta_{2}e^{-i\frac{\epsilon\alpha}{q_1}}\right)
\end{equation}
where $k_0\in\mathbb{Z}^{+}$ and $\tau$ are defined as
\begin{equation}\label{eq:k0taudef}
\left|\frac{k_0}{q_1}\right|\sqrt{2-\left|\frac{k_0}{q_1}\right|^{2}}=\tau:=\sup_{k\in\mathbb{Z}}\Re\left|\frac{k}{q_1}\right|\sqrt{2-\left|\frac{k}{q_1}\right|^{2}}
\end{equation}
and
 $$\left|\delta_{j}\right|=\frac{\delta}{2s'}\ll1,\,\left|\eta_{j}\right|\ll\left|\delta_{i}\right|$$ 
 with $\delta$ is given as Theorem \ref{thm:main1zeta}. Note that by construction, \eqref{eq:k0taudef}, $\tau\leq1$ and $k_0\sim q_1$. 
 Clearly, $B(\cdot,0)$ satisfies the orthogonality condition \eqref{eq:orthoB}.

Furthermore, with initial data above, the solution $B$  can be written as
\begin{equation}\label{eq:instabilityB}
    B(\epsilon(\alpha+\frac{1}{2\omega}t), \epsilon^2 t)
    =i+\frac{i}{\sqrt{q_1}}\left(a_1(t)e^{i\epsilon\frac{k_0}{q_1}(\alpha+\frac{1}{2\omega} t)}+a_2(t)e^{-i\epsilon\frac{k_0}{q_1}(\alpha+\frac{1}{2\omega} t)}+e(X,T)\right),
\end{equation}
where $X=\epsilon(\alpha+\frac{1}{2\omega}t),\,T=\epsilon^2t$   and $ a_j(t)=a_j(0)e^{\tau \epsilon^2},\, j=1, 2$
Then clearly by construction, the $B$ above satisfies the perturbation condition \eqref{eq:Bcondilife}. 
Denoting
\begin{equation}
    w^{L}(\alpha,t):=\frac{i}{\sqrt{q_1}}\Big( a_1(t)e^{i\epsilon\frac{k_0}{q_1}(\alpha+\frac{1}{2\omega} t)}+a_2(t)e^{i\epsilon\frac{k_0}{q_1}(\alpha+\frac{1}{2\omega} t)}\Big)
\end{equation}
then $e(X,T)$ satisfies additional estimate: for $t\in [0, \epsilon^{-2}\log\frac{\mu}{\delta}]$
\begin{equation}\label{inequality:almostlinear}
   \norm{\frac{1}{\sqrt{q_1}}e(X,T)}_{H^{s'}(q\mathbb{T})}\leq \frac{1}{2}\norm{w^{L}(\cdot,t)}_{H^{s'}(q\mathbb{T})}.
\end{equation}
Moreover, at $t_*=\epsilon^{-2}\log\frac{\mu}{\delta}$, one has
\begin{equation}\label{eq:wLinstability}
       \norm{w^{L}(\cdot,t_*)}_{H^{s'}(q\mathbb{T})}\geq c
\end{equation}
for some constant $c\gg \delta$.

\begin{rem}
See Subsection \S \ref{subsec:nonlinear} from Appendix \S \ref{sec:NLS} for details.
\end{rem}

The instability mechanism above is precisely the deriving force of the instability of the Stokes wave.

\subsubsection{The modulational instability}
Let $\gamma\in (0,\epsilon_0)$, and $\phi\geq 0$ be given. We use $\zeta_{ST}^{\gamma,\phi}$ to denote the Stokes wave with period $2\pi$, the amplitude $\gamma$, and the phase translation $\phi$.  The following result  gives the  nonlinear instability of Stokes wave.
\begin{cor}\label{cor:zetainstbality}
Let $\zeta$ be the solution as given in Theorem \ref{thm:main1zeta} with $B$ constructed as \eqref{eq:instabilityB}. Then we have
\begin{equation}\label{eq:L2instability}
    \sup_{t\in [0,\,\epsilon^{-2}\log\frac{\mu}{\delta}]}\inf_{\phi\in \mathbb{T}}\inf_{\gamma\in (0,\epsilon_0)}\norm{\Im\{ \zeta-\zeta_{ST}^{\gamma,\phi}\}}_{L^2(q\mathbb{T})}\geq c\epsilon^{1/2},
\end{equation}
for some constant $c>0$ which is uniform in $\delta$  and $\epsilon$.
\end{cor}
In particular, we conclude that the Stokes wave $\zeta_{ST}$ given in Theorem \ref{thm:main1zeta} is modulationally unstable under the long-wave perturbation in $H^s(q\mathbb{T})$.
\begin{proof}
First of all, clearly, by construction, the solution $\zeta$ given in Theorem \ref{thm:main1zeta} is a long-wave perturbation of the Stokes wave $\zeta_{ST}$  in the same theorem.

We first prove that
\begin{equation}\label{eq:H1dotinstablity}
        \sup_{t\in [0,\,\epsilon^{-2}\log\frac{\mu}{\delta}]}\inf_{\phi\in \mathbb{T}}\inf_{\gamma\in (0,\epsilon_0)}\norm{      \Im\{ \zeta-\zeta_{ST}^{\gamma,\phi}
        \}}_{\dot{H}^1 (q\mathbb{T})}\geq c\epsilon^{1/2}.
\end{equation}
Recalling that for a Stokes wave of period $2\pi$ and amplitude $\gamma$ and phase shift $\phi$, we have the asymptotic expansion
\begin{equation}\label{eq:Zexpansion}
    \zeta_{ST}^{\gamma,\phi}-\alpha=i\gamma e^{i(\alpha-\phi)+i\omega_{\gamma} t}+O(\gamma^2), \quad \quad \omega_{\gamma}:=1+\frac{\gamma^2}{2}+O(\gamma^3).
\end{equation}
By estimate \eqref{eq:estimateinitial} from Theorem \ref{thm:main1zeta}, to prove \eqref{eq:H1dotinstablity}, it suffices to show
\begin{equation}
   \sup_{t\in [0,\,\epsilon^{-2}\log\frac{\mu}{\delta}]}\inf_{\phi\in [0,\infty)}\inf_{\gamma\in (0,\epsilon_0)} \norm{ \Im\{\epsilon \zeta^{(1)}-\zeta_{ST}^{\gamma,\phi}\}}_{\dot{H}^1(q\mathbb{T})}\geq 2c\epsilon^{1/2}.
\end{equation}
From the explicit formula of $\zeta^{(1)}$, it suffices to prove  
\begin{equation}\label{estimates:unstable}
   \sup_{t\in [0,\,\epsilon^{-2}\log\frac{\mu}{\delta}]}\inf_{\phi\in [0,\infty)}\inf_{\gamma\in (0,\epsilon_0)} \norm{ \Im\{\epsilon w^Le^{i\alpha+i\omega t}+i\epsilon e^{i\alpha+i\omega t}-\zeta_{ST}^{\gamma,\phi}\}}_{\dot{H}^1(q\mathbb{T})}\geq 2c\epsilon^{1/2}.
\end{equation}
From the leading order term of the expansion \eqref{eq:Zexpansion}, we note that $\omega^L e^{ik\alpha+i\omega t}$ is orthogonal to the leading order term of $ie^{ik\alpha+i\omega t}-\zeta_{ST}^{\gamma,\phi}$. This orthogonality and (\ref{eq:wLinstability}) together imply (\ref{estimates:unstable}) from which \eqref{eq:H1dotinstablity} follows.

Finally, notice that  the leading order term of $\Im\{ \zeta-\zeta_{ST}^{\gamma,\phi}\}$ is given by $$    \Im\{\epsilon w^Le^{i\alpha+i\omega t}+i\epsilon e^{i\alpha+i\omega t}-i\gamma e^{i(\alpha-\phi)+i\omega_{\gamma} t}\}$$
whose  Fourier modes are supported around $1$.
Therefore, the estimate \eqref{eq:H1dotinstablity}  implies  \eqref{eq:L2instability}.
\end{proof}

\begin{rem}
From our explicit construction $w^L$, the instability here also holds pointwisely.
\end{rem}

\begin{rem}
Here, we just picked one special $w^L$ and \eqref{eq:instabilityB}. We should point out that there are plenty of choices to construct the unstable perturbations since each unstable solution from the scale of NLS can produce a corresponding unstable solution for the water wave system.
\end{rem}
\subsubsection{Nonlinear modulational instability in Eulerian coordinates} We can translate the instability estimate \eqref{eq:L2instability} from Corollary \ref{cor:zetainstbality} to  Eulerian coordinates.

We take $\zeta$ from Corollary \ref{cor:zetainstbality}. Then denote
\begin{equation}
x(\alpha,t):=\Re\{\zeta(\alpha,t)\}.
\end{equation}
Since $\norm{\zeta_{\alpha}-1}_{H^{s'}(q\mathbb{T})}\leq C\epsilon^{1/2}$, $x(\cdot,t): \mathbb{R}\rightarrow \mathbb{R}$ defines a diffeomorphism. We can find the inverse of $x(\cdot,t)$ as $\alpha(\cdot,t): \mathbb{R}\rightarrow \mathbb{R}$ satisfying 
\begin{equation}\label{eq:diffalpha}
    \sup_{t\in [0,\,\epsilon^{-2}\log\frac{\mu}{\delta}]}\norm{\alpha_x(\cdot,t)-1}_{H^{s}(q\mathbb{T})}\leq C\epsilon^{1/2}.
\end{equation}
In Eulearian coordinates, the elevation of the perturbed flow is given by
\begin{equation}
    \eta(x,t):=\Im \{\zeta(\alpha(x,t),t)\}.
\end{equation}
and the elevation of Stokes waves is defined as
\begin{equation}
    \eta_{ST}^{\gamma,\phi}(x,t):=\Im\{\zeta_{ST}^{\gamma,\phi}(\alpha_{ST}(x,t),t)\}.
\end{equation}
\begin{cor}With notations above, we have
\begin{equation}\label{eq:L2eulerinsta}
    \sup_{t\in [0,\,\epsilon^{-2}\log\frac{\mu}{\delta}]}\inf_{\phi\in \mathbb{T}}\inf_{\gamma\in (0,\epsilon_0)}
    \norm{\eta(\cdot,t)-\eta_{ST}^{\gamma,\phi}(\cdot,t)}_{L^2(q\mathbb{T})}\geq c\epsilon^{1/2},
\end{equation}
for some constant $c>0$ which is uniform in $\delta$ and $\epsilon$.
\end{cor}
\begin{proof}
We again start with the instability in $\dot{H}^1(q_1\pi)$. We claim that 
\begin{equation}\label{eq:Hdot1eulerinsta}
    \sup_{t\in [0,\,\epsilon^{-2}\log\frac{\mu}{\delta}]}\inf_{\phi\in \mathbb{T}}\inf_{\gamma\in (0,\epsilon_0)}
    \norm{\eta(\cdot,t)-\eta_{ST}^{\gamma,\phi}(\cdot,t)}_{\dot{H}^1(q\mathbb{T})}\geq c\epsilon^{1/2},
\end{equation}

Given the $\zeta$ from Corollary \ref{cor:zetainstbality} and a fixed Stokes wave $\zeta_{ST}^{\gamma,\phi}$, we first observe that from estimate \eqref{eq:estimateinitial} and the construction of $w^L$, we know that when $t\sim \epsilon^{-2}\log\frac{\mu}{\delta}$, $$\norm{\partial_x\Im\{\zeta(t)\}}_{L^2(q\mathbb{T})}\sim\epsilon^{1/2}.$$
On the other hand $\norm{\partial_x \Im\{\zeta^{\gamma,\phi}_{ST}(t)\}}_{L^2(q\mathbb{T})}\sim \gamma$.  Therefore, when $\gamma\ll\epsilon$ or $\gamma\gg\epsilon$,  then clearly, \eqref{eq:Hdot1eulerinsta} holds.  

It remains to consider $\gamma\sim\epsilon$. In this case we write
\begin{equation}
    \eta(\cdot,t)-\eta_{ST}^{\gamma,\phi}(\cdot,t)=\Im\{\zeta(\alpha(\cdot,t),t)-\zeta_{ST}^{\gamma,\phi}(\alpha(\cdot,t),t)\}+\Im\{\zeta_{ST}^{\gamma,\phi}(\alpha(\cdot,t),t)-\zeta_{ST}^{\gamma,\phi}(\alpha_{ST}(\cdot,t),t)\}
\end{equation}
whose leading order terms are given
\begin{equation}\label{eq:leadingeuler}
    \Im\{\epsilon w^L(\alpha,t)e^{i\alpha+i\omega t}+i\epsilon e^{i\alpha+i\omega t}-i\gamma e^{i(\alpha-\phi)+i\omega_{\gamma} t}\}+ \Im\{i\gamma e^{i(\alpha-\phi)+i\omega_{\gamma} t}-i\gamma e^{i(\alpha_{ST}-\phi)+i\omega_{\gamma} t}\}.
\end{equation}
 Applying  $\partial_x$ to \eqref{eq:leadingeuler} and then we take the $L^2$ norm.  The lower bound for the first term
\begin{equation}\label{eq:euler1}
    \norm{\partial_x \Im\{\epsilon w^L(\alpha,t)e^{i\alpha+i\omega t}+i\epsilon e^{i\alpha+i\omega t}-i\gamma e^{i(\alpha-\phi)+i\omega_{\gamma} t}\}}_{L^2(q\mathbb{T})} \geq 2c\epsilon^{1/2}
\end{equation}follows from \eqref{estimates:unstable} after applying the change of variable $\alpha$ with \eqref{eq:diffalpha}.

For the second part,when $\gamma\sim\epsilon$, we have the upper bound
\begin{equation}\label{eq:euler2}
        \norm{\partial_x \Im\{i\gamma e^{i(\alpha-\phi)+i\omega_{\gamma} t}-i\gamma e^{i(\alpha_{ST}-\phi)+i\omega_{\gamma} t}\}}_{L^2(q\mathbb{T})} \leq C \epsilon^{\frac{3}{2}}
\end{equation}
from \eqref{eq:diffalpha} and the similar one for $\alpha_{ST}$. This is of higher order in $\epsilon$.

Therefore, taking the leading order terms and the applying a simple triangle inequality, from \eqref{eq:euler1} and \eqref{eq:euler2}, we get \eqref{eq:Hdot1eulerinsta}.

Finally, by the same argument as for Corollary \ref{cor:zetainstbality}, using the support of Fourier modes, the $L^2$ version \eqref{eq:L2eulerinsta} follows from \eqref{eq:Hdot1eulerinsta}. Indeed, we can write $\alpha$ and $\alpha_{ST}$ as  $x+\mathcal{O}(\epsilon)$.  Then we expand $e^{i\alpha}$ and $e^{i\alpha_{ST}}$ in terms of powers of $\epsilon$. At the level of $\epsilon$, the Fourier modes of \eqref{eq:leadingeuler} are supported around $1$. The remaining pieces are of higher orders of $\epsilon$.  We omit the details since it is routine. For the expansion of the Stokes wave, we also refer to the formula \eqref{eq:etaexpintro1}.



\end{proof}
	\appendix

	\section{The Hilbert transform and the Cauchy integral}\label{appendix:cauchy}
In this appendix, we provide some detailed analysis of the Hilbert transform and the Cauchy integral used in this paper.  We start with some basic definitions.	
	\begin{defn}
		Let $\theta\in (0,\pi/2)$ and $\omega\in \mathbb{C}$. We define $C_{\theta}(\omega)$ as the cone
		\begin{equation}
		C_{\theta}(\omega):=\{z\in \mathbb{C}: Im\{z\}\leq \Im\{\omega\}, \tan^{-1}(\frac{|\Re\{z-\omega\}|}{|\Im\{z-\omega|\}})\leq \theta/2\}.
		\end{equation}
	\end{defn}

	\begin{defn}
	Given a chord arc $\gamma$ parametrized by $\zeta$ such that $\zeta-\alpha$ is $2q\pi$ periodic. Define the Cauchy integral as
	\begin{equation}
	C_{\gamma}f(z,t):=\frac{1}{4q\pi i}\int_{\gamma}\cot(\frac{z-\zeta}{2q})f(\zeta)\,d\zeta
	\end{equation}
where $\int_{\gamma}$ means integrating over a fundamental period of $\zeta-\alpha$.
	\end{defn}
With preparations above, we have the following properties of the Cauchy integral. 
	\begin{lemma}\label{boundaryvalueofcauchyintegral}
		Let $f$ and $\gamma$ be sufficient nice functions, and $\gamma$ has endpoints $\gamma_L, \gamma_R$. Assume that $Im\{\gamma_L\}=\Im\{\gamma_R\}$, and $f(\gamma_L)=f(\gamma_R)$. Let $\Omega$ be the region below $\gamma$.  We have the following conclusions:
		\begin{itemize}
			\item [(1)] Let $\omega\in \gamma$.Then one has that 
			\begin{equation}\label{eq:nontanlimit}
			    			C_{\gamma}f(z)\rightarrow \frac{1}{2}(I+\mathcal{H}_{\gamma})f(\omega)
			\end{equation}
			as $z\rightarrow \omega$ nontangentially. 
			
			\item [(2)] If $f(\zeta)=G(\zeta)$, for some bounded holomorphic function $G$ in $\Omega$ such that $G(-q\pi+iy)=G(q\pi+iy)$,  then 
			\begin{equation}\label{eq:nontanlimit2}
						C_{\gamma}f(z)\rightarrow f(\omega)-\frac{1}{4q\pi}\int_{\gamma} f(\zeta)\,d\zeta
			\end{equation}
			 as $z\rightarrow \omega$ nontangentially. 
		\end{itemize}
	\end{lemma}

	\begin{proof}
		Let $0<r\ll \frac{1}{q}$ be fixed. Given $\omega\in \gamma$, we denote $\gamma_1=B_r(\omega)\cap \gamma$ and $\gamma_2=\gamma-\gamma_1$. Again here, we abuse of notation that $\gamma$ here means restricting onto a fundamental period of $\zeta-\alpha$ after parameterizing the curve.  Let $z\in C_{\theta}(\omega)$. For the integral over $\gamma_2$, by the continuity of $\cot(\frac{z-\zeta}{2q})$, we have 
		\begin{equation}
		\frac{1}{4q\pi i}\int_{\gamma_2}\cot(\frac{z-\zeta}{2q}) f(\zeta)\,d\zeta\rightarrow \frac{1}{4q\pi i}\int_{\gamma_2}\cot(\frac{\omega-\zeta}{2q}) f(\zeta)\,d\zeta,\quad \quad as~z\rightarrow \omega.
		\end{equation}
		On the other hand, for the integral over $\gamma_1$, we split it into two pieces: 
		\begin{align*}
		& \frac{1}{4q\pi i}\int_{\gamma_1}\cot(\frac{z-\zeta}{2q}) f(\zeta)\,d\zeta\\
		=&\frac{1}{4q\pi i}\int_{\gamma_1}\cot(\frac{z-\zeta}{2q}) (f(\zeta)-f(\omega))\,d\zeta+ \frac{1}{4q\pi i}\int_{\gamma_1}\cot(\frac{z-\zeta}{2q}) f(\omega)\,d\zeta\\
		:=& I+II.
		\end{align*}
		For $I$, as $z\rightarrow \omega$ in $C_{\theta}(\omega)$, we observe the followings:
		\begin{itemize}
			\item [1.] On $\gamma_1$, we have 
			\begin{equation}
			\cot(\frac{z-\zeta}{2q})=\cos(\frac{z-\zeta}{2q})\frac{\frac{z-\zeta}{2q}}{\sin(\frac{z-\zeta}{2q})}\frac{2q}{z-\zeta}=\frac{2q}{z-\zeta}+O(1).
			\end{equation}

			\item [2.] $|f(\zeta)-f(\omega)|\leq \|f'\|_{L^{\infty}}|\zeta-\omega|$.
			
			\item [3.] For $z\in C_{\theta}(\omega)$, we have 
			\begin{equation}
			c_1(\theta)|z-\zeta|\leq |\omega-\zeta|\leq c_2(\theta)|z-\zeta|,
			\end{equation}
			where $c_1(\theta)$ and $c_2(\theta)$ depend continuously on $\theta$.
			
		\end{itemize}
		These facts imply 
		\begin{align*}
		|I|\leq &  \frac{1}{4q\pi }\int_{\gamma_1} (\frac{2q}{|z-\zeta|}+O(1))\|f'\|_{L^{\infty}} |\zeta-\omega| |d\zeta|\\
		\leq &c_2(\theta)\frac{1}{4q\pi }\int_{\gamma_1} (\frac{2q}{|z-\zeta|}+O(1))\|f'\|_{L^{\infty}} |z-\zeta| |d\zeta|\\
		\leq & Cr,
		\end{align*}
		for some constant $C$ depending on $\|f'\|_{L^{\infty}}$ and $\theta$, but with no dependence on $r$.   Therefore as $r\rightarrow0$, $I\rightarrow0$ (as $z\rightarrow\omega$).
		
	Next we analyze $II$. Assume that $\gamma_1$ has the starting point $\zeta_L$ and the ending point $\zeta_R$. For $II$, since $\frac{1}{2q}\cot(\frac{z-\zeta}{2q})=-\partial_{\zeta} \log \sin(\frac{z-\zeta}{2q})$, for $|z-\omega|\ll r$ (which is valid as $z\rightarrow \omega$), we have
		\begin{align*}
		II=& \frac{1}{2\pi i}(\log \sin\frac{z-\zeta_L}{2q}-\log \sin\frac{z-\zeta_R}{2q})f(\omega)\\
		=&  \frac{1}{2\pi i}(\log \sin\frac{\omega-\zeta_L}{2q}-\log \sin\frac{\omega-\zeta_R}{2q})f(\omega)\\
		&+\Big\{ (\frac{1}{2\pi i}(\log \sin\frac{z-\zeta_L}{2q}-\log \sin\frac{z-\zeta_R}{2q}))-(\frac{1}{2\pi i}(\log \sin\frac{\omega-\zeta_L}{2q}-\log \sin\frac{\omega-\zeta_R}{2q}))\}f(\omega)\\
		:=& f(\omega)(II_1+II_2).
		\end{align*}
		It is straightforward to verify that for any $r>0$ fixed,
		\begin{equation}
		|II_2|\leq C|z-\omega| |\log \sin \frac{z-\zeta_R}{2q}|\rightarrow 0, \quad as~z\rightarrow \omega.
		\end{equation}
		For $II_1$,
		we observe that as $r\rightarrow0$, one has $\frac{\omega-\zeta_L}{\omega-\zeta_R}\rightarrow-1$. Therefore $$\log \sin\frac{\omega-\zeta_L}{2q}-\log \sin\frac{\omega-\zeta_R}{2q}\rightarrow
		\log(-1)=i\pi\,\,\text{as}\,r\rightarrow0.$$
		Hence we obtain
		\begin{equation}
		II_1\rightarrow \frac{1}{2}
		\end{equation}
		Let $r\rightarrow 0$, then \eqref{eq:nontanlimit} is proved.
		
		\vspace*{2ex}
		To prove \eqref{eq:nontanlimit2}, suppose $f$ is the boundary value of the holomorphic function $G$ in $\Omega$. Take $0<r\ll 1/q$. Denote the left and right endpoints of $\gamma$ by $\gamma_L$ and $\gamma_R$, respectively. Taking $K\gg 1$, we set
		$$Q_{L,K}:=\Re\{\gamma_L\}-iK,\quad \quad Q_{R,K}:=\Re\{\gamma_R\}-iK.$$
		Let $\gamma_1$ be the segment from $\gamma_R$ to $Q_{R,K}$, $\gamma_2$ be the segment from $Q_{R,K}$ to $Q_{L,K}$, and $\gamma_3$ be the segments from $Q_{L,K}$ to $\gamma_L$. Let $\gamma_4$ be $\partial B_r(z)$, oriented clockwisely. By Cauchy's Theorem, since $f$ is the boundary value of $G$ in $\Omega$, we have 
		\begin{equation}
		C_{\zeta}f(z)+\frac{1}{4q\pi i}\int_{\gamma_1\cup \gamma_2\cup\gamma_3\cup \gamma_4} \cot(\frac{z-\zeta}{2q})G(\zeta)\,d\zeta=0.
		\end{equation}
		Note that 
		\begin{equation}
		\int_{\gamma_1 \cup\gamma_3} \cot(\frac{z-\zeta}{2q})G(\zeta)\,d\zeta=0.
		\end{equation}
		by the periodicity assumption and the orientation.
		
		As $r\rightarrow 0$, we have 
		\begin{align*}
		\frac{1}{4q\pi i}\int_{\gamma_4} \cot(\frac{z-\zeta}{2q})G(\zeta)\,d\zeta=& \frac{1}{4q\pi i}\int_{\gamma_4}(\frac{2q}{z-\zeta}+O(1))G(\zeta)\,d\zeta\rightarrow -G(z).
		\end{align*} Note that here for $\gamma_4$, the orientation is clockwise.
		
		Notice that  $\Im\{\zeta\}\rightarrow -\infty$, one has $\cot(\frac{z-\zeta}{2q})\rightarrow -i$. So for $K$ large,
		\begin{align*}
		\frac{1}{4q\pi i}\int_{\gamma_2} \cot(\frac{z-\zeta}{2q})G(\zeta)\,d\zeta=-\frac{1}{4q\pi}\int_{\gamma_2} G(\zeta)\,d\zeta+O(K^{-1}).
		\end{align*}
		By Cauchy's Theorem again, we have 
		\begin{align*}
		\frac{1}{4q\pi i}\int_{\gamma\cup \gamma_1\cup \gamma_2\cup\gamma_3} G(\zeta)\,d\zeta=0.
		\end{align*}
		Putting everything together, one has
		\begin{equation}
		-\lim_{K\rightarrow \infty} \frac{1}{4q\pi}\int_{\gamma_2} G(\zeta)\,d\zeta=\frac{1}{4q\pi}\int_{\gamma} f(\zeta)\, d\zeta.
		\end{equation}
		Therefore we conclude that
		\begin{equation}
		\lim_{z\rightarrow \omega}C_{\gamma}f(z)=f(\omega)-\frac{1}{4q\pi}\int_{\gamma} f(\zeta) \,d\zeta
		\end{equation} as desired.
	\end{proof}

	As a consequence of Lemma (\ref{boundaryvalueofcauchyintegral}), we obtain the following conclusion.
	\begin{cor}\label{corollary:hilbertboundary}
		Let $f$ and $\gamma$ be sufficient nice functions, and suppose that $\gamma$ has endpoints $\gamma_L, \gamma_R$. Let $\Omega$ be the region below $\gamma$. Assume that $\Im\{\gamma_L\}=\Im\{\gamma_R\}$, and $f(\gamma_L)=f(\gamma_R)$. 
		\begin{itemize}
			\item [(a)] $\frac{1}{2}(I+\mathcal{H}_{\gamma})f$ is the boundary value of a periodic holomorphic function $\mathcal{F}$ on $\Omega$, with $\mathcal{F}(z)\rightarrow \frac{1}{2}c_0$ as $\Im\{z\}\rightarrow -\infty$.
			
	\item [(b)]	$f$ is the boundary value of a holomorphic function $G$ on $\Omega$ satisfying $G(\Re\{\gamma_L\}+iy)=G(\Re\{\gamma_R\}+iy)$ for all $y<\Im\{\gamma_L\}$ and $G(x+iy)\rightarrow c_0$ as $y\rightarrow -\infty$ if and only if
		\begin{equation}
		(I-\mathcal{H}_{\gamma})f=c_0,
		\end{equation}
		where $c_0=\frac{1}{2q\pi}\int_{\gamma}f(\zeta)\,d\zeta$. 
		\end{itemize}
	\end{cor}
	
	\begin{proof}
		For (a), by Lemma \ref{boundaryvalueofcauchyintegral}, $\mathcal{F}(z):=C_{\gamma}f(z)$ has the boundary value $\frac{1}{2}(I+\mathcal{H}_{\gamma})f$, and $\mathcal{F}(z)\rightarrow \frac{1}{2}c_0$ as $\Im\{z\}\rightarrow -\infty$.

	For (b), in one direction,	if $f$ is the boundary value of a bounded holomorphic function in $\Omega$, then by Lemma \ref{boundaryvalueofcauchyintegral}, we have 
		\begin{equation}
		\frac{1}{2}(I+\mathcal{H}_{\gamma})f=f-\frac{1}{4q\pi}\int_{\gamma}f(\zeta)\,d\zeta,
		\end{equation}
		which implies 
		\begin{equation}
		(I-\mathcal{H}_{\gamma})f=\frac{1}{2q\pi}\int_{\gamma}f(\zeta)\,d\zeta.
		\end{equation}
		On the other hand, if $(I-\mathcal{H}_{\zeta})f=c_0$, then $\mathcal{H}_{\gamma}f=f-c_0$. Define $G(z)$ by $G(z):=C_{\gamma}f(z)+\frac{1}{2}c_0$. The boundary value of $G(z)$ is 
		$$\frac{1}{2}c_0+\frac{1}{2}(I+\mathcal{H}_{\gamma})f=\frac{1}{2}c_0+\frac{1}{2}f+(\frac{1}{2}f-\frac{1}{2}c_0)=f.$$
		Moreover,  using that fact
		$$\lim_{\Im\{z\}\rightarrow -\infty}\cot(\frac{z-\zeta}{2q})=i,$$
		we obtain
		\begin{align*}
		\lim_{\Im\{z\}\rightarrow -\infty}C_{\gamma}f(z)=\frac{1}{4q\pi i} \lim_{\Im\{z\}\rightarrow -\infty}\int_{\gamma}\cot(\frac{z-\zeta}{2q})f(\zeta)\,d\zeta =\frac{1}{4q\pi}\int_{\gamma} f(\zeta)\,d\zeta=\frac{1}{2}c_0.
		\end{align*}
		So $G(z)\rightarrow c_0$ as $\Im\{z\}\rightarrow -\infty$, and $G(z)$ has the boundary value $f$.
	\end{proof}

	\section{Identities}\label{appendix:identities}
In this appendix, we provide the proof  of Lemma \ref{lemmmaeight} in details. 
	\begin{proof}[Proof of Lemma \ref{lemmmaeight}]
		For (\ref{U1}), performing integration by parts, we rewrite the Hilbert transform as 
		\begin{equation}\label{logtransform}
		\mathcal{H}_{\zeta}f(\alpha,t)=\frac{1}{\pi i}\text{p.v.}\int_{-q\pi}^{q\pi}\log(\sin(\frac{\zeta(\alpha,t)-\zeta(\beta,t)}{2q})f_{\beta}\,d\beta.
		\end{equation}
		By direct computations, one has
		\begin{align*}
		\partial_{\alpha}\mathcal{H}_{\zeta}f(\alpha,t)=&\frac{1}{2q\pi i}\text{p.v.}\int_{-q\pi}^{q\pi}\zeta_{\alpha}\cot(\frac{\zeta(\alpha,t)-\zeta(\beta,t)}{2q})f_{\beta}\,d\beta\\
		=& \zeta_{\alpha}\mathcal{H}_{\zeta}\frac{\partial_{\alpha}}{\zeta_{\alpha}}f.
		\end{align*}
		Therefore, we can conclude that
		\begin{align*}
		[\partial_{\alpha}, \mathcal{H}_{\zeta}]f=&\partial_{\alpha}\mathcal{H}_{\zeta}f-\mathcal{H}_{\zeta}\partial_{\alpha}f=[\zeta_{\beta}, \mathcal{H}_{\zeta}]\frac{f_{\alpha}}{\zeta_{\alpha}},
		\end{align*}
		which gives (\ref{U1}) and then (\ref{U2}) is an easy consequence of (\ref{U1}).
		
		For (\ref{U3}), using (\ref{logtransform}), we get
		\begin{align*}
		\partial_t \mathcal{H}_{\zeta}f(\alpha,t)=&\frac{1}{2q\pi i}\text{p.v.}\int_{-q\pi}^{q\pi}(\zeta_{t}(\alpha,t)-\zeta_t(\beta,t))\cot(\frac{\zeta(\alpha,t)-\zeta(\beta,t)}{2q})f_{\beta}\,d\beta+\mathcal{H}_{\zeta}f_t.
		\end{align*} by direct differentiation.
		So we obtain
		\begin{align*}
		[\partial_t, \mathcal{H}_{\zeta}]f=& \partial_t \mathcal{H}_{\zeta}f-\mathcal{H}_{\zeta}f_t=[\zeta_t, \mathcal{H}_{\zeta}]\frac{f_{\alpha}}{\zeta_{\alpha}},
		\end{align*}
		which gives (\ref{U3}). And then (\ref{U2}) and (\ref{U3}) together implies (\ref{U4}).
		
		For (\ref{U5}), we note that
		\begin{equation}
		[D_t^2, \mathcal{H}_{\zeta}]=D_t[D_t, \mathcal{H}_{\zeta}]+[D_t, \mathcal{H}_{\zeta}]D_t.
		\end{equation}
		We first calculate $\partial_t[z_t, \mathcal{H}_z]g$. Notice that 
		$$\partial_t \cot(\frac{z(\alpha,t)-z(\beta,t)}{2q}) =-\frac{1}{2q}\frac{z_t(\alpha,t)-z_t(\beta,t)}{\sin(\frac{z(\alpha,t)-z(\beta,t)}{2q})}.$$
		By direct computations, one has
		\begin{align*}
		\partial_t[z_t, \mathcal{H}_z]g=& \partial_t \frac{1}{2q\pi i}\text{p.v.}\int_{-q\pi}^{q\pi}(z_t(\alpha,t)-z_t(\beta,t))\cot(\frac{z(\alpha,t)-z(\beta,t)}{2q}) g_{\beta}(\beta,t)\,d\beta\\
		=& \frac{1}{2q\pi i}\text{p.v.}\int_{-q\pi}^{q\pi}(z_{t}(\alpha,t)-z_{t}(\beta,t))\cot(\frac{z(\alpha,t)-z(\beta,t)}{2q}) g_{\beta}(\beta,t)\,d\beta\\\
		& -\frac{1}{4\pi q^2 i}\int_{-q\pi}^{q\pi}\Big(\frac{z_t(\alpha)-z_t(\beta)}{\sin(\frac{1}{2q}(z(\alpha)-z(\beta)))}\Big)^2 g_{\beta}\,d\beta\\
		&+\frac{1}{2q\pi i}\text{p.v.}\int_{-q\pi}^{q\pi}(z_t(\alpha,t)-z_t(\beta,t))\cot(\frac{z(\alpha,t)-z(\beta,t)}{2q}) \partial_tg_{\beta}(\beta,t)\,d\beta
		\end{align*}
		Changing of variables, we obtain
		\begin{align*}
		D_t[D_t, \mathcal{H}_{\zeta}]f=&[D_t^2\zeta, \mathcal{H}_{\zeta}]\frac{f_{\alpha}}{\zeta_{\alpha}}-\frac{1}{4\pi q^2 i}\int_{-q\pi}^{q\pi}\Big(\frac{D_t\zeta(\alpha)-D_t\zeta(\beta)}{\sin(\frac{1}{2q}(\zeta(\alpha)-\zeta(\beta)))}\Big)^2 f_{\beta}\,d\beta+[D_t\zeta, \mathcal{H}_{\zeta}]\frac{\partial_{\alpha}D_tf}{\zeta_{\alpha}}.
		\end{align*}
		Moreover, note that
		\begin{align*}
		D_t[D_t, \mathcal{H}_{\zeta}]f=& D_t[D_t\zeta, \mathcal{H}_{\zeta}]\frac{f_{\alpha}}{\zeta_{\alpha}}.
		\end{align*}
		Therefore we have
		\begin{equation}
		[D_t^2, \mathcal{H}_{\zeta}]f=[D_t^2\zeta, \mathcal{H}_{\zeta}]\frac{f_{\alpha}}{\zeta_{\alpha}}+2[D_t\zeta, \mathcal{H}_{\zeta}]\frac{\partial_{\alpha}D_t f}{\zeta_{\alpha}}-\frac{1}{4\pi q^2 i}\int_{-q\pi}^{q\pi}\Big(\frac{D_t\zeta(\alpha)-D_t\zeta(\beta)}{\sin(\frac{1}{2q}(\zeta(\alpha)-\zeta(\beta)))}\Big)^2 f_{\beta}\,d\beta,
		\end{equation}
		which is (\ref{U5}).  
		
		By (\ref{U2}), one has
		\begin{align*}
		[D_t^2\zeta, \mathcal{H}_{\zeta}]\frac{f_{\alpha}}{\zeta_{\alpha}}-[iA\partial_{\alpha}, \mathcal{H}_{\zeta}]f=[(D_t^2-iA\partial_{\alpha})\zeta, \mathcal{H}_{\zeta}]\frac{f_{\alpha}}{\zeta_{\alpha}}=[-i, \mathcal{H}_{\zeta}]\frac{f_{\alpha}}{\zeta_{\alpha}}=0.
		\end{align*}
		Hence putting things together, we conclude (\ref{U6}).
			\end{proof}


\section{Supplementary proofs}\label{estimatesforsomequantities}	
In this section, we provide supplementary proofs of some lemmata used in the main part of the paper.
\subsection{Proof of Lemma \ref{equivalencequantities}}
\begin{proof}[Proof of Lemma \ref{equivalencequantities}]
Recall that$$\rho:=(I-\mathcal{H}_{\zeta})\Big[\theta-\theta_{ST}-(\tilde{\theta}-\tilde{\theta}_{ST})\Big],$$ where $$\theta:=(I-\mathcal{H}_{\zeta})(\zeta-\alpha)=2(\zeta-\alpha)-(\overline{\mathcal{H}_{\zeta}}+\mathcal{H}_{\zeta})(\zeta-\alpha),$$
$$  \theta_{ST}=(I-\mathcal{H}_{\zeta_{ST}})(\zeta_{ST}-\alpha)=2(\zeta_{ST}-\alpha)-(\overline{\mathcal{H}_{\zeta_{ST}}}+\mathcal{H}_{\zeta_{ST}})(\zeta_{ST}-\alpha),$$
and 
\begin{align*}
    \tilde{\theta}-\tilde{\zeta}_{ST}:=& (I-\mathcal{H}_{\tilde{\zeta}})(\tilde{\zeta}-\alpha)- (I-\mathcal{H}_{\tilde{\zeta}_{ST}})(\tilde{\zeta}_{ST}-\alpha)\\
    =& 2(\tilde{\zeta}-\alpha)-(\overline{\mathcal{H}_{\tilde{\zeta}}}+\mathcal{H}_{\tilde{\zeta}})(\tilde{\zeta}-\alpha)-\Big(2(\tilde{\zeta}_{ST}-\alpha)-(\overline{\mathcal{H}_{\tilde{\zeta}_{ST}}}+\mathcal{H}_{\tilde{\zeta}_{ST}})(\tilde{\zeta}_{ST}-\alpha)\Big)+e,
\end{align*}
where by the explicit construction of $\tilde{\zeta}$ and $\tilde{\zeta}_{ST}$,
\begin{equation}
    \norm{e}_{H^{s'}(q\mathbb{T})}\leq C\epsilon^{7/2}\delta e^{\epsilon^2t}.
\end{equation}
Taking the difference, it follows
\begin{align*}
   & \theta-\theta_{ST}-(\tilde{\theta}-\tilde{\theta}_{ST})\\
   =& 2(\zeta-\zeta_{ST}-(\tilde{\zeta}-\tilde{\zeta}_{ST}))-e-g,
\end{align*}
where
\begin{equation}
    g:=(\overline{\mathcal{H}_{\zeta}}+\mathcal{H}_{\zeta})(\zeta-\alpha)-(\overline{\mathcal{H}_{\zeta_{ST}}}+\mathcal{H}_{\zeta_{ST}})(\zeta_{ST}-\alpha)-\Big((\overline{\mathcal{H}_{\tilde{\zeta}}}+\mathcal{H}_{\tilde{\zeta}})(\tilde{\zeta}-\alpha)-(\overline{\mathcal{H}_{\tilde{\zeta}_{ST}}}+\mathcal{H}_{\tilde{\zeta}_{ST}})(\tilde{\zeta}_{ST}-\alpha)\Big).
\end{equation}
By manipulating the differences, it is easy to obtain
\begin{equation}
    \norm{\partial_{\alpha}g}_{H^{s}(q\mathbb{T})}\leq C\epsilon^{5/2}\delta e^{\epsilon^2 t}.
\end{equation}
Combing computations and estimates above, one has
\begin{equation}
    \norm{\partial_{\alpha}(\rho-2r)}_{H^s(q\mathbb{T})}\leq C\epsilon^{5/2}\delta e^{\epsilon^2 t}.
\end{equation}
Finally by the same argument we obtain
	\begin{equation} \norm{D_t\rho-2D_tr}_{H^{s+1/2}(q\mathbb{T})}\leq C(\epsilon E_s^{1/2}+\delta e^{\epsilon^2t}\epsilon^{5/2}),
		\end{equation}
and
				\begin{equation} \norm{D_t(D_t\rho-2D_tr)}_{H^{s}(q\mathbb{T})}\leq  C\epsilon^{5/2}\delta e^{\epsilon^2 t}.
		\end{equation}
We are done.
\end{proof}

\subsection{Proof of Proposition \ref{commutator:difference:quadratic}}\label{proof:commutatordiffquadratic}

	\begin{proof}
	To suppress notations, we write $\zeta_j(\alpha,t)$ as $\zeta_j(\alpha)$.

Using elementary trigonometric identities
	\begin{align*}
	    &\cos(\frac{1}{2q}(\zeta_1(\alpha)-\zeta_1(\beta)))\sin(\frac{1}{2q}(\zeta_2(\alpha)-\zeta_2(\beta)))\\
	    &-\cos(\frac{1}{2q}(\zeta_2(\alpha)-\zeta_2(\beta)))\sin(\frac{1}{2q}(\zeta_1(\alpha)-\zeta_1(\beta)))\\
	    =& \sin(\frac{1}{2q}((\zeta_1(\alpha)-\zeta_1(\beta))-(\zeta_2(\alpha)-\zeta_2(\beta)))),
	\end{align*}
	we can write the difference as
	\begin{align*}
	   & \Big(S_{\zeta_1}-S_{\zeta_2}\Big)(f,g)\\=& [g, \mathcal{H}_{\zeta_1}]\frac{f_{\alpha}}{\partial_{\alpha}\zeta_1}-[g, \mathcal{H}_{\zeta_2}]\frac{f_{\alpha}}{\partial_{\alpha}\zeta_2}\\
	   =& \frac{1}{2q\pi i}\text{p.v.}\int_{-q\pi}^{q\pi} (g(\alpha)-g(\beta))\Big(\frac{1}{\tan(\frac{1}{2q}\Big(\zeta_1(\alpha)-\zeta_1(\beta)))}-\frac{1}{\tan(\frac{1}{2q}(\zeta_2(\alpha)-\zeta_2(\beta)))}\Big) f_{\beta}(\beta)\,d\beta\\
	   =& \frac{1}{2q\pi i}\text{p.v.}\int_{-q\pi}^{q\pi}(g(\alpha)-g(\beta))\frac{\sin(\frac{1}{2q}((\zeta_1(\alpha)-\zeta_1(\beta))-(\zeta_2(\alpha)-\zeta_2(\beta))))}{\sin(\frac{1}{2q}(\zeta_1(\alpha)-\zeta_1(\beta)))\sin(\frac{1}{2q}(\zeta_2(\alpha)-\zeta_2(\beta)))}f_{\beta}\,d\beta.
	   \end{align*}
	 By (\ref{singularversiontwo}), we have 
	 \begin{align}
	     \norm{(S_{\zeta_1}-S_{\zeta_2}\Big)(f,g)}_{H^s(q\mathbb{T})}\leq & C\|\partial_{\alpha}(\zeta_1-\zeta_2)\|_{Y}\|f\|_{Z}\|g\|_{Y}.	 \end{align}
	 	\vspace*{1ex}
Next we analyze  \eqref{eq:diffcommutator2}. We first introduce the notation:
	\begin{equation}
	        L_{j}(\alpha,\beta):= \frac{1}{2q}(\zeta_j(\alpha)-\zeta_j(\beta)).
	\end{equation}
Then we perform some elementary computations. Using trigonometric identities again, we have 
	    \begin{align*}
	        &\cot L_1(\alpha,\beta)-\cot L_2(\alpha,\beta)-\Big\{\cot L_3(\alpha,\beta)-\cot L_4(\alpha,\beta)\Big\}\\
	        =&\frac{\sin (L_1(\alpha,\beta)-L_2(\alpha,\beta))}{\sin L_1(\alpha,\beta))\sin L_2(\alpha,\beta)}- \frac{\sin (L_3(\alpha,\beta)-L_4(\alpha,\beta))}{\sin L_3(\alpha,\beta))\sin L_4(\alpha,\beta)}\\
	        =& K_1(\alpha,\beta)+K_2(\alpha,\beta),
	    \end{align*}
	    where we define
	    \begin{align}
	        &K_1(\alpha,\beta)
	        = \frac{\sin \Big(L_1(\alpha,\beta)-L_2(\alpha,\beta)\Big)-\sin \Big(L_3(\alpha,\beta)-L_4(\alpha,\beta))\Big)}{\sin L_1(\alpha,\beta))\sin L_2(\alpha,\beta)},
	    \end{align}
	    \begin{align}
	        &K_2(\alpha,\beta) =\sin (L_3(\alpha,\beta)-L_4(\alpha,\beta)) \Big(\frac{1}{\sin L_1(\alpha,\beta))\sin L_2(\alpha,\beta)}-\frac{1}{\sin L_3(\alpha,\beta))\sin L_4(\alpha,\beta)}\Big).
	    \end{align}
	    We write $\sin L_1(\alpha,\beta)\sin L_2(\alpha,\beta)$ as 
	    \begin{align}
	        &\sin L_1(\alpha,\beta)\sin L_2(\alpha,\beta)=\frac{1}{2}\cos(L_1(\alpha,\beta)-L_2(\alpha,\beta))-\frac{1}{2}\cos(L_1(\alpha,\beta)+L_2(\alpha,\beta)),
	    \end{align}
	    and similarly
	      \begin{align}
	        &\sin L_3(\alpha,\beta)\sin L_4(\alpha,\beta)=\frac{1}{2}\cos(L_3(\alpha,\beta)-L_4(\alpha,\beta))-\frac{1}{2}\cos(L_3(\alpha,\beta)+L_4(\alpha,\beta)),
	    \end{align}
Taking the difference of to expressions above, we have
\begin{align*}
     &\sin L_1(\alpha,\beta)\sin L_2(\alpha,\beta)-\sin L_3(\alpha,\beta)\sin L_4(\alpha,\beta)\\
      =&\frac{1}{2}\cos(L_1(\alpha,\beta)-L_2(\alpha,\beta))-\frac{1}{2}\cos(L_1(\alpha,\beta)+L_2(\alpha,\beta))\\
      &-\Big(\frac{1}{2}\cos(L_3(\alpha,\beta)-L_4(\alpha,\beta))-\frac{1}{2}\cos(L_3(\alpha,\beta)+L_4(\alpha,\beta))\Big)\\
      =& -\sin((L_1(\alpha,\beta)-L_2(\alpha,\beta))-(L_3(\alpha,\beta)-L_4(\alpha,\beta)))\\
      &\times \sin((L_1(\alpha,\beta)-L_2(\alpha,\beta))+(L_3(\alpha,\beta)-L_4(\alpha,\beta)))\\
      &+\sin((L_1(\alpha,\beta)+L_2(\alpha,\beta))-(L_3(\alpha,\beta)+L_4(\alpha,\beta)))\\
      &\times \sin((L_1(\alpha,\beta)+L_2(\alpha,\beta))+(L_3(\alpha,\beta)+L_4(\alpha,\beta)))\\
\end{align*}
We  further decompose $K_2(\alpha,\beta)$ as
\begin{equation}
    K_2(\alpha,\beta)=K_{21}(\alpha,\beta)+K_{22}(\alpha,\beta),
\end{equation}
where 
\begin{equation}
\begin{split}
    K_{21}(\alpha,\beta)=&-\sin (L_3(\alpha,\beta)-L_4(\alpha,\beta))\\
    &\times \frac{\sin((L_1(\alpha,\beta)-L_2(\alpha,\beta))-(L_3(\alpha,\beta)-L_4(\alpha,\beta)))}{\sin L_1(\alpha,\beta))\sin L_2(\alpha,\beta)}\\
    &\times \frac{\sin((L_1(\alpha,\beta)-L_2(\alpha,\beta))+(L_3(\alpha,\beta)-L_4(\alpha,\beta)))}{\sin L_3(\alpha,\beta))\sin L_4(\alpha,\beta)},
    \end{split}
\end{equation}
and
\begin{equation}
\begin{split}
    K_{22}(\alpha,\beta)=&\sin (L_3(\alpha,\beta)-L_4(\alpha,\beta))\\
    &\times \frac{\sin((L_1(\alpha,\beta)+L_2(\alpha,\beta))-(L_3(\alpha,\beta)+L_4(\alpha,\beta)))}{\sin L_1(\alpha,\beta))\sin L_2(\alpha,\beta)}\\
    &\times \frac{\sin((L_1(\alpha,\beta)+L_2(\alpha,\beta))+(L_3(\alpha,\beta)+L_4(\alpha,\beta)))}{\sin L_3(\alpha,\beta))\sin L_4(\alpha,\beta)}.
    \end{split}
\end{equation}
Now we can write the difference of commutators from \eqref{eq:diffcommutator2} as
	 \begin{align}
	     &\Big(S_{\zeta_1}-S_{\zeta_2}-(S_{\zeta_3}-S_{\zeta_4})\Big)(f,g)\\
	     =&\frac{1}{2q\pi i}\text{p.v.}\int_{-q\pi}^{q\pi}(g(\alpha)-g(\beta))[K_1(\alpha,\beta)+K_{21}(\alpha,\beta)+K_{22}(\alpha,\beta)]f_{\beta}\,d\beta
	 \end{align}
	 By the explicit form of $K_1$, $K_{21}$, and $K_{22}$ and (\ref{singularversiontwo}), we have 
	 \begin{align}
	    & \norm{\frac{1}{2q\pi i}\text{p.v.}\int_{-q\pi}^{q\pi}(g(\alpha)-g(\beta))K_1(\alpha,\beta)f_{\beta}\,d\beta}_{H^s(q\mathbb{T})}
	    \leq  C\norm{\partial_{\alpha}(\zeta_1-\zeta_2-(\zeta_3-\zeta_4)}_{Y}\norm{g}_{Y} \norm{f}_{Z}.
	 \end{align}
Secondly, one can estimate	 
	 \begin{align}
	      & \norm{\frac{1}{2q\pi i}\text{p.v.}\int_{-q\pi}^{q\pi}(g(\alpha)-g(\beta))K_{21}(\alpha,\beta)f_{\beta}\,d\beta}_{H^s(q\mathbb{T})}\\
	      \leq & C\|\partial_{\alpha}(\zeta_3-\zeta_4)\|_{Y}\|\partial_{\alpha}(\zeta_1-\zeta_2-(\zeta_3-\zeta_4))\|_{Y}\Big( \|\partial_{\alpha}(\zeta_1-\zeta_2)+\partial_{\alpha}(\zeta_3-\zeta_4)\|_{Y} \|f\|_{Z}\\
	       \leq & C\norm{\partial_{\alpha}(\zeta_1-\zeta_2-(\zeta_3-\zeta_4))}_{Y}\Big(\norm{\partial_{\alpha}(\zeta_1-\zeta_2)}_{Y}+\norm{\partial_{\alpha}(\zeta_3-\zeta_3)}_{Y}\Big)^2 \|f\|_Z \|g\|_{Y}.
	 \end{align}
Finally, we have	 
	  \begin{align}
	      & \norm{\frac{1}{2q\pi i}\text{p.v.}\int_{-q\pi}^{q\pi}(g(\alpha)-g(\beta))K_{22}(\alpha,\beta)f_{\beta}\,d\beta}_{H^s(q\mathbb{T})}\\
	      \leq & C\|\partial_{\alpha}(\zeta_3-\zeta_4)\|_{Y}\|\partial_{\alpha}(\zeta_1+\zeta_2-(\zeta_3+\zeta_4))\|_{Y}\Big( \|\partial_{\alpha}(\zeta_1+\zeta_2+\zeta_3+\zeta_4)\|_{Y} \|f\|_{Z}\\
	       \leq & C\|\partial_{\alpha}(\zeta_3-\zeta_4)\|_{Y}\Big( \norm{\partial_{\alpha}(\zeta_1-\zeta_3)}_{Y}+\norm{\partial_{\alpha}(\zeta_2-\zeta_4)}_{Y}\Big)\sum_{j=1}^4\Big(1+\norm{\partial_{\alpha}\zeta_j-1}_{Y}\Big)
	       \|f\|_Z \|g\|_{Y}.\nonumber
	 \end{align}
Putting three estimates above together, we obtain the desired result.
	\end{proof}

\subsection{Proof of Proposition \ref{commutator:difference:cubic}}\label{proof:commutatordiffcubic}
	\begin{proof} The proof of this proposition is purely algebraic. Firstly, we regroup the expression we are interested as
	\begin{align*}
	    &\Big(S_{\zeta_1}(g_1, f_1)-S_{\zeta_2}(g_2, f_2)\Big)-\Big(S_{\zeta_3}(g_3, f_3)-S_{\zeta_4}(g_4, f_4)\Big)\\
	    =&S_{\zeta_1}(g_1, f_1-f_2-(f_3-f_4))\\
	    &+\Big\{ S_{\zeta_1}(g_1, f_2+f_3-f_4)-S_{\zeta_2}(g_2, f_2)-\Big(S_{\zeta_3}(g_3, f_3)-S_{\zeta_4}(g_4, f_4)\Big)\Big\}
	\end{align*}
Then we  rewrite  $S_{\zeta_1}(g_1, f_2+f_3-f_4)-S_{\zeta_2}(g_2, f_2)$ as 
	\begin{align*}
	    &S_{\zeta_1}(g_1, f_2+f_3-f_4)-S_{\zeta_2}(g_2, f_2)\\
	    =& \Big(S_{\zeta_1}(g_1, f_2+f_3-f_4)-S_{\zeta_2}(g_1, f_2+f_3-f_4)\Big)+S_{\zeta_2}(g_1-g_2, f_2+f_3-f_4)\\
	    &+\Big(S_{\zeta_2}(g_2,f_2+f_3-f_4)-S_{\zeta_2}(g_2, f_2)\Big)\\
	    =& \Big(S_{\zeta_1}-S_{\zeta_2}\Big)(g_1, f_2+f_3-f_4)+S_{\zeta_2}(g_1-g_2, f_2+f_3-f_4)+S_{\zeta_2}(g_2, f_3-f_4).
	\end{align*}
We can also rewrite $S_{\zeta_3}(g_3, f_3)-S_{\zeta_4}(g_4, f_4)$ as
	\begin{align*}
	    S_{\zeta_3}(g_3, f_3)-S_{\zeta_4}(g_4, f_4)=& \Big(S_{\zeta_3}-S_{\zeta_4}\Big)(g_3, f_3)+S_{\zeta_4}(g_3-g_4, f_3)+S_{\zeta_4}(g_4, f_3-f_4).
	\end{align*}
So we can write 
	\begin{align*}
	    &S_{\zeta_1}(g_1, f_2+f_3-f_4)-S_{\zeta_2}(g_2, f_2)-\Big(S_{\zeta_3}(g_3, f_3)-S_{\zeta_4}(g_4, f_4)\Big)\\
	    =&\Big\{(S_{\zeta_1}-S_{\zeta_2})(g_1, f_2+f_3-f_4)-(S_{\zeta_3}-S_{\zeta_4}\Big)(g_3, f_3)\Big\}\\
	    &+\Big\{S_{\zeta_2}(g_1-g_2, f_2+f_3-f_4)-S_{\zeta_4}(g_3-g_4, f_3)\Big\}\\
	    &+\Big\{S_{\zeta_2}(g_2, f_3-f_4)-S_{\zeta_4}(g_4, f_3-f_4)\Big\}\\
	    :=&\it{I}+\it{II}+\it{III}.
	\end{align*}
We rewrite $I$ as 
\begin{align*}
I=& \Big(S_{\zeta_1}-S_{\zeta_2}-(S_{\zeta_3}-S_{\zeta_4})\Big)(g_1, f_2+f_3-f_4)\\
&+\Big(S_{\zeta_3}-S_{\zeta_4}\Big)(g_1-g_3, f_2+f_3-f_4)\\
&+\Big(S_{\zeta_3}-S_{\zeta_4}\Big)(g_3, f_2-f_4)\\
:=& I_1+I_2+I_3.
\end{align*}
Next, $\it{II}$ can be rewritten as 
\begin{align*}
    \it{II}=& S_{\zeta_2}\Big(g_1-g_2-(g_3-g_4), f_2+f_3-f_4\Big)+S_{\zeta_2}\Big(g_3-g_4, (f_2+f_3-f_4)-f_3\Big)\\
    &+\Big(S_{\zeta_2}-S_{\zeta_4}\Big)(g_3-g_4, f_3)\\
    =&S_{\zeta_2}\Big(g_1-g_2-(g_3-g_4), f_2+f_3-f_4\Big)\\
    &+S_{\zeta_2}\Big(g_3-g_4, f_2-f_4\Big)\\
    &+\Big(S_{\zeta_2}-S_{\zeta_4}\Big)(g_3-g_4, f_3)\\
    :=& \it{II}_1+\it{II}_2+\it{II}_3.
\end{align*}
We expand $\it{III}$ as 
\begin{align*}
    \it{III}=& S_{\zeta_2}(g_2-g_4, f_3-f_4)+\Big(S_{\zeta_2}-S_{\zeta_4}\Big)(g_4, f_3-f_4)\\
    :=& \it{III}_1+\it{III}_2.
\end{align*}
Finally, we notice that all $\it{I}_j,\,\it{II}_j,\,j=1,2,3$ and $\it{III}_j,\,j=1.2$  can be estimated separately. These estimates give the desired result. 
	\end{proof}

\subsection{Proof of Lemma \ref{lemma:quarticdifference}}\label{proof:lemmadifferencequartic}
\begin{proof}
	We first regroup the expression we are interested in as
	\begin{equation}\label{cubicdifferencefour}
	\begin{split}
	    &\sum_{j=1}^4 (-1)^{j-1}G_{\zeta_j}(g_j, h_j, f_j)\\
	    =& G_{\zeta_1}(g_1, h_1, f_1)-G_{\zeta_2}(g_1, h_2, f_2)-(G_{\zeta_3}(g_1, h_3, f_3)-G_{\zeta_4}(g_1, h_4, f_4))\\
	     &-G_{\zeta_2}(g_2-g_1, h_2, f_2)-(G_{\zeta_3}(g_3-g_1, h_3, f_3)- G_{\zeta_4}(g_4-g_1, h_4, f_4)).
	     \end{split}
	\end{equation}
	Note that we can estimate $G_{\zeta_1}(g_1, h_1, f_1)-G_{\zeta_2}(g_1, h_2, f_2)-\Big(G_{\zeta_3}(g_1, h_3, f_3)-G_{\zeta_4}(g_1, h_4, f_4)\Big)$ by the same way as $\Big(S_{\zeta_1}-S_{\zeta_2}-(S_{\zeta_3}-S_{\zeta_4})\Big)(f,g)$. 
	
	Explicitly, We have
	\begin{equation}\label{cubictoquadratic}
	\begin{split}
	    & G_{\zeta_1}(g_1, h_1, f_1)-G_{\zeta_2}(g_1, h_2, f_2)-\Big(G_{\zeta_3}(g_1, h_3, f_3)-G_{\zeta_4}(g_1, h_4, f_4)\Big)\\
	    =& G_{\zeta_1-\zeta_2-(\zeta_3-\zeta_4)}(g_1, h_1, f_2+f_3-f_4)\\
	    &+\Big(G_{\zeta_3}-G_{\zeta_4}\Big)(g_1, h_1-h_3, f_2+f_3-f_4)\\
	    &+\Big(G_{\zeta_3}-G_{\zeta_4}\Big)(g_1, h_3, f_2-f_4)\\
	    &+ G_{\zeta_2}(g_1, h_1-h_2-(h_3-h_4), f_2+f_3-f_4)\\
	    &+ G_{\zeta_2}(g_1, h_3-h_4, f_2-f_4)\\
	    &+\Big(G_{\zeta_2}-G_{\zeta_4}\Big)(g_1, h_3-h_4, f_3)\\\
	    &+G_{\zeta_2}(g_1, h_2-h_4, f_3-f_4)\\
	    &+\Big(G_{\zeta_2}-G_{\zeta_4}\Big)(g_1, h_4, f_3-f_4).
	 \end{split}
	\end{equation}
We still need to compute $$-G_{\zeta_2}(g_2-g_1, h_2, f_2)-\Big(G_{\zeta_3}(g_3-g_1, h_3, f_3)- G_{\zeta_4}(g_4-g_1, h_4, f_4)\Big)$$ and explore the cancellations. 

Note that
\begin{align*}
&G_{\zeta_{3}}\left(g_{3}-g_{1,}h_{3},f_{3}\right)-G_{\zeta_{4}}\left(g_{4}-g_{1},h_{4},f_{4}\right)\\
=& \Big(G_{\zeta_{3}}\left(g_{3}-g_{1},h_{3},f_{3}\right)-G_{\zeta_{4}}\left(g_{3}-g_{1},h_{3},f_{3}\right)\Big)+\Big(G_{\zeta_{4}}\left(g_{3}-g_{1},h_{3},f_{3}\right)-G_{\zeta_{4}}\left(g_{3}-g_{1},h_{4},f_{3}\right)\Big)\\
&+\Big(G_{\zeta_{4}}\left(g_{3}-g_{1},h_{4},f_{3}\right)-G_{\zeta_{4}}\left(g_{3}-g_{1},h_{4},f_{4}\right)\Big)+\Big(G_{\zeta_4}\left(g_3-g_1, h_4, f_4\right)-G_{\zeta_4}\left(g_4-g_1, h_4, f_4\right)\Big)\\
=& \Big(G_{\zeta_3}-G_{\zeta_4}\Big)\left(g_{3}-g_{1},h_{3},f_{3}\right)+G_{\zeta_4}\left(g_3-g_1, h_3-h_4, f_3\right)+G_{\zeta_4}\left(g_3-g_1, h_4, f_3-f_4\right)\\
&+G_{\zeta_4}\left( g_3-g_4, h_4, f_4\right).
\end{align*}
Moreover, we have
\begin{align*}
    & -G_{\zeta_2}\left(g_2-g_1, h_2, f_2\right)-G_{\zeta_4}\left( g_3-g_4, h_4, f_4\right)\\
    =&\Big(G_{\zeta_2}\left(g_1-g_2, h_2, f_2\right)-G_{\zeta_4}\left(g_1-g_2, h_2, f_2 \right)\Big)+\Big(G_{\zeta_4}\left(g_1-g_2, h_2, f_2 \right)-G_{\zeta_4}\left(g_1-g_2, h_4, f_2 \right)\Big)\\
    &+\Big(G_{\zeta_4}\left(g_1-g_2, h_4, f_2 \right)-G_{\zeta_4}\left(g_1-g_2, h_4, f_4 \right)\Big)+\Big(G_{\zeta_4}\left(g_1-g_2, h_4, f_4 \right)-G_{\zeta_4}\left(g_3-g_4, h_4, f_4 \right)\Big)\\
    =& \Big(G_{\zeta_2}-G_{\zeta_4}\Big)\left(g_1-g_2, h_2, f_2\right)+G_{\zeta_4}\left(g_1-g_2, h_2-h_4, f_2\right)\\
    &+G_{\zeta_4}\left(g_1-g_2, h_4, f_2-f_4\right)+G_{\zeta_4}\left(g_1-g_2-(g_3-g_4), h_4, f_4\right).
\end{align*}
So we can conclude that
\begin{equation}\label{cubicdifferencethree}
    \begin{split}
        &-G_{\zeta_2}(g_2-g_1, h_2, f_2)-\Big(G_{\zeta_3}(g_3-g_1, h_3, f_3)- G_{\zeta_4}(g_4-g_1, h_4, f_4)\Big)\\
        =& \Big(G_{\zeta_3}-G_{\zeta_4}\Big)\left(g_{1}-g_{3},h_{3},f_{3}\right)+G_{\zeta_4}\left(g_1-g_3, h_3-h_4, f_3\right)+G_{\zeta_4}\left(g_1-g_3, h_4, f_3-f_4\right)\\
        &+\Big(G_{\zeta_2}-G_{\zeta_4}\Big)\left(g_1-g_2, h_2, f_2\right)+G_{\zeta_4}\left(g_1-g_2, h_2-h_4, f_2\right)\\
    &+G_{\zeta_4}\left(g_1-g_2, h_4, f_2-f_4\right)+G_{\zeta_4}\left(g_1-g_2-(g_3-g_4), h_4, f_4\right)
    \end{split}
\end{equation}

Combing all pieces above with the assumptions and using the same proof of Proposition \ref{commutator:difference:quadratic} and Proposition \ref{commutator:difference:cubic}, we obtain the desired estimate.
\end{proof}

	\section{Instability of the NLS}\label{sec:NLS}
In this appendix, we provide the details on the instability of the Stokes wave in the setting of the NLS problem.  Here to abuse the notation, we consider the NLS on $q\mathbb{T}$ with $q\in\mathbb{R}^{+}$. Recall that in the main body of the article, the function $B$ which solves the NLS is defined on $q_1 \mathbb{T}$. 
	\subsection{Basic setting}

Consider the focusing cubic NLS in one dimension
\begin{equation}
i\partial_{t}u+\partial_{x}^{2}u+\left|u\right|^{2}u=0.\label{eq:nls}
\end{equation}
We are interested in the instability of the special solution given by the
Stokes wave
\begin{equation}
u\left(t,x\right)=e^{it}.\label{eq:stokes}
\end{equation}
Consider the perturbation of the form
\begin{equation}
u(t,x)=e^{it}\left(1+w\left(t,x\right)\right)\label{eq:data}
\end{equation}
where $x\in q\mathbb{T}$.

Plugging the ansatz \eqref{eq:data} above into the equation \eqref{eq:nls},
we have
\begin{equation}
i\partial_{t}w+\partial_{x}^{2}w+2\Re w=N\left(w\right)\label{eq:eqW}
\end{equation}
where
\begin{equation}
N(w)=-\left(\left|1+w\right|^{2}-1\right)\left(1+w\right)+2\Re w.\label{eq:nonlinear}
\end{equation}
Note that
\begin{align*}
\left(\left|1+w\right|^{2}-1\right)\left(1+w\right) & =\left(1+\left|w\right|^{2}+2\Re w-1\right)\left(1+w\right)\\
 & =\left(\left|w\right|^{2}+2\Re w\right)\left(1+w\right)\\
 & =\left|w\right|^{2}+\left|w\right|^{2}w+2\Re w+2w\Re w.
\end{align*}
Therefore, we have
\begin{equation}
N(w)=\left|w\right|^{2}+\left|w\right|^{2}w+2w\Re w\label{eq:nonlinearD}
\end{equation}
and
\[
N(w)\sim\mathcal{O}(w^{2}).
\]

\subsection{First order system}

Next we analyze the linear part of the equation \eqref{eq:eqW}:
\begin{equation}
i\partial_{t}w+\partial_{x}^{2}w+2\Re w=0.\label{eq:linear}
\end{equation}
Working on the  circle $q\mathbb{T}$, we can also rewrite \eqref{eq:linear}
as a first order system

\begin{equation}
\partial_{t}\left(\begin{array}{c}
\phi\\
\psi
\end{array}\right)=\left(\begin{array}{cc}
0 & -\partial_{xx}\\
\partial_{xx}+2&0
\end{array}\right)\left(\begin{array}{c}
\phi\\
\psi
\end{array}\right)\label{eq:matrix1}
\end{equation}
where $\phi$ and $\psi$ are given as
\begin{equation}
w=:\phi+i\psi.\label{eq:complex}
\end{equation}
Denoting $\bm{w}=\left(\begin{array}{c}
\phi\\
\psi
\end{array}\right)$, the equation \eqref{eq:eqW} can be rewritten as
\begin{equation}
\partial_{t}\bm{w}=J\mathcal{L}\bm{w}+\bm{N},\,\hm{w}\left(0\right)=\hm{w}_{0}\label{eq:m2}
\end{equation}
where
\begin{equation}
J=\left(\begin{array}{cc}
0 & 1\\
-1 & 0
\end{array}\right)\label{eq:J}
\end{equation}
and
\begin{equation}
\mathcal{L}=\left(\begin{array}{cc}
-\partial_{xx}-2&0\\
0 & -\partial_{xx}
\end{array}\right),\ \bm{N}=\left(\begin{array}{c}
g\\
-f
\end{array}\right)\label{eq:LN}
\end{equation}
where $g$ and $f$ are defined by
\[
N=:f+ig.
\]
By the formalism above, our problem is written as a canonical Hamiltonian
system. Using the Duhamel formula, we can write the solution to \eqref{eq:m2}
as
\begin{equation}
\hm{w}\left(t\right)=e^{tJ\mathcal{L}}\hm{w}_{0}+\int_{0}^{t}e^{\left(t-s\right)J\mathcal{L}}\hm{N}\left(s\right)\,ds.\label{eq:duhamel}
\end{equation}
Recall that for a function $f$ on $q\mathbb{T}$, the Fourier series
of $f$ is given by
\begin{equation}
f\left(x\right)=\frac{1}{q}\sum_{k\in\mathbb{Z}}f_{k}e^{i\frac{k}{q}x}\label{eq:fourier}
\end{equation}
where
\[
f_{k}=\frac{1}{2\pi}\int_{q\mathbb{T}}f(x)e^{-i\frac{k}{q}x}\,dx.
\]
For a nice function $m$, we define the Fourier multiplier as
\begin{equation}
m\left(\nabla\right)f:=\frac{1}{q}\sum_{k\in\mathbb{Z}}m\left(i\frac{k}{q}\right)f_{k}e^{i\frac{k}{q}x}.\label{eq:fourier1}
\end{equation}
With explicit formulae above, we can compute $e^{tJ\mathcal{L}}$
explicitly using the Fourier series.
\begin{lemma}
\label{lem:linearSol}Given notations above, the evolution operator
$e^{tJ\mathcal{L}}$ can be written as the following:

For $\left|\nabla\right|=\left|\frac{k}{q}\right|\leq\sqrt{2}$, one has
\[
e^{tJ\mathcal{L}}=\left(\begin{array}{cc}
\cosh\left(t\left|\nabla\right|\sqrt{2-\left|\nabla\right|^{2}}\right) & \frac{\sinh\left(t\left|\nabla\right|\sqrt{2-\left|\nabla\right|^{2}}\right)}{\left|\nabla\right|\sqrt{2-\left|\nabla\right|^{2}}}\left|\nabla\right|^{2}\\
\frac{\sinh\left(t\left|\nabla\right|\sqrt{2-\left|\nabla\right|^{2}}\right)}{\left|\nabla\right|\sqrt{2-\left|\nabla\right|^{2}}}\left(2-\left|\nabla\right|^{2}\right) & \cosh\left(t\left|\nabla\right|\sqrt{2-\left|\nabla\right|^{2}}\right)
\end{array}\right)
\]
and for $\left|\nabla\right|=\left|\frac{k}{q}\right|>\sqrt{2}$, we have
\[
e^{tJ\mathcal{L}}=\left(\begin{array}{cc}
\cos\left(t\left|\nabla\right|\sqrt{\left|\nabla\right|^{2}-2}\right) & \frac{\sin\left(t\left|\nabla\right|\sqrt{\left|\nabla\right|^{2}-2}\right)}{\left|\nabla\right|\sqrt{\left|\nabla\right|^{2}-2}}\left|\nabla\right|^{2}\\
\frac{\sin\left(t\left|\nabla\right|\sqrt{\left|\nabla\right|^{2}-2}\right)}{\left|\nabla\right|\sqrt{\left|\nabla\right|^{2}-2}}\left(2-\left|\nabla\right|^{2}\right) & \cos\left(t\left|\nabla\right|\sqrt{\left|\nabla\right|^{2}-2}\right)
\end{array}\right).
\]
\end{lemma}
\begin{rem}
The problem in the full line $\mathbb{R}$ was computed in Mu\~noz \cite{MunozNLS} using a slightly different formalism.
\end{rem}
\begin{proof}
Consider the linear problem
\[
\partial_{t}\bm{w}=J\mathcal{L}\bm{w},\,\hm{w}\left(0\right)=\hm{w}_{0}.
\]
Using the notation \eqref{eq:matrix1}, we can rewrite the problem as
\begin{equation}
\partial_{t}\left(\begin{array}{c}
\phi\\
\psi
\end{array}\right)=\left(\begin{array}{cc}
0 & -\partial_{xx}\\
\partial_{xx}+2&0
\end{array}\right)\left(\begin{array}{c}
\phi\\
\psi
\end{array}\right)\label{eq:matrix11}
\end{equation}
with
\[
\hm{w}=\left(\begin{array}{c}
\phi\\
\psi
\end{array}\right),\,\hm{w}_{0}=\left(\begin{array}{c}
\phi^{0}\\
\psi^{0}
\end{array}\right).
\]
From the equation \eqref{eq:matrix11}, we have
\[
\begin{cases}
\partial_{t}\phi=-\partial_{xx}\psi\\
\partial_{t}\psi=\partial_{xx}\phi+2\phi
\end{cases}
\]
which implies
\begin{equation}
\begin{cases}
\partial_{tt}\phi=-\partial_{x}^{4}\phi-2\partial_{x}^{2}\phi\\
\phi\left(0\right)=\phi^{0},\,\partial_{t}\phi\left(0\right)=-\partial_{x}^{2}\psi^{0}
\end{cases}\label{eq:wave1}
\end{equation}
and
\begin{equation}
\begin{cases}
\partial_{tt}\psi=-\partial_{x}^{4}\phi-2\partial_{x}^{2}\phi\\
\psi\left(0\right)=\psi^{0},\,\partial_{t}\psi\left(0\right)=\left(\partial_{x}^{2}+2\right)\phi^{0}
\end{cases}.\label{eq:wave2}
\end{equation}
To solve \eqref{eq:wave1} and \eqref{eq:wave2} using the Fourier series
on $q\mathbb{T}$ now are standard. We only illustrate the idea by
solving \eqref{eq:wave1}. Expanding $\phi$ by the Fourier series as \eqref{eq:fourier},
we obtain that
\begin{equation}
\begin{cases}
\partial_{t}^{2}\phi_{k}=-\left(i\frac{k}{q}\right)^{4}\phi_{k}-2\left(i\frac{k}{q}\right)^{2}\phi_{k}\\
\phi_{k}(0)=\left(\phi^{0}\right)_{k},\,\left(\phi_{k}\right)_{t}\left(0\right)=-\left(i\frac{k}{q}\right)^{2}\left(\psi^{0}\right)_{k}
\end{cases}.\label{eq:ode1}
\end{equation}
Solving the ODE above, we conclude that
\begin{equation}
\phi_{k}\left(t\right)=\begin{cases}
\cosh\left(t\left|\frac{k}{q}\right|\sqrt{2-\left|\frac{k}{q}\right|^{2}}\right)\phi_{k}(0)+\frac{\sinh\left(t\left|\frac{k}{q}\right|\sqrt{2-\left|\frac{k}{q}\right|^{2}}\right)}{\left|\frac{k}{q}\right|\sqrt{2-\left|\frac{k}{q}\right|^{2}}}\left(\phi_{k}\right)_{t}\left(0\right) & \left|\frac{k}{q}\right|\leq\sqrt{2}\\
\cos\left(t\left|\frac{k}{q}\right|\sqrt{\left|\frac{k}{q}\right|^{2}-2}\right)\phi_{k}(0)+\frac{\sin\left(t\left|\frac{k}{q}\right|\sqrt{\left|\frac{k}{q}\right|^{2}-2}\right)}{\left|\frac{k}{q}\right|\sqrt{\left|\frac{k}{q}\right|^{2}-2}}\left(\phi_{k}\right)_{t}\left(0\right) & \left|\frac{k}{q}\right|>\sqrt{2}
\end{cases}.\label{eq:phik}
\end{equation}
After solving the problem for $\psi$ in a similar manner, we obtain that for $\left|\frac{k}{q}\right|\leq\sqrt{2}$
\begin{equation}
\left(\begin{array}{c}
\phi_{k}\left(t\right)\\
\psi_{k}\left(t\right)
\end{array}\right)=\left(\begin{array}{cc}
\cosh\left(t\left|\frac{k}{q}\right|\sqrt{2-\left|\frac{k}{q}\right|^{2}}\right) & \frac{\sinh\left(t\left|\frac{k}{q}\right|\sqrt{2-\left|\frac{k}{q}\right|^{2}}\right)}{\left|\frac{k}{q}\right|\sqrt{2-\left|\frac{k}{q}\right|^{2}}}\left|\frac{k}{q}\right|^{2}\\
\frac{\sinh\left(t\left|\frac{k}{q}\right|\sqrt{2-\left|\frac{k}{q}\right|^{2}}\right)}{\left|\frac{k}{q}\right|\sqrt{2-\left|\frac{k}{q}\right|^{2}}}\left(2-\left|\frac{k}{q}\right|^{2}\right) & \cosh\left(t\left|\frac{k}{q}\right|\sqrt{2-\left|\frac{k}{q}\right|^{2}}\right)
\end{array}\right)\left(\begin{array}{c}
\left(\phi^{0}\right)_{k}\\
\left(\psi^{0}\right)_{k}
\end{array}\right)\label{eq:odeso1}
\end{equation}
and for $\left|\frac{k}{q}\right|>\sqrt{2}$
\begin{equation}
\left(\begin{array}{c}
\phi_{k}\left(t\right)\\
\psi_{k}\left(t\right)
\end{array}\right)=\left(\begin{array}{cc}
\cos\left(t\left|\frac{k}{q}\right|\sqrt{\left|\frac{k}{q}\right|^{2}-2}\right) & \frac{\cos\left(t\left|\frac{k}{q}\right|\sqrt{\left|\frac{k}{q}\right|^{2}-2}\right)}{\left|\frac{k}{q}\right|\sqrt{\left|\frac{k}{q}\right|^{2}-2}}\left|\frac{k}{q}\right|^{2}\\
\frac{\sin\left(t\left|\frac{k}{q}\right|\sqrt{\left|\frac{k}{q}\right|^{2}-2}\right)}{\left|\frac{k}{q}\right|\sqrt{\left|\frac{k}{q}\right|^{2}-2}}\left(2-\left|\frac{k}{q}\right|^{2}\right) & \cos\left(t\left|\frac{k}{q}\right|\sqrt{\left|\frac{k}{q}\right|^{2}-2}\right)
\end{array}\right)\left(\begin{array}{c}
\left(\phi^{0}\right)_{k}\\
\left(\psi^{0}\right)_{k}
\end{array}\right).\label{eq:odeso2}
\end{equation}
Using the multiplier notation \eqref{eq:fourier1}, we conclude the
desired results.
\end{proof}
From the explicit computations in Lemma \ref{lem:linearSol}, we can
read off the growth rate of the linear flow $e^{tJ\mathcal{L}}$ directly.

\begin{cor}
\label{cor:growthrate}Consider the first order system
\[
\partial_{t}\bm{w}=J\mathcal{L}\bm{w},\,\hm{w}(0)=\hm{w}_{0}.
\]
Define
\begin{equation}
\tau=\sup_{k\in\mathbb{Z}}\Re\left|\frac{k}{q}\right|\sqrt{2-\left|\frac{k}{q}\right|^{2}}.\label{eq:nlsrate}
\end{equation}
Then for any $s'\geq0$, we have for $t\geq0$
\begin{equation}
\left\Vert \hm{w}\left(t\right)\right\Vert _{H^{s'}}\lesssim e^{\tau t}\left\Vert \hm{w}_{0}\right\Vert _{H^{s'}}.\label{eq:NLSupper}
\end{equation}
\end{cor}

\begin{proof}
This follows from the explicit computations above. After taking the Fourier series, by a direct inspection,
from formulae \eqref{eq:odeso1} and \eqref{eq:odeso2}, when $\left|\frac{k}{q}\right|<2$,
the linear flow will result in the exponential growth. More precisely
with notations from \eqref{eq:odeso1}, we have
\[
\phi_{k}\left(t\right)=\cosh\left(t\left|\frac{k}{q}\right|\sqrt{\left|\frac{k}{q}\right|^{2}-2}\right)\left(\phi^{0}\right)_{k}+\sinh\left(t\left|\frac{k}{q}\right|\sqrt{2-\left|\frac{k}{q}\right|^{2}}\right)\frac{\left|\frac{k}{q}\right|}{\sqrt{2-\left|\frac{k}{q}\right|^{2}}}\left(\psi^{0}\right)_{k}
\]
and
\[
\psi_{k}\left(t\right)=\sinh\left(t\left|\frac{k}{q}\right|\sqrt{2-\left|\frac{k}{q}\right|^{2}}\right)\frac{\sqrt{2-\left|\frac{k}{q}\right|^{2}}}{\left|\frac{k}{q}\right|}\left(\phi^{0}\right)_{k}+\cosh\left(t\left|\frac{k}{q}\right|\sqrt{2-\left|\frac{k}{q}\right|^{2}}\right)\left(\psi^{0}\right)_{k}.
\]
These Fourier coefficients have the exponential growth rate $e^{t\left|\frac{k}{q}\right|\sqrt{2-\left|\frac{k}{q}\right|^{2}}}$
provided that
\begin{equation}\label{eq:growthcond}
\left(\phi^{0}\right)_{k}+\frac{1}{\sqrt{2-\left|\frac{k}{q}\right|^{2}}}\left|\frac{k}{q}\right|\left(\psi^{0}\right)_{k}\neq0.
\end{equation}
The desired result follows from the Fourier representation of the
solution $\hm{w}(t)$. 
\end{proof}

\subsection{Nonlinear problem}\label{subsec:nonlinear}
Given $0<\delta\ll 1$ and $s'>\frac{1}{2}$ fixed, we consider the nonlinear equation
\begin{equation}
\partial_{t}\bm{w}=J\mathcal{L}\bm{w}+\mathcal{N}\left(\hm{w}\right)\label{eq:mnp}
\end{equation}
on $q\mathbb{T}$ with initial data given in the complex form by
\begin{equation}
w_{0}=\frac{1}{\sqrt{q}}\left(\delta_{1}e^{i\frac{k_0}{q}x}+\delta_{2}e^{-i\frac{k_0}{q}x}+\eta_{1}e^{i\frac{x}{q}}+\eta_{2}e^{-i\frac{x}{q}}\right)\label{eq:initial}
\end{equation}
where $k_0\in\mathbb{Z}^{+}$  is defined as
$$\left|\frac{k_0}{q}\right|\sqrt{2-\left|\frac{k_0}{q}\right|^{2}}=\tau$$
and  $\left|\delta_{j}\right|=\frac{\delta}{2s'}\ll1$, $\left|\eta_{j}\right|\ll\left|\delta_{i}\right|$. It is straightforward to check that \eqref{eq:initial} satisfies $$\Vert \hm{w}_0\Vert_{H^{s'}(q\mathbb{T})}\leq\frac{3}{2}\delta$$
and the condition \eqref{eq:growthcond} for $k=k_0$.

Recall that by construction, $\mathcal{N}\left(\hm{w}\right)$ consists of quadratic
and cubic terms in $\hm{w}$.

\begin{thm}\label{thm:NLSinsta}
Consider the nonlinear equation \eqref{eq:mnp} with the initial data
\eqref{eq:initial}.
  Then there exist $\mu$ satisfying $\left|\delta\right|\ll\mu<1$
and $T_{0}=\log\left(\frac{\mu}{\delta}\right)$
such that
\begin{equation}
\left\Vert \hm{w}\left(t\right)\right\Vert _{H^{s'}}\lesssim\mu,\,\forall t\in\left[0,T_{0}\right]\label{eq:H1bound}
\end{equation}
and
\begin{equation}
\left\Vert \hm{w}\left(T_{0}\right)\right\Vert _{H^{s'}}\geq\frac{1}{4}\mu\gg\delta.\label{eq:unstablebound}
\end{equation}
 
\end{thm}

\begin{proof}
By the Duhamel formula, one has
\begin{equation}
\hm{w}\left(t\right)=\hm{w}_{L}(t)+\int_{0}^{t}e^{J\mathcal{L}\left(t-s\right)}\mathcal{N}\left(\hm{w}(s)\right)\,ds.\label{eq:duhamel1}
\end{equation}
\begin{equation}
\hm{w}_{L}=e^{J\mathcal{L}t}\hm{w}_{0}.\label{eq:linearM}
\end{equation}
Then clearly by construction, it follows that
\begin{equation}
\frac{1}{2}\delta e^{\tau t}\leq\left\Vert \hm{w}_{L}(t)\right\Vert _{H^{s'}}\leq\frac{3}{2}\ensuremath{\delta e^{\tau t}}.\label{eq:linearG}
\end{equation}
Define $T_{1}$ as
\begin{equation}
T_{1}:=\sup\left\{ T\geq0:\sup_{t\in\left[0,T\right]}\left\Vert \hm{w}(t)\right\Vert _{H^{s'}}\leq2\ensuremath{\delta e^{\tau t}}\right\} \label{eq:T1}
\end{equation}
and $T_{0}$ as
\begin{equation}
\delta\ll\delta e^{\tau T_{0}}=\mu<1\label{eq:T0}
\end{equation}
where $\mu$ is to be determined later. 

Note that $T_{0}\leq T_{1}$ if $\mu$ is small. Otherwise, by
contradiction, we assume that $T_{0}>T_{1}$. Taking the Sobolev norms of both side of \eqref{eq:duhamel1} and applying Corollary \ref{cor:growthrate}, we have
\begin{align}
\left\Vert \hm{w}\left(T_{1}\right)\right\Vert _{H^{s'}} & \leq\left\Vert \hm{w}_{L}\left(T_{1}\right)\right\Vert _{H^{s'}}+\int_{0}^{t}e^{\tau\left(t-s\right)}\left(4\delta^{2}e^{2\tau s}\right)\,ds\nonumber \\
 & \leq\frac{3}{2}\ensuremath{\delta e^{\tau t}}+4\mu\ensuremath{\delta e^{\tau t}}\label{eq:H1est}
\end{align}
where we used $H^{s}$ with $s>\frac{1}{2}$ is an algebra in $q\mathbb{T}$.
When $\mu$ is small, clearly, the  estimate above contradicts the
definition of $T_{1}$. Therefore indeed, we have $T_{0}\leq T_{1}$.

Evaluating $\hm{w}(t)$ at $t=T_{0}$, we have
\begin{align}
\left\Vert \hm{w}\left(T_{0}\right)\right\Vert _{L^{2}} & \geq\left\Vert \hm{w}_{L}\left(T_{0}\right)\right\Vert _{L^{2}}-\left\Vert \int_{0}^{T_{0}}e^{J\mathcal{L}\left(T_{0}-s\right)}\mathcal{N}\left(\hm{w}(s)\right)\,ds\right\Vert _{L^{2}}\nonumber \\
 & \geq\frac{1}{2}\mu-4\mu^{2}\geq\frac{1}{4}\mu\label{eq:L2instab}
\end{align}
provided that $\mu$ is small. In particular, we know that
\begin{equation}
\left\Vert \hm{w}\left(T_{0}\right)\right\Vert _{H^{s'}}\geq\frac{1}{4}\mu\label{eq:H1insta}
\end{equation}
as desired.
\end{proof}
\begin{rem}
The instability argument above holds for all initial data such that the corresponding linear flow satisfies \eqref{eq:linearG}.
\end{rem}

\section{List of notations}\label{appendix:notation}

\begin{tabular}{ll}
  $\Sigma(t)$ & The free boundary at time $t$ \\
  $\Omega(t)$ & The fluid region at time $t$\\
  $P(x+iy,t)$ & The pressure at the position $x+iy\in \Omega(t)$ at time $t$\\
  $v(x+iy,t)$ & The velocity field of the water waves at the position $x+iy\in \Omega(t)$ at time $t$ \\
    $\zeta$ & The labeling of the free interface in Wu's coordinates\\
    $\zeta_{ST}$ & The labeling of the free surface of a Stokes wave in Wu's coordinates\\
  $\epsilon$ & The leading order of the amplitude of a given Stokes wave. $0<\epsilon\ll 1$\\
  $\omega$ & The wave speed of a given Stokes wave\\
  $X$ & $X:=\epsilon(\alpha+\frac{1}{2\omega}t)$\\
  $T$ & $T:=\epsilon^2 t$\\
  $B(X,T)$ & A given solution to the NLS, $B=i+perturbation$\\
  $\mathcal{H}_{\zeta}$ & The Hilbert transform associated with a curve parametrized by $\zeta$\\
  $\mathcal{K}_{\zeta}$ & The double layer potential associated with a curve parametrized by $\zeta$\\
  $\mathcal{K}_{\zeta}^{\ast}$ & The adjoint of $\mathcal{K}_\zeta$\\
  $b(\alpha,t)$ &  A real-valued function associated with $\zeta(\alpha,t)$ and $D_t\zeta$\\
  $b_{ST}$ & A real-valued function associated with $z_{ST}$ and $\partial_t z_{ST}$\\
  $D_t$ & $D_t:=\partial_t+b\partial_{\alpha}$\\
  $D_t^{ST}$ & $D_t^{ST}:=\partial_t+b_{ST}\partial_{\alpha}$\\
  $q_1$ & A given positive number.\\
  \end{tabular}
  
  \begin{tabular}{ll}
  $q$ & $q:=\frac{q_1}{\epsilon}\in \mathbb{Q}_+$, the wavelength of the perturbation of the water waves \\
  $\mathbb{T}$ & The standard torus $\mathbb{R}/2\pi$\\
  $\zeta^{(1)}$ & $\zeta^{(1)}:=B(X,T)e^{i\alpha+i\omega t}$\\
  $\zeta_{ST}^{(1)}$ & $\zeta_{ST}^{(1)}=ie^{i\alpha+i\omega t}$\\
  $\zeta^{(2)}$ & $\zeta^{(2)}:=-\omega |B|^2$\\
  $\zeta_{ST}^{(2)}$ & $\zeta_{ST}^{(2)}:=i$\\
  $\zeta^{(3)}$ & $\zeta^{(3)}:=-\frac{1}{2}\bar{\zeta}^{(1)}H_0(|B|^2-i\bar{B})+\bar{B}\bar{B}_X+\zeta_{ST}^{(3)}$\\
  $\zeta_{ST}^{(3)}$ & $\zeta_{ST}^{(3)}:=\frac{i}{2}e^{i\alpha+i\omega t}$\\
  \end{tabular}\\
  
  \begin{tabular}{ll}
  $\mathcal{H}_{\zeta}$ & $\mathcal{H}_{\zeta}$ has the expansion $\mathcal{H}_{\zeta}=H_0+\epsilon H_1+\epsilon^2 H_2+O(\epsilon^3)$\\
  $H_0$ & The flat Hilbert transform on $q\mathbb{T}$\\
  $H_1$  & $H_1f:=[\zeta^{(1)}, H_0]\partial_{\alpha}f$\\
  $H_2$ \quad \quad\quad  & $H_2f:=[\zeta^{(2)}, H_0]f_{\alpha}-[\zeta^{(1)}, H_0]\zeta_{\alpha}^{(1)}f_{\alpha}+\frac{1}{2}[\zeta^{(1)}, [\zeta^{(1)}, H_0]]\partial_{\alpha}^2f$\\
  $\tilde{\zeta}$ & $\tilde{\zeta}:=\alpha+\epsilon\zeta^{(1)}+\epsilon^2 \zeta^{(2)}+\epsilon^3\zeta^{(3)}$\\
  $\tilde{\zeta}_{ST}$ & $\tilde{\zeta}_{ST}:=\alpha+\epsilon\zeta_{ST}^{(1)}+\epsilon^2 \zeta_{ST}^{(2)}+\epsilon^3\zeta_{ST}^{(3)}$\\
  $\zeta_{app}$ & $\zeta_{app}:=\zeta_{ST}+(\tilde{\zeta}-\tilde{\zeta}_{ST})$\\
    $\tilde{b}$ & $\tilde{b}:=\epsilon^2 b^{(2)}$, $b^{(2)}:=-\omega |B|^2$\\
  $\tilde{D}_t$ & $\tilde{D}_t=\partial_t +\tilde{b}\partial_{\alpha}$\\
  $\tilde{b}_{ST}$ & $\tilde{b}_{ST}:=\epsilon^2 b_{ST}^{(2)}$, $b^{(2)}:=-\omega $\\
  $\tilde{D}_{t}^{ST}$ & $\tilde{D}_t^{ST}:=\partial_t+\tilde{b}_{ST}\partial_{\alpha}$\\
  $\tilde{A}$ & $\tilde{A}=1$\\
  $\tilde{A}_{ST}$ & $\tilde{A}_{ST}=1$\\
    $\mathcal{P}$ & $\mathcal{P}:= D_t^2-iA\partial_{\alpha}$\\
  $\mathcal{P}_{ST}$ & $\mathcal{P}_{ST}:= (D_t^{ST})^2-iA_{ST}\partial_{\alpha}$\\
  $\tilde{\mathcal{P}}$ & $\tilde{\mathcal{P}}:= \tilde{D}_t^2-i\tilde{A}\partial_{\alpha}$\\
   $\tilde{\mathcal{P}}_{ST}$ & $\tilde{\mathcal{P}}_{ST}:= (\tilde{D}_t^{ST})^2-i\tilde{A}_{ST}\partial_{\alpha}$\\
\end{tabular}\\

\begin{tabular}{ll}
   $\mathcal{Q}$ & $\mathcal{Q}:=\mathcal{P}(I-\mathcal{H}_{\zeta})$\\
    $\mathcal{Q}_{ST}$ & $\mathcal{Q}_{ST}:=\mathcal{P}_{ST}(I-\mathcal{H}_{\zeta_{ST}})$\\
   $\tilde{\mathcal{Q}}$ & $\tilde{\mathcal{Q}}:=\tilde{\mathcal{P}}(I-\mathcal{H}_{\tilde{\zeta}})$\\
    $\tilde{\mathcal{Q}}_{ST}$ & $\tilde{\mathcal{Q}}_{ST}:=\tilde{\mathcal{P}}_{ST}(I-\mathcal{H}_{\tilde{\zeta}_{ST}})$\\
    $r$ & $r:=\zeta-\zeta_{ST}-(\tilde{\zeta}-\tilde{\zeta}_{ST})$\\
     $\tilde{\theta}$ & $\tilde{\theta}:=(I-\mathcal{H}_{\tilde{\zeta}})(\tilde{\zeta}-\alpha)$\\	    	    $\tilde{\theta}_{ST}$ & $\tilde{\theta}_{ST}:=(I-\mathcal{H}_{\tilde{\zeta}_{ST}})(\tilde{\zeta}_{ST}-\alpha)$\\
     $\theta$ & $\theta:=(I-\mathcal{H}_{\zeta})(\zeta-\alpha)$\\
	  $\theta_{ST}$ &  $\theta_{ST}:=(I-\mathcal{H}_{\zeta_{ST}})(\zeta_{ST}-\alpha)$\\
    $\rho$ & $\rho:=(I-\mathcal{H}_{\zeta})(\theta-\theta_{ST}-(\tilde{\theta}-\tilde{\theta}_{ST})$\\
    $\sigma$ & $\sigma:=(I-\mathcal{H}_{\zeta})D_t\rho$\\
    $s$ & $s\geq 4$ is fixed\\
    $s'$ & $s'= s+7$ is fixed\\
    $E_s(t)$ & $E_s(t)^{1/2}:=\norm{D_t r(\cdot,t)}_{H^{s+1/2}(q\mathbb{T})}+\norm{(r)_{\alpha}(\cdot,t)}_{H^{s}(q\mathbb{T})}+\norm{D_t^2r(\cdot,t)}_{H^s(q\mathbb{T})}$\\
    $\mathcal{E}_n$  & Defined in \S \ref{subsection:energyfunctional}\\
    $\mathcal{F}_n$ & Defined in \S \ref{subsection:energyfunctional}\\
    $\mathcal{E}$ & $\mathcal{E}:=\sum_{n=0}^s (\mathcal{E}_n+\mathcal{F}_n)$\\
    $\delta$ & The size of the perturbation of the given Stokes waves. Also, \\
    &$\delta\approx \norm{B(\alpha,0)-i}_{H^{s'}(q_1\mathbb{T})}$\\
    $\Re\{f\}$ & The real part of $f$\\
    $\Im\{f\}$ & The imaginary part of $f$
\end{tabular}




	\bibliography{stokes}{}

\begin{thebibliography}{10}

\bibitem{ai2019low}
Albert Ai.
\newblock Low regularity solutions for gravity water waves.
\newblock {\em Water Waves}, 1(1):145--215, 2019.

\bibitem{ai2020low}
Albert Ai.
\newblock Low regularity solutions for gravity water waves ii: The 2d case.
\newblock {\em Annals of PDE}, 6:4, 2020.

\bibitem{ai2019two}
Albert Ai, Mihaela Ifrim, and Daniel Tataru.
\newblock Two dimensional gravity waves at low regularity i: Energy estimates.
\newblock {\em arXiv preprint arXiv:1910.05323}, 2019.

\bibitem{ai2020two}
Albert Ai, Mihaela Ifrim, and Daniel Tataru.
\newblock Two dimensional gravity waves at low regularity ii: Global solutions.
\newblock {\em arXiv preprint arXiv:2009.11513}, 2020.

\bibitem{alazard2014cauchy}
Thomas Alazard, Nicolas Burq, and Claude Zuily.
\newblock On the cauchy problem for gravity water waves.
\newblock {\em Inventiones mathematicae}, 198(1):71--163, 2014.

\bibitem{alazard2015global}
Thomas Alazard and Jean-Marc Delort.
\newblock Global solutions and asymptotic behavior for two dimensional gravity
  water waves.
\newblock {\em Ann. Sci. {\'E}c. Norm. Sup{\'e}r.(4)}, 48(5):1149--1238, 2015.

\bibitem{ambrose2005zero}
D~Ambrose and Nader Masmoudi.
\newblock The zero surface tension limit two-dimensional water waves.
\newblock {\em Communications on pure and applied mathematics},
  58(10):1287--1315, 2005.

\bibitem{beale1993growth}
J~Thomas Beale, Thomas~Y Hou, and John~S Lowengrub.
\newblock Growth rates for the linearized motion of fluid interfaces away from
  equilibrium.
\newblock {\em Communications on Pure and Applied Mathematics},
  46(9):1269--1301, 1993.

\bibitem{benjamin1967disintegration}
T~Brooke Benjamin and JE~Feir.
\newblock The disintegration of wave trains on deep water.
\newblock {\em J. Fluid mech}, 27(3):417--430, 1967.

\bibitem{benjamin1967instability}
Thomas~Brooke Benjamin.
\newblock Instability of periodic wavetrains in nonlinear dispersive systems.
\newblock {\em Proceedings of the Royal Society of London. Series A.
  Mathematical and Physical Sciences}, 299(1456):59--76, 1967.

\bibitem{benney1967propagation}
DJ~Benney and AC~Newell.
\newblock The propagation of nonlinear wave envelopes.
\newblock {\em Journal of mathematics and Physics}, 46(1-4):133--139, 1967.

\bibitem{berti2018birkhoff}
Massimiliano Berti, Roberto Feola, and Fabio Pusateri.
\newblock Birkhoff normal form and long time existence for periodic gravity
  water waves.
\newblock {\em arXiv preprint arXiv:1810.11549}, 2018.

\bibitem{bieri2017motion}
Lydia Bieri, Shuang Miao, Sohrab Shahshahani, and Sijue Wu.
\newblock On the motion of a self-gravitating incompressible fluid with free
  boundary.
\newblock {\em Communications in Mathematical Physics}, 355(1):161--243, 2017.

\bibitem{birkhoff1962helmholtz}
Garrett Birkhoff.
\newblock Helmholtz and taylor instability.
\newblock In {\em Proc. Symp. Appl. Math}, volume~13, pages 55--76, 1962.

\bibitem{bridges1995proof}
Thomas~J Bridges and Alexander Mielke.
\newblock A proof of the benjamin-feir instability.
\newblock {\em Archive for rational mechanics and analysis}, 133(2):145--198,
  1995.

\bibitem{castro2013finite}
Angel Castro, Diego C{\'o}rboda, Charles Fefferman, Francisco Gancedo, and
  Javier G{\'o}mez-Serrano.
\newblock Finite time singularities for the free boundary incompressible euler
  equations.
\newblock {\em Annals of Mathematics}, pages 1061--1134, 2013.

\bibitem{castro2012finite}
Angel Castro, Diego C{\'o}rdoba, Charles Fefferman, Francisco Gancedo, and
  Javier G{\'o}mez-Serrano.
\newblock Finite time singularities for water waves with surface tension.
\newblock {\em Journal of Mathematical Physics}, 53(11):115622, 2012.

\bibitem{christodoulou2000motion}
Demetrios Christodoulou and Hans Lindblad.
\newblock On the motion of the free surface of a liquid.
\newblock {\em Communications on Pure and Applied Mathematics},
  53(12):1536--1602, 2000.

\bibitem{coifman1982integrale}
Ronald~Rapha{\"e}l Coifman, Alan McIntosh, and Yves Meyer.
\newblock L'int{\'e}grale de cauchy d{\'e}finit un op{\'e}rateur born{\'e} sur
  l2 pour les courbes lipschitziennes.
\newblock {\em Annals of Mathematics}, pages 361--387, 1982.

\bibitem{coutand2007well}
Daniel Coutand and Steve Shkoller.
\newblock Well-posedness of the free-surface incompressible euler equations
  with or without surface tension.
\newblock {\em Journal of the American Mathematical Society}, 20(3):829--930,
  2007.

\bibitem{coutand2014finite}
Daniel Coutand and Steve Shkoller.
\newblock On the finite-time splash and splat singularities for the 3-d
  free-surface euler equations.
\newblock {\em Communications in Mathematical Physics}, 325(1):143--183, 2014.

\bibitem{coutand2016impossibility}
Daniel Coutand and Steve Shkoller.
\newblock On the impossibility of finite-time splash singularities for vortex
  sheets.
\newblock {\em Archive for Rational Mechanics and Analysis}, 221(2):987--1033,
  2016.

\bibitem{Craig}
W.~Craig.
\newblock An existence theory for water waves and the boussinesq and
  korteweg-devries scaling limits.
\newblock {\em Comm. in P.D.E}, 10(8):787--1003, 1985.

\bibitem{craig1992nonlinear}
Walter Craig, Catherine Sulem, and Pierre-Louis Sulem.
\newblock Nonlinear modulation of gravity waves: a rigorous approach.
\newblock {\em Nonlinearity}, 5(2):497, 1992.

\bibitem{david1984operateurs}
Guy David.
\newblock Op{\'e}rateurs int{\'e}graux singuliers sur certaines courbes du plan
  complexe.
\newblock In {\em Annales scientifiques de l'{\'E}cole Normale Sup{\'e}rieure},
  volume~17, pages 157--189, 1984.

\bibitem{Deconinck2011}
Bernard Deconinck and Katie Oliveras.
\newblock The instability of periodic surface gravity waves.
\newblock {\em J. Fluid Mech.}, 675:141--167, 2011.

\bibitem{dull2016justification}
Wolf-Patrick D{\"u}ll, Guido Schneider, and C~Eugene Wayne.
\newblock Justification of the nonlinear schr{\"o}dinger equation for the
  evolution of gravity driven 2d surface water waves in a canal of finite
  depth.
\newblock {\em Archive for Rational Mechanics and Analysis}, 220(2):543--602,
  2016.

\bibitem{ebin1987equations}
David~G Ebin.
\newblock The equations of motion of a perfect fluid with free boundary are not
  well posed.
\newblock {\em Communications in Partial Differential Equations},
  12(10):1175--1201, 1987.

\bibitem{germain2012global}
Pierre Germain, Nader Masmoudi, and Jalal Shatah.
\newblock Global solutions for the gravity water waves equation in dimension 3.
\newblock {\em Annals of Mathematics}, pages 691--754, 2012.

\bibitem{ginsberg2018lifespan}
Daniel Ginsberg.
\newblock On the lifespan of three-dimensional gravity water waves with
  vorticity.
\newblock {\em arXiv preprint arXiv:1812.01583}, 2018.

\bibitem{hasimoto1972nonlinear}
Hidenori Hasimoto and Hiroaki Ono.
\newblock Nonlinear modulation of gravity waves.
\newblock {\em Journal of the Physical Society of Japan}, 33(3):805--811, 1972.

\bibitem{HunterTataruIfrim1}
John~K Hunter, Mihaela Ifrim, and Daniel Tataru.
\newblock Two dimensional water waves in holomorphic coordinates.
\newblock {\em Communications in Mathematical Physics}, 346(2):483--552, 2016.

\bibitem{hur2020benjamin}
Vera~Mikyoung Hur.
\newblock The benjamin-feir instability in the infinite depth.
\newblock {\em arXiv preprint arXiv:2009.01593}, 2020.

\bibitem{HunterTataruIfrim2}
Mihaela Ifrim and Daniel Tataru.
\newblock Two dimensional water waves in holomorphic coordinates ii: global
  solutions.
\newblock {\em arXiv preprint arXiv:1404.7583}, 2014.

\bibitem{ifrim2015two}
Mihaela Ifrim and Daniel Tataru.
\newblock Two dimensional gravity water waves with constant vorticity: I. cubic
  lifespan.
\newblock {\em arXiv preprint arXiv:1510.07732}, 2015.

\bibitem{ifrim2018nls}
Mihaela Ifrim and Daniel Tataru.
\newblock The nls approximation for two dimensional deep gravity waves.
\newblock {\em Science China Mathematics}, 62(6):1101--1120, 2019.

\bibitem{iguchi1999free}
T~Iguchi, N~Tanaka, and A~Tani.
\newblock On a free boundary problem for an incompressible ideal fluid in two
  space dimensions.
\newblock {\em Adavances in Mathematical Sciences and Applications},
  9:415--472, 1999.

\bibitem{iguchi2001well}
Tatsuo Iguchi.
\newblock Well-posedness of the initial value problem for capillary-gravity
  waves.
\newblock {\em Funkcialaj Ekvacioj Serio Internacia}, 44(2):219--242, 2001.

\bibitem{Ionescu2015}
Alexandru~D. Ionescu and Fabio Pusateri.
\newblock Global solutions for the gravity water waves system in 2d.
\newblock {\em Inventiones mathematicae}, 199(3):653--804, 2015.

\bibitem{jin2019nonlinear}
Jiayin Jin, Shasha Liao, and Zhiwu Lin.
\newblock Nonlinear modulational instability of dispersive pde models.
\newblock {\em Archive for Rational Mechanics and Analysis}, 231(3):1487--1530,
  2019.

\bibitem{krasovskii1960theory}
Yu~P Krasovskii.
\newblock The theory of steady-state waves of large amplitude.
\newblock {\em SPhD}, 5:62, 1960.

\bibitem{krasovskii1962theory}
Yu~P Krasovskii.
\newblock On the theory of steady-state waves of finite amplitude.
\newblock {\em USSR Computational Mathematics and Mathematical Physics},
  1(4):996--1018, 1962.

\bibitem{lannes2005well}
David Lannes.
\newblock Well-posedness of the water-waves equations.
\newblock {\em Journal of the American Mathematical Society}, 18(3):605--654,
  2005.

\bibitem{Levi-Civita}
T.~Levi-Civita.
\newblock Determination rigoureuse des ondes permanentes d'ampleur finie.
\newblock {\em Math. Ann}, 93(1):264--314, 1925.

\bibitem{lighthill1970contributions}
MJ~Lighthill.
\newblock Contributions to the theory of waves in non-linear dispersive
  systems.
\newblock In {\em Hyperbolic Equations and Waves}, pages 173--210. Springer,
  1970.

\bibitem{lindblad2005well}
Hans Lindblad.
\newblock Well-posedness for the motion of an incompressible liquid with free
  surface boundary.
\newblock {\em Annals of mathematics}, pages 109--194, 2005.

\bibitem{miao2020well}
Shuang Miao, Sohrab Shahshahani, and Sijue Wu.
\newblock Well-posedness for free boundary hard phase fluids with minkowski
  background.
\newblock {\em arXiv preprint arXiv:2003.02987}, 2020.

\bibitem{MunozNLS}
Claudio Mu\~noz.
\newblock Instability in nonlinear \text{Schr\"odinger} breathers.
\newblock {\em Proyecciones}, 36(4):653--683, 2017.

\bibitem{Nalimov}
V.~I. Nalimov.
\newblock The cauchy-poisson problem (in russian).
\newblock {\em Dynamika Splosh}, Sredy(18):104--210, 1974.

\bibitem{nekrasov1921steady}
Aleksandr~I Nekrasov.
\newblock On steady waves.
\newblock {\em Izv. Ivanovo-Voznesensk. Politekhn. In-ta}, 3:52--65, 1921.

\bibitem{nguyen2020proof}
Huy~Q Nguyen and Walter~A Strauss.
\newblock Proof of modulational instability of stokes waves in deep water.
\newblock {\em arXiv preprint arXiv:2007.05018}, 2020.

\bibitem{ogawa2002free}
Masao Ogawa and Atusi Tani.
\newblock Free boundary problem for an incompressible ideal fluid with surface
  tension.
\newblock {\em Mathematical Models and Methods in Applied Sciences},
  12(12):1725--1740, 2002.

\bibitem{ogawa2003incompressible}
Masao Ogawa and Atusi Tani.
\newblock Incompressible perfect fluid motion with free boundary of finite
  depth.
\newblock {\em Advances in Mathematical Sciences and Applications},
  13(1):201--223, 2003.

\bibitem{ostrovskii1967propagation}
LA~Ostrovskii.
\newblock Propagation of wave packets and space-time self-focusing in a
  nonlinear medium.
\newblock {\em Sov. Phys. JETP}, 24(4):797--800, 1967.

\bibitem{Wu1}
Sijue~Wu Rafe~Kinsey.
\newblock A priori estimates for two-dimensional water waves with angled
  crests.
\newblock {\em preprint 2014, arXiv1406:7573}.

\bibitem{schneider2011justification}
Guido Schneider and C~Eugene Wayne.
\newblock Justification of the nls approximation for a quasilinear water wave
  model.
\newblock {\em Journal of differential equations}, 251(2):238--269, 2011.

\bibitem{shatah2006geometry}
Jalal Shatah and Chongchun Zeng.
\newblock Geometry and a priori estimates for free boundary problems of the
  euler's equation.
\newblock {\em arXiv preprint math/0608428}, 2006.

\bibitem{stokes1880theory}
George~G Stokes.
\newblock On the theory of oscillatory waves.
\newblock {\em Transactions of the Cambridge philosophical society}, 1880.

\bibitem{struik1926determination}
Dirk~J Struik.
\newblock D{\'e}termination rigoureuse des ondes irrotationelles
  p{\'e}riodiques dans un canal {\`a} profondeur finie.
\newblock {\em Mathematische Annalen}, 95(1):595--634, 1926.

\bibitem{su2020long}
Qingtang Su.
\newblock Long time behavior of 2d water waves with point vortices.
\newblock {\em Communications in Mathematical Physics}, pages 1--94, 2020.

\bibitem{su2020transition}
Qingtang Su.
\newblock On the transition of the rayleigh-taylor instability in 2d water
  waves.
\newblock {\em arXiv preprint arXiv:2007.13849}, 2020.

\bibitem{su2020partial}
Qingtang Su.
\newblock Partial justification of the peregrine soliton from the 2d full water
  waves.
\newblock {\em Archive for Rational Mechanics and Analysis}, pages 1--97, 2020.

\bibitem{taylor1950instability}
Geoffrey~Ingram Taylor.
\newblock The instability of liquid surfaces when accelerated in a direction
  perpendicular to their planes. i.
\newblock {\em Proc. R. Soc. Lond. A}, 201(1065):192--196, 1950.

\bibitem{taylor2007tools}
Michael~E Taylor.
\newblock {\em Tools for PDE: pseudodifferential operators, paradifferential
  operators, and layer potentials}.
\newblock Number~81. American Mathematical Soc., 2007.

\bibitem{Totz2015}
Nathan Totz.
\newblock A justification of the modulation approximation to the 3d full water
  wave problem.
\newblock {\em Communications in Mathematical Physics}, 335(1):369--443, 2015.

\bibitem{Totz2012}
Nathan Totz and Sijue Wu.
\newblock A rigorous justification of the modulation approximation to the 2d
  full water wave problem.
\newblock {\em Communications in Mathematical Physics}, 310(3):817--883, 2012.

\bibitem{verchota1984layer}
Gregory Verchota.
\newblock Layer potentials and regularity for the dirichlet problem for
  laplace's equation in lipschitz domains.
\newblock {\em Journal of Functional Analysis}, 59(3):572--611, 1984.

\bibitem{wang2018global}
Xuecheng Wang.
\newblock Global infinite energy solutions for the 2d gravity water waves
  system.
\newblock {\em Communications on Pure and Applied Mathematics}, 71(1):90--162,
  2018.

\bibitem{whitham1967non}
GB~Whitham.
\newblock Non-linear dispersion of water waves.
\newblock {\em Journal of Fluid Mechanics}, 27(2):399--412, 1967.

\bibitem{Wu3}
Sijue Wu.
\newblock A blow-up criteria and the existence of 2d gravity water waves with
  angled crests.
\newblock {\em preprint 2015, arXiv:1502.05342}.

\bibitem{Wu2}
Sijue Wu.
\newblock On a class of self-similar 2d surface water waves.
\newblock {\em preprint 2012, arXiv1206:2208}.

\bibitem{Wu1997}
Sijue Wu.
\newblock Well-posedness in sobolev spaces of the full water wave problem in
  2-d.
\newblock {\em Inventiones mathematicae}, 130(1):39--72, 1997.

\bibitem{Wu1999}
Sijue Wu.
\newblock Well-posedness in sobolev spaces of the full water wave problem in
  3-d.
\newblock {\em J. Amer. Math. Soc.}, 12:445--495, 1999.

\bibitem{Wu2009}
Sijue Wu.
\newblock Almost global wellposedness of the 2-d full water wave problem.
\newblock {\em Inventiones mathematicae}, 177(1):45, 2009.

\bibitem{Wu2011}
Sijue Wu.
\newblock Global wellposedness of the 3-d full water wave problem.
\newblock {\em Inventiones mathematicae}, 184(1):125--220, 2011.

\bibitem{wu2019wellposedness}
Sijue Wu.
\newblock Wellposedness of the 2d full water wave equation in a regime that
  allows for non-$c^{1}$ interfaces.
\newblock {\em Inventiones mathematicae}, 217(2):241--375, 2019.

\bibitem{wu2020quartic}
Sijue Wu.
\newblock The quartic integrability and long time existence of steep water
  waves in 2d.
\newblock {\em arXiv preprint arXiv:2010.09117}, 2020.

\bibitem{Yosihara}
H.~Yosihara.
\newblock Gravity waves on the free surface of an incompressible perfect fluid
  of finite depth.
\newblock {\em RIMS Kyoto}, 18:49--96, 1982.

\bibitem{zakharov1968stability}
Vladimir~E Zakharov.
\newblock Stability of periodic waves of finite amplitude on the surface of a
  deep fluid.
\newblock {\em Journal of Applied Mechanics and Technical Physics},
  9(2):190--194, 1968.

\bibitem{zakharov2009modulation}
Vladimir~E Zakharov and LA~Ostrovsky.
\newblock Modulation instability: the beginning.
\newblock {\em Physica D: Nonlinear Phenomena}, 238(5):540--548, 2009.

\bibitem{zhang2008free}
Ping Zhang and Zhifei Zhang.
\newblock On the free boundary problem of three-dimensional incompressible
  euler equations.
\newblock {\em Communications on Pure and Applied Mathematics}, 61(7):877--940,
  2008.

\bibitem{zheng2019long}
Fan Zheng.
\newblock Long-term regularity of 3d gravity water waves.
\newblock {\em arXiv preprint arXiv:1910.01912}, 2019.

\end{thebibliography}
	\bibliographystyle{plain}
\end{document}